\documentclass[11pt]{article}
\usepackage{amssymb,amscd,amsfonts,amsbsy,amsmath,amsthm,amssymb}
\usepackage{enumerate}
\usepackage{mathrsfs}
\usepackage{epsf,epsfig}

\usepackage{color}

\newcommand{\Bk}{\color{black}}

\newcommand{\dist}{\operatorname{dist}}

\newcommand{\ree}{\mathbb{R}^{n+1}}

\newcommand{\meanint}{{\int{\mkern-19mu}-}}
\newcommand{\dmeanint}{{\int{\mkern-16mu}-}}

\newtheorem{proposition}{Proposition}[section]
\newtheorem{theorem}[proposition]{Theorem}
\newtheorem{lemma}[proposition]{Lemma}
\newtheorem{corollary}[proposition]{Corollary}

\newtheorem{remark}[proposition]{Remark}

\newtheorem{definition}[proposition]{Definition}

\newtheorem{convention}[proposition]{Convention}

\DeclareMathOperator{\diam}{diam}

\begin{document}

\title{$L^p$-Square Function Estimates on Spaces of Homogeneous Type and on 
Uniformly Rectifiable Sets    
\footnotetext[1]{The work of the authors has been supported in part by the US NSF and
the Simons Foundation}
\footnotetext[2]{{\it{\rm 2010} Mathematics Subject Classification:} 
Primary: 28A75, 42B20;
Secondary: 28A78, 42B25, 42B30}
\footnotetext[3]{{\it Key words and phrases:} Square function, quasi-metric space,
space of homogeneous type, Ahlfors-David regularity, singular integral operators, 
area function, Carleson operator, $T(1)$ theorem for the square function, local $T(b)$
theorem for the square function, uniformly rectifiable sets, tent spaces, 
variable coefficient kernels}}

\author{
Steve Hofmann\\
University of Missouri, Columbia\\
hofmanns@missouri.edu\\
\and  
Dorina Mitrea\\
University of Missouri, Columbia\\
mitread@missouri.edu\\
\and  
Marius Mitrea\\
University of Missouri, Columbia\\
mitream@missouri.edu\\
\and  
Andrew J. Morris\\
University of Missouri, Columbia\\
morrisaj@missouri.edu
}

\date{\today}

\maketitle

\begin{abstract}
We establish square function estimates for integral operators on uniformly rectifiable sets by proving a local $T(b)$ theorem and applying it to show that such estimates are stable under the so-called big pieces functor. More generally, we consider integral operators associated with Ahlfors-David regular sets of arbitrary codimension in ambient quasi-metric spaces. The local $T(b)$ theorem is then used to establish an inductive scheme in which square function estimates on so-called big pieces of an Ahlfors-David regular set are proved to be sufficient for square function estimates to hold on the entire set. Extrapolation results for $L^p$ and Hardy space versions of these 
estimates are also established. Moreover, we prove square function estimates for 
integral operators associated with variable coefficient kernels, including  
the Schwartz kernels of pseudodifferential operators acting between vector bundles
on subdomains with uniformly rectifiable boundaries on manifolds.
\end{abstract}

\newpage
\tableofcontents
\newpage

\section{Introduction}\label{Sect:1}
\setcounter{equation}{0}

The purpose of this work is three-fold:  first, to develop the so-called 
``local $T(b)$ theory" for square functions in a very general context, in which 
we allow the ambient space to be of homogeneous type, and in which the ``boundary" 
of the domain is of arbitrary (positive integer) co-dimension;  second, to use a 
special case of this local $T(b)$ theory to establish boundedness, for a rather 
general class of square functions, on uniformly rectifiable sets of codimension one 
in Euclidean space; and third, to establish an extrapolation principle whereby  an 
$L^p$ (or even weak-type $L^p$) estimate for a square function, for {\it one} fixed 
$p$, yields a full range of $L^p$ bounds.  We shall describe these results in more 
detail below, but let us first recall some of the history of the development  of 
the theory of square functions.

Referring to the role square functions play in mathematics, E. Stein wrote in 1982
(cf. \cite{St82}) that 
``{\it [square] functions are of fundamental importance in analysis,
standing as they do at the crossing of three important roads many of us have
travelled by: complex function theory, the Fourier transform (or orthogonality
in its various guises), and real-variable methods.}"  In the standard 
setting of the unit disc ${\mathbb{D}}$ in the complex plane, the classical 
square function $Sf$ of some $f:{\mathbb{T}}\to{\mathbb{C}}$ 
(with ${\mathbb{T}}:=\partial{\mathbb{D}}$) is defined 
in terms of the Poisson integral $u_f(r,\omega)$ of $f$ in ${\mathbb{D}}$ 
(written in polar coordinates) by the formula
\begin{eqnarray}\label{D-FC45}
(Sf)(z):=\Bigl(\int_{(r,\omega)\in\Gamma(x)}
|(\nabla u_f)(r,\omega)|^2\,r\,dr\,d\omega\Bigr)^{1/2},\qquad z\in{\mathbb{T}},
\end{eqnarray}
where $\Gamma(z)$ stands for the Stolz domain 
$\{(r,\omega):\,|{\rm arg}(z)-\omega|<1-r<\tfrac{1}{2}\}$ in ${\mathbb{D}}$. 
Let $v$ denote the (normalized) complex conjugate of $u_f$ in ${\mathbb{D}}$.
Then, if the analytic function $F:=u_f+iv$ is one-to-one, the quantity $(Sf)(z)^2$ 
may be naturally interpreted as the area of the region
$F(\Gamma(z))\subseteq{\mathbb{C}}$ (recall that ${\rm det}(DF)=|\nabla u_f|^2$).
The operator \eqref{D-FC45} was first considered by Lusin and the observation 
just made justifies the original name for \eqref{D-FC45} as Lusin's area function
(or Lusin's area integral). A fundamental property of $S$, originally proved by 
complex methods (cf. \cite[Theorem~3, pp.\,1092-1093]{Cal65}, and \cite{FeSt72}
for real-variable methods) is that
\begin{eqnarray}\label{D-FC46}
\|Sf\|_{L^p({\mathbb{T}})}\approx\|f\|_{H^p({\mathbb{T}})}
\quad\mbox{ for }\,\,p\in(0,\infty),
\end{eqnarray}
which already contains the $H^p$-boundedness of the Hilbert transform. 
Indeed, if $F=u+iv$ is analytic then the Cauchy-Riemann equations entail 
$|\nabla u|=|\nabla v|$ and, hence, $S(u|_{\mathbb{T}})=S(v|_{\mathbb{T}})$. 
In spite of the technical, seemingly intricate nature of \eqref{D-FC45} and its 
generalizations to higher dimensions, such as
\begin{eqnarray}\label{D-FC47}
(Sf)(x):=\Bigl(\int_{|x-y|<t}|(\nabla u_f)(y,t)|^2\,t^{1-n}dydt\Bigr)^{1/2},
\qquad x\in{\mathbb{R}}^n:=\partial{\mathbb{R}}^{n+1}_{+},
\end{eqnarray}
a great deal was known by the 1960's about the information encoded into 
the size of $Sf$, measured in $L^p$, thanks to the pioneering work of 
D.L. Burkholder, A.P. Calder\'on, C. Fefferman, R.F. Gundy, N. Lusin, 
J. Marcinkiewicz, C. Segovia, M. Silverstein, E.M. Stein, and A. Zygmund, among others. 
See, e.g., \cite{BGS}, \cite{Cal50}, \cite{Cal65}, \cite{FeSt72}, \cite{Se69}, 
\cite{St70}, \cite{St82}, \cite{STEIN}, and the references therein.

Subsequent work by B. Dahlberg, E. Fabes, D. Jerison, C. Kenig and others,
starting in the late 1970's (cf. \cite{Dah80}, \cite{DJK}, \cite{Fa88}, \cite{Ke},
\cite{M-DAH}), has brought to prominence the relevance of square function estimates 
in the context of partial differential equations in non-smooth settings, whereas 
work by D. Jerison and C. Kenig \cite{JeKe82} in the 1980's as well as G. David 
and S. Semmes in the 1990's (cf. \cite{DaSe91}, \cite{DaSe93}) has lead to 
the realization that square function estimates are also intimately connected 
with the geometry of sets (especially geometric measure theoretic aspects). 
More recently, square function estimates have played an important role in the 
solution of the Kato problem in \cite{HMc}, \cite{HLMc}, \cite{AHLMcT}.

The operator $S$ defined in \eqref{D-FC45} is obviously non-linear but 
the estimate 
\begin{eqnarray}\label{D-Fa44}
\|Sf\|_{L^p}\leq C\|f\|_{H^p}
\end{eqnarray}
may be linearized by introducing a suitable (linear) vector-valued operator.
Specifically, set $\Gamma:=\{(z,t)\in{\mathbb{R}}^{n+1}_{+}:\,|z|<t\}$ and 
consider the Hilbert space 
\begin{eqnarray}\label{D-Fa45}
{\mathscr{H}}:=\Bigl\{h:\Gamma\to{\mathbb{C}}^n:\,\mbox{ $h$ is measurable and }\,
\|h\|_{\mathscr{H}}:=
\Bigl(\int_{\Gamma}|h(z,t)|^2t^{1-n}dtdz\Bigr)^{\frac{1}{2}}<\infty\Bigr\}.
\end{eqnarray}
Also, let $\widetilde{S}f:{\mathbb{R}}^n\to{\mathscr{H}}$ be defined by the formula 
\begin{eqnarray}\label{D-Fa46}
\Bigl((\widetilde{S}f)(x)\Bigr)(z,t):=(\nabla u_f)(x-z,t),\qquad
\forall\,x\in{\mathbb{R}}^n,\,\,\,\forall\,(z,t)\in\Gamma,
\end{eqnarray}
i.e., $\widetilde{S}$ is the integral operator (mapping scalar-valued functions
defined on ${\mathbb{R}}^n$ into ${\mathscr{H}}$-valued functions defined 
on ${\mathbb{R}}^n$), whose kernel
$k:{\mathbb{R}}^{n}\times{\mathbb{R}}^{n}\setminus{\rm diagonal}\to{\mathscr{H}}$,
which is of convolution type, is 
given by $(k(x,y))(z,t):=(\nabla P_t)(x-y-z)$, for all $x,y\in{\mathbb{R}}^n$, $x\not=y$, 
and $(z,t)\in\Gamma$, where $P_t(x)$ is the Poisson kernel in ${\mathbb{R}}^{n+1}_{+}$.
Then, if $L^p({\mathbb{R}}^n,{\mathscr{H}})$ stands for the B$\hat{\rm o}$chner 
space of ${\mathscr{H}}$-valued, $p$-th power integrable functions on 
${\mathbb{R}}^n$, it follows that 
\begin{eqnarray}\label{D-Fa47}
\|Sf\|_{L^p({\mathbb{R}}^n)}\leq C\|f\|_{H^p({\mathbb{R}}^n)}\,\Longleftrightarrow\,
\|\widetilde{S}f\|_{L^p({\mathbb{R}}^n,\,{\mathscr{H}})}\leq C\|f\|_{H^p({\mathbb{R}}^n)}.
\end{eqnarray}
The relevance of the linearization procedure described in 
\eqref{D-Fa45}-\eqref{D-Fa47} is that it highlights the basic role of the 
case $p=2$ in \eqref{D-Fa44}. This is because the operator $\widetilde{S}$ falls
within the scope of theory of Hilbert space-valued singular integral operators of 
Calder\'on-Zygmund type for which boundedness on $L^2$ automatically extrapolates 
to the entire scale $L^p$, for $1<p<\infty$ (the extension to the case when 
$p\leq 1$ makes use of other specific features of $\widetilde{S}$).

From the point of view of geometry, what makes the above reduction to the case 
$p=2$ work is the fact that the upper-half space has the property that 
$x+\Gamma\subseteq{\mathbb{R}}^{n+1}_{+}$ for every $x\in\partial{\mathbb{R}}^{n+1}_{+}$.
Such a cone property actually characterizes Lipschitz domains (cf. \cite{HMT}), 
in which scenario this is the point of view adopted in 
\cite[Theorem~4.11, p.\,73]{M-LNM}. 

Hence, $S$ may be eminently regarded as a singular integral operator with 
a Hilbert space-valued Calder\'on-Zygmund kernel and, as such, establishing 
the $L^2$ bound 
\begin{eqnarray}\label{Tfgg-77}
\|\widetilde{S}f\|_{L^2({\mathbb{R}}^n,\,{\mathscr{H}})}
\leq C\|f\|_{L^2({\mathbb{R}}^n)}
\end{eqnarray}
is of basic importance to jump-start the study of the operator $S$. 
Now, as is well-known (and easy to check; see, e.g., \cite[pp.\,27-28]{STEIN}), 
\eqref{Tfgg-77} follows from Fubini's and Plancherel's theorems. 
 
For the goals we have in mind in the present work, it is worth recalling
a quote from C. Fefferman's 1974 ICM address \cite{Feff74} where he writes that
``{\it When neither the Plancherel theorem nor Cotlar's lemma applies, 
$L^2$-boundedness of singular operators presents very hard problems, 
each of which must (so far) be dealt with on its own terms.}" 
For scalar singular integral operators, this situation  began to be remedied 
in 1984 with the advent of the $T(1)$-Theorem,
proved by G. David and J.-L. Journ\'e in \cite{DJ84}. This was initially done in the Euclidean setting, using Fourier analysis methods. It was subsequently generalized 
and refined in a number of directions, including the extension to spaces of 
homogeneous type by R. Coifman (unpublished, see the discussion in \cite{Ch}), 
and the $T(b)$ Theorems proved by A.\,McIntosh and Y.\,Meyer in \cite{McM85}, 
and by G.\,David, J.L.\,Journ\'e and S.\,Semmes in \cite{DJS}. The latter reference
also contains an extension to the class of singular-integral operators with 
matrix-valued kernels. The more general case of operator-valued kernels has been 
treated by Figiel \cite{F} and by T. Hyt\"onen and L. Weis  \cite{HyWe}, who prove 
$T(1)$ Theorems in the spirit of the original work in \cite{DJ84} for singular 
integrals associated with kernels taking values in Banach spaces satisfying 
the UMD property.   Analogous $T(b)$ theorems were obtained by H\"ytonen \cite{Hy} 
(in Euclidean space) and by H\"ytonen and Martikainen \cite{HyM1} (in a metric 
measure space). Yet in a different direction, initially motivated by applications 
to the theory of analytic capacity, $L^2$-boundedness criteria which are local 
in nature appeared in the work of M. Christ \cite{Christ}. Subsequently, Christ's 
local $T(b)$ theorem has been extended to the setting of non-doubling spaces by 
F. Nazarov, S. Treil and A. Volberg in \cite{NTV}. Further extensions of the local 
$T(b)$ theory for singular integrals appear in \cite{AHMTT}, \cite{AY},
\cite{AR} and \cite{HyM2}.

Much of the theory mentioned in the preceding paragraph has also been developed in the context of square functions, as opposed to singular integrals.
In the convolution setting discussed above, \eqref{Tfgg-77} follows immediately from Plancherel's theorem, but the latter tool fails \Bk
in the case when ${\mathbb{R}}^{n+1}_{+}$ is replaced by a domain whose 
geometry is rough (so that, e.g., the cone property is violated), 
and/or one considers a square-function operator whose integral kernel ${\theta}(x,y)$ 
is no longer of convolution type (as was the case for $\widetilde{S}$).
A case in point is offered by the square-function estimate of the type 
\begin{eqnarray}\label{TD-56R}
\int_0^\infty\|\Theta_t f\|_{L^2({\mathbb{R}}^n)}^2\frac{dt}{t}
\leq C\|f\|^2_{L^2({\mathbb{R}}^n)},
\end{eqnarray}
where
\begin{eqnarray}\label{TD-56R.2}
\bigl(\Theta_t f\bigr)(x):=\int_{{\mathbb{R}}^n}{\theta}_t(x,y)f(y)\,dy,\qquad
x\in{\mathbb{R}}^n,\qquad t>0,
\end{eqnarray}
with $\{{\theta}_t(\cdot,\cdot)\}_{t>0}$ a standard Littlewood-Paley family, i.e., 
satisfying for some exponent $\alpha>0$,
\begin{eqnarray}\label{TD-56R.3A}
&& |{\theta}_t(x,y)|\leq C\frac{t^\alpha}{(t+|x-y|)^{n+\alpha}}\quad\mbox{ and}
\\[4pt]
&& |{\theta}_t(x,y)-{\theta}_t(x,y')|\leq C\frac{|y-y'|^\alpha}{(t+|x-y|)^{n+\alpha}}
\,\,\mbox{ if }\,\,|y-y'|<t/2.
\label{TD-56R.3B}
\end{eqnarray}
Then, in general, linearizing estimate \eqref{TD-56R} in a manner similar
to \eqref{D-Fa47} yields an integral operator which is no longer of convolution type.
As such, Plancherel's theorem is no longer directly effective in dealing 
with \eqref{TD-56R} given that the task at hand is establishing the $L^2$-boundedness 
of a variable kernel (Hilbert-valued) singular integral operator. 
However, M. Christ and J.-L. Journ\'e have shown in \cite{CJ} 
(under the same size/regularity conditions in \eqref{TD-56R.3A}-\eqref{TD-56R.3B})
that the square function estimate \eqref{TD-56R} is valid if the following 
Carleson measure condition holds:
\begin{eqnarray}\label{TD-5ii}
\sup_{Q\subseteq{\mathbb{R}}^n}\Bigl(
\int_0^{\ell(Q)}\meanint_Q|(\Theta_t 1)(x)|^2\,\frac{dxdt}{t}\Bigr)<\infty,
\end{eqnarray}
where the supremum is taken over all cubes $Q$ in ${\mathbb{R}}^n$.
The latter result is also implicit in the work of Coifman and Meyer \cite{CM}.
Moreover, S. Semmes' has shown in \cite{Sem90}  that, in the above 
\Bk setting, \eqref{TD-56R} holds if there exists a para-accretive 
function $b$ such that \eqref{TD-5ii} holds with ``1" replaced by ``$b$". 
%
%

Refinements of Semmes' global $T(b)$ theorem for square functions,
in the spirit of M. Christ's local $T(b)$ theorem for singular integrals \cite{Christ}, 
have subsequently been established in \cite{Au}, \cite{Ho3}, \cite{Ho4}.
The local $T(b)$ theorem for square functions which constitutes the main result 
in \cite{Ho4} reads as follows. Suppose $\Theta_t$ is as in \eqref{TD-56R.2} with 
kernel satisfying \eqref{TD-56R.3A}-\eqref{TD-56R.3B} as well as 
\begin{eqnarray}\label{TD-56R.3C}
|{\theta}_t(x,y)-{\theta}_t(x',y)|\leq C\frac{|x-x'|^\alpha}{(t+|x-y|)^{n+\alpha}}
\,\,\mbox{ if }\,\,|x-x'|<t/2.
\end{eqnarray}
In addition, assume that there exists a constant
$C_o\in(0,\infty)$ along with an exponent $q\in(1,\infty)$ and a system $\{b_Q\}_Q$ 
of functions indexed by dyadic cubes $Q$ in ${\mathbb{R}}^n$, such that for 
each dyadic cube $Q\subseteq{\mathbb{R}}^n$ one has:

(i) $\int_{{\mathbb{R}}^n}|b_Q(x)|^q\,dx\leq C_o|Q|$;

(ii) $\frac{1}{C_o}|Q|\leq\Bigl|\int_{{\mathbb{R}}^n}b_Q(x)\,dx\Bigr|$;

(iii) $\int_Q\Bigl(\int_0^{\ell(Q)}|(\Theta_t b_Q)(x)|^2\,\frac{dt}{t}\Bigr)^{q/2}\,dx
\leq C_o|Q|$.

\noindent Then the square function estimate \eqref{TD-56R} holds. The case
$q=2$ of this theorem does not require \eqref{TD-56R.3C} (just regularity 
in the second variable, as in \eqref{TD-56R.3B})\footnote{In fact, even the 
case $q\neq 2$ does not require \eqref{TD-56R.3C}, if the vertical square 
function is replaced by a conical one; see \cite{G} for details.}, and was 
already implicit in the solution of the Kato problem in~\cite{HMc}, 
\cite{HLMc}, \cite{AHLMcT}. It was formulated explicitly in \cite{Au},
\cite{Ho3}. An extension of the result of \cite{Ho4} to the case 
that the half-space is replaced by $\mathbb{R}^{n+1}\setminus E$, where $E$ 
is a closed Ahlfors-David regular set (cf. Definition \ref{Rcc-TG34} below) 
of Hausdorff dimension $n$, appears in \cite{GM}.  The latter extension has 
been used to prove a result of {\it free boundary} type, in which higher 
integrability of the Poisson kernel, in the presence of certain natural 
background hypotheses, is shown to be equivalent to uniform rectifiability 
(cf. Definition \ref{Def-unif.rect} below) of the boundary \cite{HM}, 
\cite{HMU}. Further extensions of the result of \cite{Ho4}, to the case in 
which the kernel $\theta_t$ and pseudo-accretive system $b_Q$ may be 
matrix-valued (as in the setting of the Kato problem), and in which $\theta_t$ 
need no longer satisfy the pointwise size and regularity conditions 
\eqref{TD-56R.3A}-\eqref{TD-56R.3B}, will appear in the forthcoming Ph.D. 
thesis of A. Grau de la Herran \cite{G}. 

A primary motivation for us in the present work is the connection between 
square function bounds (or their localized versions in the form of
``Carleson measure estimates"), and a quantitative, scale invariant notion
of rectifiability.  This subject has been developed extensively by 
David and Semmes \cite{DaSe91}, \cite{DaSe93} (but with some key ideas already 
present in the work of P. Jones \cite{J}). Following \cite{DaSe91}, \cite{DaSe93}, 
we shall give in the sequel (cf. Definition~\ref{Def-unif.rect}), a precise definition 
of the property that a closed set $E$ is ``Uniformly Rectifiable" (UR), but for now 
let us merely mention that UR sets are the ones on which ``nice" singular integral
operators are bounded on $L^2$.  David and Semmes have shown that these sets may also
be characterized via certain square function estimates, or equivalently, via Carleson 
measure estimates.  For example, let $E\subset \mathbb{R}^{n+1}$ be a closed set of 
codimension one, which is  ($n$-dimensional) Ahlfors-David regular (ADR) 
(cf. Definition~\ref{Rcc-TG34} below).  Then $E$ is UR if and only if  we have 
the Carleson measure estimate 
\begin{eqnarray}\label{eq1.sf}
\sup_B r^{-n}\int_{B}
\bigl|\bigl(\nabla^2{\mathcal{S}}1\bigr)(x)\bigr|^2\,\dist(x,E)\,dx<\infty,
\end{eqnarray} 
where the supremum runs over all Euclidean balls $B:=B(z,r)\subseteq{\mathbb{R}}^{n+1}$,
with $r\leq\diam(E)$, and center $z\in E$, and where ${\mathcal{S}}f$ is the harmonic 
single layer potential of the function $f$, i.e.,
\begin{eqnarray}\label{eq1.layer}
{\mathcal{S}}f(x):=c_n\,\int_{E}|x-y|^{1-n}f(y)\,d{\mathcal{H}}^n(y),\qquad
x\in{\mathbb{R}}^{n+1}\setminus E.
\end{eqnarray}
Here ${\mathcal{H}}^n$ denotes $n$-dimensional Hausdorff measure. For an appropriate normalizing constant $c_n|x|^{1-n}$ is the usual fundamental solution for the 
Laplacian in ${\mathbb{R}}^{n+1}$. We refer the reader to \cite{DaSe93} for details, 
but see also Section~\ref{Sect:4} where we present some related results.  
We note that by ``$T1$" reasoning (cf. Section~\ref{Sect:3} below), \eqref{eq1.sf} 
is equivalent to the square function bound
\begin{eqnarray}\label{eq1.sf2}
\int_{{\mathbb{R}}^{n+1}\setminus E}\bigl|\bigl(\nabla^2{\mathcal{S}}f\bigr)(x)|^2\,
\dist(x,E)\,dx\leq C\int_E |f(x)|^2\,d{\mathcal{H}}^n(x)\,.
\end{eqnarray}
Using an idea of P. Jones \cite{J}, one may derive, for UR sets, a quantitative 
version of the fact that rectifiability may be characterized in terms of existence 
a.e. of approximate tangent planes. Again, a Carleson measure is used to
express matters quantitatively. For $x\in E$ and $t>0$ we set
\begin{eqnarray}\label{eq1.beta}
\beta_2(x,t):=\inf_P\left(\frac{1}{t^{n}}
\int_{B(x,t)\cap E}\left(\frac{\mbox{dist}(y,P)}{t}\right)^2
\,d{\mathcal{H}}^n(y)\right)^{1/2},
\end{eqnarray}
where the infimum runs over all $n$-planes $P$. Then a closed, ADR set $E$ of 
codimension one is UR if and only if the following Carleson measure estimate 
holds on $E\times{\mathbb{R}}_+$:
\begin{eqnarray}\label{eq1.UR}
\sup_{x_0\in E,\,r\,>\,0}r^{-n}\int_0^r\int_{B(x_0,t)\cap E}
\beta_2(x,t)^2\,d{\mathcal{H}}^n(x)\frac{dt}{t}\,<\,\infty.
\end{eqnarray}
See \cite{DaSe91} for details, and for a formulation in the case of higher codimension.
A related result, also obtained in \cite{DaSe91}, is that a set $E$ as above is UR 
if and only if, for every odd $\psi\in C^\infty_0(\ree)$, one has the following 
discrete square function bound
\begin{eqnarray}\label{eq1.sf3}
\sum_{k=-\infty}^\infty\int_{E}
\left|\int_E2^{-kn}\psi\bigl(2^{-k}(x-y)\bigr)f(y)\,d{\mathcal{H}}^n(y)\right|^2
\,d{\mathcal{H}}^n(x)\leq C_\psi\int_E|f(x)|^2\,d{\mathcal{H}}^n(x)\,.
\end{eqnarray}
Again, there is a Carleson measure analogue, and also a version for sets 
$E$ of higher codimension.

The following theorem collects some of the main results in our present work.  
It generalizes results described earlier in the introduction, which were valid in the 
codimension one case, and in which the ambient space ${\mathscr{X}}$ was Euclidean.
To state it, recall that (in a context to be made precise below) a measurable 
function $b:E\to{\mathbb{C}}$ is called para-accretive if it is essentially bounded 
and there exist constants $c,C\in(0,\infty)$ such that the following conditions are satisfied: 
\begin{eqnarray}\label{para-acc}
\forall\,Q\in{\mathbb{D}}(E)\quad\exists\,\widetilde{Q}\in{\mathbb{D}}(E)\quad
\mbox{such that }\,\,\,
\widetilde{Q}\subseteq Q,\quad\ell(\widetilde{Q})\geq c\ell(Q),\quad
\left|\meanint_{\widetilde{Q}} b\,d\sigma\right|\geq C.
\end{eqnarray}
Other relevant definitions will be given in the sequel. 

\begin{theorem}\label{M-TTHH}
Suppose that $({\mathscr{X}},\rho,\mu)$ is an $m$-dimensional {\rm ADR} space 
for some $m>0$ and fix a number $d\in(0,m)$. Also, let 
\begin{eqnarray}\label{K234-AXXX}
{\theta}:(\mathscr{X}\times\mathscr{X})
\setminus\{(x,x):\,x\in \mathscr{X}\}\longrightarrow{{\mathbb{R}}}
\end{eqnarray}
be a function which is Borel measurable with respect to the product topology 
$\tau_\rho\times\tau_\rho$, and which has the property that there exist finite positive 
constants $C_{\theta},\,\alpha,\,\upsilon$ such that for all $x,y\in\mathscr{X}$ 
with $x\neq y$ the following hold:
\begin{eqnarray}\label{hszz-AXXX}
&& \hskip -0.40in
|{\theta}(x,y)|\leq\frac{C_{\theta}}{\rho(x,y)^{d+\upsilon}},
\\[4pt]
&& \hskip -0.40in
|{\theta}(x,y)-{\theta}(x,\widetilde{y})|\leq C_{\theta} \frac{\rho(y,\widetilde{y})^\alpha}{\rho(x,y)^{d+\upsilon+\alpha}}, 
\quad\forall\,\widetilde{y}\in\mathscr{X}\setminus\{x\}\,\,\mbox{ with }\,\,
\rho(y,\widetilde{y})\leq\tfrac{1}{2}\rho(x,y).\quad
\label{hszz-3-AXXX}
\end{eqnarray}
Assume that $E$ is a closed subset of $({\mathscr{X}},\tau_\rho)$
and that $\sigma$ is a Borel regular measure 
on $(E,\tau_{\rho|_{E}})$ such that $\bigl(E,\rho\bigl|_E,\sigma\bigr)$ 
is a $d$-dimensional {\rm ADR} space, and define the integral operator 
$\Theta=\Theta_E$ for all functions $f\in L^p(E,\sigma)$, $1\leq p\leq\infty$, by
\begin{eqnarray}\label{operator-AXXX}
(\Theta f)(x):=\int_E {\theta}(x,y)f(y)\,d\sigma(y),
\qquad\forall\,x\in\mathscr{X}\setminus E.
\end{eqnarray}
Let ${\mathbb{D}}(E)$ denote a dyadic cube structure on $E$ and, for each
$Q\in{\mathbb{D}}(E)$, denote by $T_E(Q)$ the dyadic Carleson tent over $Q$.
Finally, let $\rho_{\#}$ be the regularized version of the quasi-distance $\rho$ 
as in Theorem~\ref{JjEGh} and, for each $x\in{\mathscr{X}}$, set 
$\delta_E(x):=\inf\{\rho_{\#}(x,y):\,y\in E\}$.

Then the following are equivalent:
\begin{enumerate}
\item[(1)] $[${\tt $L^2$\! square\! function\! estimate}$]$
There exists $C\in(0,\infty)$ with the property that 
for each $f\in L^2(E,\sigma)$ one has
\begin{eqnarray}\label{G-UFXXX.2}
\int_{\mathscr{X}\setminus E}|(\Theta f)(x)|^2\delta_E(x)^{2\upsilon-(m-d)}\,
d\mu(x)\leq C\int_E|f(x)|^2\,d\sigma(x).
\end{eqnarray}
\item[(2)] $[${\tt Carleson\! measure\! condition\! on\! dyadic\! tents\!
for\! $\Theta$\! tested\! on\! $1$}$]$
There holds 
\begin{eqnarray}\label{UEHgXXX}
\sup_{Q\in{\mathbb{D}}(E)}\left(\tfrac{1}{\sigma(Q)}\int_{T_E(Q)}
\bigl|\bigl(\Theta 1\bigr)(x)\bigr|^2\,
\delta_E(x)^{2\upsilon-(m-d)}\,d\mu(x)\right)<\infty.
\end{eqnarray}
\item[(3)] $[${\tt Carleson\! measure\! condition\! on\! dyadic\! tents\!
for\! $\Theta$\! acting\! on\! $L^\infty$}$]$
There exists a constant $C\in(0,\infty)$ with the property that
for each $f\in L^\infty(E,\sigma)$ 
\begin{eqnarray}\label{UEHgXXX.22}
\sup_{Q\in{\mathbb{D}}(E)}\left(\tfrac{1}{\sigma(Q)}\int_{T_E(Q)}
|(\Theta f)(x)|^2\delta_E(x)^{2\upsilon-(m-d)}\,d\mu(x)\right)^{1/2}
\leq C\|f\|_{L^\infty(E,\sigma)}.
\end{eqnarray}
\item[(4)] $[${\tt Carleson\! measure\! condition\! on\! balls\!
for\! $\Theta$\! tested\! on\! $1$}$]$
There holds 
\begin{eqnarray}\label{UEHgXXX.2}
\sup_{x\in E,\,r>0}
\left(\tfrac{1}{\sigma\bigl(E\cap B_{\rho_{\#}}(x,r)\bigr)}
\int_{B_{\rho_{\#}}(x,r)\setminus E}
|\Theta 1|^2\delta_E^{2\upsilon-(m-d)}\,d\mu\right)<\infty.
\end{eqnarray}
\item[(5)] $[${\tt Carleson\! measure\! condition\! for\! $\Theta$\! tested\! on\! a\! 
para-accretive\! function}$]$
There exists a para-accretive function $b:E\to{\mathbb{C}}$ with the 
property that 
\begin{eqnarray}\label{UEHgXXX.2PA}
\sup_{Q\in{\mathbb{D}}(E)}
\left(\tfrac{1}{\sigma(Q)}
\int_{T_E(Q)}|(\Theta b)(x)|^2\delta_E(x)^{2\upsilon-(m-d)}\,d\mu(x)\right)<\infty.
\end{eqnarray}
\item[(6)] $[${\tt Carleson\! measure\! condition\! on\! balls\! for\! 
$\Theta$\! acting\! on\! $L^\infty$}$]$
There exists $C\in(0,\infty)$ with the property that
for each $f\in L^\infty(E,\sigma)$ 
\begin{eqnarray}\label{UEHgXXX.2S}
\sup_{x\in E,\,r>0}
\left(\tfrac{1}{\sigma\bigl(E\cap B_{\rho_{\#}}(x,r)\bigr)}
\int_{B_{\rho_{\#}}(x,r)\setminus E}
|\Theta f|^2\delta_E^{2\upsilon-(m-d)}\,d\mu\right)^{1/2}
\leq C\|f\|_{L^\infty(E,\sigma)}.
\end{eqnarray}
\item[(7)] $[${\tt Local\! $T(b)$\! condition\! on dyadic\! cubes}$]$
There exist two finite constants $C_0\geq 1$ and $c_0\in(0,1]$, along with a 
collection $\{b_Q\}_{Q\in{\mathbb{D}}(E)}$ of $\sigma$-measurable 
functions $b_Q:E\rightarrow{\mathbb{C}}$ such that for each $Q\in{\mathbb{D}}(E)$ 
the following hold:
\begin{eqnarray}\label{CON-BB}
\begin{array}{c}
\displaystyle\int_E |b_Q|^2\,d\sigma\leq C_0\sigma(Q),
\\[16pt]
\left|\int_{\widetilde{Q}} b_Q\,d\sigma\right|\geq\frac{1}{C_0}\,\sigma(\widetilde{Q})\quad
\mbox{for some }\widetilde{Q}\subseteq Q,\,\,\,
\ell(\widetilde{Q})\geq c_0\ell(Q),
\\[16pt]
\displaystyle\int_{T_E(Q)}|(\Theta\,b_Q)(x)|^2
\delta_E(x)^{2\upsilon-(m-d)}\,d\mu(x)\leq C_0\sigma(Q).
\end{array}
\end{eqnarray}
\item[(8)] $[${\tt Local\! $T(b)$\! condition\! on\! surface\! balls}$]$
There exist $C_0\in[1,\infty)$ and, for each surface ball $\Delta=\Delta(x_o,r):=B_{\rho_{\#}}(x_o,r)\cap E$, where $x_o$ is a point in $E$ and $r$ 
is a finite number in $\bigl(0,{\rm diam}_\rho(E)\bigr]$, a $\sigma$-measurable function
$b_\Delta:E\rightarrow{\mathbb{C}}$ supported in $\Delta$, such that the following 
estimates hold:
\begin{eqnarray}\label{CON-BB.789}
\begin{array}{c}
\displaystyle\int_E |b_\Delta|^2\,d\sigma\leq C_0\sigma(\Delta),\quad
\left|\int_\Delta b_\Delta\,d\sigma\right|\geq\frac{1}{C_0}\,\sigma(\Delta),
\\[16pt]
\displaystyle\int_{B_{\rho_{\#}}(x_o,2C_\rho r)\setminus E}|(\Theta\,b_\Delta)(x)|^2
\delta_E(x)^{2\upsilon-(m-d)}\,d\mu(x)\leq C_0\sigma(\Delta).
\end{array}
\end{eqnarray}
\item[(9)] $[${\tt Big\! Pieces\! of\! Square\! Function\! Estimate}$]$
The set $E$ has {\rm BPSFE} relative to the kernel ${\theta}$ 
(cf. Definition~\ref{sjvs}).
\item[(10)] $[${\tt Iterated\! Big\! Pieces\! of\! Square\! Function\! 
Estimate}$]$
The set $E$ has ${\rm (BP)}^k${\rm SFE} relative to the kernel ${\theta}$
(cf. Definition~\ref{sjvs-DDD}) for some, or any, $k\in{\mathbb{N}}$.
\item[(11)] $[${\tt Weak-$L^p$\! square\! function\! estimate}$]$
There exist an exponent $p\in(0,\infty)$ and constants $C,\kappa\in(0,\infty)$ 
such that for every $f\in L^{p}(E,\sigma)$ 
\begin{eqnarray}\label{GvBhXXX}
\hskip -0.20in
\sup_{\lambda>0}\left[\lambda\cdot
\sigma\Bigl(\Bigl\{x\in E:\int_{\Gamma_{\kappa}(x)}|(\Theta f)(y)|^2\,
\frac{d\mu(y)}{\delta_E(y)^{m-2\upsilon}}>\lambda^{2}\Bigr\}\Bigr)^{1/p}\right]
\leq C\Bigl(\int_{E}|f|^p\,d\sigma\Bigr)^{1/p},
\end{eqnarray}
where $\Gamma_\kappa(x)$ stands for the nontangential approach region defined 
in \eqref{TLjb}.
\item[(12)] $[${\tt Hardy\! and\! $L^p$\! square\! function\! estimates}$]$
Set $\gamma:=\min\,\{\alpha,(\log_2C_\rho)^{-1}\}$. Then for each
$p\in\bigl(\frac{d}{d+\gamma},\infty\bigr)$ the operator $\Theta$ 
extends to the space $H^p(E,\rho|_{E},\sigma)$, defined as the Lebesgue space  
$L^p(E,\sigma)$ if $p\in(1,\infty)$, and the Coifman-Weiss Hardy  
space on the space of homogeneous type $(E,\rho|_{E},\sigma)$ if 
$p\in\bigl(\frac{d}{d+\gamma},1\bigr]$, and this extension satisfies
\begin{eqnarray}\label{CfrdXXX}
\left\|\Bigl(\int_{\Gamma_{\kappa}(x)}|(\Theta f)(y)|^2\,
\frac{d\mu(y)}{\delta_E(y)^{m-2\upsilon}}\Bigr)^{\frac{1}{2}}
\right\|_{L^p_x(E,\sigma)}\!\!\!\leq C\|f\|_{H^p(E,\rho|_{E},\sigma)},
\end{eqnarray}
for each function $f\in H^p(E,\rho|_{E},\sigma)$.
\item[(13)] $[${\tt Mixed-norm\! space\! estimate}$]$
For each $p\in\bigl(\frac{d}{d+\gamma},\infty\bigr)$, with 
$\gamma:=\min\,\{\alpha,(\log_2C_\rho)^{-1}\}$, and each $q\in(1,\infty)$, 
the operator 
\begin{eqnarray}\label{ki-DUDXXX}
\delta_E^{\upsilon-m/q}\Theta:H^p(E,\rho|_{E},\sigma)
\longrightarrow L^{(p,q)}({\mathscr{X}},E)
\end{eqnarray}
is well-defined, linear and bounded, where $L^{(p,q)}({\mathscr{X}},E)$ is the
mixed norm space defined in \eqref{Mixed-FF7}.
\end{enumerate}
\end{theorem}

A few comments pertaining to the nature and scope of Theorem~\ref{M-TTHH} are in order.

\vskip 0.08in
$\bullet$ Theorem~\ref{M-TTHH} makes the case that estimating the square function 
in $L^p$, along with other related issues considered above, may be regarded as 
``zeroth order calculus", since only integrability and quasi-metric geometry 
are involved, without recourse to differentiability (or vector space structures). 
In particular, our approach is devoid of any PDE results and techniques. Compared 
with works in the upper-half space ${\mathbb{R}}^n\times (0,\infty)$, or so-called 
generalized upper-half spaces $E\times (0,\infty)$ (cf., e.g., \cite{HMY} and 
the references therein), here we work in an ambient ${\mathscr{X}}$ with no 
distinguished ``vertical" direction. Moreover, the set $E$ is allowed to have 
arbitrary ADR co-dimension in the ambient ${\mathscr{X}}$. In this regard we also
wish to point out that Theorem~\ref{M-TTHH} permits the consideration of fractal
subsets of the Euclidean space (such as the case when $E$ is the von Koch's 
snowflake in ${\mathbb{R}}^2$, in which scenario $d=\frac{\ln 4}{\ln 3}$).

$\bullet$ Passing from $L^2$ estimates to $L^p$ estimates is no longer done via a 
linearization procedure (since the environment no longer permits it) and, instead,
we use tent space theory and exploit the connection between the Lusin and the Carleson 
operators on spaces of homogeneous type (thus generalizing work from \cite{CoMeSt} 
in the Euclidean setting). This reinforces the philosophy that the square-function 
is a singular integral operator at least in spirit (if not in the letter). 

$\bullet$ The various quantitative aspects of the claims in items 
{\it (1)}-{\it (11)} of Theorem~\ref{M-TTHH} are naturally related to one another. 
The reader is also alerted to the fact that similar results to those contained
in Theorem~\ref{M-TTHH} are proved in the body of the manuscript for a larger class
of kernels (satisfying less stringent conditions) than in the theorem above.
The specific way in which Theorem~\ref{M-TTHH} follows from these more general 
results is discussed in \S\,\ref{Sect:6}.

\vskip 0.08in

We now proceed to describe several consequences of Theorem~\ref{M-TTHH} for 
subsets $E$ of the Euclidean space. First we record the following square function 
estimate, which extends work from \cite{DaSe91}. 

\begin{theorem}\label{sfe-cor}
Suppose that $E$ is a closed subset of ${\mathbb{R}}^{n+1}$ which is $d$-dimensional 
{\rm ADR} for some $d\in\{1,\dots,n\}$ and denote by $\sigma$ the surface measure 
induced by the $d$-dimensional Hausdorff measure on $E$. Assume that $E$ has big 
pieces of Lipschitz images of subsets of ${\mathbb{R}}^d$, i.e., there 
exist $\varepsilon,M\in(0,\infty)$ so that for every $x\in E$ and every 
$R\in(0,\infty)$, there is a Lipschitz mapping $\varphi$ with Lipschitz 
norm $\leq M$ from the ball $B^d(0,R)$ in ${\mathbb{R}}^d$ into 
${\mathbb{R}}^{n+1}$ such that
\begin{eqnarray}\label{nvdz}
\sigma\left(E\cap B(x,R)\cap\varphi(B^d(0,R))\right)\leq \varepsilon R^d.
\end{eqnarray}
Suppose $\psi:{\mathbb{R}}^{n+1}\to{\mathbb{R}}$ 
is a compactly supported, smooth, odd function and for each $k\in{\mathbb{Z}}$ set 
$\psi_k(x):=2^{-kd}\psi\bigl(\frac{x}{2^k}\bigr)$ for $x\in{\mathbb{R}}^{n+1}$.
Then for every $q\in(1,\infty)$ and every $p\in\bigl(\frac{d}{d+1},\infty\bigr)$ 
there exists $C\in(0,\infty)$ such that
\begin{eqnarray}\label{bh}
\left\|\Bigl(\sum\limits_{k=-\infty}^{+\infty}
\meanint_{y\in\Delta(x,2^k)}\Bigl|\int_E\psi_k(z-y)f(z)\,d\sigma(z)
\Bigr|^q\,d\sigma(y)\Bigr)^{1/q}
\right\|_{L^p_x(E,\sigma)}\leq C\|f\|_{H^p(E,\sigma)}
\end{eqnarray}
for every $f\in H^p(E,\sigma)$, where $\Delta(x,2^k):=\{y\in E:\,|y-x|<2^k\}$
for each $x\in E$ and $k\in{\mathbb{Z}}$.
\end{theorem}

The particular case when $p=q=2$, in which scenario \eqref{bh} takes the form 
\begin{eqnarray}\label{bh.Biss}
\sum\limits_{k=-\infty}^{+\infty}\int_E\left|
\int_E\psi_k(x-y)f(y)\,d\sigma(y)\right|^2d\sigma(x)\leq C\int_E|f|^2\,d\sigma,
\end{eqnarray}
has been treated in \cite[\S\,3, p.\,21]{DaSe91}. The main point
of Theorem~\ref{sfe-cor} is that \eqref{bh.Biss} continues to hold, 
when formulated as in \eqref{bh} for every $p\in\bigl(\frac{d}{d+1},\infty\bigr)$.
The proof of this result, presented in the last part of \S\,\ref{Sect:6}, relies
on Theorem~\ref{M-TTHH} and uses the fact that no regularity condition 
on the kernel $\theta(x,y)$ is assumed in the variable $x$ (compare with 
\eqref{hszz-AXXX}-\eqref{hszz-3-AXXX}).

Next, we discuss another consequence of Theorem~\ref{M-TTHH} in the Euclidean setting
which gives an extension of results due to G. David and S. Semmes.

\begin{theorem}\label{UR-rest}
Suppose that $K$ is a real-valued function satisfying
\begin{eqnarray}\label{K-BIS}
\begin{array}{c}
K\in C^2({\mathbb{R}}^{n+1}\setminus\{0\}),\quad K\,\,\mbox{ is odd, and}
\\[4pt]
K(\lambda x)=\lambda^{-n}K(x)\,\,\mbox{ for all }\,\lambda>0,\,\,
x\in{\mathbb{R}}^{n+1}\setminus\{0\}.
\end{array}
\end{eqnarray}
Let $E$ be a closed subset of ${\mathbb{R}}^{n+1}$ which is $n$-dimensional {\rm ADR},
denote by $\sigma$ the surface measure induced by the $n$-dimensional Hausdorff measure 
on $E$, and define the integral operator ${\mathcal{T}}$ acting on functions
$f\in L^p(E,\sigma)$, $1\leq p\leq\infty$, by 
\begin{eqnarray}\label{T-BIS.2a}
{\mathcal{T}}f(x):=\int_E K(x-y)f(y)\,d{\sigma}(y),\qquad 
\forall\,x\in{\mathbb{R}}^{n+1}\setminus E.
\end{eqnarray}
Finally, let ${\mathbb{D}}(E)$ denote a dyadic cube structure on $E$ and, for each
$Q\in{\mathbb{D}}(E)$, denote by $T_E(Q)$ the dyadic Carleson tent over $Q$.

Then, if the set $E$ is actually uniformly rectifiable ({\rm UR}), in the sense of  Definition~\ref{Def-unif.rect}, conditions {\it (1)-(5)} below hold:
\begin{enumerate}
\item[(1)] $[${\tt $L^2$\! square\! function\! estimate}$]$
There exists $C\in(0,\infty)$ with the property that 
for each $f\in L^2(E,\sigma)$ one has
\begin{eqnarray}\label{SF-BIS}
\int\limits_{{\mathbb{R}}^{n+1}\setminus E}|(\nabla{\mathcal{T}}f)(x)|^2
\,{\rm dist}\,(x,E)\,dx\leq C\int_E|f(x)|^2\,d\sigma(x).
\end{eqnarray}
\item[(2)] $[${\tt Carleson\! measure\! condition\! on\! dyadic\! tents\!
for\! ${\mathcal{T}}$\! acting\! on\! $L^\infty$}$]$
There exists a constant $C\in(0,\infty)$ with the property that
for each $f\in L^\infty(E,\sigma)$ 
\begin{eqnarray}\label{SF-BIS-3}
\sup_{Q\in{\mathbb{D}}(E)}\left(\tfrac{1}{\sigma(Q)}\int_{T_E(Q)}
|(\nabla{\mathcal{T}} f)(x)|^2\,{\rm dist}\,(x,E)\,dx\right)^{1/2}
\leq C\|f\|_{L^\infty(E,\sigma)}.
\end{eqnarray}
In particular,
\begin{eqnarray}\label{SF-BIS-2}
\sup_{Q\in{\mathbb{D}}(E)}\left(\tfrac{1}{\sigma(Q)}\int_{T_E(Q)}
|(\nabla{\mathcal{T}}1)(x)|^2\,{\rm dist}\,(x,E)\,dx\right)<\infty.
\end{eqnarray}
\item[(3)] $[${\tt Carleson\! measure\! condition\! on\! balls\!
for\! ${\mathcal{T}}$\! acting\! on\! $L^\infty$}$]$
There exists a constant $C\in(0,\infty)$ with the property that
for each $f\in L^\infty(E,\sigma)$ 
\begin{eqnarray}\label{SF-BIS-6}
\sup_{x\in E,\,r>0}\left(\tfrac{1}{\sigma\bigl(E\cap B(x,r)\bigr)}
\int_{B(x,r)\setminus E}|(\nabla{\mathcal{T}} f)(y)|^2\,{\rm dist}\,(y,E)\,dy\right)^{1/2}
\leq C\|f\|_{L^\infty(E,\sigma)}.
\end{eqnarray}
In particular,
\begin{eqnarray}\label{SF-BIS-4}
\sup_{x\in E,\,r>0}\left(\tfrac{1}{\sigma\bigl(E\cap B(x,r)\bigr)}
\int_{B(x,r)\setminus E}|(\nabla{\mathcal{T}} 1)(y)|^2\,{\rm dist}\,(y,E)\,dy\right)<\infty.
\end{eqnarray}
\item[(4)] $[${\tt Hardy\! and\! $L^p$\! square\! function\! estimates}$]$
For each $p\in\bigl(\frac{n}{n+1},\infty\bigr)$ let $H^p(E,\sigma)$ stand for the Lebesgue scale $L^p(E,\sigma)$ if $p\in(1,\infty)$, and the Coifman-Weiss scale of 
Hardy spaces on the space of homogeneous type $(E,|\cdot-\cdot|,\sigma)$ if 
$p\in\bigl(\frac{n}{n+1},1\bigr]$. Then the operator ${\mathcal{T}}$ 
extends to the space $H^p(E,\sigma)$ and this extension satisfies
\begin{eqnarray}\label{SF-BIS-10}
\hskip -0.20in
\left\|\Bigl(\int_{\Gamma_{\kappa}(x)}|(\nabla{\mathcal{T}} f)(y)|^2\,
\frac{dy}{{\rm dist}\,(y,E)^{n-1}}\Bigr)^{\frac{1}{2}}
\right\|_{L^p_x(E,\sigma)}\!\!\!\leq C\|f\|_{H^p(E,\sigma)},
\quad\forall\,f\in H^p(E,\sigma).
\end{eqnarray}
\item[(5)] $[${\tt Mixed-norm\! space\! estimate}$]$
For each $p\in\bigl(\frac{n}{n+1},\infty\bigr)$ and each $q\in(1,\infty)$ the operator 
\begin{eqnarray}\label{SF-BIS-11}
{\rm dist}\,(\cdot,E)\,\nabla{\mathcal{T}}:
H^p(E,\sigma)\longrightarrow L^{(p,q)}({\mathbb{R}}^{n+1},E)
\end{eqnarray}
is well-defined, linear and bounded, where $L^{(p,q)}({\mathbb{R}}^{n+1},E)$ is the
mixed norm space defined in \eqref{Mixed-FF7} (corresponding here to 
${\mathscr{X}}:={\mathbb{R}}^{n+1}$ and $\rho:=|\cdot-\cdot|$).
\end{enumerate}
\end{theorem}

Theorem~\ref{M-TTHH} particularized to the setting of Theorem~\ref{UR-rest} 
gives that conditions {\it (1)}-{\it (5)} above are equivalent. The fact that 
{\it (1)} holds in the special case when ${\mathcal{T}}$ is associated as 
in \eqref{T-BIS.2a} with each of the kernels $K_j(x):=x_j/|x|^{n+1}$, $1\leq j\leq n+1$, 
is due to David and Semmes \cite{DaSe93}. The new result here is that 
{\it (1)} (hence also all of {\it (1)-(5)}) holds more generally 
for the entire class of kernels described in \eqref{K-BIS}.
We shall prove the latter fact in Corollary~\ref{cor:URimSFE} below.  
Compared with \cite{DaSe91}, the class of kernels \eqref{K-BIS} is not 
tied up to any particular partial differential operator (in the manner that
the kernels $K_j(x):=x_j/|x|^{n+1}$, $1\leq j\leq n+1$, are related to the Laplacian).
Moreover, in \S\,\ref{SSect:PDO} we establish a version of Theorem~\ref{UR-rest}
for variable coefficient kernels, which ultimately applies to integral operators
on domains on manifolds associated with the Schwartz kernels of certain 
classes of pseudodifferential operators acting between vector bundles. 

The condition that the set $E$ is UR in the context of Theorem~\ref{UR-rest} 
is optimal, as seen from the converse statement stated below. This result is 
closely interfaced with the characterization of uniform rectifiability, due David 
and Semmes, in terms of \eqref{eq1.sf}-\eqref{eq1.layer}. In keeping with 
these conditions, the formulation of our result involves the 
Riesz-transform operator ${\mathcal{R}}:=\nabla{\mathcal{S}}$. 

\begin{theorem}\label{UR-rest.BIS}
Let $E$ be a closed subset of ${\mathbb{R}}^{n+1}$ which is $n$-dimensional 
{\rm ADR}, denote by $\sigma$ the surface measure induced by the $n$-dimensional
Hausdorff measure on $E$, and define the vector-valued integral operator 
${\mathcal{R}}$ acting on functions $f\in L^p(E,\sigma)$, $1\leq p\leq\infty$, by 
\begin{eqnarray}\label{T-BIS.3a}
{\mathcal{R}}f(x):=\int_E \frac{x-y}{|x-y|^{n+1}}f(y)\,d{\sigma}(y),\qquad 
\forall\,x\in{\mathbb{R}}^{n+1}\setminus E.
\end{eqnarray}
As before, let ${\mathbb{D}}(E)$ denote a dyadic cube structure on $E$ and, for each
$Q\in{\mathbb{D}}(E)$, denote by $T_E(Q)$ the dyadic Carleson tent over $Q$.
In this setting, consider the following conditions:
\begin{enumerate}
\item[(1)] $[${\tt $L^2$\! square\! function\! estimate}$]$
There exists $C\in(0,\infty)$ with the property that 
for each $f\in L^2(E,\sigma)$ one has
\begin{eqnarray}\label{SF-BIS.BIS}
\int\limits_{{\mathbb{R}}^{n+1}\setminus E}|(\nabla{\mathcal{R}}f)(x)|^2
\,{\rm dist}\,(x,E)\,dx\leq C\int_E|f(x)|^2\,d\sigma(x).
\end{eqnarray}
\item[(2)] $[${\tt Carleson\! measure\! condition\! on\! dyadic\! tents\!
for\! ${\mathcal{R}}$\! tested\! on\! $1$}$]$
There holds 
\begin{eqnarray}\label{SF-BIS-2.BIS}
\sup_{Q\in{\mathbb{D}}(E)}\left(\tfrac{1}{\sigma(Q)}\int_{T_E(Q)}
|(\nabla{\mathcal{R}}1)(x)|^2\,{\rm dist}\,(x,E)\,dx\right)<\infty.
\end{eqnarray}
\item[(3)] $[${\tt Carleson\! measure\! condition\! on\! dyadic\! tents\!
for\! ${\mathcal{R}}$\! acting\! on\! $L^\infty$}$]$
There exists a constant $C\in(0,\infty)$ with the property that
for each $f\in L^\infty(E,\sigma)$ 
\begin{eqnarray}\label{SF-BIS-3.BIS}
\sup_{Q\in{\mathbb{D}}(E)}\left(\tfrac{1}{\sigma(Q)}\int_{T_E(Q)}
|(\nabla{\mathcal{R}}f)(x)|^2\,{\rm dist}\,(x,E)\,dx\right)^{1/2}
\leq C\|f\|_{L^\infty(E,\sigma)}.
\end{eqnarray}
\item[(4)] $[${\tt Carleson\! measure\! condition\! on\! balls\!
for\! ${\mathcal{R}}$\! tested\! on\! $1$}$]$
There holds 
\begin{eqnarray}\label{SF-BIS-4.BIS}
\sup_{x\in E,\,r>0}\left(\tfrac{1}{\sigma\bigl(E\cap B(x,r)\bigr)}
\int_{B(x,r)\setminus E}|(\nabla{\mathcal{R}}1)(y)|^2\,
{\rm dist}\,(y,E)\,dy\right)<\infty.
\end{eqnarray}
\item[(5)] $[${\tt Carleson\! measure\! condition\! 
for\! ${\mathcal{R}}$\! tested\! on\! a\! para-accretive\! function}$]$
There exists a para-accretive function $b:E\to{\mathbb{C}}$ with the 
property that 
\begin{eqnarray}\label{SF-BIS-5.BIS}
\sup_{Q\in{\mathbb{D}}(E)}\left(\tfrac{1}{\sigma(Q)}
\int_{T_E(Q)}|(\nabla{\mathcal{R}}b)(x)|^2\,{\rm dist}\,(x,E)\,dx\right)<\infty.
\end{eqnarray}
\item[(6)] $[${\tt Carleson\! measure\! condition\! on\! balls\!
for\! ${\mathcal{R}}$\! acting\! on\! $L^\infty$}$]$
There exists a constant $C\in(0,\infty)$ with the property that
for each $f\in L^\infty(E,\sigma)$ 
\begin{eqnarray}\label{SF-BIS-6.BIS}
\sup_{x\in E,\,r>0}\left(\tfrac{1}{\sigma\bigl(E\cap B(x,r)\bigr)}
\int_{B(x,r)\setminus E}|(\nabla{\mathcal{R}}f)(y)|^2\,{\rm dist}\,(y,E)\,dy\right)^{1/2}
\leq C\|f\|_{L^\infty(E,\sigma)}.
\end{eqnarray}
\item[(7)] $[${\tt Local\! $T(b)$\! condition\! on dyadic\! cubes}$]$
There exist finite constants $C_0\geq 1$, $c_0\in(0,1]$ as well as a 
collection $\{b_Q\}_{Q\in{\mathbb{D}}(E)}$ of $\sigma$-measurable 
functions $b_Q:E\rightarrow{\mathbb{C}}$ such that for each $Q\in{\mathbb{D}}(E)$ 
the following hold:
\begin{eqnarray}\label{SF-BIS-7.BIS}
\begin{array}{c}
\displaystyle\int_E |b_Q|^2\,d\sigma\leq C_0\sigma(Q),
\\[16pt]
\left|\int_{\widetilde{Q}} b_Q\,d\sigma\right|\geq\frac{1}{C_0}\,\sigma(\widetilde{Q})\quad
\mbox{for some }\widetilde{Q}\subseteq Q,\,\,\,
\ell(\widetilde{Q})\geq c_0\ell(Q),
\\[16pt]
\displaystyle\int_{T_E(Q)}|(\nabla{\mathcal{R}}b_Q)(x)|^2
\,{\rm dist}\,(x,E)\,dx\leq C_0\sigma(Q).
\end{array}
\end{eqnarray}
\item[(8)] $[${\tt Local\! $T(b)$\! condition\! on surface\! balls}$]$
There exist $C_0\in[1,\infty)$ and, for each surface ball $\Delta=\Delta(x_o,r):=B(x_o,r)\cap E$, where $x_o$ is a point in $E$ and $r$ 
is a finite number in $\bigl(0,{\rm diam}(E)\bigr]$, a $\sigma$-measurable function
$b_\Delta:E\rightarrow{\mathbb{C}}$ supported in $\Delta$, such that the following 
estimates hold:
\begin{eqnarray}\label{SF-BIS-8.BIS}
\begin{array}{c}
\displaystyle\int_E |b_\Delta|^2\,d\sigma\leq C_0\sigma(\Delta),\quad
\left|\int_\Delta b_\Delta\,d\sigma\right|\geq\frac{1}{C_0}\,\sigma(\Delta),
\\[16pt]
\displaystyle\int_{B(x_o,4r)\setminus E}|(\nabla{\mathcal{R}}\,b_\Delta)(x)|^2
\,{\rm dist}\,(x,E)\,dx\leq C_0\sigma(\Delta).
\end{array}
\end{eqnarray}
\item[(9)] $[${\tt Weak-$L^p$\! square\! function\! estimate}$]$
There exist an index $p\in(0,\infty)$ and constants $C,\kappa\in(0,\infty)$ such that 
for every $f\in L^{p}(E,\sigma)$ 
\begin{eqnarray}\label{SF-BIS-9.BIS}
\hskip -0.30in
\sup_{\lambda>0}\left[\lambda\cdot
\sigma\Bigl(\Bigl\{x\in E:\int_{\Gamma_{\kappa}(x)}
\frac{|(\nabla{\mathcal{R}}f)(y)|^2}
{{\rm dist}\,(y,E)^{n-1}}\,dy>\lambda^{2}\Bigr\}\Bigr)^{1/p}\right]
\leq C\Bigl(\int_{E}|f|^p\,d\sigma\Bigr)^{1/p},
\end{eqnarray}
where $\Gamma_\kappa(x):=\bigl\{y\in{\mathbb{R}}^{n+1}\setminus E:\,
|x-y|<(1+\kappa)\,{\rm dist}(y,E)\bigr\}$ for each $x\in E$.
\end{enumerate}

Then if any of properties {\it (1)-(9)} holds it follows that $E$ is a {\rm UR} set.
\end{theorem}

The fact that condition {\it (1)} above implies that $E$ is a UR set has been 
proved by David and Semmes (see \cite[pp.\,252-267]{DaSe93}). Based on this result,
that {\it (2)}-{\it (3)} also imply that $E$ is a UR set then follows with the help 
of Theorem~\ref{M-TTHH} upon observing that the components of ${\mathcal{R}}$ 
are operators ${\mathcal{T}}$ as in \eqref{T-BIS.2a} associated with the 
kernels $K_j(x):=x_j/|x|^{n+1}$, $j\in\{1,...,n+1\}$, which satisfy 
\eqref{hszz-AXXX}-\eqref{hszz-3-AXXX}. Compared to David and Semmes' result
mentioned above (to the effect that the $L^2$ square function for the 
operators associated with the kernels $K_j$, $1\leq j\leq n+1$, implies
that the set $E$ is UR), a remarkable corollary of Theorem~\ref{UR-rest.BIS}
is that a mere weak-$L^2$ square function estimate for the operators  
associated with the kernels $K_j(x):=x_j/|x|^{n+1}$, $j\in\{1,...,n+1\}$,
as in \eqref{T-BIS.2a} implies that $E$ is a UR set.

Throughout the manuscript, we adopt the following conventions. The letter $C$ 
represents a finite positive constant that may change from one line to the next. 
The infinity symbol $\infty:=+\infty$. The set of positive integers is denoted 
by $\mathbb{N}$, and the set $\mathbb{N}_0:=\mathbb{N}\cup\{0\}$.

\section{Analysis and Geometry on Quasi-Metric Spaces}
\setcounter{equation}{0}
\label{Sect:2}

This section contains preliminary material, organized into four subsections dealing, 
respectively, with: a metrization result for arbitrary quasi-metric spaces, 
geometrically doubling quasi-metric spaces, approximations to the identity, and 
a discussion of the nature of Carleson tents in quasi-metric spaces.

\subsection{A metrization result for general quasi-metric spaces}
\label{SSect:2.1}

Here the goal is to review a sharp quantitative metrization result for 
quasi-metric spaces (cf. Theorem~\ref{JjEGh}), and record some useful 
properties of the Hausdorff outer-measure on quasi-metric spaces 
(cf. Proposition~\ref{PWRS22}). We begin, however, by introducing 
basic terminology and notation in the definition below. 

\begin{definition}\label{zbnb}
Assume that ${\mathscr{X}}$ is a set of cardinality at least two:
\begin{enumerate}
\item[(1)] A function $\rho:{\mathscr{X}}\times{\mathscr{X}}\to[0,\infty)$
is called a {\tt quasi-distance} on ${\mathscr{X}}$ provided there exist 
two constants $C_0,C_1\in[1,\infty)$ with the property that for every $x,y,z\in X$, 
one has 
\begin{eqnarray}\label{gabn-T.2}
\rho(x,y)=0\Leftrightarrow x=y,\quad
\rho(y,x)\leq C_0\rho(x,y),\quad
\rho(x,y)\leq C_1\max\{\rho(x,z),\rho(z,y)\}.
\end{eqnarray}
\item[(2)] Denote by ${\mathfrak{Q}}({\mathscr{X}})$ the collection of 
all quasi-distances on ${\mathscr{X}}$, and call a pair $({\mathscr{X}},\rho)$ 
a {\tt quasi-metric\! space} provided $\rho\in{\mathfrak{Q}}({\mathscr{X}})$.
Also, given $\rho\in{\mathfrak{Q}}({\mathscr{X}})$ and $E\subseteq{\mathscr{X}}$ 
of cardinality at lest two, denote by $\rho\bigl|_{E}\in{\mathfrak{Q}}(E)$ 
the restriction of the function $\rho$ to $E\times E$.
\item[(3)] For each $\rho\in{\mathfrak{Q}}({\mathscr{X}})$, define $C_\rho$ to be 
the smallest constant which can play the role of $C_1$ in the last inequality in \eqref{gabn-T.2}, i.e.,  
\begin{eqnarray}\label{C-RHO.111}
C_\rho:=\sup_{\stackrel{x,y,z\in{\mathscr{X}}}{\mbox{\tiny{not all equal}}}}
\frac{\rho(x,y)}{\max\{\rho(x,z),\rho(z,y)\}}\in[1,\infty),
\end{eqnarray}
and define $\widetilde{C}_\rho$ to be the smallest constant
which can play the role of $C_0$ in the first inequality in \eqref{gabn-T.2}, i.e.,  
\begin{eqnarray}\label{C-RHO.111XXX}
\widetilde{C}_\rho:=\sup_{\stackrel{x,y\in{\mathscr{X}}}{x\not=y}}
\frac{\rho(y,x)}{\rho(x,y)}\,\in[1,\infty).
\end{eqnarray}
\item[(4)] Given $\rho\in{\mathfrak{Q}}({\mathscr{X}})$, define the $\rho$-{\tt ball}
(or, simply {\tt ball} if the quasi-distance $\rho$ is clear from the context) 
centered at $x\in{\mathscr{X}}$ with radius $r\in(0,\infty)$ to be
\begin{eqnarray}\label{hdc-587}
B_\rho(x,r):=\left\{y\in{\mathscr{X}}:\rho(x,y)<r\right\}.
\end{eqnarray}
Also, call $E\subseteq{\mathscr{X}}$ $\rho$-{\tt bounded} if $E$ is contained in a $\rho$-ball, and define its $\rho$-{\tt diameter} (or, simply, {\tt diameter}) as 
\begin{eqnarray}\label{DIA-TT}
{\rm diam}_{\rho}(E):=\sup\,\bigl\{\rho(x,y):\,x,y\in E\bigr\}.
\end{eqnarray}
The $\rho$-{\tt distance} (or, simply {\tt distance}) between two arbitrary, 
nonempty sets $E,F\subseteq{\mathscr{X}}$ is naturally defined as
\begin{eqnarray}\label{HG-DD.6}
{\rm dist}_\rho(E,F):=\inf\,\{\rho(x,y):\,x\in E,\,\,y\in F\}.
\end{eqnarray}
If $E=\{x\}$ for some $x\in{\mathscr{X}}$ and $F\subseteq{\mathscr{X}}$,
abbreviate ${\rm dist}_\rho(x,F):={\rm dist}_\rho(\{x\},F)$.
\item[(5)] Given a quasi-distance $\rho\in{\mathfrak{Q}}({\mathscr{X}})$ define $\tau_{\rho}$, the {\tt topology canonically induced by $\rho$ on} ${\mathscr{X}}$,
to be the largest topology on ${\mathscr{X}}$ with the property that for each 
point $x\in{\mathscr{X}}$ the family $\{B_\rho(x,r)\}_{r>0}$ 
is a fundamental system of neighborhoods of $x$. 
\item[(6)] Call two functions
$\rho_1,\rho_2:{\mathscr{X}}\times{\mathscr{X}}\to[0,\infty)$ {\tt equivalent}, 
and write $\rho_1\approx\rho_2$, if there exist $C',C''\in(0,\infty)$ with the 
property that 
\begin{eqnarray}\label{JH-7pp}
C'\rho_1\leq\rho_2\leq C''\rho_1\,\quad\mbox{on }\,\,{\mathscr{X}}\times{\mathscr{X}}.
\end{eqnarray}
\end{enumerate}
\end{definition}

\noindent A few comments are in order. Suppose that $({\mathscr{X}},\rho)$ is 
a quasi-metric space. It is then clear that if 
$\rho':{\mathscr{X}}\times{\mathscr{X}}\rightarrow[0,\infty)$ is such that
$\rho'\approx\rho$ then $\rho'\in{\mathfrak{Q}}({\mathscr{X}})$ and 
$\tau_{\rho'}=\tau_\rho$. Also, it may be checked that
\begin{eqnarray}\label{topoQMS}
\mathcal{O}\in\tau_\rho\,\Longleftrightarrow\,
\mathcal{O}\subseteq{\mathscr{X}}\,\,\mbox{ and }\,\,\forall\,x\in\mathcal{O}\,\,
\,\,\exists\,r>0\,\mbox{ such that }\,B_\rho(x,r)\subseteq\mathcal{O}.
\end{eqnarray}

As is well-known, the topology induced by the given quasi-distance on a 
quasi-metric space is metrizable. Below we shall review a recent result proved in  
\cite{MMMM-G} which is an optimal quantitative version of this fact, and which 
sharpens earlier work from \cite{MaSe79}. To facilitate the 
subsequent discussion we first make a definition. Assume that ${\mathscr{X}}$ is an 
arbitrary, nonempty set. Given an arbitrary function 
$\rho:{\mathscr{X}}\times{\mathscr{X}}\to[0,\infty)$ define its 
{\tt symmetrization} $\rho_{sym}$ as 
\begin{eqnarray}\label{sfgu.GG}
\rho_{sym}:{\mathscr{X}}\times{\mathscr{X}}\longrightarrow[0,\infty),\qquad
\rho_{sym}(x,y):=\max\,\{\rho(x,y),\rho(y,x)\},\quad\forall\,x,y\in{\mathscr{X}}.
\end{eqnarray}
Then $\rho_{sym}$ is symmetric, i.e., $\rho_{sym}(x,y)=\rho_{sym}(y,x)$ for every 
$x,y\in{\mathscr{X}}$, and $\rho_{sym}\geq\rho$ on ${\mathscr{X}}\times{\mathscr{X}}$. 
In fact, $\rho_{sym}$ is the smallest nonnegative function defined on 
${\mathscr{X}}\times{\mathscr{X}}$ which is symmetric and pointwise $\geq\rho$. Furthermore, if $({\mathscr{X}},\rho)$ is a quasi-metric space then 
\begin{eqnarray}\label{yabv.GG}
\rho_{sym}\in{\mathfrak{Q}}({\mathscr{X}}),\,\,\,\,C_{\rho_{sym}}\leq C_\rho,
\,\,\,\,\widetilde{C}_{\rho_{sym}}=1,
\,\,\mbox{ and }\,\,\rho\leq\rho_{sym}\leq\,\widetilde{C}_\rho\,\rho. 
\end{eqnarray}

Here is the quantitative metrization theorem from \cite{MMMM-G} alluded to above.

\begin{theorem}\label{JjEGh}
Let $({\mathscr{X}},\rho)$ be a quasi-metric space and assume that 
$C_\rho,\widetilde{C}_\rho\in[1,\infty)$ are as in
\eqref{C-RHO.111}-\eqref{C-RHO.111XXX}. Introduce 
\begin{eqnarray}\label{Cro}
\alpha_\rho:=\frac{1}{\log_2 C_\rho}\in(0,\infty],
\end{eqnarray}
and define the regularization $\rho_{\#}:{\mathscr{X}}\times{\mathscr{X}}\to[0,\infty)$ 
of $\rho$ as follows. When $\alpha_\rho<\infty$, for each $x,y\in{\mathscr{X}}$ set
\begin{eqnarray}\label{R-sharp}
&& \hskip -0.60in
\rho_{\#}(x,y):=\inf\,\Bigl\{\Bigl(\sum\limits_{i=1}^N
\rho_{sym}(\xi_i,\xi_{i+1})^{\alpha_\rho}\Bigr)^{\frac{1}{\alpha_\rho}}:\,
N\in{\mathbb{N}}\mbox{ and }\xi_1,\dots,\,\xi_{N+1}\in{\mathscr{X}}, 
\nonumber
\\[4pt]
&& \hskip 1.00in
\mbox{(not necessarily distinct) such that $\xi_1=x$ and $\xi_{N+1}=y$}\Bigr\},
\end{eqnarray}
while if $\alpha_\rho=\infty$ then for each $x,y\in{\mathscr{X}}$ set
\begin{eqnarray}\label{BV+uuu}
&& \hskip -0.60in
\rho_{\#}(x,y):=\inf\,\Bigl\{\max\limits_{1\leq i\leq N}
\rho_{sym}(\xi_i,\xi_{i+1}):\,N\in{\mathbb{N}}\mbox{ and }\,
\xi_1,\dots,\,\xi_{N+1}\in{\mathscr{X}}, 
\\[4pt]
&& \hskip 1.00in
\mbox{(not necessarily distinct) such that $\xi_1=x$ and $\xi_{N+1}=y$}\Bigr\}.
\nonumber
\end{eqnarray}
Then the following properties hold:
\begin{enumerate}
\item[(1)] The function $\rho_{\#}$ is a symmetric quasi-distance on ${\mathscr{X}}$ 
and $\rho_{\#}\approx\rho$. More specifically,
\begin{eqnarray}\label{DEQV1}
(C_\rho)^{-2}\rho(x,y)\leq\rho_{\#}(x,y)\leq\widetilde{C}_\rho\,\rho(x,y),\quad
\forall\,x,y\in{\mathscr{X}}.
\end{eqnarray}
In particular, $\tau_{\rho_{\#}}=\tau_{\rho}$. Also, $C_{\rho_{\#}}\leq C_\rho$.
Furthermore, for any nonempty set $E$ of ${\mathscr{X}}$, there holds 
\begin{eqnarray}\label{RHo-evv}
(\rho|_{E})_{\#}\approx\rho|_{E}\approx(\rho_{\#})\bigl|_{E}.
\end{eqnarray}
\item[(2)] For each finite number $\beta\in(0,\alpha_{\rho}]$, the function 
\begin{eqnarray}\label{Ubhb-657.GG}
d_{\rho,\beta}:{\mathscr{X}}\times{\mathscr{X}}\to [0,\infty),\qquad
d_{\rho,\beta}(x,y):=\bigl[\rho_{\#}(x,y)\bigr]^\beta,
\qquad\forall\,x,y\in{\mathscr{X}}, 
\end{eqnarray}
is a distance on ${\mathscr{X}}$, and has the property that 
$(d_{\rho,\beta})^{1/\beta}\approx\rho$. In particular, $d_{\rho,\beta}$ induces 
the same topology on ${\mathscr{X}}$ as $\rho$, hence $\tau_\rho$ is metrizable. 
\item[(3)] For each finite number $\beta\in(0,\alpha_{\rho}]$, the function $\rho_{\#}$ 
satisfies the following H\"older-type regularity condition of order $\beta$
(in both variables, simultaneously) on ${\mathscr{X}}\times{\mathscr{X}}$:
\begin{eqnarray}\label{NNN-g8-tt}
&& \hskip -0.50in
\bigl|\rho_{\#}(x,y)-\rho_{\#}(z,w)\bigr|
\\[4pt]
&& \hskip 0.50in
\leq{\textstyle\frac{1}{\beta}}\,
\max\,\bigl\{\rho_{\#}(x,y)^{1-\beta},\rho_{\#}(z,w)^{1-\beta}\bigr\}
\bigl(\bigl[\rho_{\#}(x,z)\bigr]^\beta+\bigl[\rho_{\#}(y,w)\bigr]^\beta\bigr)
\nonumber
\end{eqnarray}
whenever $x,y,z,w\in{\mathscr{X}}$ (with the understanding that when $\beta\geq 1$ 
one also imposes the conditions $x\not=y$ and $z\not=w$). In particular, 
\begin{eqnarray}\label{CC-NNu}
\rho_{\#}:{\mathscr{X}}\times{\mathscr{X}}\longrightarrow [0,\infty)
\quad\,\mbox{is continuous},
\end{eqnarray}
when ${\mathscr{X}}\times{\mathscr{X}}$ is equipped with the natural product 
topology $\tau_\rho\times\tau_\rho$. 
\item[(4)] If $E$ is a nonempty subset of $({\mathscr{X}},\tau_{\rho})$, then 
the regularized distance function 
\begin{eqnarray}\label{REG-DDD}
\delta_E:={\rm dist}_{\rho_{\#}}(\cdot,E):{\mathscr{X}}\longrightarrow[0,\infty)
\end{eqnarray}
is equivalent to ${\rm dist}_{\rho}(\cdot,E)$. Furthermore, $\delta_E$ is 
locally H\"older of order $\beta$ on ${\mathscr{X}}$ for every 
$\beta\in(0,\min\,\{1,\alpha_\rho\}]$, in the sense that there exists 
$C\in(0,\infty)$ which depends only on $C_\rho,\widetilde{C}_\rho$ and $\beta$ 
such that 
\begin{eqnarray}\label{TBn-3}
\frac{|\delta_E(x)-\delta_E(y)|}{\rho(x,y)^\beta}\leq C\,
\Bigl(\rho(x,y)+\max\,\{{\rm dist}_\rho(x,E)\,,\,{\rm dist}_\rho(y,E)\}\Bigr)^{1-\beta}
\end{eqnarray}
for all $x,y\in{\mathscr{X}}$ with $x\not=y$. In particular,
\begin{eqnarray}\label{DEDE-CC}
\delta_E:({\mathscr{X}},\tau_\rho)\longrightarrow [0,\infty)
\quad\,\mbox{is continuous}.
\end{eqnarray}
\end{enumerate}
\end{theorem}

\noindent The key feature of the result discussed in Theorem~\ref{JjEGh} 
is the fact that if $({\mathscr{X}},\rho)$ is any quasi-metric space then 
$\rho^\beta$ is equivalent to a genuine distance on ${\mathscr{X}}$ for 
any finite number $\beta\in(0,(\log_2C_\rho)^{-1}]$. This result is 
sharp and improves upon an earlier version due to R.A.~Mac\'{\i}as and 
C. Segovia \cite{MaSe79}, in which these authors have identified a smaller, 
non-optimal upper-bound for the exponent $\beta$. 

In anticipation of briefly reviewing the notion of Hausdorff outer-measure
on a quasi-metric space, we recall a couple of definitions from measure theory.
Specifically, given an outer-measure $\mu^{\ast}$ on an arbitrary set 
${\mathscr{X}}$, consider the collection of all $\mu^{\ast}$-measurable 
sets defined as 
\begin{eqnarray}\label{MM-RE.89}
{\mathfrak{M}}_{\mu^{\ast}}:=\{A\subseteq{\mathscr{X}}:\,
\mu^{\ast}(Y)=\mu^{\ast}(Y\cap A)+\mu^{\ast}(Y\setminus A),\,\,
\forall\,Y\subseteq{\mathscr{X}}\}.
\end{eqnarray}
Carath\'eodory's classical theorem allows one to pass from a given 
outer-measure $\mu^{\ast}$ on ${\mathscr{X}}$ to a genuine measure by observing that
\begin{eqnarray}\label{MM-RE.90}
{\mathfrak{M}}_{\mu^{\ast}}\mbox{ is a sigma-algebra, and 
$\mu^{\ast}\bigl|_{{\mathfrak{M}}_{\mu^{\ast}}}$ is a complete measure}.
\end{eqnarray}
The restriction of an outer-measure $\mu^\ast$ on ${\mathscr{X}}$ to a subset 
$E$ of ${\mathscr{X}}$, denoted by $\mu^\ast\lfloor E$, is defined naturally 
by restricting the function $\mu^\ast$ to the collection of all subsets of $E$. 
We shall use the same symbol, $\lfloor\,\,$, in denoting the restriction of a measure  
to a measurable set. In this regard, it is useful to know when the measure 
associated with the restriction of an outer-measure to a set coincides with 
the restriction to that set of the measure associated with the given outer-measure.
Specifically, it may be checked that if $\mu^\ast$ is an outer-measure 
on ${\mathscr{X}}$, then 
\begin{eqnarray}\label{hgdss-F22}
\bigl(\mu^\ast\lfloor E\bigr)\Bigl|_{{\mathfrak{M}}_{(\mu^\ast\lfloor E)}}
=\bigl(\mu^\ast\bigl|_{{\mathfrak{M}}_{\mu^\ast}}\bigr)\lfloor E,\qquad
\forall\,E\in{\mathfrak{M}}_{\mu^\ast}. 
\end{eqnarray}

Next, if $({\mathscr{X}},\tau)$ is a topological space and $\mu^{\ast}$ is 
an outer-measure on ${\mathscr{X}}$ such that ${\mathfrak{M}}_{\mu^{\ast}}$ contains the Borel sets in $({\mathscr{X}},\tau)$, then call $\mu^{\ast}$ a {\tt Borel outer-measure} 
on ${\mathscr{X}}$. Furthermore, call such a Borel outer-measure $\mu^{\ast}$ a 
{\tt Borel regular outer-measure} if 
\begin{eqnarray}\label{Tgabv10}
\forall\,A\subseteq {\mathscr{X}}\,\,\,\exists 
\mbox{ a Borel set $B$ in $({\mathscr{X}},\tau)$ such that $A\subseteq B$ 
and $\mu^{\ast}(A)=\mu^{\ast}(B)$}.
\end{eqnarray}

After this digression, we now proceed to introduce the concept of 
$d$-dimensional Hausdorff outer-measure for a subset of a quasi-metric space. 

\begin{definition}\label{SKJ38}
Let $({\mathscr{X}},\rho)$ be a quasi-metric space and fix $d\geq 0$.  
Given a set $A\subseteq{\mathscr{X}}$, for every $\varepsilon>0$ define
${\mathcal{H}}_{{\mathscr{X}}\!,\,\rho,\,\varepsilon}^d(A)\in[0,\infty]$ by setting 
\begin{eqnarray}\label{SKJ38-1}
\hskip -0.20in
{\mathcal{H}}_{{\mathscr{X}}\!,\,\rho,\,\varepsilon}^d(A)
:=\inf\,\Bigl\{\sum_{j=1}^\infty({\rm diam}_{\rho}(A_j))^d:\,
A\subseteq\bigcup_{j=1}^\infty A_j\mbox{ and }
{\rm diam}_{\rho}(A_j)\leq\varepsilon\mbox{ for every }j\Bigr\}
\end{eqnarray}
(with the convention that $\inf\,\emptyset:=+\infty$), then take 
\begin{eqnarray}\label{SKJ38-2}
{\mathcal{H}}_{{\mathscr{X}}\!,\,\rho}^d(A)
:=\lim_{\varepsilon\to0^+}{\mathcal{H}}_{{\mathscr{X}}\!,\,\rho,\,\varepsilon}^d(A)
=\sup_{\varepsilon>0}{\mathcal{H}}_{{\mathscr{X}}\!,\,\rho,\,\varepsilon}^d(A)
\in[0,\infty].
\end{eqnarray}
The quantity ${\mathcal{H}}_{{\mathscr{X}}\!,\,\rho}^d(A)$ is called the 
$d$-{\tt dimensional Hausdorff outer-measure} in $({\mathscr{X}},\rho)$ of the 
set $A$. Whenever the choice of the quasi-distance $\rho$ is irrelevant or 
clear from the context, ${\mathcal{H}}_{{\mathscr{X}}\!,\,\rho}^d(A)$ is 
abbreviated as ${\mathcal{H}}_{{\mathscr{X}}}^d(A)$.
\end{definition}

It is readily verified that ${\mathcal{H}}_{{\mathscr{X}}\!,\,\rho}^0$ is equivalent 
to the counting measure. Other basic properties of the Hausdorff outer-measure are 
collected in the proposition below, proved in \cite{MMMM-B}. To state it, recall 
that a measure $\mu$ on a quasi-metric space $({\mathscr{X}},\rho)$ is called
{\tt Borel regular} provided it is Borel on $({\mathscr{X}},\tau_\rho)$ and 
\begin{eqnarray}\label{T-gc2w}
\forall\,\mbox{$\mu$-measurable}\,A\subseteq{\mathscr{X}}\,\,\,\exists
\mbox{ a Borel set $B$ in $({\mathscr{X}},\tau_\rho)$ such that $A\subseteq B$ 
and $\mu(A)=\mu(B)$}.\quad
\end{eqnarray}
Also, we make the convention that, given a quasi-metric space $({\mathscr{X}},\rho)$ 
and $d\geq 0$, 
\begin{eqnarray}\label{HHH-56}
\mbox{${\mathscr{H}}_{{\mathscr{X}}\!,\,\rho}^d$ denotes the 
{\it measure} associated with the outer-measure 
${\mathcal{H}}_{{\mathscr{X}}\!,\,\rho}^d$ as in \eqref{MM-RE.90}}.
\end{eqnarray}

\begin{proposition}\label{PWRS22}
Let $({\mathscr{X}},\rho)$ be a quasi-metric space and fix $d\geq 0$.
Then the following properties hold:
\begin{enumerate}
\item[(1)] ${\mathcal{H}}_{{\mathscr{X}}\!,\,\rho}^d$ is a Borel outer-measure on 
$({\mathscr{X}},\tau_\rho)$. In particular, ${\mathscr{H}}_{{\mathscr{X}}\!,\,\rho}^d$
(introduced in \eqref{HHH-56}) is a Borel measure on $({\mathscr{X}},\tau_\rho)$.
\item[(2)] If $\rho_{\#}$ is as in Theorem~\ref{JjEGh} then 
${\mathcal{H}}_{{\mathscr{X}}\!,\,\rho_{\#}}^d$ is actually a 
Borel regular outer-measure on $({\mathscr{X}},\tau_\rho)$. 
Moreover, ${\mathscr{H}}_{{\mathscr{X}}\!,\,\rho_{\#}}^d$ 
is a Borel regular measure on $({\mathscr{X}},\tau_\rho)$.
\item[(3)] One has ${\mathcal{H}}_{{\mathscr{X}}\!,\,\rho'}^d\approx
{\mathcal{H}}_{{\mathscr{X}}\!,\,\rho}^d$ whenever $\rho'\approx\rho$, in the sense
that there exist two finite constants $C_1,C_2>0$, which depend only on $\rho$ 
and $\rho'$, such that
\begin{eqnarray}\label{PWR-Cx}
C_1\,{\mathcal{H}}_{{\mathscr{X}}\!,\,\rho}^d(A)
\leq{\mathcal{H}}_{{\mathscr{X}}\!,\,\rho'}^d(A)
\leq C_2\,{\mathcal{H}}_{{\mathscr{X}}\!,\,\rho}^d(A)
\quad\mbox{for all }\,A\subseteq{\mathscr{X}}.
\end{eqnarray}
\item[(4)] Let $E\subseteq{\mathscr{X}}$ and consider the quasi-metric 
space $(E,\rho|_E)$. Then the $d$-dimensional Hausdorff outer-measure in 
$(E,\rho|_{E})$ is equivalent to the restriction to $E$ of the 
$d$-dimensional Hausdorff outer-measure in ${\mathscr{X}}$. 
That is, in the sense of \eqref{PWR-Cx},
\begin{eqnarray}\label{jkan}
{\mathcal{H}}_{E,\,\rho|_{E}}^d\approx
{\mathcal{H}}_{{\mathscr{X}}\!,\,\rho}^d\big\lfloor E. 
\end{eqnarray}
\item[(5)] For any $E\subseteq{\mathscr{X}}$, 
${\mathcal{H}}_{{\mathscr{X}}\!,\,\rho_{\#}}^d\big\lfloor E$ is a Borel regular 
outer-measure on $(E,\tau_{\rho|_{E}})$, and the measure associated with it 
(as in \eqref{MM-RE.90}) is a Borel regular measure on $(E,\tau_{\rho|_{E}})$.

Furthermore, if $E$ is ${\mathcal{H}}_{{\mathscr{X}}\!,\,\rho_{\#}}^d$-measurable
(in the sense of \eqref{MM-RE.89}; hence, in particular, if $E$ is a Borel subset 
of $({\mathscr{X}},\tau_\rho)$), then 
${\mathscr{H}}_{{\mathscr{X}}\!,\,\rho_{\#}}^d\lfloor E$ 
is a Borel regular measure on $(E,\tau_{\rho|_{E}})$
and it coincides with the measure associated with the outer-measure
${\mathcal{H}}_{{\mathscr{X}}\!,\,\rho_{\#}}^d\lfloor E$. 
\item[(6)] Assume that $m\in(d,\infty)$. Then for each $E\subseteq{\mathscr{X}}$ 
one has 
\begin{eqnarray}\label{haTT.2}
{\mathcal{H}}_{{\mathscr{X}}\!,\,\rho}^{d}(E)<\infty\,\Longrightarrow\, 
{\mathcal{H}}_{{\mathscr{X}}\!,\,\rho}^{m}(E)=0.
\end{eqnarray}
\end{enumerate}
\end{proposition}

\subsection{Geometrically doubling quasi-metric spaces}
\label{SSect:2.2}

In this subsection we shall work in a more specialized setting than that 
of general quasi-metric spaces considered so far, by considering geometrically 
doubling quasi-metric spaces, as described in the definition below. 

\begin{definition}\label{Gd_ZZ}
A quasi-metric space $({\mathscr{X}},\rho)$ is called {\tt geometrically doubling} 
if there exists a number $N\in{\mathbb{N}}$, called the geometric doubling constant 
of $({\mathscr{X}},\rho)$, with the property that any $\rho$-ball of radius $r$ 
in ${\mathscr{X}}$ may be covered by at most $N$ $\rho$-balls in ${\mathscr{X}}$ 
of radii $r/2$.  
\end{definition}

To put this matter into a larger perspective, recall that a subset $E$ 
of a quasi-metric space $({\mathscr{X}},\rho)$ is said to be {\tt totally bounded}
provided that for any $r\in(0,\infty)$ there exists a finite covering of $E$ with
$\rho$-balls of radii $r$. Then for a quasi-metric space $({\mathscr{X}},\rho)$
the quality of being geometrically doubling may be regarded as a scale-invariant
version of the demand that all $\rho$-balls in ${\mathscr{X}}$ are totally bounded.
In fact it may be readily verified that if $({\mathscr{X}},\rho)$ is a 
geometrically doubling quasi-metric space, then 
\begin{eqnarray}\label{RAF-e}
\begin{array}{l}
\exists\,N\in{\mathbb{N}}\mbox{ such that $\forall\,\vartheta\in(0,1)$
any $\rho$-ball of radius $r$ in ${\mathscr{X}}$} 
\\[4pt]
\mbox{may be covered by at most $N^{-[\log_2\vartheta]}$ 
$\rho$-balls in ${\mathscr{X}}$ of radii $\vartheta r$},
\end{array}
\end{eqnarray}
where $[\log_2\vartheta]$ is the smallest integer greater than or equal to $\log_2 \vartheta$. En route, let us also point out that {\it the property of being geometrically doubling is hereditary}, in the sense that if $({\mathscr{X}},\rho)$ is a geometrically 
doubling quasi-metric space with geometric doubling constant $N$, and if $E$ is an arbitrary subset of ${\mathscr{X}}$, then $(E,\rho\bigl|_E)$ is a geometrically doubling quasi-metric space with geometric doubling constant at most equal to 
$N^{\log_2 C_\rho}N$. 

The relevance of the property (of a quasi-metric space) of being 
geometrically doubling is apparent from the fact that in such a context 
a number of useful geometrical results hold, which are akin to those available 
in the Euclidean setting. A case in point, is the Whitney decomposition theorem
discussed in Proposition~\ref{H-S-Z} below. A version of the classical 
Whitney decomposition theorem in the Euclidean setting (as presented in, 
e.g., \cite[Theorem~1.1, p.\,167]{St70}) has been worked out in
\cite[Theorem~3.1, p.\,71]{CoWe71} and \cite[Theorem~3.2, p.\,623]{CoWe77}
in the context of bounded open sets in spaces of homogeneous type.
Recently, the scope of this work has been further refined  in \cite{MMMM-G} by  
allowing arbitrary open sets in geometrically doubling quasi-metric spaces, 
as presented in the following proposition.

\begin{proposition}\label{H-S-Z}
Let $({\mathscr{X}},\rho)$ be a geometrically doubling quasi-metric space. 
Then for each number $\lambda\in (1,\infty)$ there exist constants
$\Lambda\in(\lambda,\infty)$ and $N\in{\mathbb{N}}$, both depending only on $C_\rho,\widetilde{C}_\rho,\lambda$ and the geometric doubling constant 
of $({\mathscr{X}},\rho)$, and which have the following significance. 

For each open, nonempty, proper subset ${\mathcal{O}}$ of the topological space
$({\mathscr{X}},\tau_\rho)$ there exist an at most countable family of points $\{x_j\}_{j\in J}$ in ${\mathcal{O}}$ along with a family of real numbers 
$r_j>0$, $j\in J$, for which the following properties are valid:
\begin{enumerate}
\item[(1)] ${\mathcal{O}}=\bigcup\limits_{j\in J}B_\rho(x_j,r_j)$;
\item[(2)] $\sum\limits_{j\in J}{\mathbf 1}_{B_\rho(x_j,\lambda r_j)}\leq N$ 
on ${\mathcal{O}}$. In fact, there exists $\varepsilon\in(0,1)$, which depends 
only on $C_\rho,\lambda$ and the geometric doubling constant of $({\mathscr{X}},\rho)$,
with the property that for any $x_o\in{\mathcal{O}}$ 
\begin{eqnarray}\label{Lay-ff.u-c}
\hskip -0.20in
\#\,\Bigl\{j\in J:\,B_\rho\bigl(x_o,\varepsilon\,
{\rm dist}_\rho(x_o,{\mathscr{X}}\setminus{\mathcal{O}})\bigr)
\cap B_{\rho}(x_j,\lambda r_j)\not=\emptyset\Bigr\}\leq N.
\end{eqnarray}
\item[(3)] $B_\rho(x_j,\lambda r_j)\subseteq{\mathcal{O}}$ and 
$B_\rho(x_j,\Lambda r_j)\cap\bigl[{\mathscr{X}}\setminus{\mathcal{O}}\bigr]\not=\emptyset$
for every $j\in J$. 
\item[(4)] $r_i\approx r_j$ uniformly for $i,j\in J$ such that 
$B_\rho(x_i,\lambda r_i)\cap B_\rho(x_j,\lambda r_j)\not=\emptyset$, and there 
exists a finite constant $C>0$ with the property that 
$r_j\leq C\,{\rm diam}_{\rho}({\mathcal{O}})$ for each $j\in J$.
\end{enumerate}
\end{proposition}

Regarding terminology, we shall frequently employ the following convention:

\begin{convention}\label{WWVc} 
Given a geometrically doubling quasi-metric space $({\mathscr{X}},\rho)$,
an open, nonempty, proper subset ${\mathcal{O}}$ of $({\mathscr{X}},\tau_\rho)$, 
and a parameter $\lambda\in (1,\infty)$, we will refer to the 
balls $B_{\rho_{\#}}(x_j,r_j)$ obtained by treating $(\mathscr{X},\rho_{\#})$ in Proposition~\ref{H-S-Z}
as {\tt Whitney cubes}, denote the collection of these cubes 
by ${\mathbb{W}}_\lambda({\mathcal{O}})$, and for each 
$I\in{\mathbb{W}}_\lambda({\mathcal{O}})$, write $\ell(I)$ for the 
{\tt radius of $I$}. 

Furthermore, if  $I\in{\mathbb{W}}_\lambda({\mathcal{O}})$ and $c\in(0,\infty)$, 
we shall denote by $cI$ the {\tt dilate of the cube $I$ by factor $c$}, i.e., 
the ball having the same center as $I$ and radius $c\ell(I)$. 
\end{convention}

Spaces of homogeneous type, reviewed next, are an important subclass
of the class of geometrically doubling quasi-metric spaces. 

\begin{definition}\label{zbnb-hom}
A {\tt space of homogeneous type} is a triplet $({\mathscr{X}},\rho,\mu)$, where 
$({\mathscr{X}},\rho)$ is a quasi-metric space and $\mu$ is a Borel measure on
$({\mathscr{X}},\tau_\rho)$ with the property that all $\rho$-balls are $\mu$-measurable,
and which satisfies the doubling condition
\begin{eqnarray}\label{Doub-1}
0<\mu\bigl(B_{\rho}(x,2r)\bigr)\leq C\mu\bigl(B_{\rho}(x,r)\bigr)<\infty,\quad
\forall\,x\in {\mathscr{X}},\,\,\,\forall\,r>0, 
\end{eqnarray}
for some finite constant $C\geq 1$. 
\end{definition}

\noindent In the context of the above definition, call the number 
\begin{eqnarray}\label{Doub-1XX}
C_\mu:=\sup_{x\in{\mathscr{X}},r>0}\frac{\mu\bigl(B_{\rho}(x,2r)\bigr)}
{\mu\bigl(B_{\rho}(x,r)\bigr)}\in[1,\infty)
\end{eqnarray}
the {\tt doubling constant} of $\mu$. Iterating \eqref{Doub-1} then gives
\begin{eqnarray}\label{Doub-2}
\begin{array}{c}
\frac{\mu(B_1)}{\mu(B_2)}\leq C_{\mu,\rho}\Bigl(\frac{\mbox{radius of $B_1$}}
{\mbox{radius of $B_2$}}\Bigr)^{D_\mu},\qquad\mbox{for all $\rho$-balls }\,\,
B_2\subseteq B_1,
\\[8pt]
\mbox{where }\,\,D_\mu:=\log_2\,C_\mu\geq 0\,\,\mbox{ and }\,\,
C_{\mu,\rho}:=C_\mu\bigl(C_\rho\widetilde{C}_\rho\bigr)^{D_\mu}\geq 1.
\end{array}
\end{eqnarray}
The exponent $D_\mu$ is referred to as the {\tt doubling order} of $\mu$. 
For further reference, let us also record here the well-known fact that 
\begin{eqnarray}\label{DIA-MEA}
\begin{array}{c}
\mbox{given a space of homogeneous type $({\mathscr{X}},\rho,\mu)$, one has}
\\[4pt]
\mbox{${\rm diam}_{\rho}\,({\mathscr{X}})<\infty$ if and only if
$\mu({\mathscr{X}})<\infty$}.
\end{array}
\end{eqnarray}

Going further, a distinguished subclass of the class of spaces of homogeneous 
type, which is going to play a basic role in this work, is the category of 
Ahlfors-David regular spaces defined next. 

\begin{definition}\label{Rcc-TG34}
Suppose that $d>0$. A $d$-{\tt dimensional Ahlfors-David regular} (or, simply, $d$-dimensional {\rm ADR}, or $d$-{\rm ADR}) {\tt space} is a triplet $({\mathscr{X}},\rho,\mu)$, where $({\mathscr{X}},\rho)$ is a quasi-metric space and $\mu$ is a Borel measure on $({\mathscr{X}},\tau_\rho)$ with the property that all $\rho$-balls are $\mu$-measurable, and for which there exists a constant $C\in[1,\infty)$ such that
\begin{eqnarray}\label{Q3HF}
C^{-1}\,r^d\leq\mu\bigl(B_\rho(x,r)\bigr)\leq C\,r^d,
\quad\forall\,x\in {\mathscr{X}},\,\,\,\,\mbox{for every finite }\,r
\in(0,{\rm diam}_\rho({\mathscr{X}})].
\end{eqnarray}
The constant $C$ in \eqref{Q3HF} will be referred to 
as the {\rm ADR} constant of ${\mathscr{X}}$. 
\end{definition}

\noindent As alluded to earlier, if $({\mathscr{X}},\rho,\mu)$ is a 
$d$-dimensional ADR space then, trivially, $({\mathscr{X}},\rho,\mu)$ is also 
a space of homogeneous type. For further reference let us also note here that 
(cf., e.g., \cite{MMMM-B})
\begin{eqnarray}\label{Rss-TASS}
\mbox{$({\mathscr{X}},\rho,\mu)$ is $d$-ADR}\,\,\Longrightarrow\,\,
\mbox{$\bigl({\mathscr{X}},\rho_{\#},{\mathscr{H}}^d_{{\mathscr{X}},\rho_{\#}}\bigr)$ 
is $d$-ADR}.
\end{eqnarray}
In particular, it follows from \eqref{Rss-TASS}, \eqref{RHo-evv}, 
and parts {\it (3)-(5)} in Proposition~\ref{PWRS22} that 
\begin{eqnarray}\label{TASS.bis2}
\left.
\begin{array}{r}
\mbox{$({\mathscr{X}},\rho)$ quasi-metric space,}
\\[4pt]
\mbox{$E$ Borel subset of $({\mathscr{X}},\tau_\rho)$}
\\[4pt]
\mbox{$\sigma$ Borel measure on $(E,\tau_{\rho|_{E}})$}
\\[4pt]
\mbox{such that $(E,\rho|_{E},\sigma)$ is $d$-{\rm ADR}}
\end{array}
\right\}
\,\,\Longrightarrow\,\,\mbox{
$\bigl(E,\rho_{\#}\bigl|_{E},{\mathscr{H}}^d_{{\mathscr{X}},\rho_{\#}}\lfloor E\bigr)$ 
is $d$-ADR}.
\end{eqnarray}
Also, if $({\mathscr{X}},\rho,\mu)$ is
$d$-ADR, then there exists a finite constant $C>0$ such that
\begin{eqnarray}\label{PEY88-2}
&& {\mathcal{H}}_{{\mathscr{X}}\!,\,\rho_{\#}}^d(A)
\leq C\inf\limits_{{\mathcal{O}}{\mbox{\tiny{ open}}},\,A\subseteq{\mathcal{O}}}
\mu({\mathcal{O}})\quad\mbox{for every }A\subseteq{\mathscr{X}},\,\mbox{ and}
\\[4pt]
&& \mu(A)\leq C{\mathscr{H}}_{{\mathscr{X}}\!,\,\rho_{\#}}^d(A)
\quad\mbox{for every Borel subset $A$ of $({\mathscr{X}},\tau_\rho)$}.
\label{PEY88-2BIS}
\end{eqnarray}
In addition, if $\mu$ is actually a Borel regular measure, then  
\begin{eqnarray}\label{P2-GFs}
\mu(A)\approx{\mathscr{H}}_{{\mathscr{X}}\!,\,\rho_{\#}}^d(A),
\qquad\mbox{uniformly for Borel subsets $A$ of $({\mathscr{X}},\tau_\rho)$}.
\end{eqnarray}

We now discuss a couple of technical lemmas 
which are going to be useful for us later on.

\begin{lemma}\label{ME-ZZ}
Let $({\mathscr{X}},\rho,\mu)$ be an $m$-dimensional {\rm ADR} space for some 
$m\in(0,\infty)$ and suppose that $E$ is a Borel subset of $({\mathscr{X}},\tau_{\rho})$
with the property that there exists a Borel measure $\sigma$ on $(E,\tau_{\rho|_{E}})$
such that $(E,\rho\bigl|_E,\sigma)$ is a $d$-dimensional {\rm ADR}
space for some $d\in(0,m)$. Then $\mu(E)=0$. 
\end{lemma}

\begin{proof}
Fix $x\in E$. Using \eqref{PEY88-2BIS}, \eqref{TASS.bis2} and 
item {\it (6)} in Proposition~\ref{PWRS22}, we obtain
\begin{eqnarray}\label{bgSS-1}
\mu(E) \leq C {\mathscr{H}}^m_{{\mathscr{X}},\rho_{\#}}(E) = C \lim_{n\rightarrow\infty} {\mathscr{H}}^m_{{\mathscr{X}},\rho_{\#}}(E\cap B_{\rho_{\#}}(x,n))=0,
\end{eqnarray}
since ${\mathscr{H}}^d_{{\mathscr{X}},\rho_{\#}}(E\cap B_{\rho_{\#}}(x,n))\leq Cn^d<\infty$ for all $n\in{\mathbb{N}}$. 
\end{proof}

\begin{lemma}\label{segj}
Let $({\mathscr{X}},\rho)$ be a quasi-metric space. Suppose 
that $E$ is a Borel subset of $({\mathscr{X}},\tau_\rho)$ such that there exists a 
Borel measure $\sigma$ 
on $(E,\tau_{\rho|_{E}})$ with the property that $(E,\rho\bigl|_E,\sigma)$ 
is a $d$-dimensional {\rm ADR} space for some $d\in(0,\infty)$. Then there 
exists a constant $c\in(0,\infty)$ such that
\begin{eqnarray}\label{bgSS}
\begin{array}{c}
\forall\,x\in{\mathscr{X}},\,\,\forall\,r\in(0,{\rm diam}_{\rho_{\#}}(E)]
\mbox{ with $B_{\rho_{\#}}(x,r)\cap E\not=\emptyset$ there holds}
\\[4pt]
{\mathscr{H}}_{{\mathscr{X}}\!,\,\rho_{\#}}^d\bigl(B_{\rho_{\#}}(x,C_\rho r)\cap E\bigr)
\geq c\,r^d.
\end{array}
\end{eqnarray}
\end{lemma}

\begin{proof}
Fix a point $x\in{\mathscr{X}}$ with the property that 
$B_{\rho_{\#}}(x,r)\cap E\not=\emptyset$.
If we now select $y\in B_{\rho_{\#}}(x,r)\cap E$ then 
$B_{\rho_{\#}}(y,r)\subseteq B_{\rho_{\#}}(x,C_\rho r)$. Recall \eqref{TASS.bis2} 
and let $C$ be the {\rm ADR} constant of $\bigl(E,\rho_{\#}\bigl|_{E},
{\mathscr{H}}^d_{{\mathscr{X}}\!,\,\rho_{\#}}\lfloor E\bigr)$.
Then, since $y\in E$,  
\begin{eqnarray}\label{ZHvd}
{\mathscr{H}}_{{\mathscr{X}}\!,\,\rho_{\#}}^d
\bigl(B_{\rho_{\#}}(x,C_\rho r)\cap E\bigr)
\geq {\mathscr{H}}_{{\mathscr{X}}\!,\,\rho_{\#}}^d
\bigl(B_{\rho_{\#}}(y,r)\cap E\bigr)\geq C^{-1}r^d.
\end{eqnarray}
Hence \eqref{bgSS} holds with $c:=C^{-1}$.
\end{proof}

Following work in \cite{Christ} and \cite{David1988}, we now discuss the 
existence of a {\tt dyadic grid structure} on geometrically doubling quasi-metric spaces. 
The following result is essentially due to M. Christ \cite{Christ}, with  
two refinements. First, Christ's dyadic grid result is established
in the presence of a background doubling, Borel regular measure, which is more restrictive than merely assuming that the ambient quasi-metric space is 
geometrically doubling. Second, Christ's dyadic grid result involves a 
scale $\delta\in(0,1)$ which we here show may be taken to be $\frac12$,
as in the Euclidean setting. 

\begin{proposition}\label{Diad-cube}
Assume that $(E,\rho)$ is a geometrically doubling quasi-metric space and 
select $\kappa_E\in{\mathbb{Z}}\cup\{-\infty\}$ with the property that 
\begin{eqnarray}\label{Jcc-KKi}
2^{-\kappa_E-1}< {\rm diam}_\rho(E)\leq 2^{-\kappa_E}.
\end{eqnarray}
Then there exist finite constants $a_1\geq a_0>0$ such that for each 
$k\in{\mathbb{Z}}$ with $k\geq\kappa_E$, there exists a 
collection ${\mathbb{D}}_k(E):=\{Q_\alpha^k\}_{\alpha\in I_k}$ of 
subsets of $E$ indexed by a nonempty, at most countable set of indices $I_k$, 
as well as a family $\{x_\alpha^k\}_{\alpha\in I_k}$ of points in $E$, 
such that the collection of all {\tt dyadic cubes} in $E$, i.e., 
\begin{eqnarray}\label{gcEd}
{\mathbb{D}}(E):=\bigcup_{k\in{\mathbb{Z}},\,k\geq\kappa_E}{\mathbb{D}}_k(E),
\end{eqnarray}
has the following properties:
\begin{enumerate}
\item[(1)] $[${\rm All dyadic cubes are open}$]$ \\
For each $k\in{\mathbb{Z}}$ with $k\geq\kappa_E$ and each $\alpha\in I_k$, 
the set $Q_\alpha^k$ is open in $\tau_\rho$;
\item[(2)] $[${\rm Dyadic cubes are mutually disjoint within the same generation}$]$ \\
For each $k\in{\mathbb{Z}}$ with $k\geq\kappa_E$ and each $\alpha,\beta\in I_k$ with
$\alpha\not=\beta$ there holds $Q_\alpha^k\cap Q_\beta^k=\emptyset$;
\item[(3)] $[${\rm No partial overlap across generations}$]$ \\
For each $k,\ell\in{\mathbb{Z}}$ with $\ell>k\geq\kappa_E$, and each 
$\alpha\in I_k$, $\beta\in I_\ell$, either 
$Q_\beta^\ell\subseteq Q_\alpha^k$ or $Q_\alpha^k\cap Q_\beta^\ell=\emptyset$;
\item[(4)] $[${\rm Any dyadic cube has a unique ancestor in any earlier generation}$]$ \\
For each $k,\ell\in{\mathbb{Z}}$ with $k>\ell\geq\kappa_E$, and each $\alpha\in I_k$ 
there is a unique $\beta\in I_\ell$ such that $Q_\alpha^k\subseteq Q_\beta^\ell$;
\item[(5)] $[${\rm The size is dyadically related to the generation}$]$\\
For each $k\in{\mathbb{Z}}$ with $k\geq\kappa_E$ and each $\alpha\in I_k$ one has
\begin{eqnarray}\label{ha-GVV}
B_{\rho}(x_\alpha^k,a_0 2^{-k})\subseteq Q_\alpha^k\subseteq B_{\rho}(x_\alpha^k,a_1 2^{-k});
\end{eqnarray}
In particular, given a measure $\sigma$ on $E$ for which $(E,\rho,\sigma)$ 
is a space of homogeneous type, there exists $c>0$ such that if 
$Q^{k+1}_\beta\subseteq Q^k_\alpha$, then 
$\sigma(Q^{k+1}_\beta)\geq c\sigma(Q^k_\alpha)$.

\item[(6)] $[${\rm Control of the number of children}$]$\\
There exists an integer $N\in{\mathbb{N}}$ with the property that for each 
$k\in{\mathbb{Z}}$ with $k\geq\kappa_E$ one has
\begin{eqnarray}\label{ha-GMM}
\#\bigl\{\beta\in I_{k+1}:\,Q^{k+1}_\beta\subseteq Q^k_{\alpha}\bigr\}\leq N,\quad
\mbox{ for every }\,\,\alpha\in I_{k}.
\end{eqnarray}
Furthermore, this integer may be chosen such that, for each $k\in{\mathbb{Z}}$ 
with $k\geq\kappa_E$, each $x\in E$ and $r\in(0,2^{-k})$, the number of 
$Q$'s in ${\mathbb{D}}_k(E)$ that intersect $B_\rho(x,r)$ is at most $N$. 

\item[(7)] $[${\rm Any generation covers a dense subset of the entire space}$]$\\
For each $k\in{\mathbb{Z}}$ with $k\geq\kappa_E$, the set 
$\bigcup_{\alpha\in I_k}Q_\alpha^k$ is dense in $(E,\tau_\rho)$. In particular, 
for each $k\in{\mathbb{Z}}$ with $k\geq\kappa_E$ one has 
\begin{eqnarray}\label{ha-GMM.2}
E=\bigcup_{\alpha\in I_k}\bigl\{x\in E:\,{\rm dist}_{\rho}(x,Q_\alpha^k)
\leq\varepsilon 2^{-k}\bigr\},\qquad\forall\,\varepsilon>0,
\end{eqnarray}
and there exist $b_0,b_1\in(0,\infty)$ depending only on the 
geometrically doubling character of $E$ with the property that 
\begin{eqnarray}\label{ha-GL54}
\begin{array}{c}
\forall\,x_o\in E,\,\,\,\forall\,r\in(0,{\rm diam}_\rho(E)],\,\,\,
\exists\,k\in{\mathbb{Z}}\,\mbox{ with }\,k\geq\kappa_E\,\mbox{ and }\,
\exists\,\alpha\in I_k
\\[4pt]
\mbox{with the property that }\,Q^k_\alpha\subseteq B_{\rho}(x_o,r)
\,\mbox{ and }\,b_0 r\leq 2^{-k}\leq b_1 r.
\end{array}
\end{eqnarray}
Moreover, for each $k\in{\mathbb{Z}}$ with $k\geq\kappa_E$ and each $\alpha\in I_k$ 
\begin{eqnarray}\label{ha-GMM.2JJ}
\bigcup_{\beta\in I_{k+1},\,Q^{k+1}_\beta\subseteq Q^k_{\alpha}}Q^{k+1}_\beta
\,\,\mbox{ is dense in }\,\,Q^k_{\alpha},
\end{eqnarray}
and 
\begin{eqnarray}\label{ha-GMM.2II}
Q^k_{\alpha}\subseteq
\bigcup_{\beta\in I_{k+1},\,Q^{k+1}_\beta\subseteq Q^k_{\alpha}}
\bigl\{x\in E:\,{\rm dist}_{\rho}(x,Q_\beta^{k+1})
\leq\varepsilon 2^{-k-1}\bigr\},\qquad\forall\,\varepsilon>0.
\end{eqnarray}

\item[(8)] $[${\rm Dyadic cubes have thin boundaries with respect 
to a background doubling measure}$]$ \\
Given a measure $\sigma$ on $E$ for which $(E,\rho,\sigma)$ 
is a space of homogeneous type, a collection ${\mathbb{D}}(E)$ may be constructed 
as in \eqref{gcEd} such that properties {\it (1)-(7)} above hold and, in addition, 
there exist constants $\vartheta\in(0,1)$ and $c\in(0,\infty)$ such that 
for each $k\in{\mathbb{Z}}$ with $k\geq\kappa_E$ and each $\alpha\in I_k$ one has
\begin{eqnarray}\label{hacc-76es}
\sigma \left(\bigl\{x\in Q_\alpha^k:\,{\rm dist}_{\rho_{\#}}(x,E\setminus Q_\alpha^k)
\leq t\,2^{-k}\bigr\}\right)\leq c\,t^\vartheta\sigma(Q_\alpha^k),\qquad\forall\,t>0.
\end{eqnarray}
Moreover, in such a context matters may be arranged so that, 
for each $k\in{\mathbb{Z}}$ with $k\geq\kappa_E$ and each $\alpha\in I_k$,
\begin{eqnarray}\label{ihgc}
\bigl(Q^k_\alpha,\rho|_{Q^k_\alpha},\sigma\lfloor{Q^k_\alpha}\bigr)
\quad\mbox{ is a space of homogeneous type},
\end{eqnarray}
and the doubling constant of the measure $\sigma\lfloor{Q^k_\alpha}$ 
is independent of $k,\alpha$ (i.e., {\rm the quality of being a space of homogeneous 
type is hereditary at the level of dyadic cubes, in a uniform fashion}).
\item[(9)] $[${\rm All generations cover the space a.e. with respect to a doubling 
Borel regular measure}$]$\\
If $\sigma$ is a Borel measure on $E$ which is both doubling 
(cf. \eqref{Doub-1}) and Borel regular (cf. \eqref{T-gc2w})
then a collection ${\mathbb{D}}(E)$ associated with the doubling measure 
$\sigma$ as in {\it (8)} may be constructed with the additional property that 
\begin{eqnarray}\label{T-rcs}
\sigma\Bigl(E\setminus\bigcup_{\alpha\in I_k}Q_\alpha^k\Bigr)=0
\qquad\mbox{for each }\,\,k\in{\mathbb{Z}},\,\,k\geq\kappa_E.
\end{eqnarray}
In particular, in such a setting, for each $k\in{\mathbb{Z}}$ 
with $k\geq\kappa_E$ one has 
\begin{eqnarray}\label{ha-GMM.67}
\sigma\Bigl(Q^k_{\alpha}\setminus
\bigcup\limits_{\beta\in I_{k+1},\,Q^{k+1}_\beta\subseteq Q^k_{\alpha}}
Q^{k+1}_\beta\Bigr)=0,\quad\mbox{ for every }\,\,\alpha\in I_{k}.
\end{eqnarray}
\end{enumerate}
\end{proposition}

Before discussing the proof of this result we wish to say a few words clarifying
terminology. As already mentioned in the statement, sets $Q$ belonging to 
${\mathbb{D}}(E)$ will be referred to as {\it dyadic cubes} (on $E$). 
Also, following a well-established custom, whenever 
$Q_\alpha^{k+1}\subseteq Q_\beta^k$ we shall call $Q_\alpha^{k+1}$ 
a {\it child} of $Q_\beta^{k}$, and we shall say that $Q_\beta^{k}$ 
is a {\it parent} of $Q_\alpha^{k+1}$. For a given dyadic cube, an {\it ancestor} 
is then a parent, or a parent of a parent, or so on. Moreover, for each $k\in{\mathbb{Z}}$ with $k\geq\kappa_E$, we shall call ${\mathbb{D}}_k(E)$ 
the {\it dyadic cubes of generation} $k$ and, for each $Q\in{\mathbb{D}}_k(E)$, 
define the {\it side-length} of $Q$ to be $\ell(Q):=2^{-k}$, and the 
{\it center} of $Q$ to be the point $x^k_\alpha\in E$ if $Q=Q^k_\alpha$.

Finally, we make the convention that saying that 
{\it ${\mathbb{D}}(E)$ is a dyadic cube structure (or dyadic grid) on $E$} will 
always indicate that the collection ${\mathbb{D}}(E)$ is associated with $E$ as in 
Proposition~\ref{Diad-cube}. This presupposes that $E$ is the ambient set for a 
geometrically doubling quasi-metric space, in which case ${\mathbb{D}}(E)$ 
satisfies properties {\it (1)-(7)} above and that, in the presence of a background
measure $\sigma$ satisfying appropriate conditions (as stipulated in Proposition~\ref{Diad-cube}), properties {\it (8)} and {\it (9)} also hold.

We are now ready to proceed with the 

\vskip 0.08in
\begin{proof}[Proof of Proposition~\ref{Diad-cube}] 
This is a slight extension and clarification of a result proved by M.~Christ 
in \cite{Christ}, generalizing earlier work by G. David in \cite{David1988}, 
and we will limit ourselves to discussing only the novel aspects of the 
present formulation. For the sake of reference, we debut by recalling the 
main steps in the construction in \cite{Christ}. For a fixed real number 
$\delta\in(0,1)$ and for any integer $k$, Christ considers a maximal collection 
of points $z^k_\alpha\in E$ such that
\begin{eqnarray}\label{centers}
\rho_{\#}(z^k_\alpha,z^k_\beta)\geq\delta^k,\qquad\forall\,\alpha\not=\beta.
\end{eqnarray}
Hence, for each fixed $k$, the set $\{z^k_\alpha\}_{\alpha}$ is $\delta^k$-dense 
in $E$ in the sense that for each $k\in{\mathbb{Z}}$ and $x\in E$ there exists 
$\alpha$ such that $\rho_{\#}(x,z^k_\alpha)<\delta^k$.
Then (cf. \cite[Lemma~13, p.\,8]{Christ}) there exists a partial order relation 
$\preceq$ on the set $\{(k,\alpha):\,k\in{\mathbb{Z}},\alpha\in I_k\}$ 
with the following properties:

1) if $(k,\alpha)\preceq(l,\beta)$ then $k\geq l$;

2) for each $(k,\alpha)$ and $l\leq k$ there exists a unique $\beta$ such that $(k,\alpha)\preceq(l,\beta)$;

3) if $(k,\alpha)\preceq(k-1,\beta)$ then 
$\rho_{\#}(z^k_\alpha,z^{k-1}_\beta)<\delta^{k-1}$;

4) if $\rho_{\#}(z^l_\beta,z^k_\alpha)\leq 2C_\rho\delta^k$
then $(l,\beta)\preceq(k,\alpha)$.

\noindent Having established this, Christ then chooses a number $c\in(0,\frac{1}{2C_\rho})$ 
and defines 
$$
Q^k_\alpha:=\bigcup_{(l,\beta)\preceq(k,\alpha)}B_{\rho_{\#}}(z^l_\beta,c\delta^{\,l}).
$$

First, the dyadic cubes in \cite[Theorem~11, p.\,7]{Christ} 
are labeled over all $k\in{\mathbb{Z}}$. However, \eqref{ha-GVV} shows that in the
case when $E$ is bounded the index set $I_k$ becomes a singleton whenever $2^{-k}$ is
sufficiently large. Hence, in particular, ${\mathbb{D}}_k(E)$ becomes stationary as 
$k$ approaches $-\infty$, in the sense that this collection of cubes reduces to just 
the set $E$ provided $2^{-k}$ is sufficiently large. While this is not an issue in and
of itself, for later considerations we find it useful to eliminate this redundancy and
this is the reason why we restrict ourselves to only $k\geq\kappa_E$.

Second, \cite[Theorem~11, p.\,7]{Christ} is stated with $\delta^k$ replacing $2^{-k}$ 
in \eqref{ha-GVV}-\eqref{hacc-76es}, for some $\delta\in(0,1)$. The reason why we may 
always assume that $\delta=1/2$ is discussed later below. 
Third, Christ's result just mentioned is formulated in the setting 
of spaces of homogeneous type (equipped with symmetric quasi-distance), but
a cursory inspection of the proof reveals that for properties {\it (1)-(6)} 
in our statement the same type of arguments as in \cite[pp.\,7-10]{Christ} 
go through (working with the regularization $\rho_{\#}$ of $\rho$, as in 
Theorem~\ref{JjEGh}) under the weaker assumption that $(E,\rho)$ is a 
geometrically doubling quasi-metric space.

Fourth, {\it (7)} follows from a careful inspection of the proof of 
\cite[Theorem~11, p.\,7]{Christ}, which reveals that for each $k\in{\mathbb{Z}}$ 
with $k\geq\kappa_E$, and any $j\in{\mathbb{N}}$ sufficiently large (compared to $k$) 
the set $\bigcup_{\alpha\in I_k}Q_\alpha^k$ contains a $2^{-j}$-dense subset of $E$ 
that is maximal with respect to inclusion. Of course, this shows that the union in 
question is dense in $(E,\tau_\rho)$, and \eqref{ha-GMM.2} is a direct consequence of it.

Fifth, with the exception of using the regularization $\rho_{\#}$ of the original 
quasi-distance $\rho$ from Theorem~\ref{JjEGh} in place of the regularization devised 
in \cite{MaSe79}, property {\it (8)} is identical to condition (3.6) in 
\cite[Theorem~11, p.\,7]{Christ}. Sixth, property {\it (9)} corresponds to (3.1) in
\cite[Theorem~11, p.\,7]{Christ} except that we are presently assuming that the doubling
measure $\sigma$ is Borel regular. The reason for this assumption is that the proof of 
(3.1) in \cite[Theorem~11, p.\,7]{Christ} uses the Lebesgue Differentiation Theorem, 
whose proof requires that continuous functions vanishing outside bounded subsets 
of $E$ are dense in $L^1(E,\sigma)$. It is precisely here that the aforementioned
regularity of the measure intervenes and the reader is referred to 
\cite[Theorem~7.10]{MMMM-G} for a density result of this nature.

The remainder of this proof consists of a verification that, compared with 
\cite{Christ}, it is always possible to take $\delta=1/2$ (as described in the
first paragraph of this proof). In the process, we shall adopt Christ's convention of
labeling the dyadic cubes over all $k\in{\mathbb{Z}}$ (eliminating the inherent redundancy
in the case when $E$ is bounded may be done afterwards). To get started, let 
$\mathfrak{D}(E):=\bigcup\limits_{k\in{\mathbb{Z}}}\mathfrak{D}_k(E)$ 
denote a collection of dyadic cubes enjoying properties {\it (1)-(9)} listed 
in Proposition~\ref{Diad-cube} but with $\delta^k$ replacing $2^{-k}$ 
in \eqref{ha-GVV}-\eqref{hacc-76es}. Our goal in this 
part of the proof is to construct another collection of dyadic cubes, 
${\mathbb{D}}(E):=\bigcup\limits_{k\in{\mathbb{Z}}}{\mathbb{D}}_k(E)$ 
satisfying similar properties for $\delta=1/2$. We shall consider two cases.

\medskip
{\it Case~I: $1/2<\delta<1$.} Set $m_0:=0$ and, for each integer $k>0$, let 
$m_k$ be the largest positive integer such that $\delta^{m_k}\geq 2^{-k}$. Thus,
\begin{eqnarray}\label{HS1}
\delta^{m_k+1}<2^{-k}\leq \delta^{m_k}.
\end{eqnarray}  
Similarly, for each $k<0$, let $m_k$ denote the least integer such that 
$\delta^{m_k+1}< 2^{-k}$. Thus, again we have \eqref{HS1}. Of course, we shall 
have $m_k<0$ when $k<0$. The sequence $\{m_k\}_{k\in{\mathbb{Z}}}$ is 
strictly increasing. Indeed, for every $k\in \mathbb{Z}$, we have
\begin{eqnarray}\label{HS2}
m_k+1\leq m_{k+1}.
\end{eqnarray} 
To see this in the case that $k\geq 0$, observe that
\begin{eqnarray}\label{HS3}
2^{-k-1}=\tfrac{1}{2}2^{-k}\leq\tfrac{1}{2}\delta^{m_k}<\delta^{m_k+1},
\end{eqnarray} 
where in the first inequality we have used \eqref{HS1} and in the second 
that $1/2<\delta$. Thus, \eqref{HS2} holds, since by definition $m_{k+1}$ 
is the greatest integer for which $2^{-(k+1)}\leq\delta^{m_{k+1}}$.  
In the case $k\leq 0$, since $1<2\delta$ we have
\begin{eqnarray}\label{HS4}
\delta^{(m_{k+1}-1)+1}<2\delta^{m_{k+1}+1}<2\cdot2^{-(k+1)}=2^{-k}
\end{eqnarray} 
where in the second inequality we have used \eqref{HS1}.  
Since $m_k$ is the smallest integer for which $\delta^{m_k+1}<2^{-k}$, 
we again obtain \eqref{HS2}.

We then define
\begin{eqnarray}\label{HS5}
{\mathbb{D}}_k(E):=\mathfrak{D}_{m_k}(E).  
\end{eqnarray} 
It is routine to verify that ${\mathbb{D}}_k(E)$ satisfies the desired properties, 
with some of the constants possibly depending upon $\delta$.

\medskip

{\it Case~II: $0<\delta<1/2$.} In this case we reverse the roles of $1/2$ and 
$\delta$ in the construction in Case~I above, to construct a strictly increasing 
sequence of integers $\{m_k\}_{k\in{\mathbb{Z}}}$, with $m_0:=0$, for which
\begin{eqnarray}\label{2HHSS}
2^{-m_k}\leq \delta^{k}<2^{-m_k+1},\qquad\forall\,k\in{\mathbb{Z}}.
\end{eqnarray}  
It then follows that there is a fixed positive integer 
$q_0\approx\log_2(1/\delta)$ such that for each $k\in{\mathbb{Z}}$,
\begin{eqnarray}\label{3HHSS}
m_{k+1}-q_0\leq m_k<m_{k+1}.
\end{eqnarray}
Indeed, we have
\begin{eqnarray}\label{4HHSS}
2^{-m_k}\leq\delta^k =\frac1\delta\delta^{k+1}<\frac1\delta 2^{-m_{k+1}+1}
=\frac2\delta 2^{-m_{k+1}},
\end{eqnarray}
where in the two inequalities we have used \eqref{2HHSS}.   
We then obtain \eqref{3HHSS} by taking logarithms. For each $k\in\mathbb{Z}$ 
we now set
\begin{eqnarray}\label{6HHSS}
{\mathbb{D}}_j(E):=\mathfrak{D}_k(E),\qquad m_k\leq j<m_{k+1}.
\end{eqnarray}
It is now routine to check that the collection 
${\mathbb{D}}(E):=\bigcup\limits_{k\in{\mathbb{Z}}}{\mathbb{D}}_k(E)$, so defined,
satisfies the desired properties, with some of the constants possibly 
depending on $\delta$. In verifying the various properties, it is helpful to 
observe that by \eqref{3HHSS}, we have that 
\begin{eqnarray}\label{7HHSS}
2^{-j}\approx 2^{-m_k}\approx \delta^k,\qquad
\mbox{whenever $m_k\leq j<m_{k+1}$}.
\end{eqnarray}
This finishes the proof of the proposition.
\end{proof}

\subsection{Approximations to the identity on quasi-metric spaces}
\label{SSect:2.3}

This subsection is devoted to reviewing the definition and properties 
of approximations to the identity on {\rm ADR} spaces. To set the stage, we make the 
following definition. 

\begin{definition}\label{Besov-S}
Assume that $(E,\rho,\sigma)$ is a $d$-dimensional {\rm ADR} space 
for some $d>0$ and recall $\kappa_E\in{\mathbb{Z}}\cup\{-\infty\}$ from \eqref{Jcc-KKi}.
In this context, call a family $\{{\mathcal{S}}_l\}_{l\in{\mathbb{Z}},\,l\geq \kappa_E}$ 
of integral operators 
\begin{eqnarray}\label{Taga-6}
{\mathcal{S}}_l f(x):=\int_{E}S_l(x,y)f(y)\,d\sigma(y),\qquad x\in E,
\end{eqnarray}
with integral kernels $S_l:E\times E\to{\mathbb{R}}$, an 
{\tt approximation to the identity of order} $\gamma$ on $E$ 
provided there exists a finite constant $C>0$ such that, for every 
$l\in{\mathbb{Z}}$ with $l\geq \kappa_E$, the following properties hold:
\begin{enumerate}
\item[(i)] $0\leq S_l(x,y)\leq C 2^{ld}$ for all $x,y\in E$, and 
$S_l(x,y)=0$ if $\rho(x,y)\geq C2^{-l}$;
\item[(ii)] $|S_l(x,y)-S_l(x',y)|\leq C 2^{l(d+\gamma)}\rho(x,x')^{\gamma}$ 
for every $x,x',y\in E$;
\item[(iii)] $\bigl|[S_l(x,y)-S_l(x',y)]-[S_l(x,y')-S_l(x',y')]\bigr| 
\leq C2^{l(d+2\gamma)}\rho(x,x')^{\gamma}\rho(y,y')^{\gamma}$
for every point $x,x',y,y'\in E$;
\item[(iv)] $S_l(x,y)=S_l(y,x)$ for every $x,y\in E$,
and $\int_{E}S_l(x,y)\,d\sigma(y)=1$ for every $x\in E$.
\end{enumerate}
\end{definition}

Starting with the work of Coifman (cf. the discussion in 
\cite[pp.\,16-17 and p.\,40]{DJS}), the existence of approximations 
to the identity of some order $\gamma>0$ on {\rm ADR} spaces has been established in \cite[p.\,40]{DJS}, \cite[pp.\,10-11]{HaSa94}, \cite[p.\,16]{DeHa09} 
(at least when $d=1$) for various values of $\gamma>0$ and, more recently, 
in \cite{MMMM-G} for the value of the order parameter 
$\gamma$ which is optimal in relation to the quasi-metric space structure.
From \cite{MMMM-G}, we quote the following result:

\begin{proposition}\label{Besov-ST}
Let $(E,\rho,\sigma)$ be a $d$-dimensional {\rm ADR} space for 
some $d>0$ and assume that
\begin{eqnarray}\label{TFv-5tG}
0<\gamma<\min\bigl\{d+1,\alpha_\rho\bigr\},
\end{eqnarray}
where the index $\alpha_\rho\in(0,\infty]$ is associated to the quasi-distance 
$\rho$ as in \eqref{Cro}. Then, in the sense 
of Definition~\ref{Besov-S}, there exists an approximation 
to the identity of order $\gamma$ on $E$, denoted by $\{{\mathcal{S}}_l\}_{l\in {\mathbb{Z}},\,l\geq \kappa_E}$. Furthermore, given $p\in(1,\infty)$ and 
$f\in L^p(E,\sigma)$, it follows that: 
\begin{eqnarray}\label{fcc-5tnew}
\sup\limits_{l\in{\mathbb{Z}},\,l\geq \kappa_E}
\bigl\|{\mathcal{S}}_l\bigr\|_{L^p(E,\sigma)\to L^p(E,\sigma)}<+\infty,
\end{eqnarray}
\begin{eqnarray}\label{fcc-5t}
\mbox{if the measure $\sigma$ is Borel regular on $(E,\tau_{\rho})$}
\,\,\Longrightarrow\,\,
\lim_{l\to +\infty}{\mathcal{S}}_lf=f\,\,\,\mbox{ in $L^p(E,\sigma)$},
\end{eqnarray}
and 
\begin{eqnarray}\label{fcc-5t2}
\mbox{if } {\rm diam}_\rho(E)=+\infty\,\,\Longrightarrow\,\,
\lim_{l\to -\infty}{\mathcal{S}}_lf=0\,\,\,\mbox{ in $L^p(E,\sigma)$}.
\end{eqnarray}
\end{proposition}

Later on we shall need a Calder\'on-type reproducing formula involving the conditional expectation operators associated with an 
approximation to the identity, as discussed above. While this is a topic 
treated at some length in \cite{DJS}, \cite{DeHa09}, \cite{HaSa94}, we prove below
a version of this result which best suits the purposes we have in mind. 

To state the result, we first record the following preliminaries.

\begin{definition}\label{def:uncond} 
A series $\sum\limits_{j\in{\mathbb{N}}}x_j$ of vectors in a Banach space
${\mathscr{B}}$ is said to be {\tt unconditionally convergent} if the series
$\sum\limits_{j=1}^\infty x_{\sigma(j)}$ converges in ${\mathscr{B}}$ for 
all permutations $\sigma$ of ${\mathbb{N}}$.
\end{definition}
Clearly, if a series $\sum_{j\in{\mathbb{N}}}x_j$ of vectors in a Banach space ${\mathscr{B}}$ is unconditionally convergent then so is 
$\sum\limits_{j=1}^\infty x_{\sigma(j)}$ for any permutation $\sigma$ of ${\mathbb{N}}$.
It is also well-known (cf., e.g., \cite[Corollary~3.11, p.\,99]{Heil}) that, given a 
sequence of vectors $\{x_j\}_{j\in{\mathbb{N}}}$ in a Banach space
${\mathscr{B}}$, 
\begin{eqnarray}\label{UNC-11}
&& \hskip -0.50in
\mbox{$\displaystyle\sum\limits_{j\in{\mathbb{N}}}x_j$ unconditionally convergent}
\nonumber\\[-4pt]
&& \hskip 0.50in
\Longrightarrow\,
\sum\limits_{j=1}^\infty x_{\sigma_1(j)}=\sum\limits_{j=1}^\infty x_{\sigma_2(j)},
\qquad\mbox{ $\forall\,\sigma_1,\sigma_2$ permutations of ${\mathbb{N}}$}.
\end{eqnarray}
Hence, whenever $\sum\limits_{j\in{\mathbb{N}}}x_j$ is unconditionally convergent,
we may unambiguously define 
\begin{eqnarray}\label{UNC-11.BBB}
\sum\limits_{j\in{\mathbb{N}}}x_j:=\sum\limits_{j=1}^\infty x_{\sigma(j)}
\,\,\mbox{ for some (hence any) permutation $\sigma$ of ${\mathbb{N}}$}.
\end{eqnarray}
Let us also record here the following useful characterizations of 
unconditional convergence (in a Banach space setting):
\begin{eqnarray}\label{UNC-12}
\mbox{$\displaystyle\sum\limits_{j\in{\mathbb{N}}}x_j$ unconditionally convergent}
&\Longleftrightarrow &
\mbox{$\displaystyle\sum\limits_{j=1}^\infty\varepsilon_jx_j$ convergent}
\,\,\mbox{ $\forall\,\varepsilon_j=\pm 1$}
\\[4pt]
&\Longleftrightarrow &
\left\{
\begin{array}{l}
\forall\,\varepsilon>0\,\,\exists\,N_{\varepsilon}\in{\mathbb{N}}\mbox{ such that }
\Bigl\|\displaystyle\sum\limits_{j\in{\mathcal{I}}}x_j\Bigr\|<\varepsilon
\\[4pt]
\forall\,{\mathcal{I}}
\mbox{ finite subset of ${\mathbb{N}}$ with $\min\,{\mathcal{I}}\geq N_{\varepsilon}$}.
\end{array}
\right.
\nonumber
\end{eqnarray}
See, e.g., \cite[Theorem~3.10, p.\,94]{Heil} where these and other equivalent 
characterizations are proved.

The following notion of unconditional convergence applies to series indexed 
by any countable set other than ${\mathbb{N}}$.

\begin{definition}\label{def:uncond2}
For any countable set ${\mathbb{I}}$, a series $\sum_{j\in{\mathbb{I}}}x_j$ of vectors 
in a Banach space ${\mathscr{B}}$ is said to be {\tt unconditionally convergent} if 
there exists a bijection $\varphi:{\mathbb{N}}\rightarrow{\mathbb{I}}$ such that $\sum_{j\in{\mathbb{N}}}x_{\varphi(j)}$ is unconditionally convergent in the sense of
Definition~\ref{def:uncond}, in which case the sum of the series in ${\mathscr{B}}$ is 
defined as $\sum_{j\in{\mathbb{I}}}x_j:= \sum\limits_{j=1}^\infty x_{\varphi(j)}$.
\end{definition}

Note that the property of being unconditionally convergent as introduced in 
Definition~\ref{def:uncond2} is independent of the bijection ${\varphi}$ used. 
To see this, suppose that $\sum_{j\in{\mathbb{I}}}x_j$ is unconditionally convergent 
in ${\mathscr{B}}$ and let $\varphi:{\mathbb{N}}\rightarrow{\mathbb{I}}$ be a 
bijection such that $\sum_{j\in{\mathbb{N}}}x_{\varphi(j)}$ is unconditionally 
convergent in the sense of Definition~\ref{def:uncond}. If 
$\widetilde{\varphi}:{\mathbb{N}}\rightarrow{\mathbb{I}}$ is another bijection, then 
$\varphi^{-1}\circ\widetilde{\varphi}$ is a permutation of ${\mathbb{N}}$ hence,
as noted right after Definition~\ref{def:uncond},  
$\sum_{j\in{\mathbb{N}}}x_{\widetilde{\varphi}(j)}$
is also unconditionally convergent. Moreover, \eqref{UNC-11} ensures that 
$\sum\limits_{j=1}^\infty x_{\widetilde{\varphi}(j)}
=\sum\limits_{j=1}^\infty x_{\varphi(\varphi^{-1}(\widetilde{\varphi}(j)))}
=\sum\limits_{j=1}^\infty x_{\varphi(j)}
=\sum_{j\in{\mathbb{I}}}x_j$.

We also have the following equivalent characterization for unconditional convergence.

\begin{lemma}\label{SBFF}
Suppose ${\mathscr{B}}$ is a Banach space and ${\mathbb{I}}$ is a countable set. Then
a series $\sum_{j\in{\mathbb{I}}}x_j$ of vectors in ${\mathscr{B}}$ is unconditionally
convergent in ${\mathscr{B}}$ if and only if
\begin{eqnarray}\label{S+SS}
\begin{array}{c}
\forall\,\{S_i\}_{i\in{\mathbb{N}}}\mbox{ such that 
$S_i$ finite and $S_i\subseteq S_{i+1}\subseteq{\mathbb{I}}$ for each $i\in{\mathbb{N}}$},
\\[4pt]
\mbox{the sequence }\,\,\bigl\{\sum\limits_{j\in S_i}x_j\bigr\}_{i\in{\mathbb{N}}} 
\mbox{ converges in ${\mathscr{B}}$}.
\end{array}
\end{eqnarray}
\end{lemma}

\begin{proof}
Suppose $\sum_{j\in{\mathbb{I}}}x_j$ is such that \eqref{S+SS} holds and let
$\varphi:{\mathbb{N}}\to{\mathbb{I}}$ be a bijection. Also fix an arbitrary permutation
$\sigma:{\mathbb{N}}\to{\mathbb{N}}$. Then the sequence 
$S_i:=\{\varphi(\sigma(k)):\,1\leq k\leq i\}$, $i\in{\mathbb{N}}$, of subsets of
${\mathbb{I}}$ satisfies
the conditions in the first line of \eqref{S+SS}. Hence, 
$\bigl\{\sum\limits_{j=1}^i x_{\varphi(\sigma(j))}\bigr\}_{i\in{\mathbb{N}}}$ 
is convergent in ${\mathscr{B}}$, which is equivalent with $\sum\limits_{j=1}^\infty x_{\varphi(\sigma(j))}$ being convergent in ${\mathscr{B}}$. Since 
the permutation $\sigma$ of ${\mathbb{N}}$ has been arbitrarily chosen, 
this shows that $\sum_{j\in{\mathbb{N}}} x_{\varphi(j)}$ is
unconditionally convergent in ${\mathscr{B}}$, thus  $\sum_{j\in{\mathbb{I}}}x_j$ is unconditionally convergent in ${\mathscr{B}}$. For the converse implication, suppose 
that $\sum_{j\in{\mathbb{I}}}x_j$ is unconditionally convergent in ${\mathscr{B}}$. 
Thus, for any bijection $\varphi:{\mathbb{N}}\to{\mathbb{I}}$ we have that
$\sum_{j\in{\mathbb{N}}}x_{\varphi(j)}$ is unconditionally convergent in ${\mathscr{B}}$. 
Let $\{S_i\}_{i\in{\mathbb{N}}}$ be as in the first line of \eqref{S+SS} 
and set $S:=\bigcup\limits_{i\in{\mathbb{N}}}S_i$. Using \eqref{UNC-12}, it follows that 
$\sum\limits_{j\in{\mathbb{N}}\setminus\varphi^{-1}({\mathbb{I}}\setminus S)}x_{\varphi(j)}$
is also unconditionally convergent in ${\mathscr{B}}$. In turn, the latter readily 
implies that $\bigl\{\sum\limits_{j\in S_i}x_j\bigr\}_{i\in{\mathbb{N}}}$ is convergent in
${\mathscr{B}}$, as wanted. 
\end{proof}

We now state the aforementioned Calder\'on-type reproducing formula.

\begin{proposition}\label{HS-PP.3}
Let $(E,\rho,\sigma)$ be a $d$-dimensional {\rm ADR} space for some $d>0$ and 
assume that the measure $\sigma$ is Borel regular on $(E,\tau_\rho)$.
In this context, recall $\kappa_E$ from \eqref{Jcc-KKi} and, for some fixed $\gamma$ as 
in \eqref{TFv-5tG}, let $\{{\mathcal{S}}_l\}_{l\in{\mathbb{Z}},\,l\geq \kappa_E}$ be 
an approximation to the identity of order $\gamma$ on $E$ 
(cf. Proposition~\ref{Besov-ST}). Finally, introduce the integral operators 
(see \cite{DJS})
\begin{eqnarray}\label{opD}
D_l:={\mathcal{S}}_{l+1}-{\mathcal{S}}_l,\quad l\in{\mathbb{Z}},\,\,\,l\geq \kappa_E.
\end{eqnarray}

Then there exist a linear and bounded operator $R$ on $L^2(E,\sigma)$ and
a family $\bigl\{\widetilde{D}_l\bigr\}_{l\in{\mathbb{Z}},\,l\geq \kappa_E}$ 
of linear operators on $L^2(E,\sigma)$ with the property that
\begin{eqnarray}\label{PKc-2}
\sum\limits_{l\in{\mathbb{Z}},\,l\geq \kappa_E}\|\widetilde{D}_lf\|^2_{L^2(E,\sigma)}
\leq C\|f\|^2_{L^2(E,\sigma)},\quad\mbox{ for each }\,\,f\in L^2(E,\sigma),
\end{eqnarray}
and, with $I$ denoting the identity operator on $L^2(E,\sigma)$,
\begin{eqnarray}\label{PKc}
I+{\mathcal{S}}_{\kappa_E}R=\sum_{l\in{\mathbb{Z}},\,
l\geq \kappa_E}D_l\widetilde{D}_l
\qquad\mbox{pointwise unconditionally in $L^2(E,\sigma)$},
\end{eqnarray}
with the convention (taking effect when ${\rm diam}_\rho(E)=+\infty$) 
that ${\mathcal{S}}_{-\infty}:=0$.
\end{proposition}

As a preamble to the proof of the above proposition we momentarily digress 
and record a version of the Cotlar-Knapp-Stein lemma which suits our purposes.

\begin{lemma}\label{L-CKS}
Assume that ${\mathscr{H}}_0$, ${\mathscr{H}}_1$ are two Hilbert spaces 
and consider a family of operators $\{T_j\}_{j\in{\mathbb{I}}}$, indexed by 
a countable set ${\mathbb{I}}$, with $T_j:{\mathscr{H}}_0\to{\mathscr{H}}_1$ 
linear and bounded for every $j\in{\mathbb{I}}$. Then, if the $T_j$'s are almost 
orthogonal in the sense that
\begin{eqnarray}\label{Gvvv-42E}
C_0:=\sup_{j\in{\mathbb{I}}}\Bigl(\sum_{k\in{\mathbb{I}}}
\sqrt{\Vert T_j^\ast T_k\Vert_{{\mathscr{H}}_0\to{\mathscr{H}}_0}}\,\Bigr)<\infty,
\quad 
C_1:=\sup_{k\in{\mathbb{I}}}\Bigl(\sum_{j\in{\mathbb{I}}}
\sqrt{\Vert T_j T_k^\ast\Vert_{{\mathscr{H}}_1\to{\mathscr{H}}_1}}\,\Bigr)<\infty
\end{eqnarray}
it follows that for any subset $J$ of ${\mathbb{I}}$, 
\begin{eqnarray}\label{Gvvv-43E}
\begin{array}{c}
\mbox{$\displaystyle\sum_{j\in J}T_jx$ converges unconditionally in ${\mathscr{H}}_1$
for each $x\in{\mathscr{H}}_0$, and}
\\[6pt]
\mbox{if $\displaystyle\Bigl(\sum_{j\in J}T_j\Bigr)x:=\sum_{j\in J}T_jx$ then }\,\,
\displaystyle\Bigl\Vert\sum_{j\in J}T_j\Bigr\Vert_{{\mathscr{H}}_0\to{\mathscr{H}}_1}
\leq\sqrt{C_0C_1}. 
\end{array}
\end{eqnarray}
Furthermore, 
\begin{eqnarray}\label{Gvvv-43Ej}
\Bigl(\sum_{j\in{\mathbb{I}}}\|T_jx\|_{{\mathscr{H}}_1}^2\Bigr)^{1/2}
\leq 2\sqrt{C_0C_1}\|x\|_{{\mathscr{H}}_0},\qquad\forall\,x\in{\mathscr{H}}_0.
\end{eqnarray}
\end{lemma}

\begin{proof} 
This result is typically stated with $J$ finite and without including \eqref{Gvvv-43Ej}.
See, for example, \cite[Lemma~4.1, p.\,285]{Tor86} as well as
\cite[Theorem~1, p.280 and comment following it]{STEIN}.
The fact that the more general version formulated above 
holds is an immediate consequence of the standard version of the Cotlar-Knapp-Stein lemma
as stated in the aforementioned references
and the abstract functional analytic result contained in Lemma~\ref{Tgv-x99} below.
\end{proof}

\begin{lemma}\label{Tgv-x99}
Let ${\mathscr{H}}$ be a Hilbert space with norm $\|\cdot\|_{\mathscr{H}}$ and assume that $\{x_j\}_{j\in{\mathbb{I}}}$ is a sequence of vectors in ${\mathscr{H}}$ indexed by a countable set ${\mathbb{I}}$. Then 
\begin{eqnarray}\label{Hc-77.UU}
\Bigl(\sum\limits_{j\in{\mathbb{I}}}\|x_j\|_{\mathscr{H}}^2\Bigr)^{1/2}
\leq 2\cdot
\sup\limits_{\stackrel{J_o\subseteq{\mathbb{I}}}{J_o\,\mbox{\tiny{finite}}}}
\Bigl\|\sum\limits_{j\in J_o}x_j\Bigr\|_{\mathscr{H}},
\end{eqnarray}
and 
\begin{eqnarray}\label{FAFFF}
\mbox{$\displaystyle\sum\limits_{j\in{\mathbb{I}}}x_j$ is unconditionally convergent}\,
\Longleftrightarrow\,
\sup\limits_{\stackrel{J_o\subseteq{\mathbb{I}}}{J_o\,\mbox{\tiny{finite}}}}
\Bigl\|\sum\limits_{j\in J_o}x_j\Bigr\|_{\mathscr{H}}<\infty.
\end{eqnarray}
Moreover, if the above supremum is finite, then 
\begin{eqnarray}\label{TGbb-88u}
\Bigl\|\sum\limits_{j\in{\mathbb{I}}}x_j\Bigr\|_{\mathscr{H}}\leq 
\sup\limits_{\stackrel{J_o\subseteq{\mathbb{I}}}{J_o\,\mbox{\tiny{finite}}}}
\Bigl\|\sum\limits_{j\in J_o}x_j\Bigr\|_{\mathscr{H}}.
\end{eqnarray}
\end{lemma}

\begin{proof}
It suffices to assume that ${\mathbb{I}}={\mathbb{N}}$. This follows from Definition~\ref{def:uncond2}, since for any bijection $\varphi: {\mathbb{N}} \rightarrow {\mathbb{I}}$, we have
\begin{eqnarray}\label{eq:NIequiv}
\sup\limits_{\stackrel{J_o\subseteq{\mathbb{N}}}{J_o\,\mbox{\tiny{finite}}}}
\Bigl\|\sum\limits_{n\in J_o}x_{\varphi(n)}\Bigr\|_{\mathscr{H}}
=\sup\limits_{\stackrel{J_o\subseteq{\mathbb{I}}}{J_o\,\mbox{\tiny{finite}}}}
\Bigl\|\sum\limits_{j\in J_o}x_j\Bigr\|_{\mathscr{H}}.
\end{eqnarray}
We begin by establishing \eqref{Hc-77.UU}. To get started, 
let $\{x_j\}_{j\in{\mathbb{N}}}\subseteq{\mathscr{H}}$ be such that 
\begin{eqnarray}\label{Hc-77.A}
C:=\sup\limits_{\stackrel{J_o\subseteq{\mathbb{N}}}{J_o\,\mbox{\tiny{finite}}}}
\Bigl\|\sum\limits_{j\in J_o}x_j\Bigr\|_{\mathscr{H}}<\infty.
\end{eqnarray}
Assume that $\{r_j\}_{j\in{\mathbb{N}}}$ is Rademacher's system of 
functions on $[0,1]$, i.e., for each $j\in{\mathbb{N}}$,
\begin{eqnarray}\label{Hc-77.B1}
r_j(t)={\rm sign}\,\bigl(\sin(2^j\pi t)\bigr)\in\{-1,0,+1\},
\quad\mbox{ for all $t\in[0,1]$}.
\end{eqnarray}
Hence, in particular,  
\begin{eqnarray}\label{Hc-77.B2}
\int_0^1 r_j(t)r_k(t)\,dt=\delta_{jk},\qquad\forall\,j,k\in{\mathbb{N}}.
\end{eqnarray}
Consequently, if $\langle\cdot,\cdot\rangle_{\mathscr{H}}$ stands for 
the inner product in ${\mathscr{H}}$, then for any finite set 
$J_o\subseteq{\mathbb{N}}$, 
\begin{eqnarray}\label{Hc-77.C}
\int_0^1\Bigl\|\sum_{j\in J_o}r_j(t)x_j\Bigl\|^2_{\mathscr{H}}\,dt
&=& \int_0^1\Big\langle\sum_{j\in J_o}r_j(t)x_j,\sum_{k\in J_o}r_k(t)x_k
\Big\rangle_{\mathscr{H}}\,dt
\\[4pt]
&=& \sum_{j,k\in J_o}\Bigl(\int_0^1r_j(t)r_k(t)\,dt\Bigr)
\langle x_j,x_k\rangle_{\mathscr{H}}=\sum_{j\in J_o}\|x_j\|^2_{\mathscr{H}}.
\nonumber
\end{eqnarray}
On the other hand, thanks to \eqref{Hc-77.B1},
for each $t\in[0,1]$ we may estimate 
\begin{eqnarray}\label{Hc-77.D}
\Bigl\|\sum_{j\in J_o}r_j(t)x_j\Bigl\|_{\mathscr{H}}
&=&\Bigl\|\Bigl(\sum_{j\in J_o,\,r_j(t)=+1}x_j\Bigr)
-\Bigl(\sum_{j\in J_o,\,r_j(t)=-1}x_j\Bigr)\Bigl\|_{\mathscr{H}}
\nonumber\\[4pt]
&\leq & \Bigl\|\sum_{j\in J_o,\,r_j(t)=+1}x_j\Bigl\|_{\mathscr{H}}
+\Bigl\|\sum_{j\in J_o,\,r_j(t)=-1}x_j\Bigl\|_{\mathscr{H}}
\leq 2C.
\end{eqnarray}
By combining \eqref{Hc-77.C} and \eqref{Hc-77.D} we therefore obtain 
\begin{eqnarray}\label{Hc-77.E}
\sum_{j\in J_o}\|x_j\|^2_{\mathscr{H}}\leq 4C^2,
\qquad\mbox{for every finite subset $J_o$ of ${\mathbb{N}}$},
\end{eqnarray}
from which \eqref{Hc-77.UU} readily follows. 

Moving on, assume that \eqref{Hc-77.A} holds yet $\sum_{j\in{\mathbb{N}}}x_j$ 
does not converge unconditionally, and seek a contradiction.  
Then (cf. the first equivalence in \eqref{UNC-12}),
there exists a choice of signs $\varepsilon_j\in\{\pm 1\}$, $j\in{\mathbb{N}}$, 
with the property that the sequence of partial sums of the 
series $\sum_{j\in{\mathbb{N}}}\varepsilon_jx_j$ is not Cauchy in ${\mathscr{H}}$. 
In turn, this implies that there exist $\vartheta>0$ along with two sequences 
$\{a_i\}_{i\in{\mathbb{N}}}$, $\{b_i\}_{i\in{\mathbb{N}}}$ of numbers in ${\mathbb{N}}$,
such that
\begin{eqnarray}\label{Hc-77.G}
a_i\leq b_i<a_{i+1}\quad\mbox{ and }\quad
\Bigl\|\sum_{a_i\leq j\leq b_i}\varepsilon_j
x_j\Bigl\|_{\mathscr{H}}\geq\vartheta,
\quad\mbox{ for every $i\in{\mathbb{N}}$}.
\end{eqnarray}
In this scenario, consider the sequence $\{y_i\}_{i\in{\mathbb{N}}}$ of vectors 
in ${\mathscr{H}}$ defined by 
\begin{eqnarray}\label{Hc-77.H-1}
y_i:=\sum_{a_i\leq j\leq b_i}\varepsilon_j x_j
\,\,\mbox{ for every $i\in{\mathbb{N}}$},
\end{eqnarray}
and note that, by \eqref{Hc-77.G}, 
\begin{eqnarray}\label{Hc-77.H}
\|y_i\|_{\mathscr{H}}\geq\vartheta,\quad\mbox{ for every $i\in{\mathbb{N}}$}.
\end{eqnarray}
Fix now an arbitrary finite subset $I_o$ of ${\mathbb{N}}$ and set 
$J_o:=\bigl\{j\in{\mathbb{N}}:\,\exists\,i\in I_o\mbox{ such that }
a_i\leq j\leq b_i\bigr\}$. Thus, $J_o$ is a finite subset of ${\mathbb{N}}$.
Then with the constant $C$ as in \eqref{Hc-77.A}, we have
\begin{eqnarray}\label{Hc-77.Axxx-1}
\Bigl\|\sum\limits_{i\in I_o}y_i\Bigr\|_{\mathscr{H}}
&=& \Bigl\|\sum\limits_{i\in I_o}\Bigl(\sum_{a_i\leq j\leq b_i}
\varepsilon_j x_j\Bigr)\Bigr\|_{\mathscr{H}}
= \Bigl\|\Bigl(\sum\limits_{j\in J_o,\,\varepsilon_j=+1}x_j\Bigr)
-\Bigl(\sum\limits_{j\in J_o,\,\varepsilon_j=-1}x_j\Bigr)
\Bigr\|_{\mathscr{H}}
\nonumber\\[4pt]
&\leq & \Bigl\|\sum\limits_{j\in J_o,\,\varepsilon_j=+1}x_j\Bigr\|_{\mathscr{H}}
+\Bigl\|\sum\limits_{j\in J_o,\,\varepsilon_j=-1}x_j\Bigr\|_{\mathscr{H}}\leq 2C,
\end{eqnarray}
where the second equality relies on the fact from~\eqref{Hc-77.G} that $a_i\leq b_i < a_{i+1}$. Hence, 
\begin{eqnarray}\label{Hc-77.Axxx}
\sup\limits_{\stackrel{I_o\subseteq{\mathbb{N}}}{I_o\,\mbox{\tiny{finite}}}}
\Bigl\|\sum\limits_{i\in I_o}y_i\Bigr\|_{\mathscr{H}}\leq 2C.
\end{eqnarray}
Having established this, \eqref{Hc-77.UU} then gives 
$\sum_{i\in{\mathbb{N}}}\|y_i\|^2_{\mathscr{H}}\leq 16C^2<\infty$ which, in particular, 
forces $\lim\limits_{i\to\infty}\|y_i\|_{\mathscr{H}}=0$. This, however, contradicts 
\eqref{Hc-77.H}. 

To summarize, the proof so far shows that if \eqref{Hc-77.A} holds then 
the series $\sum_{j\in{\mathbb{N}}}x_j$ is unconditionally convergent. 
Of course, once the (norm) convergence of the series has been established 
then  \eqref{Hc-77.A} also gives
$\displaystyle\Bigl\|\sum\limits_{j\in{\mathbb{N}}}x_j\Bigr\|_{\mathscr{H}}
\leq\limsup\limits_{N\to\infty}\Bigl\|\sum\limits_{j=1}^N x_j\Bigr\|_{\mathscr{H}}
\leq C$, proving \eqref{TGbb-88u}.

There remains to prove that the finiteness condition in \eqref{Hc-77.A} holds 
if the series $\sum_{j\in{\mathbb{N}}}x_j$ is unconditionally convergent. 
With $N_1\in{\mathbb{N}}$ denoting the integer $N_{\varepsilon}$ corresponding 
to taking $\varepsilon=1$ in the last condition in \eqref{UNC-12}, consider 
\begin{eqnarray}\label{Hc-77.FV}
M:=\sup\limits_{I_o\subseteq\{1,...,N_1\}}
\Bigl\|\sum\limits_{j\in I_o}x_j\Bigr\|_{\mathscr{H}}<\infty.
\end{eqnarray}
Then, given any finite subset $J_o$ of ${\mathbb{N}}$ we may write 
\begin{eqnarray}\label{Hc-77.FV2}
\Bigl\|\sum\limits_{j\in J_o}x_j\Bigr\|_{\mathscr{H}}
\leq\Bigl\|\sum\limits_{j\in J_o\cap\{1,...,N_1\}}x_j\Bigr\|_{\mathscr{H}}
+\Bigl\|\sum\limits_{j\in J_o\setminus\{1,...,N_1\}}x_j\Bigr\|_{\mathscr{H}}
\leq M+1,
\end{eqnarray}
from which the finiteness condition in \eqref{Hc-77.A} follows.
\end{proof}

For further reference, given an ambient quasi-metric space $({\mathscr{X}},\rho)$ 
and a set $E$ with the property that there exists a Borel measure $\sigma$ on 
$(E,\tau_{\rho|_{E}})$ such that $(E,\rho_{\#}|_{E},\sigma)$ is a space of 
homogeneous type, we shall denote by $M_E$ the {\tt Hardy-Littlewood maximal 
function} in this context, i.e., 
\begin{eqnarray}\label{HL-MAX}
(M_E f)(x):=\sup_{r>0}\frac{1}{\sigma\bigl(B_{\rho_{\#}}(x,r)\bigr)}
\int_{B_{\rho_{\#}}(x,r)}|f(y)|\,d\sigma(y),\qquad x\in E.
\end{eqnarray}

We next present the 

\vskip 0.08in
\begin{proof}[Proof of Proposition~\ref{HS-PP.3}]
For each $l\in{\mathbb{Z}}$ with $l\geq \kappa_E$, denote by $h_l(\cdot,\cdot)$ 
the integral kernel of the operator $D_l$. Thus, 
$h_l(\cdot,\cdot)=S_{l+1}(\cdot,\cdot)-S_l(\cdot,\cdot)$ and, 
as a consequence of properties $(i)-(iv)$ in Definition~\ref{Besov-S}, we see 
that $h_l(\cdot,\cdot)$ is a symmetric function on $E\times E$, and there 
exists $C\in(0,\infty)$ such that for each $l\in{\mathbb{Z}}$ with $l\geq \kappa_E$ we have 
\begin{eqnarray}\label{condh-1}
&& |h_l(\cdot,\cdot)|\leq C2^{\,ld}{\mathbf{1}}_{\{\rho(\cdot,\cdot)\leq C2^{-l}\}},
\qquad\mbox{ on }\,\,E\times E,
\\[4pt]
&&|h_l(x,y)-h_l(x',y)|\leq C2^{\,l(d+\gamma)}\rho(x,x')^\gamma
\quad\mbox{ $\forall\, x,x',y\in E$},
\label{condh-2}
\\[4pt]
&& \int_Eh_l(x,y)\,d\sigma(x)=0\quad\forall\,y\in E.
\label{condh-4}
\end{eqnarray}
Of course, due to the symmetry of $h$, smoothness and cancellation conditions in the 
second variable, similar to \eqref{condh-2} and \eqref{condh-4}, respectively, also hold.

Furthermore, for each $j,k\in{\mathbb{Z}}$ with $j,k\geq \kappa_E$, using first 
\eqref{condh-4}, then \eqref{condh-1} and \eqref{condh-2}, and then the fact that 
$(E,\rho,\sigma)$ is $d$-ADR, we may write
\begin{eqnarray}\label{TT-sH}
\left|\int_{E}h_j(x,z)h_k(z,y)\,d\sigma(z)\right|
& = & \left|\int_{E}[h_j(x,z)-h_j(x,y)]h_k(z,y)\,d\sigma(z)\right|
\nonumber\\[4pt]
& \leq & C2^{j(d+\gamma)}\int_E\rho_{\#}(y,z)^\gamma 2^{kd}
{\mathbf{1}}_{\{\rho_{\#}(y,\cdot)\leq C2^{-k}\}}(z)\,d\sigma(z)
\nonumber\\[4pt]
& \leq & C2^{j(d+\gamma)}2^{-k\gamma}.
\end{eqnarray}
Similarly, 

\begin{eqnarray}\label{TT-sH-2}
\left|\int_{E}h_j(x,z)h_k(z,y)\,d\sigma(z)\right|
& = & \left|\int_{E}h_j(x,z)[h_k(z,y)-h_k(x,y)]\,d\sigma(z)\right|
\nonumber\\[4pt]
& \leq & C2^{k(d+\gamma)}2^{-j\gamma}.
\end{eqnarray}
Combining \eqref{TT-sH}, \eqref{TT-sH-2}, and the support condition \eqref{condh-1},
it follows that for each $j,k\in{\mathbb{Z}}$ with $j,k\geq \kappa_E$, there holds
(compare with \cite[p.\,15]{DJS} and \cite[(1.14), p.\,16]{DeHa09})
\begin{eqnarray}\label{hvs-RT}
\left|\int_{E}h_j(x,z)h_k(z,y)\,d\sigma(z)\right|\leq C2^{-|j-k|\gamma}\,
2^{d\cdot\min(j,k)}{\mathbf{1}}_{\{\rho(x,y)\leq C2^{-\min(j,k)}\}},
\quad\forall\,x,y\in E.
\end{eqnarray}
Note that for each $j,k\in{\mathbb{Z}}$ with 
$j,k\geq\kappa_E$ we have that $D_jD_k:L^2(E,\sigma)\to L^2(E,\sigma)$ is 
a linear and bounded integral operator whose integral kernel is given 
by $\int_{E}h_j(x,z)h_k(z,y)\,d\sigma(z)$, for $x,y\in E$. Based on this and 
\eqref{hvs-RT} we may then conclude that for each $j,k\in{\mathbb{Z}}$ with 
$j,k\geq \kappa_E$, 
\begin{eqnarray}\label{hvs-Raaa}
\bigl|(D_jD_kf)(x)\bigr| &\leq & C2^{-|j-k|\gamma}
\meanint_{B_{\rho_{\#}}(x,C2^{-\min(j,k)})}|f(y)|\,d\sigma(y)
\nonumber\\[4pt]
&\leq & C2^{-|j-k|\gamma}M_E(f)(x),\qquad\forall\,x\in E,
\end{eqnarray}
for every $f\in L^1_{loc}(E,\sigma)$. In turn, the boundedness of $M_E$ and
\eqref{hvs-Raaa} yield 
\begin{eqnarray}\label{hvs-RTS}
\|D_jD_k\|_{L^2(E,\sigma)\to L^2(E,\sigma)}\leq C2^{-|j-k|\gamma},
\qquad\forall\,j,k\in{\mathbb{Z}},\,\,j,k\geq\kappa_E.
\end{eqnarray}

Having established \eqref{hvs-RTS}, it follows that the family of linear operators
$\bigl\{D_l\bigr\}_{l\in{\mathbb{Z}},\,l\geq\kappa_E}$, from $L^2(E,\sigma)$ into 
itself, is almost orthogonal. As such, Lemma~\ref{L-CKS} applies and gives that
\begin{eqnarray}\label{fcc-5tNN2}
\sup\limits_{\stackrel{J\subseteq{\mathbb{Z}}}{J\,\mbox{\tiny{finite}}}}
\Bigl\|\sum_{l\in J,\,l\geq\kappa_E}D_l\Bigr\|_{L^2(E,\sigma)\to L^2(E,\sigma)}
\leq C<\infty,
\end{eqnarray}
the following Littlewood-Paley estimate holds
\begin{eqnarray}\label{FC+MN}
\Bigl(\sum\limits_{l\in{\mathbb{Z}},\,l\geq \kappa_E}\|D_lf\|^2_{L^2(E,\sigma)}
\Bigr)^{1/2}\leq C\|f\|_{L^2(E,\sigma)},\quad\mbox{ for each }\,\,f\in L^2(E,\sigma),
\end{eqnarray}
and, making use of \eqref{fcc-5t} and \eqref{fcc-5t2} as well, we have
\begin{eqnarray}\label{fcc-5tNN1}
\begin{array}{l}
\bigl(I-{\mathcal{S}}_{\kappa_E}\bigr)f
=\displaystyle\sum\limits_{l\in{\mathbb{Z}},\,l\geq \kappa_E}D_l f
\quad\mbox{ for each $f\in L^2(E,\sigma)$},
\\[4pt]
\mbox{where the series converges unconditionally in $L^2(E,\sigma)$}.
\end{array}
\end{eqnarray}

To proceed, fix a number $N\in{\mathbb{N}}$. Based on \eqref{fcc-5tNN2}, we may square \eqref{fcc-5tNN1} and obtain, pointwise in $L^2(E,\sigma)$,
\begin{eqnarray}\label{fcc-5A.1}
\bigl(I-{\mathcal{S}}_{\kappa_E}\bigr)^2 
&=& \lim_{M\to\infty}
\Bigl[\Bigl(\sum_{j\in{\mathbb{Z}},\,j\geq\kappa_E,\,|j|\leq M}D_j\Bigr)
\Bigl(\sum_{k\in{\mathbb{Z}},\,k\geq \kappa_E,\,|k|\leq M}D_k\Bigr)\Bigr]
\nonumber\\[4pt]
&=& \lim_{M\to\infty}\Bigl(\sum\limits_{\stackrel{|j-k|\leq N}{\,j,k\geq \kappa_E,\,|j|,|k|
\leq M}}D_jD_k
+\sum\limits_{\stackrel{|j-k|>N}{j,k\geq \kappa_E,\,|j|,|k|\leq M}}D_jD_k\Bigr).
\end{eqnarray}
Going further, fix $i\in{\mathbb{Z}}$ and consider the 
family $\{T_l\}_{l\in J_i}$ of operators on $L^2(E,\sigma)$, where 
\begin{eqnarray}\label{T-GBn89}
T_l:=D_{l+i}D_l\quad\mbox{ for every }\quad
l\in J_i:=\bigl\{l\in{\mathbb{Z}}:\,l\geq\max\{\kappa_E,\kappa_E-i\}\bigr\}.
\end{eqnarray}
Then, with $\|\cdot\|$ temporarily abbreviating 
$\|\cdot\|_{L^2(E,\sigma)\to L^2(E,\sigma)}$, for each $j,k\in J_i$ we may estimate 
\begin{eqnarray}\label{T-77-AB}
\|T_j^\ast T_k\| &\leq &\min\,\Bigl\{\|D_j\|\|D_{j+i}D_{k+i}\|\|D_k\|\,,\,
\|D_jD_{j+i}\|\|D_{k+i}\|\|D_k\|\Bigr\}
\nonumber\\[4pt]
&\leq & C\,\min\,\Bigl\{2^{-|k-j|\gamma}\,,\,2^{-|i|\gamma}\Bigr\},
\end{eqnarray}
thanks to \eqref{fcc-5tNN2} and \eqref{hvs-RTS}. 
This readily implies that $\sup\limits_{j\in J_i}\Bigl(\sum\limits_{k\in J_i}
\sqrt{\|T_j^\ast T_k\|}\Bigr)\leq C(1+|i|)2^{-|i|\gamma/2}$
and $\sup\limits_{k\in J_i}\Bigl(\sum\limits_{j\in J_i} \sqrt{\|T_j^{\phantom{\ast}}T_k^\ast \|}\Bigr)\leq C(1+|i|)2^{-|i|\gamma/2}$ for some $C\in(0,\infty)$ 
independent of $i$. Hence, for each $i\in{\mathbb{Z}}$, the family 
$\bigl\{D_{l+i}D_l\bigr\}_{l\in{\mathbb{Z}},\,l\geq\max\{\kappa_E,\kappa_E-i\}}$ is 
almost orthogonal, and by Lemma~\ref{L-CKS} there exists some constant $C\in(0,\infty)$ independent 
of $i$ such that for every set $J\subseteq J_i$ we have that
$\sum\limits_{l\in J}D_{l+i}D_l$ converges pointwise 
unconditionally in $L^2(E,\sigma)$ and 
\begin{eqnarray}\label{T-GBn88}
\Bigl\|\sum\limits_{l\in J}D_{l+i}D_l
\Bigr\|_{L^2(E,\sigma)\to L^2(E,\sigma)}\leq C(1+|i|)2^{-|i|\gamma/2}.
\end{eqnarray} 
Next, fix $N\in{\mathbb{N}}$ and let ${\mathscr{I}}$ be an arbitrary finite subset of
$\{(l,m)\in{\mathbb{Z}}\times{\mathbb{Z}}:\,l,\,m\geq\kappa_E\}$. Then for each function $f\in L^2(E,\sigma)$ with $\|f\|_{L^2(E,\sigma)}=1$, using \eqref{T-GBn88} we may estimate
\begin{eqnarray}\label{RGbb-6Y}
&& \hskip -0.20in
\Bigl\|\sum_{(j,k)\in{\mathscr{I}},\,|j-k|>N}D_jD_kf\Bigr\|_{L^2(E,\sigma)}
=\Bigl\|\sum_{i\in{\mathbb{Z}},\,|i|>N}
\Bigl(\sum_{l\in{\mathbb{Z}},\,(l+i,l)\in{\mathscr{I}}}
D_{l+i}D_l f\Bigr)\Bigr\|_{L^2(E,\sigma)}
\\[4pt]
&& \hskip 0.15in
\leq\sum_{i\in{\mathbb{Z}},\,|i|>N}
\Bigl\|\sum_{l\in{\mathbb{Z}},\,(l+i,l)\in{\mathscr{I}}}
D_{l+i}D_l f\Bigr\|_{L^2(E,\sigma)} \leq  \sum_{i\in{\mathbb{Z}},\,|i|>N}C(1+|i|)2^{-|i|\gamma/2} \leq C_\gamma N2^{-N\gamma/2},
\nonumber
\end{eqnarray}
for some finite constant $C_\gamma>0$ which is independent of $N$. In turn, based on \eqref{FAFFF}, \eqref{TGbb-88u} and \eqref{RGbb-6Y} we deduce that   
\begin{eqnarray}\label{dec-id2}
\begin{array}{c}
R_N:=\displaystyle\sum\limits_{\stackrel{|j-k|>N}{j,k\geq \kappa_E}}D_jD_k
\,\,\mbox{ converges pointwise unconditionally in $L^2(E,\sigma)$, and}
\\[4pt]
\mbox{there exists $C_\gamma\in(0,\infty)$ such that }\,\,
\|R_N\|_{L^2(E,\sigma)\to L^2(E,\sigma)}\leq C_\gamma N2^{-N\gamma/2}.
\end{array} 
\end{eqnarray} 
In a similar fashion to \eqref{T-GBn88}-\eqref{dec-id2}, we may also deduce that 
\begin{eqnarray}\label{de-FFws}
T_N:=\sum\limits_{\stackrel{|j-k|\leq N}{j,k\geq \kappa_E}}D_jD_k
\,\,\mbox{ converges pointwise unconditionally in $L^2(E,\sigma)$}.
\end{eqnarray} 
Consequently, if we now set 
\begin{eqnarray}\label{dTF-gAA}
D_l^N:=\sum\limits_{\stackrel{i\in{\mathbb{Z}},\,|i|\leq N}{i\geq \kappa_E-l}}D_{l+i},
\qquad\mbox{for each }\,\,l\in{\mathbb{Z}},
\end{eqnarray}
then (cf. \eqref{UNC-11}) the series $T_N$ may be rearranged as 
\begin{eqnarray}\label{dec-id1}
T_N =\sum\limits_{l\in{\mathbb{Z}},\,l\geq \kappa_E}D_lD_l^N, 
\end{eqnarray} 
where the sum converges pointwise unconditionally in $L^2(E,\sigma)$. Combining \eqref{fcc-5A.1}, \eqref{dec-id2} and \eqref{de-FFws}, we arrive at 
the identity
\begin{eqnarray}\label{dec-id1A}
\bigl(I-{\mathcal{S}}_{\kappa_E}\bigr)^2=R_N+T_N\quad\mbox{on }\,\,L^2(E,\sigma),
\end{eqnarray}
which is convenient to further re-write as
\begin{eqnarray}\label{dec-id1ABIS}
I=R_N+\widetilde{T}_N\quad\mbox{on }\,\,L^2(E,\sigma),\qquad
\mbox{where }\quad
\widetilde{T}_N:=T_N+{\mathcal{S}}_{\kappa_E}\bigl(2I-{\mathcal{S}}_{\kappa_E}\bigr).
\end{eqnarray}
Thanks to the estimate in \eqref{dec-id2}, it follows from \eqref{dec-id1ABIS} that  
\begin{eqnarray}\label{dec-id1B}
\mbox{$\widetilde{T}_N:L^2(E,\sigma)\to L^2(E,\sigma)$ is boundedly invertible
for $N\in{\mathbb{N}}$ sufficiently large}.
\end{eqnarray}
Hence, for $N$ sufficiently large and fixed, based on \eqref{dec-id1B} 
we may write that $I=\widetilde{T}_N(\widetilde{T}_N)^{-1}$, and keeping in mind  
\eqref{dec-id1} and \eqref{dec-id1ABIS}, we arrive at the following 
Calder\'on-type reproducing formula
\begin{eqnarray}\label{PKc.TG}
I=\Bigl(\sum_{l\in{\mathbb{Z}},\,l\geq \kappa_E}D_l\widetilde{D}_l\Bigr)
+{\mathcal{S}}_{\kappa_E}\bigl(2I-{\mathcal{S}}_{\kappa_E}\bigr)(\widetilde{T}_N)^{-1},
\end{eqnarray}
where the sum converges pointwise unconditionally in $L^2(E,\sigma)$, and
\begin{eqnarray}\label{egFvo}
\widetilde{D}_l:=D_l^N(\widetilde{T}_N)^{-1},\qquad\forall\,l\in{\mathbb{Z}}
\,\,\mbox{ with }\,\,l\geq \kappa_E.
\end{eqnarray}
From this \eqref{PKc} follows with
$R:=\bigl({\mathcal{S}}_{\kappa_E}-2I\bigr)(\widetilde{T}_N)^{-1}$.
Finally, \eqref{PKc-2} is a consequence of \eqref{egFvo}, the fact that the 
sum in \eqref{dTF-gAA} has a finite number of terms, \eqref{dec-id1B} and 
\eqref{FC+MN}.
\end{proof}

\subsection{Dyadic Carleson tents}
\label{SSect:2.4}

Suppose that $({\mathscr{X}},\rho)$ is a geometrically doubling quasi-metric space 
and that $E$ is a nonempty, closed, proper subset of $({\mathscr{X}},\tau_\rho)$. 
It follows from the discussion below Definition~\ref{Gd_ZZ} that $(E,\rho\bigl|_E)$ is also a geometrically doubling quasi-metric space. 
We now introduce dyadic Carleson tents in this setting. These are sets in 
${\mathscr{X}}\setminus E$ that are adapted to $E$ in the same way that classical Carleson boxes or tents in the upper-half space $\mathbb{R}^{n+1}_+$ are adapted to $\mathbb{R}^n$. 
We require a number of preliminaries before we introduce these sets in \eqref{gZSZ-3} below. First, fix a collection ${\mathbb{D}}(E)$ of dyadic cubes contained in $E$ as in Proposition~\ref{Diad-cube}. Second, choose $\lambda\in[2C_\rho,\infty)$ and fix a Whitney covering ${\mathbb{W}}_\lambda({\mathscr{X}}\setminus E)$ of balls contained in ${\mathscr{X}}\setminus E$ as in Proposition~\ref{H-S-Z}. Following Convention~\ref{WWVc}, we refer to these $\rho_{\#}$-balls as Whitney cubes, and for each $I\in{\mathbb{W}}_\lambda({\mathscr{X}}\setminus E)$, we use the notation $\ell(I)$ for the radius of $I$. Third, choose $C_\ast\in[1,\infty)$, and for each $Q\in{\mathbb{D}}(E)$, define the following collection of Whitney cubes:
\begin{equation}\label{gZSZa}
W_Q:=\{I\in{\mathbb{W}}_\lambda({\mathscr{X}}
\setminus E):\,C_\ast^{-1}\ell(I)\leq\ell(Q)\leq C_\ast\ell(I)\mbox{ and }
{\rm dist}_\rho(I,Q)\leq\ell(Q)\}.
\end{equation}
Fourth, for each $Q\in{\mathbb{D}}(E)$, define the following subset of $({\mathscr{X}},\tau_{\rho})$:
\begin{eqnarray}\label{gZSZb}
{\mathcal{U}}_Q:=\bigcup\limits_{I\in W_Q}I.
\end{eqnarray}
Since from Theorem~\ref{JjEGh} we know that the regularized quasi-distance 
$\rho_\#$ is continuous, it follows that the $\rho_\#$-balls are open. 
As such, that each $I$ in $W_Q$, hence ${\mathcal{U}}_Q$ itself, is open.

Finally, for each $Q\in{\mathbb{D}}(E)$, the {\tt dyadic Carleson tent} $T_E(Q)$ 
{\tt over} $Q$ is defined as follows:
\begin{eqnarray}\label{gZSZ-3}
T_E(Q):=\bigcup_{Q'\in{\mathbb{D}}(E),\,\,Q'\subseteq Q}{\mathcal{U}}_{Q'}.
\end{eqnarray}

For most of the subsequent work we will assume that the Whitney covering 
${\mathbb{W}}_\lambda({\mathscr{X}}\setminus E)$ and the constant $C_\ast$ 
are chosen as in the following lemma. 

\begin{lemma}\label{Lem:CQinBQ-N}
Let $({\mathscr{X}},\rho)$ be a geometrically doubling quasi-metric space and 
suppose that $E$ is a nonempty, closed, proper subset of $({\mathscr{X}},\tau_\rho)$. 
Fix a collection ${\mathbb{D}}(E)$ of dyadic cubes in $E$ as in
Proposition~\ref{Diad-cube}. Next, choose $\lambda\in[2C_\rho,\infty)$, 
fix a Whitney covering ${\mathbb{W}}_\lambda({\mathscr{X}}\setminus E)$ of
${\mathscr{X}}\setminus E$, and let $\Lambda$ denote the constant associated 
with $\lambda$ as in Proposition~\ref{H-S-Z}. Finally, choose
\begin{eqnarray}\label{NeD-67}
C_\ast \in [4C_\rho^4\,\Lambda,\infty),
\end{eqnarray}
and define the collection $\{{\mathcal{U}}_Q\}_{Q\in{\mathbb{D}}(E)}$ associated with ${\mathbb{W}}_\lambda({\mathscr{X}}\setminus E)$ and $C_\ast$ as in \eqref{gZSZa}-\eqref{gZSZb}.

Then there exists $\epsilon\in(0,1)$, depending only on $\lambda$ 
and geometry, with the property that
\begin{eqnarray}\label{doj.cF}
\bigl\{x\in{\mathscr{X}}\setminus E:\,\delta_E(x)<\epsilon\,{\rm diam}_{\rho}(E)\bigr\}
\subseteq\bigcup\limits_{Q\in{\mathbb{D}}(E)}{\mathcal{U}}_Q.
\end{eqnarray} 
\end{lemma}
\begin{proof}
If ${\rm diam}_{\rho}(E) = \infty$, then both sides of \eqref{doj.cF} are equal to ${\mathscr{X}}\setminus E$ for all $\epsilon\in(0,1)$, since the Whitney cubes cover ${\mathscr{X}}\setminus E$, so the result is immediate. Now assume that ${\rm diam}_{\rho}(E) < \infty$. Fix some integer $N\in{\mathbb{N}}$, to be specified later, and consider an arbitrary
point $x\in{\mathscr{X}}\setminus E$ with $\delta_E(x)<2^{-N}{\rm diam}_{\rho}(E)$. 
Then by \eqref{topoQMS} and \eqref{Jcc-KKi} we have $0<\delta_E(x)<2^{-N-\kappa_E}$, hence there 
exists $k\in{\mathbb{Z}}$ with $k\geq\kappa_E$ such 
that $2^{-N-k-1}\leq\delta_E(x)<2^{-N-k}$. Now, select a ball 
$I=B_{\rho_{\#}}(x_I,\ell(I))\in{\mathbb{W}}_\lambda({\mathscr{X}}\setminus E)$ 
such that $x\in I$. Then, by {\it (3)} in Proposition~\ref{H-S-Z}, there exists 
$z\in E$ such that $\rho_{\#}(x_I,z)<\Lambda\ell(I)$. Consequently, 
\begin{eqnarray}\label{jaf-UU.1}
\delta_{E}(x)\leq\rho_{\#}(x,z)\leq C_\rho\,\max\,\{\rho_{\#}(x,x_I),\rho_{\#}(x_I,z)\}
<C_\rho\Lambda\ell(I).
\end{eqnarray}
In addition, {\it (3)} in Proposition~\ref{H-S-Z} also gives 
that $B_{\rho_{\#}}(x_I,\lambda\ell(I))\subseteq{\mathscr{X}}\setminus E$ and, hence, 
for every $y\in E$ 
\begin{eqnarray}\label{jaf-UU.2}
2C_\rho\ell(I) &\leq & \lambda\ell(I)\leq\rho_{\#}(x_I,y)
\leq C_\rho\,\rho_{\#}(x_I,x)+C_\rho\,\rho_{\#}(x,y)
\nonumber\\[4pt]
&\leq & C_\rho\,\ell(I)+C_\rho\,\rho_{\#}(x,y).
\end{eqnarray}
After canceling like-terms in the most extreme sides of \eqref{jaf-UU.2} and 
taking the infimum over all $y\in E$, we arrive at 
\begin{eqnarray}\label{jaf-UU.3}
\ell(I)\leq\delta_{E}(x).
\end{eqnarray}
Next, since $\delta_E(x)<2^{-N-k}$, there exists $x_0\in E$ such that 
$\rho_{\#}(x,x_0)<2^{-N-k}$. Furthermore, by invoking {\it (7)} 
in Proposition~\ref{Diad-cube} we may choose $Q\in{\mathbb{D}}_k(E)$ 
with the property that $B_{\rho_{\#}}(x_0,2^{-N-k})\cap Q$ contains at least 
one point $x_1$. Thus, by \eqref{DEQV1} we have
\begin{eqnarray}\label{jaf-UU.4}
{\rm dist}_\rho(I,Q) &\leq &{\rm dist}_\rho(x,Q)\leq\rho(x,x_1)
\leq C_\rho^2\rho_{\#}(x,x_1)
\\[4pt]
&\leq & {C_\rho\!\!\!\!\phantom{.}^2}C_{\rho_{\#}}\max\,\{\rho_{\#}(x,x_0),\rho_{\#}(x_0,x_1)\}
<C_\rho^3 2^{-N-k}=C_\rho^3 2^{-N}\ell(Q).
\nonumber
\end{eqnarray}
This shows that 
\begin{eqnarray}\label{jaf-UU.5}
2^N>C_\rho^3\,\Longrightarrow\,
{\rm dist}_\rho(I,Q)\leq\ell(Q).
\end{eqnarray}
Starting with \eqref{jaf-UU.3} and keeping in mind that $\delta_E(x)<2^{-N-k}$, 
we obtain 
\begin{eqnarray}\label{jaf-UU.6}
\ell(I)<2^{-N-k}=2^{-N}\ell(Q)\leq\ell(Q).
\end{eqnarray}
Finally, with the help of \eqref{jaf-UU.1} we write 
$2^{-N-1}\ell(Q)=2^{-N-k-1}\leq\delta_E(x)\leq C_\rho\Lambda\ell(I)$
which further entails
\begin{eqnarray}\label{jaf-UU.8}
C_\ast\geq 2^{N+1}C_\rho\Lambda\,\Longrightarrow\,
\ell(I)\geq C_\ast^{-1}\ell(Q).
\end{eqnarray}
At this stage, by choosing $N\in\mathbb{N}$ such that
\begin{eqnarray}\label{jaf-UU.7}
N-1 \leq \log_2\bigl(C_\rho^3\bigr) < N,
\end{eqnarray}
we may conclude from \eqref{jaf-UU.5}, \eqref{jaf-UU.6} and \eqref{jaf-UU.8}
that $I\in W_Q$ when $C_\ast\geq 4 C_\rho^4\,\Lambda$. This, in turn, 
forces $x\in I\subseteq{\mathcal{U}}_Q$. Taking $\epsilon:=2^{-N}$ with 
$N$ as in \eqref{jaf-UU.7} then justifies \eqref{doj.cF},
and finishes the proof of the lemma.
\end{proof}

We now return to the context introduced in the first paragraph of this subsection, and in particular, where $\lambda\in[2C_\rho,\infty)$ and $C_\ast\in[1,\infty)$. For further reference, we note that then there exists $C_o\in[1,\infty)$ such that 
\begin{eqnarray}\label{UUU-rf}
C_o^{-1}\ell(Q)\leq\delta_E(x)\leq C_o\ell(Q),\qquad\forall\,Q\in{\mathbb{D}}(E)
\,\mbox{ and }\,\forall\,x\in{\mathcal{U}}_Q.
\end{eqnarray}
Indeed, an inspection of \eqref{jaf-UU.1}, \eqref{jaf-UU.3}, \eqref{gZSZa} and \eqref{gZSZb}  
shows that \eqref{UUU-rf} holds when
\begin{eqnarray}\label{UUU-rf-N} 
C_o:=C_\ast C_\rho\Lambda,
\end{eqnarray}
where $\Lambda$ is the constant associated with $\lambda$ as in Proposition~\ref{H-S-Z}.

The reader should be aware of the fact that even when \eqref{doj.cF} holds 
it may happen that some ${\mathcal{U}}_Q$'s are empty. However, under the 
assumption that $({\mathscr{X}},\rho,\mu)$ is an $m$-dimensional {\rm ADR} space 
and granted the existence of a measure $\sigma$ such that $(E,\rho|_{E},\sigma)$ 
becomes a $d$-dimensional {\rm ADR} space for some $d\in(0,m)$, matters may be arranged
so that this eventuality never materializes. In particular, if $C_\ast$ is large enough (depending on $\lambda$ and geometry), then ${\mathcal{U}}_Q\not=\emptyset$ 
for all $Q\in{\mathbb{D}}(E)$. This is a simple consequence of the 
following lemma, which is proved in \cite{MMMM-B}.

\begin{lemma}\label{FR-DF-4}
Let $({\mathscr{X}},\rho,\mu)$ be an $m$-dimensional {\rm ADR} space, for some $m>0$, 
and assume that $E$ is a closed subset of $({\mathscr{X}},\tau_\rho)$ with the 
property that there exists a measure $\sigma$ on $E$ for which $(E,\rho|_{E},\sigma)$ 
is a $d$-dimensional {\rm ADR} space for some $d\in(0,m)$. 

Then there exists 
$\vartheta\in(0,1)$ such that for each $x_0\in{\mathscr{X}}$ and each finite 
$r\in(0,{\rm diam}_\rho({\mathscr{X}})]$ one may find $x\in{\mathscr{X}}$ 
with the property that $B_\rho(x,\vartheta r)\subseteq B_\rho(x_0,r)\setminus E$. 
\end{lemma}

We again return to the context introduced in the first paragraph of this subsection, and in particular, where $\lambda\in[2C_\rho,\infty)$ and $C_\ast\in[1,\infty)$. For each $Q\in{\mathbb{D}}(E)$, recall the dyadic Carleson tent $T_E(Q)$ over $Q$ from~\eqref{gZSZ-3}:
\begin{eqnarray}\label{gZSZ-3aux}
T_E(Q):=\bigcup_{Q'\in{\mathbb{D}}(E),\,\,Q'\subseteq Q}{\mathcal{U}}_{Q'}.
\end{eqnarray}
A property that will be needed later is the fact that 
\begin{eqnarray}\label{dFvK}
\begin{array}{c}
\mbox{there exists }\,C\in(0,\infty)\,
\mbox{ depending only on $C_\ast$ from \eqref{gZSZa} and $\rho$}
\\[4pt]
\mbox{such that }\,\,
T_E(Q)\subseteq B_\rho\bigl(x,C\ell(Q)\bigr)\setminus E,\quad
\forall\,Q\in{\mathbb{D}}(E),\,\,\forall\,x\in Q.
\end{array}
\end{eqnarray}
Indeed, if $Q\in{\mathbb{D}}(E)$ and $y\in T_E(Q)$ are arbitrary, then there exists
$Q'\in{\mathbb{D}}(E)$, $Q'\subseteq Q$ such that $y\in I$, for some $I\in W_{Q'}$. 
Hence, for each $x\in Q$, we have
\begin{eqnarray}\label{wqMY}
\rho(y,x) &\leq & C{\rm diam}_\rho(I)+C{\rm dist}_\rho(I,Q')+C{\rm diam}_\rho(Q)
\nonumber\\[4pt]
&\leq & C\ell(Q')+C\ell(Q)
\leq C\ell(Q),
\end{eqnarray}
where $C$ is a finite positive geometric constant. Now \eqref{dFvK} follows from
\eqref{wqMY}.

The following lemma compliments the containment in~\eqref{dFvK}. 

\begin{lemma}\label{b:SV}
Assume all of the hypotheses contained in the first paragraph of Lemma~\ref{Lem:CQinBQ-N} and recall the family of dyadic Carleson tents $\{T_E(Q)\}_{Q\in{\mathbb{D}}(E)}$ defined in \eqref{gZSZ-3}.

Then there exists $\varepsilon\in(0,1)$, depending only on $\lambda$ 
and geometry, with the property that
\begin{eqnarray}\label{zjrh}
B_{\rho_{\#}}\bigl(x_Q,\varepsilon\ell(Q)\bigr)\setminus E\subseteq T_E(Q),
\qquad\forall\,Q\in{\mathbb{D}}(E).
\end{eqnarray} 
\end{lemma}

\begin{proof}
Fix $\varepsilon\in(0,1)$ to be specified later and let $N$ be as in \eqref{jaf-UU.7}.
Also take an arbitrary $Q\in{\mathbb{D}}(E)$ and fix 
$x\in B_{\rho_{\#}}\bigl(x_Q,\varepsilon\ell(Q)\bigr)\setminus E$. Then 
$\rho_{\#}(x_Q,x)<\varepsilon\ell(Q)$ and making the restriction 
\begin{eqnarray}\label{varep-1}
\varepsilon<2^{-N-1}
\end{eqnarray}
we have
\begin{eqnarray}\label{jve-1}
\delta_E(x)\leq\rho_{\#}(x_Q,x)<\varepsilon\ell(Q)\leq 2\varepsilon\,{\rm diam}_\rho(E)
<2^{-N}{\rm diam}_\rho(E).
\end{eqnarray}
Thus, all considerations in the first part of the proof of 
Lemma~\ref{Lem:CQinBQ-N} up to \eqref{jaf-UU.3} apply. In particular, it follows that
$\delta_E(x)<\min\,\{2^{-N-k},\varepsilon\ell(Q)\}$. Hence, there exists $x_0\in E$
such that $\rho_{\#}(x,x_0)<\min\,\{2^{-N-k},\varepsilon\ell(Q)\}$. Applying property
{\it (7)} in Proposition~\ref{Diad-cube}, we may choose $Q'\in{\mathbb{D}}_k(E)$ such
that $B_{\rho_{\#}}\bigl(x_0,\varepsilon\ell(Q)\bigr)\cap Q'\not=\emptyset$.
At this point we make the claim that 
\begin{eqnarray}\label{jve-2}
B_{\rho_{\#}}\bigl(x_0,\varepsilon\ell(Q)\bigr)\cap E\subseteq Q
\quad\mbox{ if $\varepsilon$ is sufficiently small.}
\end{eqnarray}
Indeed, first observe that 
\begin{eqnarray}\label{jve-3}
\rho_{\#}(x_0,x_Q)\leq C_\rho\max\,\{\rho_{\#}(x_0,x),\rho_{\#}(x,x_Q)\}
<\varepsilon C_\rho\ell(Q).
\end{eqnarray}
Consequently, if $y\in B_{\rho_{\#}}\bigl(x_0,\varepsilon\ell(Q)\bigr)\cap E$ is arbitrary, then
\begin{eqnarray}\label{jve-4}
\rho_{\#}(x_Q,y)\leq C_\rho\max\,\{\rho_{\#}(x_Q,x_0),\rho_{\#}(x_0,y)\}
<\varepsilon C_\rho^2\ell(Q).
\end{eqnarray}
Property \eqref{ha-GVV} ensures that 
$B_{\rho_{\#}}\bigl(x_Q,a_0C_\rho^{-2}\ell(Q)\bigr)\cap E\subseteq Q$, so 
\eqref{jve-4} implies that $y\in Q$ if 
\begin{eqnarray}\label{varep-2}
\varepsilon<a_0C_\rho^{-4},
\end{eqnarray}
proving the claim in \eqref{jve-2}. In turn, if we assume that $\varepsilon$ is
sufficiently small, the inclusion in \eqref{jve-2} implies
\begin{eqnarray}\label{jve-5}
Q'\cap Q\not=\emptyset.
\end{eqnarray}
On the other hand, using the reasoning in the proof of Lemma~\ref{Lem:CQinBQ-N}
that yielded \eqref{jaf-UU.4}-\eqref{jaf-UU.7}, this time with $Q'$ replacing $Q$,
we obtain that if $N$ is as in \eqref{jaf-UU.7} (recall that we are assuming 
that $C_\ast$ satisfies \eqref{NeD-67}), then 
\begin{eqnarray}\label{jve-6}
x\in I\subseteq {\mathcal{U}}_{Q'}. 
\end{eqnarray}
Thus, using also \eqref{UUU-rf}, we have
\begin{eqnarray}\label{jve-7}
\ell(Q')\leq C_o\delta_E(x)\leq C_o\rho_{\#}(x_Q,x)<C_o\varepsilon\ell(Q).
\end{eqnarray}
Hence, under the additional restriction $\varepsilon<C_o^{-1}$, we arrive at the conclusion
that $Q'\in {\mathbb{D}}_j(E)$ for some $j>k$, which when combined with \eqref{jve-5}
and {\it (3)} in Proposition~\ref{Diad-cube}, forces $Q'\subseteq Q$. This in concert
with \eqref{jve-6} and \eqref{gZSZ-3}, shows that $x\in T_E(Q)$ provided 
\begin{eqnarray}\label{varep-3}
0<\varepsilon<\min\{2^{-N-1}, a_0C_\rho^{-4},C_o^{-1}\}.
\end{eqnarray}
The proof of the lemma is now complete.
\end{proof}

Next we prove a finite overlap property for the sets in
$\{{\mathcal{U}}_Q\}_{Q\in\mathbb{D}(E)}$ from \eqref{gZSZb}. 
Throughout the manuscript, we agree that ${\mathbf{1}}_A$ stands 
for the characteristic (or indicator) function of the set $A$.

\begin{lemma}\label{Lem:CQinBQ}
Let $({\mathscr{X}},\rho)$ be a geometrically doubling quasi-metric space and suppose that 
$E$ is a nonempty, closed, proper subset of $({\mathscr{X}},\tau_\rho)$. 
Fix $a\in[1,\infty)$, a collection ${\mathbb{D}}(E)$ of dyadic cubes in $E$ as in Proposition~\ref{Diad-cube}, and $C_\ast\in[1,\infty)$. 

If $\lambda\in[a,\infty)$, and we fix a Whitney covering ${\mathbb{W}}_\lambda({\mathscr{X}}\setminus E)$ of ${\mathscr{X}}\setminus E$ as in Proposition~\ref{H-S-Z}, then there exists $N\in{\mathbb{N}}$, depending only on $\lambda$, $C_\ast$ and geometry, such that  
\begin{eqnarray}\label{doj}
\sum_{Q\in{\mathbb{D}}(E)}{\mathbf{1}}_{{\mathcal{U}}_{Q}^\ast}\leq N,
\end{eqnarray}
where $\{\mathcal{U}_Q\}_{Q\in\mathbb{D}(E)}$ is the collection associated with ${\mathbb{W}}_\lambda({\mathscr{X}}\setminus E)$ and $C_\ast$ as in \eqref{gZSZa}-\eqref{gZSZb}, and for each $Q\in{\mathbb{D}}(E)$, the set (compare with \eqref{gZSZb})
\begin{eqnarray}\label{doj.222}
{\mathcal{U}}_{Q}^\ast:=\bigcup_{I\in W_Q}aI.
\end{eqnarray}
\end{lemma}

\begin{proof}
Let ${\mathbb{D}}(E)$ be the collection of dyadic cubes obtained by 
applying Proposition~\ref{Diad-cube}. Fix $a\in[1,\infty)$ and consider a Whitney covering
${\mathbb{W}}_\lambda({\mathscr{X}}\setminus E)$ as in Proposition~\ref{H-S-Z}
with $\lambda\in[a,\infty)$. In particular, 
\begin{eqnarray}\label{AAA-NNN}
\sum_{I\in{\mathbb{W}}_\lambda({\mathscr{X}}\setminus E)}{\mathbf{1}}_{\lambda I}
\leq N_1,\qquad\mbox{for some $N_1\in{\mathbb{N}}$.}
\end{eqnarray}
To proceed, define 
\begin{eqnarray}\label{3a-YU.1}
{\mathcal{I}}:=\bigcup_{Q\in{\mathbb{D}}(E)}W_Q
\subseteq{\mathbb{W}}_\lambda({\mathscr{X}}\setminus E)
\end{eqnarray}
and, for each $I\in{\mathcal{I}}$, 
\begin{eqnarray}\label{3a-YU.2}
q_I:=\{Q\in{\mathbb{D}}(E):\,I\in W_Q\}.
\end{eqnarray}
Then, using \eqref{AAA-NNN}, we estimate 
\begin{eqnarray}\label{3a-YU.3}
\sum_{Q\in{\mathbb{D}}(E)}{\mathbf{1}}_{{\mathcal{U}}_{Q}^\ast}
&\leq & \sum_{Q\in{\mathbb{D}}(E)}\sum_{I\in W_Q}{\mathbf{1}}_{\lambda I}
=\sum_{I\in{\mathcal{I}}}(\#\,q_I)\cdot{\mathbf{1}}_{\lambda I}
\nonumber\\[4pt]
&\leq & \Bigl(\sup_{I\in{\mathbb{W}}_\lambda({\mathscr{X}}\setminus E)}\#\,q_I\Bigr)
\sum_{I\in{\mathbb{W}}_\lambda({\mathscr{X}}\setminus E)}{\mathbf{1}}_{\lambda I}
\leq N_1\cdot\Bigl(\sup_{I\in{\mathbb{W}}_\lambda({\mathscr{X}}\setminus E)}\#\,q_I\Bigr).
\hskip 0.30in
\end{eqnarray}
Hence, once we show that there exists $N_2\in{\mathbb{N}}$ such that 
\begin{eqnarray}\label{3a-YU.4}
\#\,q_I\leq N_2,\qquad\forall\,I\in{\mathbb{W}}_\lambda({\mathscr{X}}\setminus E),
\end{eqnarray}
the desired estimate, \eqref{doj}, follows with $N:=N_1N_2$. To prove \eqref{3a-YU.4}, 
fix an arbitrary $I\in{\mathbb{W}}_\lambda({\mathscr{X}}\setminus E)$ and assume that
$Q\in{\mathbb{D}}(E)$ is such that $I\in W_Q$. Then, from \eqref{gZSZa} we deduce that 
\begin{eqnarray}\label{3a-YU.5}
C_\ast^{-1}\ell(I)\leq\ell(Q)\leq C_\ast\ell(I)\quad\mbox{and}\quad
{\rm dist}_\rho(I,Q)\leq C_\ast\ell(I).
\end{eqnarray}
Now, \eqref{3a-YU.4} follows from \eqref{3a-YU.5} and the fact that 
$\bigl(E,\rho\bigl|_{E}\bigr)$ is geometrically doubling.
\end{proof}

\section{$T(1)$ and local $T(b)$ Theorems for Square Functions}
\setcounter{equation}{0}
\label{Sect:3}

This section consists of two parts, dealing with a $T(1)$ Theorem and 
a local $T(b)$ Theorem for square functions on sets of arbitrary co-dimension,
relative to an ambient quasi-metric space (the notion of dimension refers
to the degree of Ahlfors-David regularity). The $T(1)$ Theorem generalizes the Euclidean co-dimension one result proved by M.~Christ and J.-L.~Journ\'e in~\cite{CJ}
(cf. also \cite[Theorem~20, p.\,69]{Ch}). The local $T(b)$ Theorem generalizes the Euclidean co-dimension one result that was implicit in the solution of the Kato problem in \cite{HMc,HLMc,AHLMcT}, and formulated explicitly in \cite{Au,Ho3,HMc2}.

We consider the following context. Fix two real numbers $d,m$ such that $0<d<m$, an $m$-dimensional {\rm ADR} space $({\mathscr{X}},\rho,\mu)$, 
a closed subset $E$ of $({\mathscr{X}},\tau_\rho)$, and 
a Borel measure $\sigma$ 
on $(E,\tau_{\rho|_{E}})$ with the property that $(E,\rho\bigl|_E,\sigma)$ 
is a $d$-dimensional {\rm ADR} space. In this context, suppose that 
\begin{eqnarray}\label{K234}
\begin{array}{c}
{\theta}:(\mathscr{X}\setminus E)\times E\longrightarrow{{\mathbb{R}}}
\quad\mbox{is Borel measurable with respect to}
\\[4pt]
\mbox{the relative topology induced
by the product topology $\tau_\rho\times\tau_\rho$ on $(\mathscr{X}\setminus E)\times E$},
\end{array}
\end{eqnarray}
and has the property that there exist finite positive constants
$C_{\theta},\,\alpha,\,\upsilon$, and $a\in[0,\upsilon)$ such that for all 
$x\in\mathscr{X}\setminus E$ and $y\in E$ the following hold:
\begin{eqnarray}\label{hszz}
&& \hskip -0.35in
|{\theta}(x,y)|\leq\frac{C_{\theta}}{\rho(x,y)^{d+\upsilon}}\,\Bigl(
\frac{{\rm dist}_\rho(x,E)}{\rho(x,y)}\Bigr)^{-a},
\\[4pt]
&& \hskip -0.40in
\begin{array}{l}
\displaystyle|{\theta}(x,y)-{\theta}(x,\widetilde{y})|\leq C_{\theta} \frac{\rho(y,\widetilde{y})^\alpha}{\rho(x,y)^{d+\upsilon+\alpha}}\,\Bigl(
\frac{{\rm dist}_\rho(x,E)}{\rho(x,y)}\Bigr)^{-a-\alpha}, 
\\[12pt]
\qquad\forall\,\widetilde{y}\in E\,\,\mbox{ with }\,\,
\rho(y,\widetilde{y})\leq\tfrac{1}{2}\rho(x,y).
\end{array}
\label{hszz-3}
\end{eqnarray}
Then define the integral operator $\Theta$ for all functions 
$f\in L^p(E,\sigma)$, ${1\leq p\leq\infty}$, by
\begin{eqnarray}\label{operator}
(\Theta f)(x):=\int_E {\theta}(x,y)f(y)\,d\sigma(y),\qquad\forall\,x\in\mathscr{X}\setminus E.
\end{eqnarray}
It follows from H\"older's inequality and Lemma~\ref{Gkwvr} that the integral 
in \eqref{operator} is absolutely convergent for each $x\in\mathscr{X}\setminus E$. 

\begin{remark}\label{rem:a}
The factors in parentheses in \eqref{hszz}-\eqref{hszz-3} are greater than or equal to $1$, since for every $x\in\mathscr{X}\setminus E$ and every $y\in E$ we have $\rho(x,y)\geq{\rm dist}_\rho(x,E)>0$, hence
\eqref{hszz}-\eqref{hszz-3} are less demanding than their respective versions in which these factors are omitted.
\end{remark}

We proceed to prove square function versions of the $T(1)$ Theorem and the local $T(b)$ Theorem for the integral operator $\Theta$. As usual, we prove the local $T(b)$ Theorem by verifying the hypotheses of the $T(1)$ Theorem, to which we now turn.

\subsection{An arbitrary codimension $T(1)$ theorem for square functions}
\label{SSect:3.1}
The main result in this subsection is a $T(1)$ theorem for square functions,
to the effect that {\it a square function estimate for the integral operator $\Theta$ 
holds if and only if $|\Theta(1)|^2$, appropriately weighted by a power of the 
distance to $E$, is the density (relative to $\mu$) of a Carleson measure on 
${\mathscr{X}}\setminus E$}. To state this formally, the reader is advised 
to recall the dyadic cube grid from Proposition~\ref{Diad-cube} and the 
regularized distance function to a set from \eqref{REG-DDD}.

\begin{theorem}\label{SChg}
Let $d,m$ be two real numbers such that $0<d<m$. Assume that $({\mathscr{X}},\rho,\mu)$
is an $m$-dimensional {\rm ADR} space, $E$ is a closed subset of
$({\mathscr{X}},\tau_\rho)$, and $\sigma$ is a Borel regular measure on
$(E,\tau_{\rho|_{E}})$ with the property that $(E,\rho\bigl|_E,\sigma)$ is a 
$d$-dimensional {\rm ADR} space.

Suppose that $\Theta$ is the integral operator defined in \eqref{operator} 
with a kernel ${\theta}$ as in \eqref{K234}, \eqref{hszz}, \eqref{hszz-3}. 
Furthermore, let ${\mathbb{D}}(E)$ denote a dyadic cube structure on $E$, 
consider a Whitney covering ${\mathbb{W}}_\lambda(\mathscr{X}\setminus E)$ of 
$\mathscr{X}\setminus E$ and a constant $C_\ast$ as in Lemma~\ref{Lem:CQinBQ-N} 
and, corresponding to these, recall the dyadic Carleson tents from \eqref{gZSZ-3}. 

In this context, if 
\begin{eqnarray}\label{UEHg}
\sup_{Q\in{\mathbb{D}}(E)}\left(\tfrac{1}{\sigma(Q)}\int_{T_E(Q)}
|\Theta 1(x)|^2\delta_E(x)^{2\upsilon-(m-d)}\,d\mu(x)\right)<\infty,
\end{eqnarray}
then there exists a finite constant $C>0$ depending only on the 
constants $C_{\theta}$, the {\rm ADR} constants of $E$ and ${\mathscr{X}}$, and
the value of the supremum in \eqref{UEHg},
such that for each function $f\in L^2(E,\sigma)$ one has
\begin{eqnarray}\label{G-UF.22}
\int\limits_{\mathscr{X}\setminus E}
|(\Theta f)(x)|^2\delta_E(x)^{2\upsilon-(m-d)}\,
d\mu(x)\leq C\int_E|f(x)|^2\,d\sigma(x),\qquad\forall\,f\in L^2(E,\sigma).
\end{eqnarray}
Finally, the converse of the implication discussed above is also true. 
In fact, the following stronger claim holds: under the original background 
assumptions, except that the regularity requirement \eqref{hszz-3} is now dropped, 
the fact that 
\begin{eqnarray}\label{G-UF}
\int\limits_{\stackrel{x\in\mathscr{X}}{0<\delta_E(x)<\eta\,{\rm diam}_\rho(E)}}
\!\!\!\!\!\!\!\!|(\Theta f)(x)|^2\delta_E(x)^{2\upsilon-(m-d)}\,
d\mu(x)\leq C\int_E|f(x)|^2\,d\sigma(x),\quad\forall\,f\in L^2(E,\sigma),
\end{eqnarray}
holds for some $\eta\in(0,\infty)$ implies that \eqref{UEHg} holds as well.
\end{theorem}

Before presenting the actual proof of Theorem~\ref{SChg} we shall discuss a 
number of preliminary lemmas, starting with the following discrete Carleson estimate.

\begin{lemma}\label{brFC}
Assume $(E,\rho,\sigma)$ is a space of homogeneous type with the property that
$\sigma$ is Borel regular, and denote by ${\mathbb{D}}(E)$ a dyadic cube structure 
on $E$. If a sequence $\bigl\{B_Q\bigr\}_{Q\in{\mathbb{D}}(E)}\subseteq [0,\infty]$
satisfies the discrete Carleson condition
\begin{eqnarray}\label{TAkB}
C:=\sup_{R\in{\mathbb{D}}(E)}\Bigl[
\frac{1}{\sigma(R)}\sum\limits_{Q\in{\mathbb{D}}(E),\,Q\subseteq R} B_Q
\Bigr]<\infty,
\end{eqnarray}
then for every sequence $\bigl\{A_Q\bigr\}_{Q\in{\mathbb{D}}(E)}\subseteq{\mathbb{R}}$ 
one has
\begin{eqnarray}\label{TAkB-2}
\sum\limits_{Q\in{\mathbb{D}}(E)}A_Q B_Q\leq C\int_E A^\ast\,d\sigma,
\end{eqnarray}
where $A^\ast:E\rightarrow [0,\infty]$ is the function defined by 
\begin{eqnarray}\label{TAkB-2B}
A^\ast(x):=0\,\,\mbox{ if }\,\,x\in E\setminus\!\!\bigcup_{Q\in{\mathbb{D}}(E)}Q
\quad\mbox{ and }\quad
A^\ast(x):=\!\!\sup\limits_{Q\in{\mathbb{D}}(E),\,x\in Q}|A_Q|
\,\,\,\mbox{ if }\,\,\,x\in\!\!\!\bigcup_{Q\in{\mathbb{D}}(E)}\!\!Q.
\end{eqnarray}
\end{lemma}

\begin{proof}
For each $t>0$ define ${\mathcal{O}}_t:=\{x\in E:\,A^\ast(x)>t\}$. Then it is 
immediate from definitions that 
${\mathcal{O}}_t=\!\!\!\bigcup\limits_{Q\in{\mathbb{D}}(E),\,A_Q>t}\!\!\!Q$
for every $t>0$. This shows that ${\mathcal{O}}_t$ is open in $(E,\tau_\rho)$
(cf. {\it (1)} in Proposition~\ref{Diad-cube}) and, hence, $A^\ast$ is 
$\sigma$-measurable. Note that if $A^\ast\in L^1(E,\sigma)$ (otherwise 
there is nothing to prove), then by Tschebyshev's inequality, 
\begin{eqnarray}\label{TAkB-1}
\sigma({\mathcal{O}}_t)\leq \frac{1}{t}\int_EA^\ast(x)\,d\sigma(x)<\infty,
\qquad\forall\,t>0.
\end{eqnarray}
This ensures that for each $t>0$ we may meaningfully define $D_t\subseteq{\mathbb{D}}(E)$, 
the collection of maximal dyadic cubes contained in ${\mathcal{O}}_t$, i.e., 
\begin{eqnarray}\label{TAkB-3}
D_t:=\bigl\{R\in{\mathbb{D}}(E):\,R\subseteq{\mathcal{O}}_t\mbox{ and }\not\!\exists\,
Q\in{\mathbb{D}}(E)\mbox{ such that 
$R\subseteq Q\subseteq{\mathcal{O}}_t$ and $R\not=Q$}\bigr\}.
\end{eqnarray}
The cubes in $D_t$ are pairwise disjoint, and
\begin{eqnarray}\label{TAkB-4}
{\mathcal{O}}_t=\bigcup\limits_{R\in D_t}R.
\end{eqnarray}
Now for each $Q\in{\mathbb{D}}(E)$ define  
\begin{eqnarray}\label{TAkB-5}
h_Q:(0,\infty)\longrightarrow{\mathbb{R}},\qquad
h_Q(t):=\left\{
\begin{array}{cc}
1,&\mbox{if }\,0<t<A_Q,
\\[4pt]
0,&\mbox{otherwise.}
\end{array}
\right.
\end{eqnarray}
Then, for each $t>0$ we have 
\begin{eqnarray}\label{TAkB-6}
\sum\limits_{Q\in{\mathbb{D}}(E)}h_Q(t)B_Q
& = & \sum\limits_{Q\in{\mathbb{D}}(E),\,Q\subseteq{\mathcal{O}}_t}B_Q
=\sum\limits_{R\in D_t}\left(\sum\limits_{Q\in{\mathbb{D}}(E),\,Q\subseteq R}
B_Q\right)
\nonumber\\[4pt]
& \leq & C \sum\limits_{R\in D_t}\sigma(R)
= C\sigma({\mathcal{O}}_t),
\end{eqnarray}
where for the first inequality in \eqref{TAkB-6} we have used \eqref{TAkB}, while
the last equality follows from \eqref{TAkB-4}. Hence,
\begin{eqnarray}\label{TAkB-7}
\sum\limits_{Q\in{\mathbb{D}}(E)}A_QB_Q
& = & \int_0^\infty\sum\limits_{Q\in{\mathbb{D}}(E)}h_Q(t)B_Q\,dt
\leq C\int_0^\infty\sigma({\mathcal{O}}_t)\,dt
\nonumber\\[4pt]
& = & C \int_0^\infty\int_E{\mathbf{1}}_{\{A^\ast>t\}}(x)\,d\sigma(x)\,dt
=C\int_E \int_0^\infty{\mathbf{1}}_{\{A^\ast>t\}}(x)\,dt\,d\sigma(x)
\nonumber\\[4pt]
& = & C\int_EA^\ast(x)\,d\sigma(x),
\end{eqnarray}
completing the proof of the lemma.
\end{proof}

We continue by recording a quantitative version of the classical Urysohn lemma
in the context of H\"older functions on quasi-metric spaces from \cite{MMMM-G}
(cf. also \cite{AMM} for a refinement).

\begin{lemma}\label{GVa2} 
Let $(E,\rho)$ be a quasi-metric space and assume that $\beta$ is a real 
number with the property that $0<\beta\leq\left[\log_2 C_{\rho}\right]^{-1}$.
Assume that $F_0,F_1\subseteq E$ are two nonempty sets with the property 
that ${\rm dist}_\rho(F_0,F_1)>0$. Then, there exists a function 
$\eta:E\to{\mathbb{R}}$ such that
\begin{eqnarray}\label{PMab5}
0\leq\eta\leq 1\,\,\,\mbox{ on }\,\,E,\quad
\eta\equiv 0\,\,\mbox{ on }\,\,F_0,\quad
\eta\equiv 1\,\,\mbox{ on }\,\,F_1,
\end{eqnarray}
and for which there exists a finite constant $C>0$, depending only on $\rho$, 
such that 
\begin{eqnarray}\label{PMab6}
\sup_{\stackrel{x,y\in E}{x\not=y}}\frac{|\eta(x)-\eta(y)|}{\rho(x,y)^\beta}
\leq C\bigl({\rm dist}_\rho(F_0,F_1)\bigr)^{-\beta}.
\end{eqnarray}
\end{lemma}

In the proof of Theorem~\ref{SChg} we shall also need a couple of 
results of geometric measure theoretic nature, which we next discuss. 

\begin{lemma}\label{Gkwvr}
Let $({\mathscr{X}},\rho)$ be a quasi-metric space.
Suppose $E\subseteq{\mathscr{X}}$ is nonempty and $\sigma$ is a measure on $E$ such that
$(E,\rho\bigl|_E,\sigma)$ becomes a $d$-dimensional {\rm ADR} space, for some $d>0$. 
Fix a real number $m>d$. Then there exists $C\in(0,\infty)$ depending only on 
$m$, $\rho$, and the {\rm ADR} constant of $E$ such that 
\begin{eqnarray}\label{mMji}
\int_E\frac{1}{\rho_{\#}(x,y)^{m}}\,d\sigma(y)\leq C\delta_E(x)^{d-m},
\qquad\forall\,x\in {\mathscr{X}}\setminus E.
\end{eqnarray}
Also, for each $\varepsilon>0$ and $c>0$, there exists $C\in(0,\infty)$ depending 
only on $\varepsilon$, $c$, $\rho$, and the {\rm ADR} constant of $E$ such that 
for every $\sigma$-measurable function $f:E\to[0,\infty]$ one has
\begin{eqnarray}\label{WBA}
\int\limits_{y\in E,\,\rho_{\#}(y,x)>cr}
\frac{r^\varepsilon}{\rho_{\#}(y,x)^{d+\varepsilon}}f(y)\,d\sigma(y)
\leq C\,M_E(f)(x)\qquad\forall\,x\in E,\quad\forall\,r>0,
\end{eqnarray}
where $M_E$ is as in \eqref{HL-MAX}.
\end{lemma}

\begin{proof}
Fix $x\in{\mathscr{X}}\setminus E$. Then 
\begin{eqnarray}\label{jrap}
\int_E \frac{1}{\rho_{\#}(y,x)^{m}}\,d\sigma(y) 
&\leq & \int_E \mathbf{1}_{\{z:\,\rho_{\#}(z,x)\geq\delta_E(x)\}}(y) 
\frac{1}{\rho_{\#}(y,x)^{m}}\,d\sigma(y) 
\nonumber\\[4pt]
&= & C\sum_{j=0}^\infty\int_E 
\mathbf{1}_{\{z:\,\rho_{\#}(z,x)\in[2^j\delta_E(x),2^{j+1}\delta_E(x))\}}(y) 
\frac{1}{\rho_{\#}(y,x)^{m}}\,d\sigma(y) 
\nonumber\\[4pt]
&\leq& C\sum_{j=0}^\infty\frac{1}{[2^j\delta_E(x)]^{m}} 
\sigma\left(B_{\rho_{\#}}(x,2^{j+1}\delta_E(x))\cap E\right) 
\nonumber\\[4pt]
&\leq& C\sum_{j=0}^\infty\frac{1}{[2^j\delta_E(x)]^{m}} 
\bigl[2^{j+1}\delta_E(x)\bigr]^{d}
\nonumber\\[4pt]
&=& C\delta_E(x)^{d-m},
\end{eqnarray}
where for the last inequality in \eqref{jrap} we have used the fact that 
$(E,\rho_{\#}|_E,\sigma)$ is a $d$-dimensional {\rm ADR} space, while the last 
equality uses the condition $m>d$. This concludes the proof of \eqref{mMji}.
Finally, \eqref{WBA} is proved similarly, by decomposing the domain of 
integration in dyadic annuli centered at $x$, at scale $r$, and then using 
the fact that $(E,\rho_{\#}|_E,\sigma)$ is a $d$-dimensional {\rm ADR} space.
\end{proof}

For a proof of our second result of geometric measure theoretic nature
the interested reader is referred to \cite{MMMM-B}, where more general results 
of this type are established.

\begin{lemma}\label{geom-lem}
Assume that $({\mathscr{X}},\rho,\mu)$ is an $m$-dimensional {\rm ADR} space 
for some $m>0$ and let $E\subseteq{\mathscr{X}}$ be nonempty, closed. 
Suppose there exists a measure $\sigma$ on $E$ such that $(E,\rho\bigl|_E,\sigma)$
is a $d$-dimensional {\rm ADR} space with $0<d<m$. If $\gamma<m-d$, then there 
exists a finite positive constant $C_0$ which depends only on $\gamma$ and 
the {\rm ADR} constants of $E$ and ${\mathscr{X}}$, such that
\begin{eqnarray}\label{lbzF}
\int\limits_{x\in B_\rho(x_0,R),\,\delta_E(x)<r}
\delta_E(x)^{\,-\gamma}\,d\mu(x)\leq C_0\,r^{m-d-\gamma}R^d,
\end{eqnarray}
for every $x_0\in E$ and every $r,R>0$.
\end{lemma}

At this stage, we are ready to present the

\vskip 0.08in
\begin{proof}[Proof of Theorem~\ref{SChg}]
For notational simplicity, abbreviate $A_Qf:=\dmeanint_Qf\,d\sigma$, for every
cube $Q\in{\mathbb{D}}(E)$ whenever $f:E\to{\mathbb{C}}$ is locally integrable. 
Also, recall our convention that $\delta_E(x)$ stands for 
${\rm dist}_{\rho_{\#}}(x,E)$, for every $x\in\mathscr{X}\setminus E$, 
and recall the Hardy-Littlewood maximal operator $M_E$ from \eqref{HL-MAX}.
We break down the proof of the implication 
``\eqref{UEHg}$\Rightarrow$\eqref{G-UF.22}" into a number of steps.

\vskip 0.10in
\noindent{\tt Step~I.} {\it We claim that for every $r\in(1,\infty)$ there exist 
finite positive constants $C$ and $\beta$ such that for each $l,k\in{\mathbb{Z}}$
with $l,k\geq\kappa_E$ and every $Q\in{\mathbb{D}}_k(E)$, fixed thereafter, the 
following inequality holds}:
\begin{eqnarray}\label{Lvsn}
\sup\limits_{x\in{\mathcal{U}}_Q}\Bigl|
\delta_E(x)^\upsilon\bigl(\Theta(D_lg)(x)-(\Theta 1)(x)A_Q(D_lg)\bigr)\Bigr|
\leq C2^{-|k-l|\beta}\inf\limits_{{w}\in Q}\Bigl[
M_E^2(|g|^r)({w})\Bigr]^{\frac{1}{r}},
\end{eqnarray}
{\it for every $g:E\to{\mathbb{R}}$ locally integrable}. Here, $D_l$ is the operator
introduced in \eqref{opD}, whose integral kernel is denoted by $h_l(\cdot,\cdot)$ 
(cf. the discussion in the proof of Proposition~\ref{HS-PP.3}).

To justify \eqref{Lvsn}, fix $k\in{\mathbb{Z}}$ with $k\geq\kappa_E$, 
$Q\in{\mathbb{D}}_k(E)$, and pick a number $k_0\in{\mathbb{N}}_0$ to be specified 
later, purely in terms of geometrical constants. We distinguish two cases.

\vskip 0.10in
{\it Case I: $k+{k_0}\geq l$.}  
As a preamble, we remark that $k+{k_0}\geq l$ forces 
\begin{eqnarray}\label{RRT}
2^{-(k+{k_0}-l)}\approx 2^{-|k-l|},
\end{eqnarray}
where the comparability constants depend only on ${k_0}$. Indeed, 
observe that if $j\in{\mathbb{R}}$ is such that $j\geq -{k_0}$,
then $j\leq|j|\leq j+2{k_0}$, hence $2^{-j}\approx 2^{-|j|}$. Taking now $j:=k-l$,
\eqref{RRT} follows.

Turning now to the proof of \eqref{Lvsn} in earnest,
using Fubini's Theorem, for each $l\in{\mathbb{Z}}$, we write
\begin{eqnarray}\label{FDcs}
&&\hskip -0.60in
\delta_E(x)^\upsilon\bigl(\Theta(D_lg)(x)-(\Theta 1)(x)A_Q(D_lg)\bigr)
\nonumber\\[4pt]
&&\,\,=\delta_E(x)^\upsilon\int_E{\theta}(x,y)\int_Eh_l(y,z)g(z)\,d\sigma(z)\,d\sigma(y)
\nonumber\\[4pt]
&&\hskip 0.20in
-\delta_E(x)^\upsilon(\Theta 1)(x)\meanint_Q\int_E h_l(y,z)
g(z)\,d\sigma(z)\,d\sigma(y)
\nonumber\\[4pt]
&&\,\, = \int_E\left[\int_E\Phi(x,y)h_l(y,z)\,d\sigma(y)
\right]g(z)\,d\sigma(z),\qquad\forall\,x\in{\mathcal{U}}_Q,
\end{eqnarray} 
where
\begin{eqnarray}\label{FDcs-2}
\Phi(x,y):=\delta_E(x)^\upsilon\Bigl[{\theta}(x,y)-
\tfrac{1}{\sigma(Q)}(\Theta 1)(x){\mathbf{1}}_Q(y)\Bigr],\quad
\forall\,x\in\mathscr{X}\setminus E,\quad\forall\,y\in E.
\end{eqnarray}
Note that, by design, 
\begin{eqnarray}\label{FDcs-3}
\int_E\Phi(x,y)\,d\sigma(y)=0,\qquad\forall\,x\in\mathscr{X}\setminus E,
\end{eqnarray}
and we claim that 
\begin{eqnarray}\label{FDcs-4}
|\Phi(x,y)|\leq\frac{C}{\sigma(Q)},\qquad\forall\,x\in{\mathcal{U}}_Q,\quad\forall\,y\in E.
\end{eqnarray}
Indeed, if $x\in{\mathcal{U}}_Q$ then $\delta_E(x)\approx\ell(Q)$ 
(with constants independent of $x$) and making use of
\eqref{hszz} and the fact that $\delta_E(\cdot)\approx{\rm dist}_\rho(\cdot,E)$
(see {\it (4)} in Theorem~\ref{JjEGh}), we obtain
\begin{eqnarray}\label{FDcs-5}
\delta_E(x)^\upsilon|{\theta}(x,y)|
\leq\frac{C\delta_E(x)^{\upsilon-a}}{\rho(x,y)^{d+\upsilon-a}}
\leq\frac{C\delta_E(x)^{\upsilon-a}}{\delta_E(x)^{d+\upsilon-a}}
\leq\frac{C}{\ell(Q)^d}
\leq\frac{C}{\sigma(Q)},\qquad\forall\,y\in E.
\end{eqnarray}
In addition,
\begin{eqnarray}\label{FDcs-6}
\delta_E(x)^\upsilon|(\Theta1)(x)|\leq C\delta_E(x)^{\upsilon-a}
\int_{y\in E}\frac{d\sigma(y)}{\rho_{\#}(x,y)^{d+\upsilon-a}}
\leq C,\qquad\forall\,x\in{\mathcal{U}}_Q,
\end{eqnarray}
where for the last inequality in \eqref{FDcs-6} we made use of \eqref{mMji}. Now 
\eqref{FDcs-4} follows from \eqref{FDcs-2}, \eqref{FDcs-5} and \eqref{FDcs-6}.

Denote by $x_Q$ the center of $Q$ and let $\varepsilon\in(0,1)$ and $C_0>0$ be fixed, 
to be specified later. Then for every ${w}\in Q$ fixed, due to \eqref{FDcs-3}
for each $z\in E$ we may write
\begin{eqnarray}\label{FDcs-7}
&& \hskip -0.20in
\left|\int_E\Phi(x,y)h_l(y,z)\,d\sigma(y)\right|
=\left|\int_E\Phi(x,y)[h_l(y,z)-h_l({w},z)]\,d\sigma(y)\right|
\nonumber\\[4pt]
&& \hskip 0.70in
\leq\int\limits_{y\in E,\,\rho_{\#}(y,x_Q)\leq C_02^{(k+{k_0}-l)\varepsilon}\ell(Q)}
|\Phi(x,y)||h_l(y,z)-h_l({w},z)|\,d\sigma(y)
\nonumber\\[4pt]
&& \hskip 0.80in
+\int\limits_{y\in E,\,\rho_{\#}(y,x_Q)>C_02^{(k+{k_0}-l)\varepsilon}\ell(Q)}
|\Phi(x,y)||h_l(y,z)-h_l({w},z)|\,d\sigma(y)
\nonumber\\[4pt]
&& \hskip 0.70in
=:I_1+I_2.
\end{eqnarray}
In order to estimate $I_1$ we make the claim that if $\widetilde{C}$ is chosen
large enough (compared to finite positive background constants and $C_0$) then 
\begin{eqnarray}\label{vsKJb}
\begin{array}{l}
|h_l(y,z)-h_l({w},z)|\leq C\,2^{\,l(d+\gamma)}2^{(k+{k_0}-l)\varepsilon\gamma}
\ell(Q)^\gamma
{\mathbf{1}}_{\{\rho_{\#}({w},\cdot)\leq\widetilde{C}2^{-l}\}}(z)
\\[10pt]
\mbox{whenever }\,\,z\in E,\,y\in E,\,{w}\in Q\,\,\mbox{ and }\,\,
\rho_{\#}(y,x_Q)\leq C_02^{(k+{k_0}-l)\varepsilon}\ell(Q),
\end{array}
\end{eqnarray}
with $\gamma$ as in \eqref{condh-1}. To justify this claim, first note that 
if $C_0$ is large, then since $k+{k_0}-l\geq 0$, we have
\begin{eqnarray}\label{vsKJb-2}
\hskip -0.15in
y\in E,\,{w}\in Q\mbox{ and }\rho_{\#}(y,x_Q)\leq C_02^{(k+{k_0}-l)\varepsilon}\ell(Q)
\,\Rightarrow\,\rho_{\#}(y,{w})\leq CC_02^{(k+{k_0}-l)\varepsilon}\ell(Q).
\end{eqnarray}
From now on, assume that $C_0$ is large enough to ensure the validity of \eqref{vsKJb-2}.
Second, if $y,{w}$ are as in \eqref{vsKJb-2} and $z\in E$ is such that
$\rho_{\#}(z,{w})\geq\widetilde{C}2^{-l}$, then
\begin{eqnarray}\label{vsKJb-3}
\widetilde{C}2^{-l} &\leq &
\rho_{\#}(z,{w})\leq C(\rho_{\#}(z,y)+\rho_{\#}(y,{w}))
\leq C\rho_{\#}(z,y)+CC_02^{(k+{k_0}-l)\varepsilon}\ell(Q)
\nonumber\\[4pt]
& \leq & C\rho_{\#}(z,y)+2^{k_0}C_0C2^{-l},
\end{eqnarray}
for some finite geometric constant $C>0$. Now choosing $\widetilde{C}:=2^{{k_0}+1}C_0C$
(which is permissible since, in the end, the parameter ${k_0}\in{\mathbb{N}}_0$ is 
chosen to depend only on finite positive background geometrical constants) we may absorb 
$2^{k_0}C_0C2^{-l}$ into $\widetilde{C}2^{-l}$ yielding (with $C_1:=2^{k_0}C_0C$)
\begin{eqnarray}\label{vsKJb-4}
\left.
\begin{array}{r}
y\in E,\,{w}\in Q,\,\rho_{\#}(y,x_Q)\leq C_02^{(k+{k_0}-l)\varepsilon}\ell(Q)
\\[4pt]
\mbox{and }\,\,\rho_{\#}(z,{w})\geq\widetilde{C}2^{-l}
\end{array}
\right\}\,\Longrightarrow\,\rho_{\#}(z,y)>C_12^{-l}.
\end{eqnarray}
Moreover, we can further increase $C_0$ and, in turn, $\widetilde{C}$ to insure
that the constant $C_1$ in the last inequality in \eqref{vsKJb-4} is larger that
the constant $C$ in \eqref{condh-1}. Henceforth, assume that such a choice has 
been made. Then a combination of \eqref{vsKJb-2}, \eqref{vsKJb-4} and 
\eqref{condh-2} yields \eqref{vsKJb}. 

Next, we use \eqref{FDcs-4} and \eqref{vsKJb} in order to estimate
\begin{eqnarray}\label{AsQ}
I_1 & \leq & \frac{C}{\sigma(Q)}\,
2^{\,l(d+\gamma)}\,2^{(k+{k_0}-l)\varepsilon\gamma}\,
2^{-k\gamma}\,{\mathbf{1}}_{\{\rho_{\#}({w},\cdot)\leq\widetilde{C}2^{-l}\}}(z)
\!\!\!\!
\int\limits_{y\in E,\,\rho_{\#}(y,x_Q)\leq C_02^{(k+{k_0}-l)\varepsilon}\ell(Q)}
\!\!\!\!\!\!\!\!\!\!\!\!\!\!\!\!1\,d\sigma(y)
\nonumber\\[4pt]
& \leq & C \,2^{-(k+{k_0}-l)[\gamma-\varepsilon(d+\gamma)]}2^{\,dl}\,
{\mathbf{1}}_{\{\rho_{\#}({w},\cdot)\leq\widetilde{C}2^{-l}\}}(z),
\qquad\forall\,z\in E,
\end{eqnarray}
where for the last inequality in \eqref{AsQ} we have used the fact that 
$(E,\rho|_E,\sigma)$ is a $d$-dimensional {\rm ADR} space. At this point we choose 
$0<\varepsilon<\frac{\gamma}{d+\gamma}<1$ which ensures that 
$\beta_1:=\gamma-\varepsilon(d+\gamma)>0$, hence
\begin{eqnarray}\label{AsQ-2}
I_1\leq C\,2^{\,dl}\,2^{-(k+{k_0}-l)\beta_1}\,{\mathbf{1}}_{\{\rho_{\#}({w},\cdot)
\leq\widetilde{C}2^{-l}\}}(z),\qquad\forall\,z\in E.
\end{eqnarray}

To estimate the contribution from $I_1$ in the context of \eqref{FDcs}, 
based on \eqref{AsQ-2} and \eqref{RRT} we write 
(recall that $\beta_1=\gamma-\varepsilon(d+\gamma)>0$ is a fixed constant)
\begin{eqnarray}\label{Mxr}
\int_EI_1|g(z)|\,d\sigma(z)
& \leq & C\,2^{\,dl}\,2^{-(k+{k_0}-l)\beta_1}\,\int_E|g(z)|
{\mathbf{1}}_{\{\rho_{\#}({w},\cdot)\leq C2^{-l}\}}(z)\,d\sigma(z)
\nonumber\\[4pt]
& = & C2^{-(k+{k_0}-l)\beta_1}\meanint_{z\in E,\,\rho_{\#}(z,{w})\leq C2^{-l}}|g(z)|\,d\sigma(z)
\nonumber\\[4pt]
& = & C2^{-(k+{k_0}-l)\beta_1} (M_Eg)({w})
\nonumber\\[4pt]
& \approx & 2^{-|k-l|\beta_1} (M_Eg)({w}),
\quad\mbox{uniformly in }{w}\in Q.
\end{eqnarray}

Next, we turn our attention to $I_2$ from \eqref{FDcs-7}. Note that since 
we are currently assuming that $k+{k_0}\geq l$, the condition 
$\rho_{\#}(y,x_Q)\geq C_02^{(k+{k_0}-l)\varepsilon}\ell(Q)$ forces $y\not\in c_1Q$
for some finite positive constant $c_1$, which may be further increased as desired 
by suitably increasing the value of $C_0$. Thus, assuming that $C_0$ 
is sufficiently large to guarantee $c_1>1$, we obtain ${\mathbf{1}}_Q(y)=0$ 
if $y\in E$ and $\rho_{\#}(y,x_Q)\geq C_02^{(k+{k_0}-l)\varepsilon}\ell(Q)$.
In turn, this implies that $\Phi(x,y)=\delta_E(x)^\upsilon {\theta}(x,y)$ on the domain 
of integration in $I_2$. Thus, for each $z\in E$, we have
\begin{eqnarray}\label{AsQ-3-A}
I_2 & \leq & C2^{-k\upsilon}\int\limits_{y\in E,\,\rho_{\#}(y,x_Q)>C_02^{(k+{k_0}-l)\varepsilon}\ell(Q)}
|{\theta}(x,y)|\,|h_l(y,z)|\,d\sigma(y)
\\[4pt]
&&\,\,+C2^{-k\upsilon}|h_l({w},z)|
\int\limits_{y\in E,\,\rho_{\#}(y,x_Q)>C_02^{(k+{k_0}-l)\varepsilon}\ell(Q)}
|{\theta}(x,y)|\,d\sigma(y)=:I_3+I_4.
\nonumber
\end{eqnarray}
We also remark that the design of ${\mathcal{U}}_Q$ and the fact 
that $k+{k_0}-l\geq 0$ ensure that 
\begin{eqnarray}\label{AsQ-3}
y\in E,\,\,\rho_{\#}(y,x_Q)>C_02^{(k+{k_0}-l)\varepsilon}\ell(Q)\,\Longrightarrow\,
\left\{
\begin{array}{l}
\rho_{\#}(x,y)\approx\rho_{\#}({w},y)\approx\rho_{\#}(x_Q,y),
\\[4pt]
\mbox{uniformly for $x\in{\mathcal{U}}_Q$ and ${w}\in Q$}.
\end{array}\right.
\end{eqnarray}
Making first use of \eqref{hszz} combined with \eqref{AsQ-3} and the fact that since
$x\in{\mathcal{U}}_Q$ we have $\delta_E(x)\approx\ell(Q)\approx 2^{-k}$, and then of
\eqref{condh-1}, we may further estimate
\begin{eqnarray}\label{AsQ-4}
I_3 & \leq & C2^{-k\upsilon}\int\limits_{y\in E,\,
\rho_{\#}(y,{w})>C2^{(k+{k_0}-l)\varepsilon}\ell(Q)}
\frac{\delta_E(x)^{-a}}{\rho_{\#}(y,{w})^{d+\upsilon-a}}\,|h_l(y,z)|\,d\sigma(y)
\nonumber\\[4pt]
& \leq & C2^{-k\upsilon}2^{\,dl}\int\limits_{y\in E,\,
\rho_{\#}(y,{w})>C2^{(k+{k_0}-l)\varepsilon}\ell(Q)}
\frac{2^{\,ak}}{\rho_{\#}(y,{w})^{d+\upsilon-a}}\,{\mathbf{1}}_{\{\rho_{\#}(y,\cdot)
\leq C2^{-l}\}}(z)\,d\sigma(y)
\nonumber\\[4pt]
& = & C\,2^{-(k+{k_0}-l)\varepsilon(\upsilon-a)}2^{dl}
\int_{y\in E,\,\rho_{\#}(y,{w})\geq Cr}
\frac{r^{\upsilon-a}}{\rho_{\#}(y,{w})^{d+\upsilon-a}}\,
{\mathbf{1}}_{\{\rho_{\#}(y,\cdot)\leq C2^{-l}\}}(z)\,d\sigma(y),\qquad
\end{eqnarray}
for each $z\in E$, where we have set
\begin{eqnarray}\label{yfwz}
r:=2^{(k+{k_0}-l)\varepsilon-k}.
\end{eqnarray}
Consequently, by \eqref{AsQ-4}, Fubini's theorem, \eqref{WBA} and \eqref{RRT}, we obtain
\begin{eqnarray}\label{Mxr-2}
&& \hskip -0.40in
\int_EI_3|g(z)|\,d\sigma(z)
\leq C\,2^{-(k+{k_0}-l)\varepsilon(\upsilon-a)}\,\int_E|g(z)|\times
\nonumber\\[4pt]
&& \hskip 1.20in
\times \int_{y\in E,\,\rho_{\#}(y,{w})\geq Cr}
\frac{r^{\upsilon-a}}{\rho_{\#}(y,{w})^{d+\upsilon-a}}\,2^{\,dl}\,
{\mathbf{1}}_{\{\rho_{\#}(y,\cdot)\leq C2^{-l}\}}(z)\,d\sigma(y)\,d\sigma(z)
\nonumber\\[4pt]
&& \hskip 0.75in
\leq C\,2^{-(k+{k_0}-l)\varepsilon(\upsilon-a)}\,
\int_{y\in E,\,\rho_{\#}(y,{w})\geq Cr}
\frac{r^{\upsilon-a}}{\rho_{\#}(y,{w})^{d+\upsilon-a}}\,(M_Eg)(y)\,d\sigma(y)
\nonumber\\[4pt]
&& \hskip 0.75in
\leq C2^{-(k+{k_0}-l)\varepsilon(\upsilon-a)}(M_E^2g)({w})
\nonumber\\[4pt]
&& \hskip 0.75in
\approx 2^{-|k-l|\varepsilon(\upsilon-a)}(M_E^2g)({w}),
\quad\mbox{uniformly in }{w}\in Q.
\end{eqnarray}
As for $I_4$, invoking again \eqref{hszz}, \eqref{AsQ-3}, the fact that
$\delta_E(x)\approx2^{-k}$, and \eqref{condh-1} we write
\begin{eqnarray}\label{AsQ-5}
I_4 & \leq & C2^{-k\upsilon}2^{\,dl}
{\mathbf{1}}_{\{\rho_{\#}({w},\cdot)\leq C2^{-l}\}}(z)
\int\limits_{y\in E,\,\rho_{\#}(y,x_Q)>Cr}
\frac{2^{ak}}{\rho_{\#}(y,x_Q)^{d+\upsilon-a}}\,d\sigma(y)
\nonumber\\[4pt]
& = & C2^{-(k+{k_0}-l)\varepsilon(\upsilon-a)}\,2^{\,dl}
{\mathbf{1}}_{\{\rho_{\#}({w},\cdot)\leq C2^{-l}\}}(z)
\int\limits_{y\in E,\,\rho_{\#}(y,x_Q)>Cr}
\frac{r^{\upsilon-a}}{\rho_{\#}(y,x_Q)^{d+\upsilon-a}}\,d\sigma(y)
\nonumber\\[4pt]
& = & C2^{-(k+{k_0}-l)\varepsilon(\upsilon-a)}\,2^{\,dl}
{\mathbf{1}}_{\{\rho_{\#}({w},\cdot)\leq C2^{-l}\}}(z),\qquad\forall\,z\in E,
\end{eqnarray}
by \eqref{WBA} (used with $f\equiv 1$). Finally, based on \eqref{AsQ-5} 
and \eqref{RRT}, we obtain
\begin{eqnarray}\label{Mxr-3}
\int_EI_4|g(z)|\,d\sigma(z)
& \leq & C\,2^{-(k+{k_0}-l)\varepsilon(\upsilon-a)}\,2^{\,dl}\,
\int_E|g(z)|{\mathbf{1}}_{\{\rho_{\#}({w},\cdot)\leq C2^{-l}\}}(z)\,d\sigma(z)
\nonumber\\[4pt]
& \leq & C\,2^{-(k+{k_0}-l)\varepsilon(\upsilon-a)}\,(M_Eg)({w})
\nonumber\\[4pt]
& \approx & 2^{-|k-l|\varepsilon(\upsilon-a)}(M_Eg)({w}),
\quad\mbox{uniformly in }{w}\in Q.
\end{eqnarray}
Collectively, \eqref{AsQ-4} and \eqref{AsQ-5} yield an estimate for $I_2$,
in view of \eqref{AsQ-3-A}. In order to express this estimate as well as \eqref{Mxr}
in a manner consistent with \eqref{Lvsn} requires an extra adjustment. Concretely, 
as a consequence of Lebesgue's Differentiation Theorem (which holds in our context
given that $\sigma$ is Borel regular), the monotonicity 
of the Hardy-Littlewood maximal operator and H\"older's inequality, 
for every $r\in[1,\infty)$ and any $g$ locally integrable on $E$ we have 
\begin{eqnarray}\label{bvez}
M_E g\leq\bigl[M_E^2(|g|^r)\bigr]^{\frac{1}{r}}\quad\mbox{and}\quad
M_E^2 g\leq\bigl[M_E^2(|g|^r)\bigr]^{\frac{1}{r}}\qquad\mbox{pointwise in $E$}.
\end{eqnarray}
Thus, if we define $\beta_2:=\min\{\beta_1,\varepsilon(\upsilon-a)\}>0$, then a combination
of \eqref{FDcs}, \eqref{FDcs-7}, \eqref{Mxr}, \eqref{AsQ-3-A}, \eqref{Mxr-2}, \eqref{Mxr-3}, 
and \eqref{bvez} proves \eqref{Lvsn} with $\beta$ replaced by $\beta_2$ in {\it Case~I}.

\vskip 0.10in
{\it Case II: $k+k_0<l$.} In this scenario, in order to obtain estimate \eqref{Lvsn} 
we shall use the cancellation property of the operator $D_l$, and the H\"older 
regularity of ${\theta}(\cdot,\cdot)$ in the second variable. To get started, Fubini's 
Theorem allows us to write
\begin{eqnarray}\label{hgfAW}
\delta_E(x)^\upsilon\Theta(D_lg)(x)
& = &\int_E\delta_E(x)^\upsilon\int_E{\theta}(x,y)h_l(y,z)\,d\sigma(y)\,g(z)\,d\sigma(z)
\nonumber\\[4pt]
& = &\int_E\Psi(x,z)\,g(z)\,d\sigma(z),\qquad\forall\,x\in{\mathcal{U}}_Q,
\end{eqnarray}
where we have set 
\begin{eqnarray}\label{hgfAW-BB}
\Psi(x,z):=\delta_E(x)^\upsilon\int_E{\theta}(x,y)h_l(y,z)\,d\sigma(y),\qquad
\forall\,x\in\mathscr{X}\setminus E,\quad\forall\,z\in E.
\end{eqnarray}
To proceed, fix $x\in{\mathcal{U}}_Q$ arbitrary. 
Based on \eqref{condh-4} and \eqref{condh-1}, we have
\begin{eqnarray}\label{hgfAW-2}
\hskip -0.30in
|\Psi(x,z)| &=& \left|\delta_E(x)^\upsilon\int_E\bigl({\theta}(x,y)-{\theta}(x,z)\bigr)h_l(y,z)\,d\sigma(y)\right|
\nonumber\\[4pt]
& \leq & \delta_E(x)^\upsilon\!\!\!\!\!\!\!\!
\int\limits_{y\in E,\,\rho_{\#}(y,z)\leq C2^{-l}}\!\!\!\!\!\!\!\!
|{\theta}(x,y)-{\theta}(x,z)||h_l(y,z)|\,d\sigma(y),
\qquad\forall\,z\in E.
\end{eqnarray}
As a consequence of \eqref{hszz-3}, we have
\begin{eqnarray}\label{ghhzt}
|{\theta}(x,y)-{\theta}(x,z)|\leq C\,
\frac{\rho(y,z)^\alpha\delta_E(x)^{-a-\alpha}}{\rho(x,y)^{d+\upsilon-a}}
\qquad\mbox{if }\,\,y,z\in E,\,\,\rho(y,z)<\tfrac{1}{2}\rho(x,y).
\end{eqnarray}
Observe that, since we are currently assuming $k+{k_0}<l$, if the points 
$y,z\in E$ are such that $\rho_{\#}(y,z)\leq C2^{-l}$, then
\begin{eqnarray}\label{ghhzt-2}
\rho(y,z)\leq C\rho_{\#}(y,z)\leq C2^{-l}\leq C 2^{-{k_0}}2^{-k}
\leq C2^{-{k_0}}\delta_E(x)\leq C2^{-{k_0}}\rho(x,y)<\tfrac{1}{2}\rho(x,y),
\end{eqnarray}
where the last inequality follows by choosing ${k_0}$ large. For the remainder of
the proof fix such a ${k_0}\in{\mathbb{N}}_0$. Then \eqref{ghhzt} holds when
$y$ belongs to the domain of integration of the last integral in \eqref{hgfAW-2}. 
For ${w}\in Q$ arbitrary we claim that
\begin{eqnarray}\label{ghhzt-3}
\left.
\begin{array}{l}
\exists\,C'>0\mbox{ such that }\,\forall\,z,y\in E
\\[4pt]
\mbox{ satisfying }\rho(y,z)\leq C2^{-l}
\end{array}
\right\}
\,\,\mbox{one has}\,\,
\left\{
\begin{array}{l}
\rho(x,y)\geq C'[2^{-k}+\rho_{\#}({w},z)],
\\[4pt]
\mbox{ uniformly for }
x\in{\mathcal{U}}_Q,\,{w}\in Q.
\end{array}
\right.
\end{eqnarray}
Indeed, if $y,z$ are as in the left hand-side of \eqref{ghhzt-3}, then  
$\rho(x,y)\geq C\delta_E(x)\approx\ell(Q)\approx 2^{-k}$. In addition,
\begin{eqnarray}\label{dxgkr}
\rho_{\#}({w},z) & \leq & C\rho({w},z)\leq C(\rho({w},x)+\rho(x,y)+\rho(y,z))
\leq C(\ell(Q)+\rho(x,y)+2^{-l})
\nonumber\\[4pt]
& \leq & C(2^{-k}+\rho(x,y))\leq C\rho(x,y).
\end{eqnarray}
This proves \eqref{ghhzt-3}. Combining \eqref{ghhzt-3}, \eqref{ghhzt}, 
\eqref{condh-1} and \eqref{hgfAW-2}, we obtain
\begin{eqnarray}\label{hgfAW-4}
\hskip -0.30in
|\Psi(x,z)| & \leq & C2^{-k\upsilon}\int\limits_{y\in E,\,\rho_{\#}(y,z)\leq C2^{-l}}
\frac{2^{k(a+\alpha)}2^{-l\alpha}}{[2^{-k}+\rho_{\#}({w},z)]^{d+\upsilon-a}}
\,2^{\,dl}\,d\sigma(y)
\nonumber\\[4pt]
& \leq & C2^{-k\upsilon}\frac{2^{k(a+\alpha)}2^{-l\alpha}}
{[2^{-k}+\rho_{\#}({w},z)]^{d+\upsilon-a}}
\nonumber\\[4pt]
& = & C2^{-|k-l|\alpha}
\frac{2^{-k(\upsilon-a)}}{[2^{-k}+\rho_{\#}({w},z)]^{d+\upsilon-a}},
\qquad\forall\,z\in E.
\end{eqnarray}
For the second inequality in \eqref{hgfAW-4} we used the fact that $(E,\rho|_E,\sigma)$ 
is a $d$-dimensional {\rm ADR} space, while the last equality is a simple consequence 
of the fact that $k<l$. Thus, returning with \eqref{hgfAW-4} in \eqref{hgfAW}, then making
use of \eqref{WBA}, and then recalling \eqref{bvez}, we arrive at the conclusion that 
for each $r\in[1,\infty)$,
\begin{eqnarray}\label{hgfAW-5}
\bigl|\delta_E(x)^\upsilon\Theta(D_lg)(x)\bigr|
& \leq & C 2^{-|k-l|\alpha}\int_E
\frac{2^{-k(\upsilon-a)}}{[2^{-k}+\rho_{\#}({w},z)]^{d+\upsilon-a}}|g(z)|\,d\sigma(z)
\nonumber\\[4pt]
& \leq & C 2^{-|k-l|\alpha}(M_Eg)({w})
\nonumber\\[4pt]
& \leq & C2^{-|k-l|\alpha}\bigl[M_E^2(|g|^r)({w})\bigr]^\frac{1}{r},
\qquad\mbox{uniformly for }\,\,{w}\in Q.
\end{eqnarray}

In order to estimate $|\delta_E(x)^\upsilon (\Theta 1)(x)A_Q(D_lg)|$, first note that \eqref{FDcs-6} holds in this case, so by Fubini's Theorem we have
\begin{eqnarray}\label{hgfAW-6}
|\delta_E(x)^\upsilon (\Theta 1)(x)A_Q(D_lg)| \leq C\left|\int_E\Bigl\{\tfrac{1}{\sigma(Q)}\int_Qh_l(y,z)\,d\sigma(y)
\Bigr\}g(z)\,d\sigma(z)
\right|.
\end{eqnarray}
To continue, for some fixed $\varepsilon\in(0,1)$, define
\begin{eqnarray}\label{hgfAW-7}
S_Q:=\bigl\{x\in Q:\,{\rm dist}_{\rho_{\#}}(x,E\setminus Q)
\leq C 2^{-|k-l|\varepsilon}\ell(Q)\bigr\}
\quad\mbox{ and }\quad F_Q:=Q\setminus S_Q.
\end{eqnarray}
Also, consider a function 
\begin{eqnarray}\label{zhzM}
\begin{array}{c}
\eta_Q:E\to{\mathbb{R}}\,\,\mbox{ such that }\,\,{\rm supp}\,\eta_Q\subseteq Q,\quad
0\leq\eta_Q\leq 1,\quad\eta_Q=1\mbox{ on }F_Q\,\mbox{ and}
\\[4pt]
\Bigl|\eta_Q(x)-\eta_Q(y)\Bigr|\leq C
\Bigl(\frac{\rho(x,y)}{2^{-|k-l|\varepsilon}\ell(Q)}
\Bigr)^\gamma,\qquad\forall\,x\in E,\,\,\forall\,y\in E,
\end{array}
\end{eqnarray}
for some $\gamma\in(0,1)$. That such a function exists is a consequence of 
Lemma~\ref{GVa2}. Hence,
\begin{eqnarray}\label{zhzMB}
\hskip -0.30in
\left|\meanint_Qh_l(y,z)\,d\sigma(y)\right|
&\leq & \tfrac{1}{\sigma(Q)}\left|\int_E\bigl({\mathbf{1}}_Q-\eta_Q(y)\bigr)
h_l(y,z)\,d\sigma(y)\right|
\\[4pt]
&&\hskip 0.06in 
+\tfrac{1}{\sigma(Q)}\left|\int_E\eta_Q(y)h_l(y,z)\,d\sigma(y)\right|
=:II_1(z)+II_2(z),\quad\forall\,z\in E.
\nonumber
\end{eqnarray}
Fix $z\in E$. To estimate $II_2(z)$, recall \eqref{condh-4}, \eqref{condh-1},
\eqref{zhzM} and the fact that $E$ is Ahlfors-David regular. Based on these, we
may write
\begin{eqnarray}\label{zhzM-2}
II_2(z) &=& \tfrac{1}{\sigma(Q)}\left|\int_E\bigl(\eta_Q(y)-\eta_Q(z)\bigr)
h_l(y,z)\,d\sigma(y)\right|
\\[4pt]
&\leq &  \tfrac{1}{\sigma(Q)}\!\!\!\!
\int\limits_{y\in E,\,\rho_{\#}(y,z)\leq C2^{-l}}\!\!\!\!\!\!\!\!\!\!
\bigl|\eta_Q(y)-\eta_Q(z)\bigr||h_l(y,z)|\,d\sigma(y)
\leq \tfrac{1}{\sigma(Q)}\left[\frac{2^{-l}}{2^{-|k-l|\varepsilon}2^{-k}}
\right]^{\gamma}.
\nonumber
\end{eqnarray}
In addition, since whenever $y\in{\rm supp}\,\eta_Q\subseteq Q$ and 
$\rho_{\#}(y,z)\leq C2^{-l}\leq C\ell(Q)$ one necessarily has 
$\rho_{\#}({w},z)\leq C\ell(Q)$, for all ${w}\in Q$, it follows 
that one may strengthen \eqref{zhzM-2} to
\begin{eqnarray}\label{zhzM-3}
II_2(z) \leq C \tfrac{1}{\sigma(Q)}2^{-|k-l|(1-\varepsilon)\gamma}
{\mathbf{1}}_{\{\rho_{\#}({w},\cdot)\leq C\ell(Q)\}}(z),\qquad
\mbox{ for all ${w}\in Q$}.
\end{eqnarray}
Hence, by also recalling \eqref{bvez}, for each $r\in[1,\infty)$ we further obtain
\begin{eqnarray}\label{zhzM-4}
\hskip -0.40in
\int_E II_2(z)|g(z)|\,d\sigma(z) 
& \leq & C 2^{-|k-l|(1-\varepsilon)\gamma}\tfrac{1}{\sigma(Q)}
\int_{z\in E,\,\rho_{\#}(z,{w})\leq C\ell(Q)}|g(z)|\,d\sigma(z)
\nonumber\\[4pt]
& \leq & C 2^{-|k-l|(1-\varepsilon)\gamma}(M_Eg)({w})
\leq C 2^{-|k-l|(1-\varepsilon)\gamma}\bigl[M_E^2(|g|^r)({w})\bigr]^\frac{1}{r},
\end{eqnarray}
for all ${w}\in Q$. This bound suits our purposes. 

Next, we turn our attention to $II_1(z)$. Pick $r\in(1,\infty)$ and let $r'$ be such that
$\frac{1}{r}+\frac{1}{r'}=1$. Note that since 
${\rm supp}\,({\mathbf{1}}_Q-\eta_Q)\subseteq S_Q$ and 
$0\leq {\mathbf{1}}_Q-\eta_Q\leq 1$ we may write
\begin{eqnarray}\label{zhzM-5}
\hskip -0.40in
\int_E II_1(z)|g(z)|\,d\sigma(z) 
& \leq & \tfrac{1}{\sigma(Q)}\int_E\int_{S_Q}|h_l(y,z)|\,d\sigma(y)\,|g(z)|\,d\sigma(z) 
\nonumber\\[4pt]
&=& \tfrac{1}{\sigma(Q)}\int_{S_Q}\int_E|g(z)||h_l(y,z)|\,d\sigma(z)\,d\sigma(y)
\nonumber\\[4pt]
&\leq & C\tfrac{1}{\sigma(Q)}\int_{S_Q}\int_{z\in E,\,\rho_{\#}(y,z)\leq C2^{-l}}
2^{\,dl}|g(z)|\,d\sigma(z)\,d\sigma(y)
\nonumber\\[4pt]
&\leq & C\tfrac{1}{\sigma(Q)}\int_Q{\mathbf{1}}_{S_Q}(y)(M_Eg)(y)\,d\sigma(y)
\nonumber\\[4pt]
&\leq & C\left[\frac{\sigma(S_Q)}{\sigma(Q)}\right]^{\frac{1}{r'}}
\left[\meanint_Q(M_Eg)^r\,d\sigma\right]^{\frac{1}{r}}
\nonumber\\[4pt]
&\leq & C\left[\frac{\sigma(S_Q)}{\sigma(Q)}\right]^{\frac{1}{r'}}
\bigl[M_E^2(|g|^r)({w})\bigr]^{\frac{1}{r}},\qquad\mbox{ for all }{w}\in Q.
\end{eqnarray}
The second inequality in \eqref{zhzM-5} is based on \eqref{condh-1}, the third
is immediate, the fourth uses H\"older's inequality, while the last one is a consequence
of \eqref{bvez}. By virtue of the ``thin boundary" property described in item {\it (8)}
of Proposition~\ref{Diad-cube}, we have
\begin{eqnarray}\label{AAzrhg}
\exists\,c>0\,\,\mbox{ and }\,\,
\exists\,\tau\in(0,1)\,\,\mbox{ such that }\,\, 
\sigma(S_Q)\leq c2^{-|k-l|\varepsilon\tau}\sigma(Q),
\end{eqnarray}
which, when used in concert with \eqref{zhzM-5}, yields
\begin{eqnarray}\label{zhzM-6}
\int_E II_1(z)|g(z)|\,d\sigma(z) 
\leq C 2^{-|k-l|\varepsilon\tau/r'}
\bigl[M_E^2(|g|^r)({w})\bigr]^{\frac{1}{r}},
\qquad\mbox{for all }\,{w}\in Q.
\end{eqnarray}
Now choose $\beta_3:=\min\{(1-\varepsilon)\gamma,\frac{\varepsilon\tau}{r'},\alpha\}>0$.
Then \eqref{Lvsn} follows in the current case with $\beta$ replaced by $\beta_3$, 
by combining \eqref{hgfAW-5}, \eqref{hgfAW-6}, \eqref{zhzMB}, \eqref{zhzM-4} and 
\eqref{zhzM-6}.

Now the proof of the claim made in Step~I is completed by combining what we proved in 
{\it Case I} and {\it Case II} above and by taking $\beta:=\min\{\beta_2,\beta_3\}>0$.

\vskip 0.10in
\noindent{\tt Step~II.} {\it We claim that there exists a finite constant $C>0$ 
with the property that for every $f\in L^2(E,\sigma)$ there holds}
\begin{eqnarray}\label{N-NKy}
\sum\limits_{k\in{\mathbb{Z}},\,k\geq\kappa_E}\sum\limits_{Q\in{\mathbb{D}}_k(E)}
\int_{{\mathcal{U}}_Q}\bigl|\delta_E(x)^\upsilon
\bigl((\Theta f)(x)-(\Theta 1)(x)A_Qf\bigr)\bigr|^2
\frac{d\mu(x)}{\delta_E(x)^{m-d}}
\leq C\int_E|f|^2\,d\sigma.
\end{eqnarray}
To justify this claim, fix $r\in(1,2)$ and let $\beta>0$ be such that \eqref{Lvsn} holds. 
Then, for an arbitrary function $f\in L^2(E,\sigma)$, using \eqref{PKc}, we may write 
\begin{eqnarray}\label{N-NKy-2AA}
&&\hskip -0.20in
\sum\limits_{k\in{\mathbb{Z}},\,k\geq\kappa_E}\sum\limits_{Q\in{\mathbb{D}}_k(E)}
\int_{{\mathcal{U}}_Q}\bigl|\delta_E(x)^\upsilon
\bigl(\Theta-(\Theta 1)(x)A_Q\bigr)(f)(x)\bigr|^2\frac{d\mu(x)}{\delta_E(x)^{m-d}}
\\[4pt]
&& 
\leq 2\sum\limits_{k\in{\mathbb{Z}},\,k\geq\kappa_E}\sum\limits_{Q\in{\mathbb{D}}_k(E)}
\int_{{\mathcal{U}}_Q}\Bigl|\sum\limits_{l\in{\mathbb{Z}},\,l\geq \kappa_E}
\delta_E(x)^\upsilon\bigl(\Theta -(\Theta 1)(x)A_Q\bigr)(D_l\widetilde{D}_lf)(x)\Bigr|^2
\frac{d\mu(x)}{\delta_E(x)^{m-d}}
\nonumber\\[4pt]
&& \hskip 0.10in
+2\sum\limits_{k\in{\mathbb{Z}},\,k\geq\kappa_E}\sum\limits_{Q\in{\mathbb{D}}_k(E)}
\int_{{\mathcal{U}}_Q}\Bigl|\delta_E(x)^\upsilon
\bigl(\Theta -(\Theta 1)(x)A_Q\bigr)({\mathcal{S}}_{\kappa_E}(Rf))(x)\Bigr|^2
\frac{d\mu(x)}{\delta_E(x)^{m-d}}=:A_1+A_2.
\nonumber
\end{eqnarray}
Pick now $\varepsilon\in(0,\beta)$ arbitrary and proceed to estimate $A_1$ as follows:
\begin{eqnarray}\label{N-NKy-2}
A_1\!&=& \!2\sum\limits_{\stackrel{k\in{\mathbb{Z}}}{k\geq\kappa_E}}
\sum\limits_{Q\in{\mathbb{D}}_k(E)}
\int_{{\mathcal{U}}_Q}\Bigl|\sum\limits_{\stackrel{l\in{\mathbb{Z}}}{l\geq\kappa_E}}\!\!\!
2^{-|k-l|\varepsilon}2^{|k-l|\varepsilon}
\delta_E(x)^\upsilon\bigl(\Theta -(\Theta 1)(x)A_Q\bigr)(D_l\widetilde{D}_lf)(x)\Bigr|^2
\frac{d\mu(x)}{\delta_E(x)^{m-d}}
\nonumber\\[4pt]
&\leq& 2\sum\limits_{\stackrel{k\in{\mathbb{Z}}}{k\geq\kappa_E}}
\sum\limits_{Q\in{\mathbb{D}}_k(E)}
\int_{{\mathcal{U}}_Q}\Bigl(\sum\limits_{l\in{\mathbb{Z}}}2^{-2|k-l|\varepsilon}\Bigr)
\times
\nonumber\\[4pt]
&& \hskip 1.00in
\times\Bigl(\sum\limits_{\stackrel{l\in{\mathbb{Z}}}{l\geq\kappa_E}}2^{2|k-l|\varepsilon}
\bigl|\delta_E(x)^\upsilon\bigl(\Theta -(\Theta 1)(x)A_Q\bigr)
(D_l\widetilde{D}_lf)(x)\bigr|^2\Bigr)\frac{d\mu(x)}{\delta_E(x)^{m-d}}
\nonumber\\[4pt]
&\leq& C\sum\limits_{\stackrel{l\in{\mathbb{Z}}}{l\geq\kappa_E}}
\sum\limits_{\stackrel{k\in{\mathbb{Z}}}{k\geq\kappa_E}}
\sum\limits_{Q\in{\mathbb{D}}_k(E)}2^{2|k-l|\varepsilon}\int_{{\mathcal{U}}_Q}
\bigl|\delta_E(x)^\upsilon\bigl(\Theta -(\Theta 1)(x)A_Q\bigr)
(D_l\widetilde{D}_lf)(x)\bigr|^2\frac{d\mu(x)}{\delta_E(x)^{m-d}}
\nonumber\\[4pt]
&\leq& C\sum\limits_{\stackrel{l\in{\mathbb{Z}}}{l\geq\kappa_E}}
\sum\limits_{\stackrel{k\in{\mathbb{Z}}}{k\geq\kappa_E}}
\sum\limits_{Q\in{\mathbb{D}}_k(E)}2^{2|k-l|\varepsilon}2^{-2|k-l|\beta}
\inf_{{w}\in Q}\Bigl[M_E^2\bigr(|\widetilde{D}_lf|^r\bigr)({w})\Bigr]^\frac{2}{r}
\int_{{\mathcal{U}}_Q}2^{k(m-d)}\,d\mu
\nonumber\\[4pt]
&\leq& C\sum\limits_{\stackrel{l\in{\mathbb{Z}}}{l\geq\kappa_E}}
\sum\limits_{\stackrel{k\in{\mathbb{Z}}}{k\geq\kappa_E}}
\sum\limits_{Q\in{\mathbb{D}}_k(E)}2^{-2|k-l|(\beta-\varepsilon)}
\int_Q\Bigl[M_E^2\bigr(|\widetilde{D}_lf|^r\bigr)\Bigr]^\frac{2}{r}\,d\sigma
\nonumber\\[4pt]
&=& C\sum\limits_{\stackrel{l\in{\mathbb{Z}}}{l\geq\kappa_E}}
\sum\limits_{\stackrel{k\in{\mathbb{Z}}}{k\geq\kappa_E}}
2^{-2|k-l|(\beta-\varepsilon)}
\int_E\Bigl[M_E^2\bigr(|\widetilde{D}_lf|^r\bigr)\Bigr]^\frac{2}{r}\,d\sigma
\leq C\sum\limits_{\stackrel{l\in{\mathbb{Z}}}{l\geq\kappa_E}}
\int_E\Bigl[M_E^2\bigr(|\widetilde{D}_lf|^r\bigr)\Bigr]^\frac{2}{r}\,d\sigma
\nonumber\\[4pt]
&\leq& C\sum\limits_{\stackrel{l\in{\mathbb{Z}}}{l\geq\kappa_E}}
\int_E|\widetilde{D}_lf|^2\,d\sigma\leq C\int_E|f|^2\,d\sigma.
\end{eqnarray}
The first inequality in \eqref{N-NKy-2} uses the Cauchy-Schwarz inequality, 
the second inequality uses the fact that
$\sum\limits_{l\in{\mathbb{Z}}}2^{-2|k-l|\varepsilon}=C$,
for some finite positive constant independent of $k\in{\mathbb{Z}}$,
the third inequality employs \eqref{Lvsn}, while the forth inequality is based on 
the fact that $\mu(\mathcal{U}_Q)\leq C2^{-km}$ and $2^{-kd}\leq C\sigma(Q)$, 
for all $Q\in{\mathbb{D}}_k(E)$. Since $\varepsilon\in(0,\beta)$, we have 
$\sum\limits_{k\in{\mathbb{Z}}}2^{-2|k-l|(\beta-\varepsilon)}=C$, which is used in
the fifth inequality in \eqref{N-NKy-2}. The sixth inequality in 
\eqref{N-NKy-2} follows from the boundedness of the Hardy-Littlewood maximal operator
$M_E$ on $L^{\frac{2}{r}}(E,\sigma)$ (recall that $2/r>1$) and the fact that 
$|\widetilde{D}_lf|^r\in L^{\frac{2}{r}}(E,\sigma)$. Finally, the last inequality in 
\eqref{N-NKy-2} uses \eqref{PKc-2}. 

There remains to obtain a similar bound for $A_2$ introduced in \eqref{N-NKy-2AA}.
Note that if $E$ is unbounded then $\kappa_E=-\infty$ so actually $A_2=0$ since we
agreed that ${\mathcal{S}}_{-\infty}=0$. Consider therefore the case when $E$ is 
bounded, in which scenario we have $k\geq\kappa_E\in{\mathbb{Z}}$. An inspection of
the proof of \eqref{Lvsn} in {\it Case~I} (of Step~I) shows that the function 
$D_lg={\mathcal{S}}_{l+1}g-{\mathcal{S}}_l g$ may actually be decoupled, i.e., 
be replaced by, say, ${\mathcal{S}}_l g$. This is because in {\it Case~I} we have 
only made use of the regularity of the integral kernel of $D_l$ (as opposed to 
{\it Case~II} where the vanishing condition of the integral kernel of $D_l$ is used)
and the integral kernel of ${\mathcal{S}}_l$ exhibits the same type of regularity.
Consequently, the same proof as before gives that for every $r\in(1,\infty)$ there exist 
finite positive constants $C$ and $\beta$ such that for each $k,l\in{\mathbb{Z}}$
with $k,l\geq\kappa_E$ such that $k\geq l$ and every $Q\in{\mathbb{D}}_k(E)$, 
there holds
\begin{eqnarray}\label{Lvsn-Sl}
\sup\limits_{x\in{\mathcal{U}}_Q}\Bigl|
\delta_E(x)^\upsilon\bigl(\Theta({\mathcal{S}}_lg)(x)
-(\Theta 1)(x)A_Q({\mathcal{S}}_lg)\bigr)\Bigr|
\leq C2^{-|k-l|\beta}\inf\limits_{{w}\in Q}
\Bigl[M_E^2(|g|^r)({w})\Bigr]^{\frac{1}{r}},
\end{eqnarray}
for every $g:E\to{\mathbb{R}}$ locally integrable. Applying \eqref{Lvsn-Sl} with
$l:=\kappa_E$ and $g:=Rf$ then yields
\begin{eqnarray}\label{N-NKy-A2}
A_2
&\leq & C\sum\limits_{k\in{\mathbb{Z}},\,k\geq\kappa_E}\sum\limits_{Q\in{\mathbb{D}}_k(E)}
2^{-2|k-\kappa_E|\beta}
\inf_{{w}\in Q}\Bigl[M_E^2\bigr(|Rf|^r\bigr)({w})\Bigr]^\frac{2}{r}
\int_{{\mathcal{U}}_Q}2^{k(m-d)}\,d\mu
\nonumber\\[4pt]
&\leq & C\sum\limits_{\stackrel{k\in{\mathbb{Z}}}{k\geq\kappa_E}}
\sum\limits_{Q\in{\mathbb{D}}_k(E)}2^{-2|k-\kappa_E|\beta}
\int_Q\Bigl[M_E^2\bigr(|Rf|^r\bigr)\Bigr]^\frac{2}{r}\,d\sigma
\nonumber\\[4pt]
&=& C\sum\limits_{\stackrel{k\in{\mathbb{Z}}}{k\geq\kappa_E}}
2^{-2|k-\kappa_E|\beta}
\int_E\Bigl[M_E^2\bigr(|Rf|^r\bigr)\Bigr]^\frac{2}{r}\,d\sigma
\leq C\int_E\Bigl[M_E^2\bigr(|Rf|^r\bigr)\Bigr]^\frac{2}{r}\,d\sigma
\nonumber\\[4pt]
&\leq & C\int_E|Rf|^2\,d\sigma\leq C\int_E|f|^2\,d\sigma,
\end{eqnarray}
since $R$ is a bounded operator on $L^2(E,\sigma)$. Now \eqref{N-NKy-2AA}, 
\eqref{N-NKy-2} and \eqref{N-NKy-A2} imply \eqref{N-NKy} completing the proof 
of the claim made in Step~II.

\vskip 0.10in
\noindent{\tt Step~III.} {\it The end-game in the proof of the implication
``\eqref{UEHg}$\Rightarrow$\eqref{G-UF.22}"}. Fix $f\in L^2(E,\sigma)$ and 
recall $\epsilon\in(0,1)$ from Lemma~\ref{Lem:CQinBQ-N} (here is where we use 
that $C_\ast$ is as in \eqref{NeD-67}). Then
by \eqref{doj.cF} and \eqref{doj} we may write
\begin{eqnarray}\label{WLQ}
&& \hskip -0.40in
\int_{\bigl\{x\in{\mathscr{X}}\setminus E:\,\delta_E(x)<\epsilon\,
{\rm diam}_{\rho}(E)\bigr\}}
|(\Theta f)(x)|^2\delta_E(x)^{2\upsilon-(m-d)}\,d\mu(x) 
\nonumber\\[4pt]
&& \hskip 0.40in
\leq\int_{\bigcup\limits_{Q\in{\mathbb{D}}(E)}{\mathcal{U}}_Q}
|(\Theta f)(x)|^2\delta_E(x)^{2\upsilon-(m-d)}\,d\mu(x) 
\nonumber\\[4pt]
&& \hskip 0.40in
\leq C\sum\limits_{k\in{\mathbb{Z}},\,k\geq\kappa_E}\sum\limits_{Q\in{\mathbb{D}}_k(E)}
\int_{{\mathcal{U}}_Q}\bigl|(\Theta f)(x)-(\Theta 1)(x)A_Qf\bigr|^2
\delta_E(x)^{2\upsilon-(m-d)}\,d\mu(x)
\nonumber\\[4pt]
&& \hskip 0.50in
+C\sum\limits_{k\in{\mathbb{Z}},\,k\geq\kappa_E}\sum\limits_{Q\in{\mathbb{D}}_k(E)}
\int_{{\mathcal{U}}_Q}\bigl|(\Theta 1)(x)A_Qf\bigr|^2
\delta_E(x)^{2\upsilon-(m-d)}\,d\mu(x).
\end{eqnarray}
Observe that if we set 
$B_Q:=\int_{{\mathcal{U}}_Q}|(\Theta 1)(x)|^2\delta_E(x)^{2\upsilon-(m-d)}\,d\mu(x)$
for each $Q\in{\mathbb{D}}(E)$, then in view of \eqref{gZSZ-3}, \eqref{doj}, 
and \eqref{UEHg} there holds
\begin{eqnarray}\label{WLQ-5}
\sum\limits_{Q'\in{\mathbb{D}}(E),\,Q'\subseteq Q}B_Q
\leq C\int_{T_E(Q)}|(\Theta 1)(x)|^2\delta_E(x)^{2\upsilon-(m-d)}\,d\mu(x)\leq C\sigma(Q),
\quad\forall\,Q\in{\mathbb{D}}(E).
\end{eqnarray}
Thus, the numerical sequence $\bigl\{B_Q\bigr\}_{Q\in{\mathbb{D}}(E)}$ 
satisfies \eqref{TAkB}. Consequently, Lemma~\ref{brFC} applies and gives 
\begin{eqnarray}\label{WLQ-2}
&& \hskip -0.70in
\sum\limits_{Q\in{\mathbb{D}}(E)}
\int_{{\mathcal{U}}_Q}\bigl|(\Theta 1)(x)A_Qf\bigr|^2
\delta_E(x)^{2\upsilon-(m-d)}\,d\mu(x)
\nonumber\\[4pt]
&& \hskip 0.50in
\leq C\int_E\Bigl[\,\sup_{Q\in{\mathbb{D}}(E),\,x\in Q}\meanint_Q|f|\,d\sigma
\Bigr]^2\,d\sigma(x)
\nonumber\\[4pt]
&& \hskip 0.50in
\leq C\int_E(M_E f)^2\,d\sigma\leq C\int_E|f|^2\,d\sigma,
\end{eqnarray}
where for the last inequality in \eqref{WLQ-2} we have used the boundedness
of $M_E$ on $L^2(E,\sigma)$. By combining \eqref{WLQ}, \eqref{N-NKy} and \eqref{WLQ-2}
we therefore obtain
\begin{eqnarray}\label{WLD-1s}
\int_{\bigl\{x\in{\mathscr{X}}\setminus E:\,\delta_E(x)<\epsilon\,
{\rm diam}_{\rho}(E)\bigr\}}
|(\Theta f)(x)|^2\delta_E(x)^{2\upsilon-(m-d)}\,d\mu(x)\leq C\int_E|f|^2\,d\sigma.
\end{eqnarray}
Of course, this takes care of \eqref{G-UF.22} in the case when ${\rm diam}_\rho(E)=\infty$. 
To prove that \eqref{G-UF.22} continues to hold
in the case when $E$ is bounded, let $R:={\rm diam}_\rho(E)\in(0,\infty)$, 
fix $x_0\in E$ and set 
${\mathcal{O}}:=\bigl\{x\in{\mathscr{X}}\setminus E:\,\epsilon R\leq\delta_E(x)\bigr\}$.
Then for each $x\in{\mathcal{O}}$ there exists $y\in E$ such that
$\rho_{\#}(x,y)<2\delta_E(x)$, hence
\begin{eqnarray}\label{Bvvc}
\rho(x,x_0)\leq C_\rho^2\rho_{\#}(x,x_0)\leq C_\rho^2\max\{\rho_{\#}(x,y),
\rho_{\#}(y,x_0)\}\leq C_\rho^2\max\{2,\tfrac{1}{\epsilon}\}\delta_E(x).
\end{eqnarray}
Thus, $\rho(x,x_0)\approx\delta_E(x)$ uniformly for $x\in{\mathcal{O}}$. Based on this, 
estimate \eqref{hszz}, and H\"older's inequality we then obtain the pointwise estimate 
$|(\Theta f)(x)|^2\leq CR^{\,d}\|f\|_{L^2(E,\sigma)}^2\rho(x,x_0)^{-2(d+\upsilon)}$ 
for each $x\in{\mathcal{O}}$.  Consequently, for some sufficiently small $c>0$ and some 
$C\in(0,\infty)$ independent of $f$ and $R$, we may estimate
\begin{eqnarray}\label{WL-1B22}
&& \hskip -0.40in
\int_{\mathcal{O}}|(\Theta f)(x)|^2\delta_E(x)^{2\upsilon-(m-d)}\,d\mu(x)
\nonumber\\[4pt]
&& \hskip 0.40in
\leq CR^{\,d}\|f\|_{L^2(E,\sigma)}^2
\int_{{\mathscr{X}}\setminus B_{\rho_{\#}}(x_0,cR)}\rho_{\#}(x,x_0)^{-m-d}\,d\mu(x)
\nonumber\\[4pt]
&& \hskip 0.40in
\leq CR^{\,d}\|f\|_{L^2(E,\sigma)}^2\sum\limits_{j=1}^\infty
\int_{B_{\rho_{\#}}(x_0,c2^{j+1}R)\setminus B_{\rho_{\#}}(x_0,c2^j R)}
\rho_{\#}(x,x_0)^{-m-d}\,d\mu(x)
\nonumber\\[4pt]
&& \hskip 0.40in
\leq CR^{\,d}\|f\|_{L^2(E,\sigma)}^2\sum\limits_{j=1}^\infty(2^j R)^{-m-d}
\mu\bigl(B_{\rho_{\#}}(x_0,c2^{j+1}R)\bigr)
\nonumber\\[4pt]
&& \hskip 0.40in
\leq CR^{\,d}\|f\|_{L^2(E,\sigma)}^2\sum\limits_{j=1}^\infty(2^j R)^{-m-d}
(2^j R)^m\leq C\|f\|_{L^2(E,\sigma)}^2.
\end{eqnarray}
Now \eqref{G-UF.22} follows by combining \eqref{WLD-1s} and \eqref{WL-1B22}.

At this stage in the proof of the theorem, we are left with establishing 
the converse implication with the regularity assumption \eqref{hszz-3} on the kernel
now dropped. With the goal of proving \eqref{UEHg}, suppose that \eqref{G-UF} holds 
for some $\eta\in(0,\infty)$. Assume first that ${\rm diam}_\rho(E)<\infty$ 
and pick an arbitrary $\eta_o\in(\eta,\infty)$. We may then estimate 
\begin{eqnarray}\label{WLQ-AA}
&& \hskip -0.40in
\int_{\bigl\{x\in{\mathscr{X}}\setminus E:\,\delta_E(x)<\eta_o\,
{\rm diam}_{\rho}(E)\bigr\}}
|(\Theta f)(x)|^2\delta_E(x)^{2\upsilon-(m-d)}\,d\mu(x) 
\nonumber\\[4pt]
&& \hskip 0.30in
=\int_{\bigl\{x\in{\mathscr{X}}\setminus E:\,\delta_E(x)<\eta\,
{\rm diam}_{\rho}(E)\bigr\}}
|(\Theta f)(x)|^2\delta_E(x)^{2\upsilon-(m-d)}\,d\mu(x) 
\\[4pt]
&& \hskip 0.40in
+\int_{\bigl\{x\in{\mathscr{X}}\setminus E:\,
\eta\,{\rm diam}_{\rho}(E)\leq\delta_E(x)<\eta_o\,{\rm diam}_{\rho}(E)\bigr\}}
|(\Theta f)(x)|^2\delta_E(x)^{2\upsilon-(m-d)}\,d\mu(x),
\nonumber 
\end{eqnarray} 
and then observing that since by \eqref{hszz} and H\"older's inequality we have the
pointwise estimate 
$|(\Theta f)(x)|^2\leq C\|f\|_{L^2(E,\sigma)}^2[{\rm diam}_{\rho}(E)]^{-d-2\upsilon}$ 
whenever $\eta\,{\rm diam}_{\rho}(E)\leq\delta_E(x)<\eta_o\,{\rm diam}_{\rho}(E)$, 
the last integral in \eqref{WLQ-AA} may also be bounded by $C\int_E|f|^2\,d\sigma$, 
for some finite positive geometric constant independent of ${\rm diam}_\rho(E)$.
The bottom line is that there is no loss of generality in assuming that 
$\eta>0$ appearing in \eqref{G-UF} is as large as desired.

Assuming that this is the case, fix $Q\in{\mathbb{D}}(E)$ and, for some large finite
positive constant $C_o$, write 
\begin{eqnarray}\label{UEHg.CC}
&& \hskip -0.40in
\tfrac{1}{\sigma(Q)}\int_{T_E(Q)}
|\Theta 1(x)|^2\delta_E(x)^{2\upsilon-(m-d)}\,d\mu(x)
\nonumber\\[4pt]
&& \hskip 0.40in
\leq\tfrac{2}{\sigma(Q)}\int_{T_E(Q)}
\bigl|\bigl(\Theta {\mathbf{1}}_{E\cap B_{\rho_{\#}}(x_Q,C_o\ell(Q))}\bigr)(x)\bigr|^2
\delta_E(x)^{2\upsilon-(m-d)}\,d\mu(x)
\nonumber\\[4pt]
&& \hskip 0.50in
+\tfrac{2}{\sigma(Q)}\int_{T_E(Q)}
\bigl|\bigl(\Theta{\mathbf{1}}_{E\setminus B_{\rho_{\#}}(x_Q,C_o\ell(Q))}\bigr)(x)\bigr|^2
\delta_E(x)^{2\upsilon-(m-d)}\,d\mu(x)
\nonumber\\[4pt]
&& \hskip 0.40in
=:{\mathcal{I}}_1+{\mathcal{I}}_2.
\end{eqnarray}
Then, granted \eqref{G-UF} and keeping in mind \eqref{dFvK}
and the fact that $\eta$ is large, we may write 
\begin{eqnarray}\label{UEHg.CC2}
{\mathcal{I}}_1 &\leq &\tfrac{2}{\sigma(Q)}
\int_{\{x\in{\mathscr{X}}:\,0<\delta_E(x)<\eta\,{\rm diam}_{\rho}(E)\}}
\bigl|\bigl(\Theta {\mathbf{1}}_{E\cap B_{\rho_{\#}}(x_Q,C_o\ell(Q))}\bigr)(x)\bigr|^2
\delta_E(x)^{2\upsilon-(m-d)}\,d\mu(x)
\nonumber\\[4pt]
&\leq & \tfrac{C}{\sigma(Q)}\int_{E}
\bigl|{\mathbf{1}}_{E\cap B_{\rho_{\#}}(x_Q,C_o\ell(Q))}(x)\bigr|^2\,d\sigma(x)
\nonumber\\[4pt]
&\leq & C\frac{\sigma\bigl(E\cap B_{\rho_{\#}}(x_Q,C_o\ell(Q))\bigr)}{\sigma(Q)}\leq C,
\end{eqnarray}
given that $\sigma$ is doubling. To estimate ${\mathcal{I}}_2$, observe that 
there exists $C\in(0,\infty)$ with the property that  
\begin{eqnarray}\label{UEHg.CC3}
&& \hskip -0.30in
\bigl|\bigl(\Theta{\mathbf{1}}_{E\setminus B_{\rho_{\#}}(x_Q,C_o\ell(Q))}\bigr)(x)\bigr|
\\[4pt]
&& \hskip 0.20in
\leq C\!\!\!\int\limits_{E\setminus B_{\rho_{\#}}(x_Q,C_o\ell(Q))}
\frac{\delta_E(x)^{-a}}{\rho_{\#}(x,y)^{d+\upsilon-a}}\,d\sigma(y)
\leq C\ell(Q)^{-(\upsilon-a)}\delta_E(x)^{-a},\qquad\forall\,x\in T_E(Q).
\nonumber
\end{eqnarray}
This is based on \eqref{hszz}, \eqref{WBA} (used here with $f\equiv 1$ 
and $\varepsilon=\upsilon-a>0$) and on the fact that $C_o$ has been 
chosen sufficiently large (compare with \eqref{AsQ-3}). Consequently, 
from \eqref{UEHg.CC3}, \eqref{dFvK}, and \eqref{lbzF} in Lemma~\ref{geom-lem}
(for which we recall that $\upsilon-a>0$)
\begin{eqnarray}\label{UEHg.CC4}
{\mathcal{I}}_2 &\leq & \tfrac{C}{\sigma(Q)}\ell(Q)^{-2(\upsilon-a)}
\int_{T_E(Q)}\delta_E(x)^{2(\upsilon-a)-(m-d)}\,d\mu(x)
\nonumber\\[4pt]
&\leq & \tfrac{C}{\sigma(Q)}\ell(Q)^{-2(\upsilon-a)}
\int_{B_{\rho_{\#}}(x_Q,C\ell(Q))}\delta_E(x)^{2(\upsilon-a)-(m-d)}\,d\mu(x)
\nonumber\\[4pt]
&\leq & \tfrac{C}{\sigma(Q)}\ell(Q)^{-2(\upsilon-a)}
\ell(Q)^{m-d+2(\upsilon-a)-(m-d)}\ell(Q)^d\leq C<\infty,
\end{eqnarray}
given that $\sigma(Q)\approx\ell(Q)^d$.
In concert, \eqref{UEHg.CC}-\eqref{UEHg.CC4} prove \eqref{UEHg}, and this 
finishes the proof of the theorem. 
\end{proof}

\subsection{An arbitrary codimension local $T(b)$ theorem for square functions}
\label{SSect:3.2}

We continue to work in the context introduced at the beginning of Section~\ref{Sect:3}. The main result in this subsection is a local $T(b)$ theorem for square functions,
to the effect that {\it a square function estimate for the integral operator $\Theta$ 
holds if there exists a suitably nondegenerate family of functions $\{b_Q\}$, indexed by dyadic cubes $Q$ in $E$, for which there is local scale-invariant $L^2$ control of $\Theta b_Q$, appropriately weighted by a power of the distance to $E$}. To state this formally, the reader is again advised 
to recall the dyadic cube grid from Proposition~\ref{Diad-cube} and the 
regularized distance function to a set from \eqref{REG-DDD}.

\begin{theorem}\label{Thm:localTb}
Let $d,m$ be two real numbers such that $0<d<m$. Assume that $({\mathscr{X}},\rho,\mu)$
is an $m$-dimensional {\rm ADR} space, $E$ is a closed subset 
of $({\mathscr{X}},\tau_\rho)$, and $\sigma$ is a Borel regular measure 
on $(E,\tau_{\rho|_{E}})$ with the property that $(E,\rho\bigl|_E,\sigma)$ is a 
$d$-dimensional {\rm ADR} space.

Suppose that $\Theta$ is the integral operator defined in \eqref{operator} 
with a kernel ${\theta}$ as in \eqref{K234}, \eqref{hszz}, \eqref{hszz-3}. 
Furthermore, let ${\mathbb{D}}(E)$ denote a dyadic cube structure on $E$, 
consider a Whitney covering ${\mathbb{W}}_\lambda(\mathscr{X}\setminus E)$ 
of $\mathscr{X}\setminus E$ and a constant $C_\ast$ as in Lemma~\ref{Lem:CQinBQ-N}
and, corresponding to these, recall the dyadic Carleson tents from \eqref{gZSZ-3}. 

For these choices, assume that there exist finite constant $C_0\geq 1$, $c_0\in(0,1]$,
and a collection $\{b_Q\}_{Q\in{\mathbb{D}}(E)}$ of $\sigma$-measurable functions $b_Q:E\rightarrow{\mathbb{C}}$ such that for each $Q\in{\mathbb{D}}(E)$ the following
estimates hold:
\begin{enumerate}
\item $\int_E |b_Q|^2\,d\sigma\leq C_0\sigma(Q)$;
\item there exists $\widetilde{Q}\in{\mathbb{D}}(E)$, $\widetilde{Q}\subseteq Q$,
$\ell(\widetilde{Q})\geq c_0\ell(Q)$, and 
$\left|\int_{\widetilde{Q}}b_Q\,d\sigma\right|\geq\frac{1}{C_0}\,\sigma(\widetilde{Q})$;
\item $\int_{T_E(Q)}|(\Theta\,b_Q)(x)|^2
\delta_E(x)^{2\upsilon-(m-d)}\,d\mu(x)\leq C_0\sigma(Q)$.
\end{enumerate}

Then there exists a finite constant $C>0$ depending only on $C_0$, $C_{\theta}$, and the 
{\rm ADR} constants of $E$ and ${\mathscr{X}}$, as well as on ${\rm diam}_\rho(E)$
in the case when $E$ is bounded, such that for each function $f\in L^2(E,\sigma)$ 
one has
\begin{eqnarray}\label{G-UF-2}
\int_{\mathscr{X}\setminus E}|(\Theta f)(x)|^2\delta_E(x)^{2\upsilon-(m-d)}\,
d\mu(x)\leq C\int_E|f(x)|^2\,d\sigma(x).
\end{eqnarray}
\end{theorem}

Before giving the proof of Theorem~\ref{Thm:localTb} we present a
stopping-time construction and elaborate on the way this is used. 

\begin{lemma}\label{brjo}
Assume $(E,\rho,\sigma)$ is a space of homogeneous type with the property that $\sigma$
is Borel regular, and denote by ${\mathbb{D}}(E)$ a dyadic cube structure on $E$. 
Suppose that there exist finite constants $C_0\geq 1$, $c_0\in(0,1)$, and a collection
$\{b_Q\}_{Q\in{\mathbb{D}}(E)}$ of $\sigma$-measurable functions
$b_Q:E\rightarrow{\mathbb{C}}$ such that 
\begin{eqnarray}\label{dbjpz}
&& \hskip -0.40in
\int_E |b_Q|^2\ d\sigma \leq C_0\sigma(Q)\quad
\mbox{for every $Q\in{\mathbb{D}}(E)$, and}
\\[4pt]
&& \hskip -0.40in
\forall\,Q\in{\mathbb{D}}(E)\quad\exists\,
\widetilde{Q}\in{\mathbb{D}}(E),\,\widetilde{Q}\subseteq Q, 
\,\ell(\widetilde{Q})\geq c_0\ell(Q)\,\,\,\mbox{with}\,\,\,
\left| \int_{\widetilde{Q}} b_Q\ d\sigma \right|\geq\frac{1}{C_0}\,\sigma(\widetilde{Q}).
\label{dbjpz-extra}
\end{eqnarray}
Then there exists a number $\eta\in(0,1)$ such that for every cube 
$Q\in{\mathbb{D}}(E)$, and each fixed $\widetilde{Q}$ as in \eqref{dbjpz-extra}, 
one can find a sequence $\bigl\{Q_j\bigr\}_{j\in J}\subseteq{\mathbb{D}}(E)$ of 
pairwise disjoint cubes satisfying the following properties:
\begin{enumerate}
\item[(i)] $Q_j\subseteq\widetilde{Q}$ for every $j\in J$ and 
$\sigma\Bigl(\widetilde{Q}\setminus\bigcup_{j\in J}Q_j\Bigr)\geq\eta\,\sigma(\widetilde{Q})$;
\item[(ii)] if
\begin{eqnarray}\label{lvB}
{\mathcal{F}}_Q:=\bigl\{Q'\in{\mathbb{D}}(E):\,Q'\subseteq\widetilde{Q}
\mbox{ and $Q'$ is not contained in $Q_j$ for every $j\in J$}\bigr\},
\end{eqnarray}
then $\left|\dmeanint_{Q'}b_Q\,d\sigma\right|\geq\frac{1}{2}$ for every
$Q'\in{\mathcal{F}}_Q$.
\end{enumerate}
\end{lemma}

\begin{proof}
Granted the regularity of the measure $\sigma$, it follows from 
{\it (3)} and {\it (9)} in Proposition~\ref{Diad-cube} that
for each $k\in{\mathbb{Z}}$ and each $Q\in{\mathbb{D}}_k(E)$ we have 
\begin{eqnarray}\label{T-rcs-BV}
\sigma\Bigl(Q\setminus\bigcup_{Q'\subseteq Q,\,Q'\in{\mathbb{D}}_\ell(E)}Q'\Bigr)=0
\qquad\mbox{for every $\ell\in{\mathbb{Z}}$ with $\ell\geq k$}.
\end{eqnarray}
Thanks to \eqref{dbjpz}-\eqref{dbjpz-extra}, we may re-normalize the functions 
$\{b_Q\}_{Q\in{\mathbb{D}}(E)}$ so that $\dmeanint_{\widetilde{Q}} b_Q\,d\sigma=1$ for each
$Q\in{\mathbb{D}}(E)$, where $\widetilde{Q}$ is as in \eqref{dbjpz-extra}. In the process,
the first inequality in \eqref{dbjpz} becomes 
\begin{eqnarray}\label{Fs23EE}
\int_E |b_Q|^2\ d\sigma\leq C_0^3\sigma(Q),\qquad
\mbox{for each }\,\,\,Q\in{\mathbb{D}}(E). 
\end{eqnarray}
Fix $Q\in{\mathbb{D}}(E)$ and a corresponding $\widetilde{Q}$ as in \eqref{dbjpz-extra}. 
In particular we have
\begin{eqnarray}\label{meYY}
\sigma(Q)\leq C_1\sigma(\widetilde{Q})\quad\mbox{for some $C_1\in[1,\infty)$ 
independent of $Q,\widetilde{Q}$}.
\end{eqnarray}
Next, perform a stopping-time argument for $\widetilde{Q}$ by successively dividing 
it into dyadic sub-cubes $Q'\subseteq \widetilde{Q}$ and stopping whenever 
${\rm Re}\,\dmeanint_{Q'}b_Q\,d\sigma\leq\frac{1}{2}$. That this is doable is ensured 
by \eqref{T-rcs-BV} and the re-normalization of $b_Q$. This yields a family of cubes 
$\bigl\{Q_j\bigr\}_{j\in J}\subseteq{\mathbb{D}}(E)$ such that:
\begin{enumerate}
\item[(1)] $Q_j\subseteq \widetilde{Q}\subseteq Q$ for each $j\in J$ and 
$Q_j\cap Q_{j'}=\emptyset$ whenever $j,j'\in J$, $j\not=j'$;
\item[(2)] ${\rm Re}\,\dmeanint_{Q_j}b_Q\,d\sigma\leq\frac{1}{2}$ for each $j\in J$;
\item[(3)] the family $\bigl\{Q_j\bigr\}_{j\in J}$ is maximal with respect to 
$(1)$ and $(2)$ above, i.e., if $Q'\in{\mathbb{D}}(E)$ is such that 
$Q'\subseteq \widetilde{Q}$, then either there exists $j_0\in J$ such that 
$Q'\subseteq Q_{j_0}$, or ${\rm Re}\,\dmeanint_{Q'}b_Q\,d\sigma>\frac{1}{2}$.
\end{enumerate}

Then we may write
\begin{eqnarray}\label{kvnS}
\sigma(\widetilde{Q}) & = &\int_{\widetilde{Q}}b_Q\,d\sigma
={\rm Re}\int_{\widetilde{Q}\setminus(\bigcup_{j\in J}Q_j)}b_Q\,d\sigma
+\sum\limits_{j\in J}{\rm Re}\int_{Q_j}b_Q\,d\sigma
\nonumber\\[4pt]
&\leq & \Bigl(\int_E|b_Q|^2d\sigma\Bigr)^{\frac{1}{2}}
\sigma\bigl(\widetilde{Q}\setminus\cup_{j\in J}Q_j\bigr)^{\frac{1}{2}}
+\tfrac{1}{2}\sum\limits_{j\in J}\sigma(Q_j)
\nonumber\\[4pt]
&\leq & C_0^{\frac{3}{2}}\sigma(Q)^{\frac{1}{2}}
\sigma\bigl(\widetilde{Q}\setminus\cup_{j\in J}Q_j\bigr)^{\frac{1}{2}}
+\tfrac{1}{2}\sigma(\widetilde{Q})
\nonumber\\[4pt]
&\leq & C_1^{\frac{1}{2}}C_0^{\frac{3}{2}}\sigma(\widetilde{Q})^{\frac{1}{2}}
\sigma\bigl(\widetilde{Q}\setminus\cup_{j\in J}Q_j\bigr)^{\frac{1}{2}}
+\tfrac{1}{2}\sigma(\widetilde{Q}),
\end{eqnarray}
where the first inequality in \eqref{kvnS} is based on H\"older's inequality and condition
$(2)$ above, the second inequality uses \eqref{Fs23EE} and $(1)$ above, while the last
inequality is a consequence of \eqref{meYY}. After absorbing 
$\tfrac{1}{2}\sigma(\widetilde{Q})$ in the leftmost side of \eqref{kvnS} and setting
$\eta:=\frac{1}{4C_1C_0^3}\in(0,1)$, it follows that 
$\sigma\bigl(\widetilde{Q}\setminus\cup_{j\in J}Q_j\bigr)\geq\eta\sigma(\widetilde{Q})$,
thus condition $(i)$ holds for the family $\bigl\{Q_j\bigr\}_{j\in J}$ constructed 
above. In addition it is immediate from property 
$(3)$ that condition $(ii)$ is also satisfied.
\end{proof}
 
A typical application of Lemma~\eqref{brjo} is exemplified by our next result.

\begin{lemma}\label{NIGz}
Assume that $({\mathscr{X}},\rho)$ is a geometrically doubling quasi-metric space,
$\mu$ is a Borel measure on $({\mathscr{X}},\tau_\rho)$ and that $E$ is a nonempty,
closed, proper subset of $({\mathscr{X}},\tau_\rho)$. Also, suppose that $\sigma$ is a 
Borel regular measure on $E$ with the property that $(E,\rho\bigl|_E,\sigma)$ is a 
space of homogeneous type and denote by ${\mathbb{D}}(E)$ a dyadic cube structure 
on $E$. Next, assume that ${\mathbb{W}}_\lambda(\mathscr{X}\setminus E)$ is a
Whitney covering of $\mathscr{X}\setminus E$ as in Lemma~\ref{Lem:CQinBQ}
(for some fixed $a\geq 1$), and recall the regions $\{\mathcal{U}_Q\}_{Q\in\mathbb{D}(E)}$ from \eqref{gZSZb} relative to this cover. Finally, assume the hypotheses of Lemma~\ref{brjo}, and for each cube $Q\in{\mathbb{D}}(E)$, recall the collection ${\mathcal{F}}_Q$ from \eqref{lvB} and define
\begin{eqnarray}\label{tnjG}
E^{\ast}_Q:=\bigcup\limits_{Q'\in{\mathcal{F}}_Q}{\mathcal{U}}_{Q'}.
\end{eqnarray}

Then for every $\gamma\in{\mathbb{R}}$ and every $\mu$-measurable function
$u:\mathscr{X}\setminus E\to{\mathbb{R}}$ it follows that
\begin{eqnarray}\label{BNKa}
\int_{E^{\ast}_Q}|u(x)|^2\delta_E(x)^\gamma\,d\mu(x)
\approx\sum\limits_{Q'\in{\mathcal{F}}_Q}\int_{{\mathcal{U}}_{Q'}}
\bigl|u(x){\textstyle{\dmeanint_{Q'}b_Q\,d\sigma}}\bigr|^2\delta_E(x)^\gamma\,d\mu(x),
\end{eqnarray}
with finite positive equivalence constants, 
depending only on $C_0$ from \eqref{dbjpz}-\eqref{dbjpz-extra}.
\end{lemma}

\begin{proof}
This readily follows by combining \eqref{doj}, \eqref{dbjpz}-\eqref{dbjpz-extra} 
and $(ii)$ in Lemma~\ref{brjo}.
\end{proof}

We are now ready to present the

\vskip 0.08in
\begin{proof}[Proof of Theorem~\ref{Thm:localTb}]
Based on Theorem~\ref{SChg}, it suffices to show that 
$|\Theta 1|^2\delta_E^{2\upsilon-(m-d)}\,d\mu$ is a Carleson measure in
$\mathscr{X}\setminus E$ relative to $E$, that is, that \eqref{UEHg} holds.
In a first stage, we shall show that \eqref{UEHg} holds for $\Theta$ replaced by 
some truncated operators $\Theta_i$, $i\in{\mathbb{N}}$. More precisely, 
for each $i\in{\mathbb{N}}$ consider the kernel
\begin{eqnarray}\label{LIH}
{\theta}_i(x,y):={\mathbf{1}}_{\{1/i<\delta_E<i\}}(x){\theta}(x,y),
\qquad\forall\,x\in\mathscr{X}\setminus E,\quad\forall\,y\in E,
\end{eqnarray}
and introduce the integral operator mapping $f:E\to{\mathbb{R}}$ into
\begin{eqnarray}\label{LIH-2}
(\Theta_if)(x):=\int_E{\theta}_i(x,y)f(y)\,d\sigma(y),\qquad
\forall\,x\in\mathscr{X}\setminus E.
\end{eqnarray}
Clearly, 
\begin{eqnarray}\label{LIH-3}
\Theta_i={\mathbf{1}}_{\{1/i<\delta_E<i\}}\Theta,\qquad\forall\,i\in{\mathbb{N}},
\end{eqnarray}
and, with $C_{\theta}$ as in \eqref{hszz}, for every $x\in\mathscr{X}\setminus E$ we have
\begin{eqnarray}\label{hszz-i}
&& \hskip -0.40in
|{\theta}_i(x,y)|\leq C_{\theta}\frac{\delta_E(x)^{-a}}{\rho(x,y)^{d+\upsilon-a}},\qquad\forall\,y\in E,
\\[4pt]
&& \hskip -0.40in
|{\theta}_i(x,y)-{\theta}_i(x,\widetilde{y})|\leq C_{\theta}\frac{\rho(y,\widetilde{y})^\alpha\delta_E(x)^{-a-\alpha}}
{\rho(x,y)^{d+\upsilon-a}}, 
\qquad\forall\,\widetilde{y},y\in E,\quad
\rho(y,\widetilde{y}) \leq \tfrac{1}{2}\rho(x,y).
\label{hszz-3i}
\end{eqnarray}
Then for each $i\in{\mathbb{N}}$ and each $x\in\mathscr{X}\setminus E$ 
by \eqref{LIH-2}, \eqref{hszz-i} and Lemma~\ref{Gkwvr} (given that $\upsilon-a>0$) we have
\begin{eqnarray}\label{LIH-4}
|(\Theta_i 1)(x)| &\leq & C{\mathbf{1}}_{\{1/i<\delta_E<i\}}(x)
\int_E\frac{\delta_E(x)^{-a}}{\rho_{\#}(x,y)^{d+\upsilon-a}}\,d\sigma(y)
\leq C {\mathbf{1}}_{\{1/i<\delta_E<i\}}(x)[\delta_E(x)]^{-\upsilon}
\nonumber\\[4pt]
& \leq & C i^\upsilon{\mathbf{1}}_{\{1/i<\delta_E<i\}}(x).
\end{eqnarray}
Recalling now \eqref{dFvK}, estimate \eqref{LIH-4} further yields
(with $x_Q$ denoting the center of $Q$)
\begin{eqnarray}\label{LIH-5}
&& \hskip -0.50in
\int_{T_E(Q)}|(\Theta_i 1)(x)|^2\delta_E(x)^{2\upsilon-(m-d)}\,d\mu(x)
\nonumber\\[4pt]
&& \hskip 0.40in
\leq Ci^{2\upsilon}\int_{x\in B_{\rho_{\#}}(x_Q,C\ell(Q)),\,\delta_E(x)<i}
\delta_E(x)^{2\upsilon-(m-d)}\,d\mu(x)
\nonumber\\[4pt]
&& \hskip 0.40in
\leq C i^{4\upsilon}\ell(Q)^d\leq Ci^{4\upsilon}\sigma(Q),\qquad
\forall\,Q\in{\mathbb{D}}(E),
\end{eqnarray}
for some constant $C\in(0,\infty)$ which does not depend on $Q$ and $i$, where
the second inequality in \eqref{LIH-5} is a consequence of Lemma~\ref{geom-lem}.
Hence, if we now define
\begin{eqnarray}\label{LIH-6}
c_i:=\sup_{Q\in{\mathbb{D}}(E)}
\Bigl[\tfrac{1}{\sigma(Q)}\int_{T_E(Q)}|(\Theta_i 1)(x)|^2\delta_E(x)^{2\upsilon-(m-d)}\,d\mu(x)\Bigr],
\qquad \forall\,i\in{\mathbb{N}},
\end{eqnarray}
then \eqref{LIH-5} implies $0\leq c_i\leq Ci^{4\upsilon}$ for each $i\in{\mathbb{N}}$. 
In particular, each $c_i$ is finite. Our goal is to show that actually 
\begin{eqnarray}\label{LIH-6B}
\sup\limits_{i\in{\mathbb{N}}}c_i<\infty.
\end{eqnarray}
To this end, fix $Q\in{\mathbb{D}}(E)$ and for this $Q$ consider some $\widetilde{Q}$
satisfying {\it 2.} in the hypothesis of Theorem~\ref{Thm:localTb} (recall the constant
$c_0$). In particular, we have that 
\begin{eqnarray}\label{nzvd}
\exists\,p\in{\mathbb{N}}\quad\mbox{satisfying}\quad
p\leq -\log_{2}(c_0)\quad\mbox{and such that}\quad
\widetilde{Q}\in{\mathbb{D}}_p(E).
\end{eqnarray}
Next recall Lemma~\ref{brjo} and Lemma~\ref{NIGz} and the notation therein. 
Then, from the definition of $E_Q^\ast$, \eqref{nzvd} and \eqref{gZSZ-3}, we have 
\begin{eqnarray}\label{nzvd-Dsf}
T_E(Q)\subseteq E_Q^\ast\cup\Bigl(\bigcup\limits_{j\in J}T_E(Q_j)\Bigr)
\cup\Bigl(\bigcup\limits_{Q''\in{\mathbb{D}}_p(E),\,
Q''\subseteq Q,\,Q''\not=\widetilde{Q}}T_E(Q'')\Bigr).
\end{eqnarray}
Consequently, for each $i\in{\mathbb{N}}$ we may write
\begin{eqnarray}\label{Vhx}
&& \hskip -0.30in
\int_{T_E(Q)}|(\Theta_i 1)(x)|^2\delta_E(x)^{2\upsilon-(m-d)}\,d\mu(x)
\\[4pt]
&& \hskip 0.20in
\leq\int\limits_{E_Q^\ast}|(\Theta_i 1)(x)|^2\delta_E(x)^{2\upsilon-(m-d)}\,d\mu(x)
+\sum\limits_{j\in J}\,\,\int\limits_{T_E(Q_j)}|(\Theta_i 1)(x)|^2
\delta_E(x)^{2\upsilon-(m-d)}\,d\mu(x)
\nonumber\\[4pt]
&& \hskip 0.25in
+\sum\limits_{\stackrel{Q''\in{\mathbb{D}}_p(E),\,Q''\subseteq Q}{Q''\not=\widetilde{Q}}}
\,\,\int\limits_{T_E(Q'')}|(\Theta_i 1)(x)|^2\delta_E(x)^{2\upsilon-(m-d)}\,d\mu(x).
\nonumber
\end{eqnarray}
To estimate the first integral in the right hand-side of \eqref{Vhx} start with
\eqref{BNKa} written for $u:=\Theta_i 1$. Keeping in mind that 
$|\Theta_i1|\leq|\Theta 1|$ for all $i\in{\mathbb{N}}$, we obtain
\begin{eqnarray}\label{AMpR}
&&\hskip -0.20in
\int_{E^{\ast}_Q}|\Theta_i 1(x)|^2\delta_E(x)^{2\upsilon-(m-d)}\,d\mu(x)
\leq C\sum\limits_{Q'\in{\mathcal{F}}_Q}\int_{{\mathcal{U}}_{Q'}}
|(\Theta 1)(x)A_{Q'}b_Q|^2\delta_E(x)^{2\upsilon-(m-d)}\,d\mu(x)
\nonumber\\[4pt]
&&\qquad \leq C\sum\limits_{Q'\in{\mathbb{D}}(E),\,Q'\subseteq Q}
\int_{{\mathcal{U}}_{Q'}}|(\Theta b_Q)(x)|^2\delta_E(x)^{2\upsilon-(m-d)}\,d\mu(x)
\nonumber\\[4pt]
&&\qquad \quad+C\!\!\!\sum\limits_{Q'\in{\mathbb{D}}(E)}\int_{{\mathcal{U}}_{Q'}}
\bigl|(\Theta b_Q)(x)-(\Theta 1)(x)A_{Q'}b_Q\bigr|^2\delta_E(x)^{2\upsilon-(m-d)}\,d\mu(x)
\nonumber\\[4pt]
&&\qquad \leq C\int_{T_E(Q)}|(\Theta b_Q)(x)|^2\delta_E(x)^{2\upsilon-(m-d)}\,d\mu(x)
+C\int_E|b_Q|^2\,d\sigma\leq C\sigma(Q).
\end{eqnarray}
The second inequality in \eqref{AMpR} is immediate, the third uses
\eqref{gZSZ-3} and \eqref{N-NKy} (the latter applied with $f:=b_Q$), while the fourth 
uses assumptions {\it 1} and {\it 3} of Theorem~\ref{Thm:localTb}.

Consider next the first sum in the right hand-side of \eqref{Vhx}. Upon recalling
\eqref{LIH-6} and the properties of $Q_j$'s in Lemma~\ref{brjo} we may write
\begin{eqnarray}\label{Vhx-2}
\hskip -0.30in
\sum\limits_{j\in J}\int\limits_{T_E(Q_j)}
|(\Theta_i 1)(x)|^2\delta_E(x)^{2\upsilon-(m-d)}\,d\mu(x)
&\leq & c_i \sum\limits_{j\in J}\sigma(Q_j)
=c_i\,\sigma\bigl(\cup_{j\in J}Q_j\bigr)
\nonumber\\[4pt]
&=& c_i\,\sigma(\widetilde{Q})
-c_i\,\sigma\bigl(\widetilde{Q}\setminus \cup_{j\in J}Q_j\bigr)
\nonumber\\[4pt]
& \leq & c_i\,(1-\eta)\,\sigma(\widetilde{Q}).
\end{eqnarray}
Upon recalling \eqref{LIH-6} we obtain
\begin{eqnarray}\label{Vh-Es}
\sum\limits_{\stackrel{Q''\in{\mathbb{D}}_p(E),\,Q''\subseteq Q}{Q''\not=\widetilde{Q}}}
\,\,\int\limits_{T_E(Q'')}|(\Theta_i 1)(x)|^2\delta_E(x)^{2\upsilon-(m-d)}\,d\mu(x)
\leq c_i\sigma(Q\setminus\widetilde{Q}).
\end{eqnarray}
In concert, \eqref{Vhx}, \eqref{AMpR}, \eqref{Vhx-2}, and \eqref{Vh-Es} imply that 
there exists a finite constant $C>0$ with the property that for every 
$i\in{\mathbb{N}}$ there holds
\begin{eqnarray}\label{Vhx-3}
\int_{T_E(Q)}|\Theta_i1|^2\delta_E^{2\upsilon-(m-d)}\,d\mu
& \leq & c_i\,(1-\eta)\,\sigma(\widetilde{Q})+c_i\sigma(Q\setminus\widetilde{Q})+C\sigma(Q)
\nonumber\\[4pt]
& \leq & c_i\,\sigma(Q)-c_i\eta\,\sigma(\widetilde{Q})+C\sigma(Q)
\nonumber\\[4pt]
& \leq & c_i\,\sigma(Q)-c_i\eta C_1^{-1}\,\sigma(Q)+C\sigma(Q)
\nonumber\\[4pt]
& = & c_i\,(1-\eta C_1^{-1})\,\sigma(Q)+C\sigma(Q),
\quad\forall\,Q\in{\mathbb{D}}(E),
\end{eqnarray}
where for the last inequality in \eqref{Vhx-3} we have used \eqref{meYY}.
Dividing both sides of \eqref{Vhx-3} by $\sigma(Q)$ and then taking the supremum over 
all $Q\in{\mathbb{D}}(E)$ we obtain $c_i\leq c_i\,(1-\eta C_1^{-1})+C$ and furthermore,
since $c_i\leq Ci^{4\upsilon}<\infty$, that $c_i\leq\eta^{-1}C_1C$ for all
$i\in{\mathbb{N}}$. This finishes the proof of \eqref{LIH-6B}. 

Having established this, for each $Q\in{\mathbb{D}}(E)$
we may then write, using \eqref{LIH-3} and Lebesgue's Monotone Convergence Theorem,
\begin{eqnarray}\label{Vhx-4} 
\int_{T_E(Q)}|(\Theta 1)(x)|^2\delta_E(x)^{2\upsilon-(m-d)}\,d\mu(x)
& = & \lim\limits_{i\to\infty}\int_{T_E(Q)}|(\Theta_i 1)(x)|^2\delta_E(x)^{2\upsilon-(m-d)}\,d\mu(x)
\nonumber\\[4pt]
& \leq &(\sup\limits_{i\in{\mathbb{N}}}c_i)\,\sigma(Q)\leq C\sigma(Q),
\end{eqnarray}
for some finite constant $C>0$ independent of $Q$. This completes the proof of 
\eqref{UEHg} and finishes the proof of Theorem~\ref{Thm:localTb}.
\end{proof}

\section{An Inductive Scheme for Square Function Estimates}
\setcounter{equation}{0}
\label{Sect:4}

We now apply the local $T(b)$ Theorem from the previous section to establish an inductive scheme for square function estimates. More specifically, we show that an integral operator $\Theta$, associated with an Ahlfors-David regular set $E$ as in~\eqref{operator}, satisfies square function estimates whenever the set $E$ contains (uniformly, at all scales and locations) so-called big pieces of sets on which square function estimates for $\Theta$ hold. In short, we say that big pieces of square function estimates (BPSFE) imply square function estimates (SFE). We emphasize that this 
``big pieces" functor is applied to square function estimates for an individual, 
fixed $\Theta$. Thus, the result to be proved in this section is not a consequence 
of the stability of UR sets under the big pieces functor, as our particular square function bounds may {\it not} be equivalent to the property that $E$ is UR. 

We continue to work in the context introduced at the beginning of Section~\ref{Sect:3}, except we must assume in addition that the integral kernel $\theta$ is not adapted to a fixed set~$E$. In particular, fix two real numbers $d,m$ such that $0<d<m$, and an $m$-dimensional {\rm ADR} space $({\mathscr{X}},\rho,\mu)$. In this context, suppose that 
\begin{eqnarray}\label{K234-A}
\begin{array}{c}
{\theta}:(\mathscr{X}\times\mathscr{X})
\setminus\{(x,x):\,x\in \mathscr{X}\}\longrightarrow{{\mathbb{R}}}
\\[4pt]
\mbox{is Borel measurable with respect to the product topology 
$\tau_\rho\times\tau_\rho$},
\end{array}
\end{eqnarray}
and has the property that there exist finite positive constants $C_{\theta}$, 
$\alpha$, $\upsilon$ such that for all $x,y\in\mathscr{X}$ with $x\neq y$ the 
following hold:
\begin{eqnarray}\label{hszz-A}
&& \hskip -0.40in
|{\theta}(x,y)|\leq\frac{C_{\theta}}{\rho(x,y)^{d+\upsilon}},
\\[4pt]
&& \hskip -0.40in
|{\theta}(x,y)-{\theta}(x,\widetilde{y})|\leq C_{\theta} \frac{\rho(y,\widetilde{y})^\alpha}{\rho(x,y)^{d+\upsilon+\alpha}}, 
\quad\forall\,\widetilde{y}\in\mathscr{X}\setminus\{x\}\,\,\mbox{ with }\,\,
\rho(y,\widetilde{y})\leq\tfrac{1}{2}\rho(x,y).\quad
\label{hszz-3-A}
\end{eqnarray}
Then for each closed subset $E$ of $({\mathscr{X}},\tau_\rho)$, 
and each Borel regular measure $\sigma$ 
on $(E,\tau_{\rho|_{E}})$ with the property that $(E,\rho\bigl|_E,\sigma)$ 
is a $d$-dimensional {\rm ADR} space, define the integral operator 
$\Theta_E$ for all functions $f\in L^p(E,\sigma)$, $1\leq p\leq\infty$, by
\begin{eqnarray}\label{operator-A}
(\Theta_E f)(x):=\int_E {\theta}(x,y)f(y)\,d\sigma(y),
\qquad\forall\,x\in\mathscr{X}\setminus E.
\end{eqnarray}

We begin by defining what it means for a set to have big pieces of square function estimates.

\begin{definition}\label{sjvs}
Consider two numbers $d,m\in(0,\infty)$ such that $m>d$, suppose that 
$({\mathscr{X}},\rho,\mu)$ is an $m$-dimensional {\rm ADR} space, and 
assume that ${\theta}$ is as in \eqref{K234-A}-\eqref{hszz-3-A}. In this context, 
a set $E\subseteq{\mathscr{X}}$ is said to have 
{\tt Big Pieces of Square Function Estimate} (or, simply {\tt BPSFE}) 
relative to the kernel $\theta$ provided the following conditions are satisfied:
\begin{enumerate}
\item[(i)] the set $E$ is closed in $({\mathscr{X}},\tau_\rho)$ and 
has the property that there exists a Borel regular measure $\sigma$ 
on $(E,\tau_{\rho|_{E}})$ such that $\bigl(E,\rho\bigl|_E,\sigma\bigr)$ 
is a $d$-dimensional {\rm ADR} space;
\item[(ii)] there exist finite positive constants $\eta$, $C_1$, and $C_2$ with the 
property that for each $x\in E$ and each real number $r\in(0,{\rm diam}_{\rho_{\#}}(E)]$ 
there exists a closed subset $E_{x,r}$ of $(\mathscr{X},\tau_\rho)$ such that if  
\begin{eqnarray}\label{HHH-57}
\sigma_{x,r}:={\mathscr{H}}_{{\mathscr{X}}\!,\,\rho_{\#}}^d\lfloor E_{x,r},
\quad\mbox{ where ${\mathscr{H}}_{{\mathscr{X}}\!,\,\rho_{\#}}^d$ 
is as in \eqref{HHH-56}},
\end{eqnarray}
then $\bigl(E_{x,r},\rho\big|_{E_{x,r}},\sigma_{x,r}\bigr)$ is a 
$d$-dimensional {\rm ADR} space, with {\rm ADR} constant $\leq C_1$, 
and which satisfies
\begin{eqnarray}\label{yvg}
\sigma\bigl(E_{x,r}\cap E\cap B_{\rho_{\#}}(x,r)\bigr)\geq\eta\,r^d,
\end{eqnarray}
as well as
\begin{eqnarray}\label{avhai}
\begin{array}{c}
\displaystyle
\int_{\mathscr{X}\setminus E_{x,r}}|\Theta_{E_{x,r}}f(z)|^2\,
{\rm dist}_{\rho_{\#}}(z,E_{x,r})^{2\upsilon-(m-d)}\,d\mu(z) 
\leq C_2\int_{E_{x,r}}|f|^2\,d\sigma_{x,r}\qquad
\\[8pt]
\mbox{for each function }\,\,f\in L^2(E_{x,r},\sigma_{x,r}),
\end{array}
\end{eqnarray}
where $\Theta_{E_{x,r}}$ is the operator associated with $E_{x,r}$ as 
in \eqref{operator-A}.
\end{enumerate}
\end{definition}

\noindent In the context of the above definition, the constants $\eta,C_1,C_2$ will 
collectively be referred to as the {\tt BPSFE character} of the set $E$.

The property of having BPSFE may be dyadically discretized as explained 
in the lemma below.

\begin{lemma}\label{hrsbrli}
Let $d,m\in(0,\infty)$ be such that $m>d$, assume that 
$({\mathscr{X}},\rho,\mu)$ is an $m$-dimensional {\rm ADR} space, and  
suppose that ${\theta}$ is as in \eqref{K234-A}-\eqref{hszz-3-A}. In addition, let 
$E\subseteq{\mathscr{X}}$ be such that {\it (i)} in Definition~\ref{sjvs} holds
and consider the dyadic cube structure ${\mathbb{D}}(E)$ on $E$ as in Proposition~\ref{Diad-cube}. 

Then the set $E$ has {\rm BPSFE} (relative to $\theta$) if and only if there 
exist finite positive constants $\eta,C_1,C_2$ with the property that for each ${Q\in{\mathbb{D}}(E)}$ there exists a closed set $E_Q\subseteq\mathscr{X}$ such 
that if 
\begin{eqnarray}\label{HHH-60}
\sigma_{Q}:={\mathscr{H}}_{{\mathscr{X}}\!,\,\rho_{\#}}^d\lfloor E_{Q},
\end{eqnarray}
then $\bigl(E_Q,\rho\bigl|_{E_Q},\sigma_Q\bigr)$ is a $d$-dimensional {\rm ADR} 
space with {\rm ADR} constant $\leq C_1$ which satisfies
\begin{eqnarray}\label{yvg2}
{\mathscr{H}}_{{\mathscr{X}}\!,\,\rho_{\#}}^d(E_Q\cap Q)
\geq\eta\,{\mathscr{H}}_{{\mathscr{X}}\!,\,\rho_{\#}}^d(Q)
\end{eqnarray}
as well as
\begin{eqnarray}\label{avhai2}
\begin{array}{c}
\displaystyle
\int_{\mathscr{X}\setminus E_Q}|\Theta_{E_Q} f(x)|^2\,
{\rm dist}_{\rho_{\#}}(x,E_Q)^{2\upsilon-(m-d)}\,d\mu(x)
\leq C_2\int_{E_Q} |f|^2\ d\sigma_Q,
\\[8pt]
\mbox{for each function }\,\,f\in L^2(E_Q,\sigma_Q).
\end{array}
\end{eqnarray}
\end{lemma}

\begin{proof}
The left-to-right implication is a simple consequence of \eqref{ha-GVV},
while the opposite one follows with the help of \eqref{ha-GL54}.
\end{proof}

We now state and prove the main result in this section. 

\begin{theorem}\label{Thm:BPSFtoSF}
Consider two numbers $d,m\in(0,\infty)$ such that $m>d$, suppose that 
$({\mathscr{X}},\rho,\mu)$ is an $m$-dimensional {\rm ADR} space, and 
assume that ${\theta}$ is as in \eqref{K234-A}-\eqref{hszz-3-A}. 

If the set $E\subseteq{\mathscr{X}}$ has {\rm BPSFE} relative to $\theta$ then 
there exists a finite constant $C>0$, depending only on $\rho$, $m$, $d$, 
$\upsilon$, $C_{\theta}$ (from \eqref{hszz-A}-\eqref{hszz-3-A}), 
the {\rm BPSFE} character of $E$, and the {\rm ADR} constants of $E$ and 
${\mathscr{X}}$, such that 
\begin{eqnarray}\label{vlnGG}
\int_{\mathscr{X}\setminus E}|\Theta_E f(x)|^2\,
\delta_E(x)^{2\upsilon-(m-d)}\,d\mu(x) 
\leq C\int_E|f|^2\,d\sigma,\qquad\forall\,f\in L^2(E,\sigma),
\end{eqnarray}
where 
\begin{eqnarray}\label{HHH-58}
\sigma:={\mathscr{H}}_{{\mathscr{X}}\!,\,\rho_{\#}}^d\lfloor E,
\quad\mbox{ with ${\mathscr{H}}_{{\mathscr{X}}\!,\,\rho_{\#}}^d$ as in \eqref{HHH-56}}.
\end{eqnarray}
\end{theorem}

\begin{proof}
Since $E$ has BPSFE relative to $\theta$, by Lemma~\ref{hrsbrli}, for 
each $Q\in{\mathbb{D}}(E)$ there exists $E_Q$ satisfying \eqref{yvg2}-\eqref{avhai2}.
For each $Q\in{\mathbb{D}}(E)$, we then define the function $b_Q:E\rightarrow{\mathbb{R}}$ by setting
\begin{eqnarray}\label{BBss}
b_Q(y):={\mathbf{1}}_{Q\cap E_Q}(y),\qquad\forall\,y\in E.
\end{eqnarray}
The strategy for proving \eqref{vlnGG} is to invoke Theorem~\ref{Thm:localTb}
for the family $\{b_Q\}_{Q\in{\mathbb{D}}(E)}$. As such, matters are reduced 
to checking that conditions {\it 1.}--{\it 3.} in Theorem~\ref{Thm:localTb} 
hold for the collection $\{b_Q\}_{Q\in{\mathbb{D}}(E)}$ defined in \eqref{BBss}. 
Now, condition {\it 1.} is immediate, while the validity of condition {\it 2.} 
(with $\widetilde{Q}:=Q$)
is a consequence of \eqref{yvg2}. Thus, it remains to check that condition {\it 3.}
holds as well. To this end, fix $Q\in{\mathbb{D}}(E)$ and for some constant
$C_1\in(1,\infty)$ to be specified later write (employing notation introduced in
\eqref{REG-DDD} in relation to both $E$ and $E_Q$)
\begin{eqnarray}\label{KLGB}
&& \hskip -0.60in
\int_{T_E(Q)}|\Theta_E b_Q(x)|^2\delta_E(x)^{2\upsilon-(m-d)}\,d\mu(x) 
\nonumber\\[4pt]
&& \hskip 0.40in
=\int_{T_E(Q)}|\Theta_E b_Q(x)|^2 
{\mathbf{1}}_{\{z\in{\mathscr{X}}:\,\delta_{E_Q}(z)>C_1\delta_E(z)\}}(x)\,
\delta_E(x)^{2\upsilon-(m-d)}\,d\mu(x)
\nonumber\\[4pt]
&& \hskip 0.50in
+\int_{T_E(Q)}|\Theta_E b_Q(x)|^2 
{\mathbf{1}}_{\{z\in{\mathscr{X}}:\,
C_1^{-1}\delta_E(z)\leq\delta_{E_Q}(z)\leq C_1\delta_E(z)\}}(x)\,
\delta_E(x)^{2\upsilon-(m-d)}\,d\mu(x)
\nonumber\\[4pt]
&& \hskip 0.50in
+\int_{T_E(Q)}|\Theta_E b_Q(x)|^2 
{\mathbf{1}}_{\{z\in{\mathscr{X}}:\,\delta_{E_Q}(z)<C_1^{-1}\delta_E(z)\}}(x)\,
\delta_E(x)^{2\upsilon-(m-d)}\,d\mu(x) 
\nonumber\\[4pt]
&& \hskip 0.40in
=: I_1+I_2+I_3.
\end{eqnarray}

To proceed with estimating $I_1$ we first obtain a pointwise bound for 
$\Theta_E b_Q$. To this end, first observe that 
\begin{eqnarray}\label{setO}
{\mathcal{O}}:=\bigl\{z\in{\mathscr{X}}:\,\delta_{E_Q}(z)>C_1\delta_E(z)\bigr\}
\,\Longrightarrow\,{\mathcal{O}}\cap E_Q=\emptyset.
\end{eqnarray}
Hence, \eqref{BBss}, \eqref{hszz-A} and \eqref{mMji} in Lemma~\ref{Gkwvr} give 
that, for some finite $C>0$ independent of the dyadic cube $Q$, 
\begin{eqnarray}\label{Lz}
|\Theta_E b_Q(x)| &=& \left|\int_E {\theta}(x,y)\,b_Q(y)\,d\sigma(y)\right|
\nonumber\\[4pt]
&\leq & \int_{E_Q} |{\theta}(x,y)|\,d\sigma(y)\leq\frac{C}{\delta_{E_Q}(x)^\upsilon},
\qquad\forall\,x\in{\mathcal{O}}.
\end{eqnarray}
Also, \eqref{yvg2} guarantees that $Q\cap E_Q\not=\emptyset$, and we fix a point 
$x_0\in Q\cap E_Q$. By \eqref{dFvK}, there exists $c\in(0,\infty)$ such that
$T_E(Q)\subseteq B_{\rho_{\#}}(x_0,c\,\ell(Q))$ which when combined with 
\eqref{Lz} gives
\begin{eqnarray}\label{gwo}
I_1\leq C\int_{B_{\rho_{\#}}(x_0,c\ell(Q))\cap{\mathcal{O}}}
\delta_{E_Q}(x)^{-2\upsilon}\delta_E(x)^{2\upsilon-(m-d)}\,d\mu(x).
\end{eqnarray}

At this stage, select a constant $M\in(C_\rho^2,\infty]$, choose $C_1\in(M,\infty)$, and observe that if 
$x\in{\mathcal{O}}$ then $\frac{\delta_{E_Q}(x)}{M}>\delta_E(x)$, 
hence $B_{\rho_{\#}}\bigl(x,\delta_{E_Q}(x)/M\bigr)\cap E\not=\emptyset$. Recalling
Lemma~\ref{segj}, it follows that there exists $C\in(0,\infty)$ such that 
\begin{eqnarray}\label{dxb}
\frac{\delta_{E_Q}(x)^d}{M^d}\leq C
{\mathscr{H}}_{{\mathscr{X}}\!,\,\rho_{\#}}^d
\Bigl(B_{\rho_{\#}}\Bigl(x,C_\rho\frac{\delta_{E_Q}(x)}{M}\Bigr)\cap E\Bigr),
\qquad\forall\,x\in{\mathcal{O}}.
\end{eqnarray}
Using this in \eqref{gwo} we obtain
\begin{eqnarray}\label{gwo-2}
I_1\leq C\!\!\!
\int\limits_{B_{\rho_{\#}}(x_0,c\ell(Q))\cap{\mathcal{O}}}
\Bigl(\int\limits_{B_{\rho_{\#}}
\bigl(x,C_\rho\delta_{E_Q}(x)/M\bigr)\cap E}
\hskip -0.40in
1\,d{\mathscr{H}}_{{\mathscr{X}}\!,\,\rho_{\#}}^d(z)\Bigr)
\delta_{E_Q}(x)^{-2\upsilon-d}\delta_E(x)^{2\upsilon-(m-d)}\,d\mu(x).
\end{eqnarray}
We make the claim that for each $\vartheta\in(0,1)$
\begin{eqnarray}\label{gwo-3}
\left.
\begin{array}{r}
\mbox{if $x\in{\mathscr{X}}\setminus E_Q$ and $z\in {\mathscr{X}}$ are}
\\[4pt]
\mbox{such that }\,\rho_{\#}(x,z)<\frac{\vartheta}{C_\rho}\delta_{E_Q}(x)
\end{array}
\right\} 
\Longrightarrow \,\,
\frac{1-\vartheta}{C_\rho}\delta_{E_Q}(x)\leq\delta_{E_Q}(z)\leq C_\rho\delta_{E_Q}(x).
\end{eqnarray}
Indeed, for each $\eta>1$ close to $1$ we may take $y\in E_Q$ satisfying
$\rho_{\#}(y,x)<\eta\delta_{E_Q}(x)$ which implies that 
$\delta_{E_Q}(z)\leq\rho_{\#}(y,z)\leq C_\rho\max\{\rho_{\#}(y,x),\rho_{\#}(x,z)\}
\leq C_\rho \eta\delta_{E_Q}(x)$. Upon letting $\eta\searrow 1$ we therefore 
obtain $\delta_{E_Q}(z)\leq C_\rho\delta_{E_Q}(x)$.
On the other hand, for each $w\in E_Q$ we have 
$\delta_{E_Q}(x)\leq \rho_{\#}(x,w)\leq C_\rho\rho_{\#}(x,z)
+C_\rho\rho_{\#}(z,w)\leq\vartheta\delta_{E_Q}(x)+C_\rho\rho_{\#}(z,w)$,
which further yields $\delta_{E_Q}(x)\leq \frac{C_\rho}{1-\vartheta}\delta_{E_Q}(z)$. 
This concludes the proof of \eqref{gwo-3}. 

Going further, fix $x\in B_{\rho_{\#}}(x_0,c\ell(Q))\cap{\mathcal{O}}$ and
$z\in B_{\rho_{\#}}\bigl(x,C_\rho\delta_{E_Q}(x)/M\bigr)\cap E$ and make two observations.
First, an application of \eqref{gwo-3} with $\vartheta:=C_\rho^2/M\in(0,1)$ yields
\begin{eqnarray}\label{gwo-3a}
\hskip -0.20in
\frac{M-C_\rho^2}{MC_\rho}\delta_{E_Q}(x)\leq\delta_{E_Q}(z)\leq C_\rho\delta_{E_Q}(x)
\,\,\mbox{ and }\,\,
\rho_{\#}(z,x)<\frac{C_\rho}{M}\delta_{E_Q}(x)
\leq\frac{C_\rho^2}{M-C_\rho^2}\delta_{E_Q}(z),
\end{eqnarray}
hence $x\in B_{\rho_{\#}}\bigl(z,\tfrac{C_\rho^2}{M-C_\rho^2}\delta_{E_Q}(z)\bigr)$.
Second, recalling that $x_0\in E_Q$, $M>C_\rho^2$ and $C_\rho\geq 1$, we obtain
\begin{eqnarray}\label{gwo-3b}
\rho_{\#}(x_0,z) & \leq & C_\rho\max\bigl\{\rho_{\#}(x_0,x),\rho_{\#}(x,z)\bigr\}
<C_\rho\max\bigl\{c\,\ell(Q),\tfrac{C_\rho}{M}\delta_{E_Q}(x)\bigr\}
\\[4pt]
&\leq & C_\rho\max\bigl\{c\,\ell(Q),\tfrac{1}{C_\rho}\delta_{E_Q}(x)\bigr\}
\leq C_\rho\max\bigl\{c\,\ell(Q),\tfrac{1}{C_\rho}\rho_{\#}(x_0,x)\bigr\}
=C_\rho c\,\ell(Q),
\nonumber
\end{eqnarray}
which shows that $z\in B_{\rho_{\#}}(x_0,C_\rho c\ell(Q))$. Combining these 
observations with \eqref{gwo-3a}, keeping in mind \eqref{setO}, and using 
Fubini's Theorem in \eqref{gwo-2}, we may write
\begin{eqnarray}\label{gwo-5}
I_1 & \!\!\!\leq\!\!\! &  
C\int\limits_{B_{\rho_{\#}}\bigl(x_0,C_\rho c\ell(Q)\bigr)\cap (E\setminus E_Q)}
\hskip -0.30in
\delta_{E_Q}(z)^{-2\upsilon-d}
\Bigl(\int\limits_{
B_{\rho_{\#}}\bigl(z,\frac{C_\rho^2}{M-C_\rho^2}\delta_{E_Q}(z)\bigr)\setminus E_Q}
\hskip -0.30in\delta_E(x)^{2\upsilon-(m-d)}\,d\mu(x)\Bigr)
\,d{\mathscr{H}}_{{\mathscr{X}}\!,\,\rho_{\#}}^d(z)
\nonumber\\[4pt]
& \!\!\!\leq \!\!\! & C\int\limits_{B_{\rho_{\#}}\bigl(x_0,C_\rho c\ell(Q)\bigr)
\cap (E\setminus E_Q)}
\delta_{E_Q}(z)^{-2\upsilon-d}\delta_{E_Q}(z)^{2\upsilon+d}
\,d{\mathscr{H}}_{{\mathscr{X}}\!,\,\rho_{\#}}^d(z)
\nonumber\\[4pt]
& \!\!\!\leq\!\!\! & C{\mathscr{H}}_{{\mathscr{X}}\!,\,\rho_{\#}}^d
\Bigl(B_{\rho_{\#}}\bigl(x_0,C_\rho c\,\ell(Q)\bigr)\cap E\Bigr)
\leq C\ell(Q)^d\leq C\sigma(Q),
\end{eqnarray}
where for the second inequality in \eqref{gwo-5} we used Lemma~\ref{geom-lem}
(with $\gamma:=(m-d)-2\upsilon$ and 
$r:=R:=\frac{C_\rho^2}{M-C_\rho^2}\delta_{E_Q}(z)$), and for
the last two inequalities the fact that 
$\bigl(E,\rho\bigl|_{E},{\mathscr{H}}_{{\mathscr{X}}\!,\,\rho_{\#}}^d\lfloor E\bigr)$
is $d$-{\rm ADR} and $x_0\in E$.

To estimate $I_3$, we first note that since $T_E(Q)\cap E=\emptyset$ 
then \eqref{hszz-A} and \eqref{mMji} give that
\begin{eqnarray}\label{lbpz.33}
\hskip -0.20in
|\Theta_E b_Q(x)|=\left|\int_E {\theta}(x,y)\,b_Q(y)\,d\sigma(y)\right|
\leq\int_E |{\theta}(x,y)|\,d\sigma(y)\leq\frac{C}{\delta_{E}(x)^\upsilon},\quad
\forall\,x\in T_E(Q),
\end{eqnarray}
for some finite $C>0$ independent of $Q$. Also (compare with \eqref{Lz}), 
$|\Theta_E b_Q(x)|\leq C\delta_{E_Q}(x)^{-\upsilon}$ for each 
$x\in T_{E}(Q)\setminus E_Q$. Fix $\alpha,\beta>0$ such that $\alpha+\beta=\upsilon$.
A logarithmically convex combination of these inequalities then yields 
\begin{eqnarray}\label{lbpz.34}
|\Theta_E b_Q(x)|\leq C\delta_{E_Q}(x)^{-\alpha}
\delta_{E}(x)^{-\beta}\qquad\forall\,x\in T_E(Q)\setminus E_Q.
\end{eqnarray}
Observe that, by assumptions and Lemma~\ref{ME-ZZ}, we have $\mu(E_Q)=0$. 
Using this and \eqref{lbpz.34} in place of \eqref{Lz}, we obtain 
(compare with \eqref{gwo})
\begin{eqnarray}\label{gwo-N}
I_3\leq C\int_{B_{\rho_{\#}}(x_0,c\ell(Q))\cap({\widetilde{\mathcal{O}}}\setminus E_Q)}
\delta_E(x)^{-2\beta+2\upsilon-(m-d)}\delta_{E_Q}(x)^{-2\alpha}\,d\mu(x), 
\end{eqnarray}
where, this time, we have set 
\begin{eqnarray}\label{setO-N}
\widetilde{\mathcal{O}}:=\bigl\{z\in{\mathscr{X}}:\,
\delta_E(z)>C_1\delta_{E_Q}(z)\bigr\}.
\end{eqnarray}
Given the nature of \eqref{gwo-N}, \eqref{setO-N}, the same reasoning leading 
up to \eqref{gwo-5} used with $E$ and $E_Q$ interchanged this time gives 
\begin{eqnarray}\label{gwo-5-N}
I_3 & \!\!\!\leq\!\!\! &  
C\int\limits_{B_{\rho_{\#}}\bigl(x_0,C_\rho c\ell(Q)\bigr)\cap (E_Q\setminus E)}
\hskip -0.30in
\delta_{E}(z)^{-2\beta+2\upsilon-m}
\Bigl(\int\limits_{
B_{\rho_{\#}}\bigl(z,\frac{C_\rho^2}{M-C_\rho^2}\delta_{E}(z)\bigr)\setminus E}
\hskip -0.30in\delta_E(x)^{-2\alpha}\,d\mu(x)\Bigr)
\,d{\mathscr{H}}_{{\mathscr{X}}\!,\,\rho_{\#}}^d(z)
\nonumber\\[4pt]
& \!\!\!\leq \!\!\! & C\int\limits_{B_{\rho_{\#}}\bigl(x_0,C_\rho c\ell(Q)\bigr)
\cap (E_Q\setminus E)}
\delta_{E}(z)^{-2\beta+2\upsilon-m}\delta_{E}(z)^{-2\alpha+m}
\,d{\mathscr{H}}_{{\mathscr{X}}\!,\,\rho_{\#}}^d(z)
\nonumber\\[4pt]
& \!\!\!\leq\!\!\! & C{\mathscr{H}}_{{\mathscr{X}}\!,\,\rho_{\#}}^d
\Bigl(B_{\rho_{\#}}\bigl(x_0,C_\rho c\,\ell(Q)\bigr)\cap E_Q\Bigr)
\leq C\ell(Q)^d\leq C\sigma(Q).
\end{eqnarray}
Above, the second inequality follows from Lemma~\ref{geom-lem}
(used with  $r:=R:=\frac{C_\rho^2}{M-C_\rho^2}\delta_{E}(z)$) provided 
we choose $0<\alpha<(m-d)/2$ to begin with. Also, for the last two inequalities 
in \eqref{gwo-5-N} we have made use of the fact that both  
$\bigl(E,\rho\bigl|_{E},{\mathscr{H}}_{{\mathscr{X}}\!,\,\rho_{\#}}^d\lfloor E\bigr)$
and $\bigl(E_Q,\rho\bigl|_{E_Q},
{\mathscr{H}}_{{\mathscr{X}}\!,\,\rho_{\#}}^d\lfloor E_Q\bigr)$
are $d$-{\rm ADR} spaces and that $x_0\in E\cap E_Q$.

We are left with estimating $I_2$. With $C_1$ as above,
thanks to \eqref{BBss}, \eqref{operator-A}, and \eqref{haTT.2} we have 
\begin{eqnarray}\label{KLGB.3}
I_2=\int_{T_E(Q)\setminus E_Q}|\Theta_{E_Q} b_Q(x)|^2 
{\mathbf{1}}_{\{z\in{\mathscr{X}}:\,
C_1^{-1}\delta_E(z)\leq\delta_{E_Q}(z)\leq C_1\delta_E(z)\}}(x)\,
\delta_E(x)^{2\upsilon-(m-d)}\,d\mu(x).
\end{eqnarray}
Hence, we may further use \eqref{KLGB.3} and \eqref{avhai2} in order 
to write (with $\sigma_Q$ as in \eqref{HHH-60})
\begin{eqnarray}\label{KLGB-1}
I_2 &\leq & C\int_{\mathscr{X}\setminus E_Q}
|\Theta_{E_Q} b_Q(x)|^2\,\delta_{E_Q}(x)^{2\upsilon-(m-d)}\,d\mu(x)
\nonumber\\[4pt]
&\leq & C\int_{E_Q}|b_Q|^2\,d\sigma_Q 
\leq C{\mathscr{H}}_{{\mathscr{X}}\!,\,\rho_{\#}}^d\bigl(Q\cap E_Q\bigr)
\leq C\sigma(Q),
\end{eqnarray}
which is of the correct order. 

Now the fact that condition {\it 3} in Theorem~\ref{Thm:localTb} is satisfied for
our choice of $b_Q$'s follows by combining \eqref{KLGB}, \eqref{gwo-5},
\eqref{gwo-5-N} and \eqref{KLGB-1}. This finishes the proof of the theorem.
\end{proof}

We conclude this section by taking a closer look at a higher order version of 
the notion of ``big pieces of square function estimates." To set the stage,
in the context of Definition~\ref{sjvs} let us say that a closed subset $E$ of $({\mathscr{X}},\tau_\rho)$ with the property that there exists a Borel regular 
measure $\sigma$ on $(E,\tau_{\rho|_{E}})$ such that $\bigl(E,\rho\bigl|_E,\sigma\bigr)$ 
is a $d$-dimensional {\rm ADR} space has {\tt (BP)$^0$SFE} {\tt relative to}
$\theta$, or simply  {\tt SFE} {\tt relative to} $\theta$ 
(``Square Function Estimates relative to $\theta$"), 
provided there exists a finite positive constant $C$ such that
\begin{eqnarray}\label{avhai-TTFsa}
\begin{array}{c}
\displaystyle
\int_{\mathscr{X}\setminus E}|\Theta_{E}f(z)|^2\,
{\rm dist}_{\rho_{\#}}(z,E)^{2\upsilon-(m-d)}\,d\mu(z) 
\leq C\int_{E}|f|^2\,d\sigma\qquad
\\[8pt]
\mbox{for each function }\,\,f\in L^2(E,\sigma).
\end{array}
\end{eqnarray}
In addition, we shall say that $E$ has {\tt (BP)$^1$SFE} whenever $E$ has {\tt BPSFE}. 

We may then iteratively interpret ``$E$ has {\tt (BP)$^{k+1}$SFE}" as
the property that $E$ contains big pieces of sets having {\tt (BP)$^k$SFE},
in a uniform fashion. More specifically, we make the following definition. 

\begin{definition}\label{sjvs-DDD}
Consider two numbers $d,m\in(0,\infty)$ such that $m>d$, suppose that 
$({\mathscr{X}},\rho,\mu)$ is an $m$-dimensional {\rm ADR} space, and 
assume that ${\theta}$ is as in \eqref{K234-A}-\eqref{hszz-3-A}. Also, 
suppose that $k\in{\mathbb{N}}$. In this context, a set $E\subseteq{\mathscr{X}}$ 
is said to have {\tt (BP)$^{k+1}$SFE} relative to $\theta$
provided the following conditions are satisfied:
\begin{enumerate}
\item[(i)] the set $E$ is closed in $({\mathscr{X}},\tau_\rho)$ and 
has the property that there exists a Borel regular measure $\sigma$ 
on $(E,\tau_{\rho|_{E}})$ such that $\bigl(E,\rho\bigl|_E,\sigma\bigr)$ 
is a $d$-dimensional {\rm ADR} space;
\item[(ii)] there exist finite positive constants $\eta$, $C_1$, and $C_2$ with the 
property that for each $x\in E$ and each real number $r\in(0,{\rm diam}_{\rho_{\#}}(E)]$ 
there exists a closed subset $E_{x,r}$ of $(\mathscr{X},\tau_\rho)$ such that if  
\begin{eqnarray}\label{HHH-57-yyy}
\sigma_{x,r}:={\mathscr{H}}_{{\mathscr{X}}\!,\,\rho_{\#}}^d\lfloor E_{x,r},
\quad\mbox{ where ${\mathscr{H}}_{{\mathscr{X}}\!,\,\rho_{\#}}^d$ 
is as in \eqref{HHH-56}},
\end{eqnarray}
then $\bigl(E_{x,r},\rho\big|_{E_{x,r}},\sigma_{x,r}\bigr)$ is a 
$d$-dimensional {\rm ADR} space, with {\rm ADR} constant $\leq C_1$, 
and which satisfies
\begin{eqnarray}\label{yvg.Ifav}
\sigma\bigl(E_{x,r}\cap E\cap B_{\rho_{\#}}(x,r)\bigr)\geq\eta\,r^d,
\end{eqnarray}
as well as
\begin{eqnarray}\label{yvg-5tVV}
\mbox{$E_{x,r}$ has {\tt (BP)$^k$SFE} relative to $\theta$, with character 
controlled by $C_2$}.
\end{eqnarray}
\end{enumerate}
In this context, we shall refer to $\eta,C_1,C_2$ as the 
{\tt (BP)$^{k+1}$SFE} character of $E$.
\end{definition}

The following result may be regarded as a refinement of Theorem~\ref{Thm:BPSFtoSF}.

\begin{theorem}\label{Thm:BPSFtoSF.XXX}
Consider two numbers $d,m\in(0,\infty)$ such that $m>d$, suppose that 
$({\mathscr{X}},\rho,\mu)$ is an $m$-dimensional {\rm ADR} space, and 
assume that ${\theta}$ is as in \eqref{K234-A}-\eqref{hszz-3-A}. 
Also, suppose that the set $E$ is closed in $({\mathscr{X}},\tau_\rho)$ and 
has the property that there exists a Borel regular measure $\sigma$ 
on $(E,\tau_{\rho|_{E}})$ such that $\bigl(E,\rho\bigl|_E,\sigma\bigr)$ 
is a $d$-dimensional {\rm ADR} space. 

Then the following claims are equivalent:
\begin{enumerate}
\item[(i)] $E$ has {\tt (BP)$^k$SFE} relative to $\theta$ 
for some $k\in{\mathbb{N}}$;
\item[(ii)] $E$ has {\tt (BP)$^k$SFE} relative to $\theta$ 
for every $k\in{\mathbb{N}}$;
\item[(iii)] $E$ has {\tt (BP)$^0$SFE} relative to $\theta$. 
\end{enumerate}
\end{theorem}

\begin{proof}
It is clear that if $E$ has {\tt (BP)$^k$SFE} relative to $\theta$ 
for some $k\in{\mathbb{N}}_0$ then $E$ also has {\tt (BP)$^{k+1}$SFE} 
relative to $\theta$, since obviously $E$ has big pieces of itself.
This gives the implications {\it (iii)} $\Rightarrow$ {\it (i)} and 
{\it (iii)} $\Rightarrow$ {\it (ii)} in the statement of the theorem.
Also, the implication {\it (ii)} $\Rightarrow$ {\it (i)} is trivial. 
Finally, in light of Definition~\ref{sjvs-DDD}, Theorem~\ref{Thm:BPSFtoSF} 
combined with an induction argument gives that {\it (i)} $\Rightarrow$ {\it (iii)},
completing the proof.
\end{proof}

\section{Square Function Estimates on Uniformly Rectifiable Sets}\label{Sect:SFE}
\setcounter{equation}{0}

Given an $n$-dimensional Ahlfors-David regular set $\Sigma$ in $\mathbb{R}^{n+1}$ that
has so-called big pieces of Lipschitz graphs (BPLG), the inductive scheme established 
in the previous section  allows us to deduce square function estimates for an integral
operator $\Theta_\Sigma$, as in~\eqref{operator-A}, whenever square function estimates
are satisfied by $\Theta_\Gamma$ for all Lipschitz graphs $\Gamma$ in $\mathbb{R}^{n+1}$.
Furthermore, induction allows us to prove the same result when the set $\Sigma$ 
only has (BP)$^k$LG for any $k\in\mathbb{N}$. The definition of (BP)$^k$LG is given 
in Definition~\ref{Def-BPnLG}. A recent result by J.~Azzam and R.~Schul 
(cf. \cite[Corollary~1.7]{AS}) proves that uniformly rectifiable sets have 
(BP)$^2$LG (the converse implication also holds and can be found in~\cite[p.\,16]{DaSe93}), and this allows us to obtain square function 
estimates on uniformly rectifiable sets.

We work in the Euclidean codimension one setting throughout this section. 
In particular, fix $n\in\mathbb{N}$ and let $\mathbb{R}^{n+1}$ be the ambient 
space, so that in the notation of Section~\ref{Sect:4}, we would have $d=n$, 
$m=n+1$ and $(\mathscr{X},\rho,\mu)$ is $\mathbb{R}^{n+1}$ with the Euclidean 
metric and Lebesgue measure. We also restrict our attention to the following 
class of kernels in order to obtain square function estimates on Lipschitz graphs. 
Suppose that $K:{\mathbb{R}}^{n+1}\setminus\{0\} \rightarrow \mathbb{R}$ satisfies
\begin{eqnarray}\label{prop-K}
K\in C^2({\mathbb{R}}^{n+1}\setminus\{0\}),\quad
K(\lambda x)=\lambda^{-n}K(x)\,\mbox{ for all }\,\lambda>0,\,\,
x\in{\mathbb{R}}^{n+1}\setminus\{0\},\quad K\mbox{ is odd},\ \
\end{eqnarray}
and has the property that there exists a finite positive constant $C_K$ 
such that for all $j\in\{0,1,2\}$ the following holds:
\begin{eqnarray}\label{prop-K2}
|\nabla^j K(x)|\leq C_K|x|^{-n-j},\quad \forall\,x\in{\mathbb{R}}^{n+1}\setminus\{0\}.
\end{eqnarray}
Then for each closed subset $\Sigma$ of ${\mathbb{R}}^{n+1}$, denote 
by $\sigma:={\mathscr{H}}^n_{\mathbb{R}^{n+1}}\lfloor{\Sigma}$ the surface measure induced by 
the $n$-dimensional Hausdorff measure on $\Sigma$ from \eqref{HHH-56}, and define the integral operator ${\mathcal{T}}$ for all functions $f\in L^p(\Sigma,\sigma)$, $1\leq p\leq\infty$, by
\begin{eqnarray}\label{sbnvjb}
{\mathcal{T}}f(x):=\int_{\Sigma} K(x-y)f(y)\,d{\sigma}(y),\qquad 
\forall\,x\in{\mathbb{R}}^{n+1}\setminus\Sigma.
\end{eqnarray}

In the notation of Section~\ref{Sect:4}, we consider the set $E=\Sigma$ and the operator $\Theta_E=\nabla\mathcal{T}$ with integral kernel $\theta=\nabla K$. We begin by proving square function estimates for $\nabla\mathcal{T}$ in the case when $\Sigma$ is a Lipschitz graph. The inductive scheme from the previous section then allows us to extend that result to the case when $\Sigma$ has (BP)$^k$LG for any $k\in\mathbb{N}$, and hence when $\Sigma$ is uniformly rectifiable.

\subsection{Square function estimates on Lipschitz graphs}
\label{SSect:LG}

The main result in this subsection is the square function estimate for Lipschitz graphs contained 
in the theorem below.  A parabolic variant of this result appears in \cite{HL}, 
and the present proof is based on the arguments given there, and in \cite{Ho}.

\begin{theorem}\label{Theorem1.1}
Let $A:{\mathbb{R}}^n\to{\mathbb{R}}$ be a Lipschitz function and set 
$\Sigma:=\{(x,A(x)):\,x\in{\mathbb{R}}^n\}$. Moreover, assume that $K$ is as in 
\eqref{prop-K} and consider the operator ${\mathcal{T}}$ as in \eqref{sbnvjb}. 
Then there exists a finite constant $C>0$ depending only on 
$\|\partial^\alpha K\|_{L^\infty(S^{n})}$ for $|\alpha|\leq 2$, and the Lipschitz 
constant of $A$ such that for each function $f\in L^2(\Sigma,\sigma)$ 
one has
\begin{eqnarray}\label{hbraH}
\int_{{\mathbb{R}}^{n+1}\setminus\Sigma}
|(\nabla{\mathcal{T}}f)(x)|^2\,{\rm dist}(x,\Sigma)\,dx
\leq C\int_{\Sigma}|f|^2\,d\sigma.
\end{eqnarray}
\end{theorem}

As a preamble to the proof of Theorem~\ref{Theorem1.1}, we state and prove a couple 
of technical lemmas. The first has essentially appeared previously in \cite{Ch},
and is based upon ideas of \cite{J}. 

\begin{lemma}\label{lgih} Assume that $A:{\mathbb{R}}^n\to{\mathbb{R}}$ is a
locally integrable function such that $\nabla A\in L^2({\mathbb{R}}^n)$. 
Pick a smooth, real-valued, nonnegative, compactly supported function $\phi$ 
defined in ${\mathbb{R}}^n$ with $\int_{{\mathbb{R}}^n}\phi(x)\,dx=1$ and for 
each $t>0$ set $\phi_t(x):=t^{-n}\phi(x/t)$ for $x\in{\mathbb{R}}^n$. 
Finally, define 
\begin{eqnarray}\label{1.28}
E_A(t,x,y):=
A(x)-A(y)-{\langle}\nabla_{x}(\phi_t\ast A)(x),(x-y){\rangle},\qquad
\forall\,x,y\in{\mathbb{R}}^n,\,\,\,\forall\,t>0.
\end{eqnarray}
Then, for some finite positive constant $C=C(\phi,n)$,
\begin{eqnarray}\label{1.29}
\int_0^{\infty}t^{-n-2}\int_{{\mathbb{R}}^n}\int_{|x-y|\leq \lambda t}
|E_A(t,x,y)|^2\,dy\,dx\,\tfrac{dt}{t}
\leq C\lambda^{n+3}\|\nabla A\|^2_{L^2({\mathbb{R}}^n)},\qquad\forall\,\lambda\geq 1.
\end{eqnarray}
\end{lemma}

\begin{proof}
Starting with the changes of variables $t=\lambda^{-1}\tau$, $y=x+h$ and then 
employing Plancherel's theorem in the variable $x$, we may write (with `hat' denoting 
the Fourier transform)
\begin{eqnarray}\label{ldbs}
&& \int_0^{\infty}t^{-n-2}\int_{{\mathbb{R}}^n}\int_{|x-y|\leq \lambda t}
|E_A(t,x,y)|^2\,dy\,dx\,\frac{dt}{t}
\\[4pt]
&& \hskip 0.50in
=\lambda^{n+2}\int_0^{\infty}\tau^{-n-2}\int_{{\mathbb{R}}^n}\int_{|h|\leq\tau}
\bigl|A(x)-A(x+h)+{\langle}\nabla_{x}(\phi_{\lambda^{-1}\tau}\ast A)(x),h{\rangle}\bigr|^2
\,dh\,dx\,\frac{d\tau}{\tau}
\nonumber\\[4pt]
&& \hskip 0.50in
=\lambda^{n+2}\int_0^{\infty}\tau^{-n-2}\int_{{\mathbb{R}}^n}\int_{|h|\leq\tau}
\bigl|1-e^{i\langle\zeta,h\rangle}+i\langle\zeta,h\rangle\,
\widehat{\phi}(\lambda^{-1}\tau\zeta)\bigr|^2
\tfrac{|\widehat{\nabla A}(\zeta)|^2}{|\zeta|^2}
\,dh\,d\zeta\,\frac{d\tau}{\tau}
\nonumber\\[4pt]
&& \hskip 0.50in
=\lambda^{n+2}\int_0^{\infty}\int_{{\mathbb{R}}^n}\int_{|w|\leq 1}
\frac{\bigl|1-e^{i\tau\langle\zeta,w\rangle}+i\tau\langle\zeta,w\rangle\,
\widehat{\phi}(\lambda^{-1}\tau\zeta)\bigr|^2}{\tau^2|\zeta|^2}
|\widehat{\nabla A}(\zeta)|^2\,dw\,d\zeta\,\frac{d\tau}{\tau},
\nonumber
\end{eqnarray}
where the last equality in \eqref{ldbs} is based on the change of variables $h=\tau w$.

Next we observe that for every $\zeta\in{\mathbb{R}}^n$ and 
$w\in{\mathbb{R}}^n$ with $|w|\leq 1$ there holds
\begin{eqnarray}\label{zawgi}
\frac{\bigl|1-e^{i\tau\langle\zeta,w\rangle}+i\tau\langle\zeta,w\rangle\,
\widehat{\phi}(\lambda^{-1}\tau\zeta)\bigr|}{\tau|\zeta|}
\leq C\,\min\Bigl\{\tau|\zeta|,\frac{\lambda}{\tau|\zeta|}\Bigr\}
\end{eqnarray}
for some $C>0$ depending only on $\phi$. To see why \eqref{zawgi} is true, 
analyze the following two cases.

\vskip 0.08in
\noindent {\it Case~1.} $\tau|\zeta|\leq\sqrt{\lambda}$\,:\\
In this situation the minimum in the right hand-side of \eqref{zawgi} is equal
to $\tau|\zeta|$. In addition, if we use Taylor expansions about zero for the 
complex exponential function and $\widehat{\phi}$, we obtain 
(keeping in mind that $\widehat{\phi}(0)=1$, $\lambda\geq 1$ and $|w|\leq 1$)
\begin{eqnarray}\label{zaw-cgi}
\bigl|1-e^{i\tau\langle\zeta,w\rangle}+i\tau\langle\zeta,w\rangle
+i\tau\langle\zeta,w\rangle\,(\widehat{\phi}(\lambda^{-1}\tau\zeta)-1)\bigr|
\leq C\tau^2|\zeta|^2,
\end{eqnarray}
which shows that \eqref{zawgi} holds in this case.

\vskip 0.08in
\noindent {\it Case~2.} $\tau|\zeta|>\sqrt{\lambda}$\,:\\
In this scenario the minimum in the right hand-side of \eqref{zawgi} is equal
to $\frac{\lambda}{\tau|\zeta|}$. Moreover, 
\begin{eqnarray}\label{pkbi}
\bigl|1-e^{i\tau\langle\zeta,w\rangle}
+\tau\langle\zeta,w\rangle\,\widehat{\phi}(\lambda^{-1}\tau\zeta)\bigr|
&\leq& 2+\tau|\zeta|\,\bigl|\widehat{\phi}(\lambda^{-1}\tau\zeta)\bigr|
\nonumber\\[4pt]
&\leq& 2+C\tau|\zeta|\,(1+\lambda^{-1}\tau|\zeta|)^{-1}
\leq C\lambda,
\end{eqnarray}
since the Schwartz function $\widehat{\phi}$ decays, $\lambda\geq 1$ 
and $|w|\leq 1$. Consequently, \eqref{zawgi} holds in this case as well.

With \eqref{zawgi} in hand, we proceed to integrate in $\tau\in(0,\infty)$ 
with respect to the Haar measure to further obtain
\begin{eqnarray}\label{knoB}
&& \int_0^{\infty}
\frac{\bigl|1-e^{i\tau\langle\zeta,w\rangle}+i\tau\langle\zeta,w\rangle\,
\widehat{\phi}(\lambda^{-1}\tau\zeta)\bigr|^2}{\tau^2|\zeta|^2}\,\frac{d\tau}{\tau}
\leq \int_0^{\infty}\min\Bigl\{\tau^2|\zeta|^2,\frac{\lambda^2}{\tau^2|\zeta|^2}\Bigr\}
\,\frac{d\tau}{\tau}
\nonumber\\[4pt]
&& \hskip 0.50in
=\int_0^{\sqrt{\lambda}/|\zeta|}\tau|\zeta|^2\,d\tau
+\int_{\sqrt{\lambda}/|\zeta|}^\infty\frac{\lambda^2}{\tau^3|\zeta|^2}\,d\tau
\leq C\lambda.
\end{eqnarray}
A combination of \eqref{ldbs}, \eqref{knoB} and Plancherel's theorem
now yields \eqref{1.29}, finishing the proof of Lemma~\ref{lgih}.
\end{proof}

The second lemma needed here has essentially appeared previously in \cite{MMT}.

\begin{lemma}\label{Lemma1.3}
Let $F:{\mathbb{R}}^{n+1}\setminus\{0\}\to{\mathbb{R}}$ be a continuous 
function which is even and positive homogeneous of degree $-n-1$. 
Then for any $a\in{\mathbb{R}}^n$ and any $t>0$ there holds
\begin{eqnarray}\label{1.31}
\int_{{\mathbb{R}}^n}F(y,a\cdot y+t)\,dy
&=& \frac{1}{2t}\int_{S^{n-1}}\int_{-\infty}^\infty F(\omega,s)\,ds\,d\omega 
\nonumber\\[4pt]
&=&\int_{{\mathbb{R}}^n}F(y,t)\,dy.
\end{eqnarray}
In particular, if $F$ is some first-order partial derivative, 
say $F=\partial_jG$, $j\in\{1,..,n+1\}$, of a function 
$G\in C^1({\mathbb{R}}^{n+1}\setminus\{0\})$ which is 
odd and homogeneous of degree $-n$, then 
\begin{eqnarray}\label{1.31ASBN}
\int_{{\mathbb{R}}^n} F(y,a\cdot y+t)\,dy=0
\qquad\mbox{ for any $a\in{\mathbb{R}}^n$ and $t>0$}.
\end{eqnarray}
\end{lemma}

\begin{proof} 
Fix $a\in{\mathbb{R}}^n$ and $t>0$. By the homogeneity of $F$ we 
have $|F(x)|\leq\|F\|_{L^\infty(S^n)}|x|^{-n-1}$ for every 
$x\in{\mathbb{R}}^{n+1}\setminus\{0\}$. Also, $\{(y,a\cdot y+t):\,|y|\leq 1\}$ 
is a compact subset of ${\mathbb{R}}^{n+1}\setminus\{0\}$. Hence, given that
$F$ is continuous on ${\mathbb{R}}^{n+1}\setminus\{0\}$, it follows that 
$\int_{{\mathbb{R}}^n}|F(y,a\cdot y+t)|\,dy<\infty$.  

To proceed, by passing to polar coordinates $y=r\omega$, $r>0$, $\omega\in S^{n-1}$, 
and using the homogeneity of $F$ we may write
\begin{eqnarray}\label{1.32}
\int_{{\mathbb{R}}^n}F(y,a\cdot y+t)\,dy=
\int_{S^{n-1}}\int_0^\infty F(\omega,a\cdot\omega+t/r)\,r^{-2}dr\,d\omega.
\end{eqnarray}
Now, setting $s:=a\cdot\omega+t/r$ the last integral above becomes
$t^{-1}\int_{S^{n-1}}\int_{a\cdot\omega}^\infty F(\omega,s)\,ds\,d\omega$
which, by making the change of variables $(\omega,s)\mapsto(-\omega,-s)$ 
and using the fact that $F$ is even, may be written as 
\begin{eqnarray}\label{1.33}
\frac{1}{t}\int_{S^{n-1}}\int_{a\cdot\omega}^\infty F(\omega,s)\,ds\,d\omega
=\frac{1}{2t}\int_{S^{n-1}}\int_{-\infty}^\infty F(\omega,s)\,ds\,d\omega.
\end{eqnarray}
This analysis gives the first equality in \eqref{1.31}. Furthermore, the integral 
in the right side of \eqref{1.33} is independent of $a\in{\mathbb{R}}^n$
and, hence, so is the original one. In particular, its value does
not change if we take $a=0$ and this is precisely what the second 
equality in \eqref{1.31} says.

Finally, consider the claim made in \eqref{1.31ASBN} under the assumption that
$F=\partial_jG$ for some $j\in\{1,..,n+1\}$ and some   
$G\in C^1({\mathbb{R}}^{n+1}\setminus\{0\})$ which is 
odd and homogeneous of degree $-n$. In particular, there exists $C\in(0,\infty)$ 
such that $|G(x)|\leq C|x|^{-n}$ for all $x\in{\mathbb{R}}^{n+1}\setminus\{0\}$. 
Then, using the decay of $G$ and integration by parts, if $j\neq n+1$ the third 
integral in \eqref{1.31} vanishes whereas if $j=n+1$ the second one does so.
\end{proof}

After this preamble, we are ready to present the 

\vskip 0.08in
\begin{proof}[Proof of Theorem~\ref{Theorem1.1}]
A moment's reflection shows that it suffices to establish \eqref{hbraH} with the domain 
of integration ${\mathbb{R}}^{n+1}\setminus\Sigma$ in the left-hand side replaced by 
\begin{eqnarray}\label{Gbbab}
\Omega:=\{(x,t)\in{\mathbb{R}}^{n+1}:t>A(x)\}. 
\end{eqnarray}
Assume that this is the case and note that by making the bi-Lipschitz change of 
variables ${\mathbb{R}}^n\times(0,\infty)\ni (x,t)\mapsto (x,A(x)+t)\in\Omega$,
(whose Jacobian is equivalent to a finite constant) the estimate \eqref{hbraH} 
follows from the boundedness of
\begin{eqnarray}\label{1.20}
&& T^j:L^2({\mathbb{R}}^n,dx)\to
L^2({\mathbb{R}}^{n+1}_+,\tfrac{dt}{t}dx), 
\\
&& T^jf(x,t):=\int_{{\mathbb{R}}^n}K^j_t(x,y)f(y)\,dy
\end{eqnarray}
for $j=1,..,n+1$, where the family of kernels $\{K^j_t(x,y)\}_{t>0}$ is given by
\begin{eqnarray}\label{1.21}
K^j_t(x,y):=t\,(\partial_j K)(x-y,A(x)-A(y)+t),\qquad
x,y\in{\mathbb{R}}^n,\,t>0,\,j=1,\dots,n+1.
\end{eqnarray}
The approach we present utilizes ideas developed in \cite{CJ} and \cite{Ho}.
Based on \eqref{prop-K}-\eqref{prop-K2} it is not difficult to check that the family
$\{K^j_t(x,y)\}_{t>0}$ is standard, i.e., there hold
\begin{eqnarray}\label{1.22}
|K^j_t(x,y)|&\leq & C\,t(t+|x-y|)^{-(n+1)}
\\[4pt]
|\nabla_{x}K^j_t(x,y)|+|\nabla_{y}K^j_t(x,y)|
&\leq & C\,t(t+|x-y|)^{-(n+2)}.
\label{1.22b}
\end{eqnarray}
As such, a particular version of Theorem~\ref{SChg} gives 
that the operators in \eqref{1.20} are bounded as soon as we show that 
for each $j=1,\dots,n+1$,
\begin{eqnarray}\label{cfg-aa}
|T^j(1)(x,t)|^2\,\tfrac{dt}{t}dx\,\,
\mbox{is a Carleson measure in}\,\,{\mathbb{R}}^{n+1}_+.
\end{eqnarray}
To this end, fix $j\in\{1,\dots,n+1\}$ and select a real-valued, nonnegative function 
$\phi\in C^\infty_c({\mathbb{R}}^n)$, vanishing for $|x|\geq 1$,
with $\int_{{\mathbb{R}}^n}\phi(x)\,dx=1$ and, as usual, 
for every $t>0$, set $\phi_t(x):=t^{-n}\phi(x/t)$ for $x\in{\mathbb{R}}^n$.
We write $T^j(1)=(T^j(1)-\widetilde{T}^j(1))+\widetilde{T}^j(1)$ where
\begin{eqnarray}\label{1.24}
\widetilde{T}^jf(x,t):=\int_{{\mathbb{R}}^n}\widetilde{K}^j_t(x,y)f(y)\,dy,
\qquad x\in{\mathbb{R}}^n,\,t>0,
\end{eqnarray}
with
\begin{eqnarray}\label{1.25}
\widetilde{K}^j_t(x,y):=t\,(\partial_j K)
( x-y,\langle\nabla_{x}(\phi_t\ast A)(x),x-y\rangle+t),\qquad
x,y\in{\mathbb{R}}^n,\,t>0.
\end{eqnarray}
To prove that $|(T^j-\widetilde{T}^j)(1)(x,t)|^2\,dx\frac{dt}{t}$ is a Carleson 
measure, fix $x_0$ in ${\mathbb{R}}^n$, $r>0$, and split
\begin{eqnarray}\label{1.25b}
(T^j-\widetilde{T}^j)(1)=(T^j-\widetilde{T}^j)({\mathbf{1}}_{B(x_0,100r)})+
(T^j-\widetilde{T}^j)({\mathbf{1}}_{{\mathbb{R}}^n\setminus B(x_0,100r)}),
\end{eqnarray}
where, for any set $S$, ${\mathbf{1}}_S$ stands for the characteristic function of $S$. 
Using \eqref{1.22} and the fact that a similar estimate holds for 
$\widetilde{K}^j_t(x,y)$, we may write
\begin{eqnarray}\label{1.26}
&&\int_0^r\int_{B(x_0,r)}
|(T^j-\widetilde{T}^j)({\mathbf{1}}_{{\mathbb{R}}^n\setminus B(x_0,100r)})(x,t)|^2\,dx\,\tfrac{dt}{t}
\nonumber\\[4pt]
&&\hskip 0.50in
\leq C\,\int_0^r\int_{B(x_0,r)}\Bigl(
\int_{{\mathbb{R}}^n\setminus B(x_0,100r)}
\frac{t}{|x-y|^{n+1}}\,dy\Bigr)^2\,dx\,\tfrac{dt}{t}
\nonumber\\[4pt]
&&\hskip 0.50in
=C\,\int_0^r\int_{B(x_0,r)}\Bigl(
\int_{{\mathbb{R}}^n\setminus B(0,99r)}
\frac{t}{|z|^{n+1}}\,dz\Bigr)^2\,dx\,\tfrac{dt}{t}
=C\,r^n,
\end{eqnarray}
a bound which is of the right order. We are therefore left with proving 
an estimate similar to \eqref{1.26} with ${\mathbb{R}}^n\setminus B(x_0,100r)$ 
replaced by $B(x_0,100r)$. More precisely, the goal is to show that
\begin{eqnarray}\label{1.26-BB}
\int_0^r\int_{B(x_0,r)}
|(T^j-\widetilde{T}^j)({\mathbf{1}}_{B(x_0,100r)})(x,t)|^2\,dx\,\tfrac{dt}{t}\leq C\,r^n.
\end{eqnarray}
For this task we make use of the Lemma~\ref{lgih}. This requires an adjustment which
we now explain. Concretely, fix a function $\Phi\in C^\infty({\mathbb{R}})$ such that 
$0\leq\Phi\leq 1$, $\mbox{supp}\,\Phi\subseteq [-150r,150r]$, 
$\Phi\equiv 1$ on $[-125r,125r]$, and $\|\Phi'\|_{L^\infty({\mathbb{R}})}\leq c/r$. 
If we now set $\widetilde{A}(x):=\Phi(|x-x_0|)(A(x)-A(x_0))$ for every $x\in{\mathbb{R}}^n$,
it follows that 
\begin{eqnarray}\label{g5ya}
\begin{array}{c}
\widetilde{A}(x)-\widetilde{A}(y)=A(x)-A(y)\,\,\mbox{ and }\,\,
\nabla(\phi_t\ast\widetilde{A})(x)=\nabla(\phi_t\ast A)(x)
\\[4pt]
\mbox{whenever }\,\,x\in B(x_0,r),\,\,y\in B(x_0,100r),\,\,t\in(0,r).
\end{array}
\end{eqnarray}
Hence, the expression $(T^j-\widetilde{T}^j)({\mathbf{1}}_{B(x_0,100r)})(x,t)$ 
does not change for $x\in B(x_0,r)$ and $t\in(0,r)$ if we replace $A$ by $\widetilde{A}$.
In addition, since $\|\nabla\widetilde{A}\|_{L^\infty({\mathbb{R}}^n)}
\leq C\|\nabla A\|_{L^\infty({\mathbb{R}}^n)}$, taking into account the support of 
$\widetilde{A}$ we have 
\begin{eqnarray}\label{bhius}
\|\nabla\widetilde{A}\|_{L^2({\mathbb{R}}^n)}
\leq Cr^{n/2}\|\nabla A\|_{L^\infty({\mathbb{R}}^n)}
\end{eqnarray}
for some $C>0$ independent of $r$. Hence, there is no loss of generality in assuming 
that the original Lipschitz function $A$ has the additional property that
\begin{eqnarray}\label{bhius-CC}
\|\nabla A\|_{L^2({\mathbb{R}}^n)}\leq Cr^{n/2}\|\nabla A\|_{L^\infty({\mathbb{R}}^n)}.
\end{eqnarray}
Under this assumption we now return to the task of proving \eqref{1.26-BB}.
To get started, recall \eqref{1.28}. We claim that there exists $C=C(A,\phi)>0$ such that 
\begin{eqnarray}\label{1.27}
|K^j_t(x,y)-\widetilde{K}^j_t(x,y)|\leq C\,t(t+|x-y|)^{-(n+2)}|E_A(t,x,y)|,\quad
\forall\,x,y\in{\mathbb{R}}^n,\,\forall\,t>0.
\end{eqnarray}
Indeed, by making use of the Mean-Value Theorem and \eqref{prop-K2}, the claim 
will follow if we show that there exists $C=C(A,\phi)>0$ with the property that
\begin{eqnarray}\label{skgn}
\sup_{\xi\in I}\,[|\xi|+|x-y|]^{-(n+2)}\leq C[t+|x-y|]^{-(n+2)},
\end{eqnarray}
where $I$ denotes the interval with endpoints
$t+A(x)-A(y)$ and $t+{\langle}\nabla_{x}(\phi_t\ast A)(x),(x-y){\rangle}$. 
From the properties of $A$ and $\phi$ we see that 
$\xi=t+{{O}}(|x-y|)$, with constants depending only on $A$ and $\phi$.
In particular, there exists some small $\varepsilon=\varepsilon(A,\phi)>0$ 
such that if $|x-y|<\varepsilon t$ then $t\leq C|\xi|\leq C(|\xi|+|x-y|)$. 
On the other hand, if $|x-y|\geq\varepsilon t$ then clearly $t\leq C(|\xi|+|x-y|)$.
Thus, there exists $C>0$ such that $t\leq C(|\xi|+|x-y|)$ for $\xi\in I$, which
implies that for some $C=C(A,\phi)>0$ there holds $t+|x-y|\leq C(|\xi|+|x-y|)$ 
whenever $\xi\in I$, proving \eqref{skgn}. 

Next, making use of \eqref{1.27}, we may write 
\begin{eqnarray}\label{1.26-DD}
&& \int_0^r\int_{B(x_0,r)}
|(T^j-\widetilde{T}^j)({\mathbf{1}}_{B(x_0,100r)})(x,t)|^2\,dx\,\frac{dt}{t}
\nonumber\\[4pt]
&& \hskip 0.50in
\leq C\int_0^\infty\int_{{\mathbb{R}}^n}\Bigl(
\int_{{\mathbb{R}}^n}\frac{t}{(t+|x-y|)^{n+2}}\,
|E_A(t,x,y)|\,dy\Bigr)^2\,dx\,\frac{dt}{t}
\nonumber\\[4pt]
&& \hskip 0.50in
\leq C\int_0^\infty\int_{{\mathbb{R}}^n}\Bigl(t^{-n-1}
\int_{B(x,t)}|E_A(t,x,y)|\,dy\Bigr)^2\,dx\,\frac{dt}{t}
\nonumber\\[4pt]
&& \hskip 0.60in
+ C\int_0^\infty\int_{{\mathbb{R}}^n}\Bigl(\sum\limits_{\ell=0}^\infty
\int\limits_{B(x,2^{\ell+1}t)\setminus B(x,2^\ell t)}\frac{t}{|x-y|^{n+2}}\,
|E_A(t,x,y)|\,dy\Bigr)^2\,dx\,\frac{dt}{t}
\nonumber\\[4pt]
&& \hskip 0.50in
\leq C\int_0^\infty\int_{{\mathbb{R}}^n}\Bigl(\sum\limits_{\ell =0}^\infty
2^{-\ell }(2^\ell t)^{-n-1}\int\limits_{B(x,2^{\ell +1}t)}|E_A(t,x,y)|\,dy\Bigr)^2
\,dx\,\frac{dt}{t}.
\end{eqnarray}
Now, we apply Minkowski's inequality in order to obtain
\begin{eqnarray}\label{btua9}
&& \int_0^\infty\int_{{\mathbb{R}}^n}\Bigl(\sum\limits_{\ell =0}^\infty
2^{-\ell }(2^\ell t)^{-n-1}\int\limits_{B(x,2^{\ell +1}t)}|E_A(t,x,y)|\,dy\Bigr)^2
\,dx\,\frac{dt}{t}
\nonumber\\[4pt]
&& \hskip 0.40in 
\leq\left(\sum\limits_{\ell =0}^\infty\Bigl[\int_0^\infty\int_{{\mathbb{R}}^n}
2^{-2\ell }(2^\ell t)^{-2n-2}\Bigl(\int\limits_{B(x,2^{\ell +1}t)}|E_A(t,x,y)|\,dy\Bigr)^2
\,dx\,\frac{dt}{t}\Bigr]^{1/2}\right)^2.\qquad
\end{eqnarray}
By the Cauchy-Schwarz inequality, the last expression above is dominated by 
\begin{eqnarray}\label{btu56}
\left(\sum\limits_{\ell =0}^\infty\Bigl[2^{\ell (-n-4)}
\int_0^\infty\int_{{\mathbb{R}}^n}t^{-n-2}
\int\limits_{B(x,2^{\ell +1}t)}|E_A(t,x,y)|^2\,dy\,dx\,\frac{dt}{t}\Bigr]^{1/2}\right)^2.
\end{eqnarray}
Invoking now Lemma~\ref{lgih} with $\lambda:=2^{\ell+1}\geq 1$ for 
$\ell\in{\mathbb{N}}\cup\{0\}$, each inner triple integral in \eqref{btu56} is dominated
by $C2^{\ell(n+3)}\|\nabla A\|^2_{L^2({\mathbb{R}}^n)}$ with $C>0$ finite constant 
independent of $\ell$. Thus, the entire expression in \eqref{btu56} is 
\begin{eqnarray}\label{btu57}
\leq C\left(\sum\limits_{\ell =0}^\infty\bigl[2^{\ell(-n-4)}\cdot
2^{\ell(n+3)}\|\nabla A\|^2_{L^2({\mathbb{R}}^n)}\bigr]^{1/2}\right)^2
=C\|\nabla A\|^2_{L^2({\mathbb{R}}^n)}\leq Cr^n,
\end{eqnarray}
where for the last inequality in \eqref{btu57} we have used \eqref{bhius-CC}.
This finishes the proof of \eqref{1.26-BB}. In turn, when \eqref{1.26-BB} is
combined with \eqref{1.26}, we obtain 
\begin{eqnarray}\label{cfg-ab}
|(T^j-\widetilde{T}^j)(1)(x,t)|^2\,\tfrac{dt}{t}dx\,\,
\mbox{is a Carleson measure in}\,\,{\mathbb{R}}^{n+1}_+.
\end{eqnarray}
At this stage, there remains to observe that, thanks to Lemma~\ref{Lemma1.3},
we have 
\begin{eqnarray}\label{1.30}
\widetilde{T}^j(1)(x,t)=\int\limits_{{\mathbb{R}}^n}t\,(\partial_j K)
\bigl(x-y,{\langle}\nabla_{x}(\phi_t\ast A)(x),(x-y){\rangle}+t\bigr)\,dy\equiv 0
\quad\forall\,x\in{\mathbb{R}}^n,\,\forall\,t>0.
\end{eqnarray}
The proof of Theorem~\ref{Theorem1.1} is now completed.
\end{proof}

\subsection{Square function estimates on (BP)$^k$LG sets}
\label{SSect:BPmLG}

We continue to work in the context of $\mathbb{R}^{n+1}$ introduced at the beginning of Section~\ref{Sect:SFE}, and abbreviate the $n$-dimensional Hausdorff outer measure from Definition~\ref{SKJ38} as ${\mathcal{H}}^n:={\mathcal{H}}^n_{{\mathbb{R}}^{n+1}}$. We prove that square function
estimates are stable under the so-called big pieces functor. Square function estimates on
uniformly rectifiable sets then follow as a simple corollary. Let us begin by reviewing 
the concept of uniform rectifiability. In particular, following G.~David and S.~Semmes
\cite{DaSe91}, we make the following definition.

\begin{definition}\label{Def-unif.rect}
A closed set ${\Sigma}\subseteq{\mathbb{R}}^{n+1}$ is called {\tt uniformly rectifiable} provided
it is $n$-dimensional Ahlfors-David regular and the following property holds. There exist
$\varepsilon$, $M\in(0,\infty)$ (called the {\rm UR} constants of ${\Sigma}$) such that for each
$x\in {\Sigma}$ and $r>0$, there is a Lipschitz map 
$\varphi:B^{n}_r\rightarrow {\mathbb{R}}^{n+1}$ (where $B^{n}_r$ is a ball of radius $r$ 
in ${\mathbb{R}}^{n}$) with Lipschitz constant at most equal to $M$, such that
\begin{eqnarray}\label{3.1.9aS}
{\mathcal{H}}^{n}\bigl({\Sigma}\cap B(x,r)\cap\varphi(B^{n}_r)\bigr)\geq\varepsilon r^{n}.
\end{eqnarray}
If ${\Sigma}$ is compact, then this is only required for $r\in (0,{\rm diam}\,({\Sigma})]$.  
\end{definition}

There are a variety of equivalent characterizations of uniform rectifiability 
(cf., e.g., \cite[Theorem~I.1.5.7, p.\,22]{DaSe93}); the version above is often specified 
by saying that  ${\Sigma}$ has {\tt Big Pieces of Lipschitz Images} (or, simply {\tt BPLI}).
Another version, in which Lipschitz maps are replaced with Bi-Lipschitz maps, is specified
by saying that ${\Sigma}$ has {\tt Big Pieces of Bi-Lipschitz Images} (or, simply {\tt BPBI}). 
The equivalence between BPLI and BPBI can be found in~\cite[p.\,22]{DaSe93}. 
We also require the following notion of sets having big pieces of Lipschitz graphs.

\begin{definition}\label{Def-BPnLG}
A set ${\Sigma}\subseteq \mathbb{R}^{n+1}$ is said to have {\tt Big Pieces of Lipschitz Graphs}
(or, simply {\tt BPLG}) provided it is $n$-dimensional Ahlfors-David regular and the
following property holds. There exist $\varepsilon$, $M\in(0,\infty)$ (called the 
{\rm BPLG} constants of ${\Sigma}$) such that for each $x\in {\Sigma}$ and $r>0$, there is an
$n$-dimensional Lipschitz graph $\Gamma\subseteq\mathbb{R}^{n+1}$ with Lipschitz 
constant at most equal to $M$, such that
\begin{eqnarray}\label{3.1.9a}
{\mathcal{H}}^{n}\bigl({\Sigma}\cap B(x,r)\cap\Gamma)\geq\varepsilon r^{n}.
\end{eqnarray}
If ${\Sigma}$ is compact, then this is only required for 
$r\in (0,{\rm diam}\,({\Sigma})]$.

We also write {\tt (BP)$^1$LG} to mean {\tt BPLG}. For each $k\geq1$, a set 
${\Sigma}\subseteq \mathbb{R}^{n+1}$ is said to have {\tt Big Pieces of (BP)$^{k}$LG} 
(or, simply {\tt (BP)$^{k+1}$LG}) provided it is $n$-dimensional Ahlfors-David regular 
and the following property holds.  There exist $\delta$, $\varepsilon$, $M\in(0,\infty)$
(called the {\rm (BP)$^{k+1}$LG} constants of ${\Sigma}$) such that for each 
$x\in {\Sigma}$ and $r>0$, there is a set $\Omega\subseteq\mathbb{R}^{n+1}$ 
that has {\rm (BP)$^{k}$LG} with {\rm ADR} constant at most equal to $M$, and 
{\rm (BP)$^{k}$LG} constants $\varepsilon$, $M$, such that 
\begin{eqnarray}\label{3.1.9Z}
{\mathcal{H}}^{n}\bigl({\Sigma}\cap B(x,r)\cap \Omega)\geq\delta r^{n}.
\end{eqnarray}
If ${\Sigma}$ is compact, then this is only required for 
$r\in (0,{\rm diam}\,({\Sigma})]$.
\end{definition}

We now combine the inductive scheme from Section~\ref{Sect:4} with the square function estimates for Lipschitz graphs from Subsection~\ref{SSect:LG} to prove that square function estimates are stable under the so-called big pieces functor.

\begin{theorem}\label{thm:BPmLGimSFE}
Let $k\in\mathbb{N}$ and suppose that ${\Sigma}\subseteq \mathbb{R}^{n+1}$ has (BP)$^k$LG. 
Let $K$ be a real-valued kernel satisfying \eqref{prop-K}, and let ${\mathcal{T}}$ 
denote the integral operator associated with $\Sigma$ as in \eqref{sbnvjb}. 
Then there exists a constant $C\in(0,\infty)$ depending only on $n$, the (BP)$^k$LG constants of ${\Sigma}$, and $\|\partial^\alpha K\|_{L^\infty(S^{n})}$ for $|\alpha|\leq 2$, 
such that 
\begin{eqnarray}\label{vlnGG-2}
\int_{\mathbb{R}^{n+1}\setminus {\Sigma}}|\nabla{\mathcal{T}}f(x)|^2\,{\rm dist}(x,{\Sigma})\,dx 
\leq C\int_{\Sigma}|f|^2\,d\sigma,\qquad\forall\,f\in L^2({\Sigma},\sigma),
\end{eqnarray}
where $\sigma:={\mathscr{H}}^n\lfloor{{\Sigma}}$ is the measure induced by the $n$-dimensional
Hausdorff measure on ${\Sigma}$.
\end{theorem}

\begin{proof}
The proof proceeds by induction on ${\mathbb{N}}$. For the case $k=1$, suppose that
${\Sigma}\subseteq \mathbb{R}^{n+1}$ has BPLG with BPLG constants $\varepsilon_0$,
$C_0\in(0,\infty)$. For each $x\in {\Sigma}$ and $r>0$, there is an $n$-dimensional Lipschitz
graph $\Gamma\subseteq\mathbb{R}^{n+1}$ with Lipschitz constant at most equal to $C_0$, 
such that 
\begin{eqnarray}\label{3.1.9aB}
{\mathcal{H}}^{n}\bigl({\Sigma}\cap B(x,r)\cap \Gamma)\geq \varepsilon_0 r^{n}.
\end{eqnarray}
It follows from Theorem~\ref{Theorem1.1} that ${\Sigma}$ has BPSFE with BPSFE character 
(cf. Definition~\ref{sjvs}) depending only on the BPLG constants of ${\Sigma}$, and
$\|\partial^\alpha K\|_{L^\infty(S^{n})}$ for $|\alpha|\leq 2$. It then follows from
Theorem~\ref{Thm:BPSFtoSF} that~\eqref{vlnGG-2} holds for some $C\in(0,\infty)$ depending
only on $n$, the BPLG constants of ${\Sigma}$, and $\|\partial^\alpha K\|_{L^\infty(S^{n})}$ 
for $|\alpha|\leq 2$.

Now let $j\in\mathbb{N}$ and assume that the statement of the theorem holds in the case
$k=j$. Suppose that ${\Sigma}\subseteq \mathbb{R}^{n+1}$ has (BP)$^{j+1}$LG with (BP)$^{j+1}$LG
constants $\varepsilon_1$, $\varepsilon_2$, $C_1\in(0,\infty)$. For each $x\in {\Sigma}$ and
$r>0$, there is a set $\Omega\subseteq\mathbb{R}^{n+1}$ that has {\rm (BP)$^{j}$LG} with 
{\rm ADR} constant at most equal to $C_1$, and {\rm (BP)$^{j}$LG} constants 
$\varepsilon_1$, $C_1$, such that 
\begin{eqnarray}\label{3.1.9aQ}
{\mathcal{H}}^{n}\bigl({\Sigma}\cap B(x,r)\cap \Omega)\geq \varepsilon_2 r^{n}.
\end{eqnarray}
It follows by the inductive assumption that ${\Sigma}$ has BPSFE with BPSFE character depending only on the constants specified in the theorem in the case $k=j$. Applying again Theorem~\ref{Thm:BPSFtoSF} we obtain that~\eqref{vlnGG-2} holds for some $C\in(0,\infty)$
depending only on $n$, the (BP)$^{j+1}$LG constants of ${\Sigma}$, and $\|\partial^\alpha K\|_{L^\infty(S^{n})}$ for $|\alpha|\leq 2$. This completes the proof.
\end{proof}

The recent result by J. Azzam and R. Schul (cf. \cite[Corollary~1.7]{AS}) that uniformly rectifiable sets have (BP)$^2$LG allows us to obtain the following as an immediate corollary of Theorem~\ref{thm:BPmLGimSFE}.

\begin{corollary}\label{cor:URimSFE}
Suppose that ${\Sigma}\subseteq \mathbb{R}^{n+1}$ is a uniformly rectifiable set. Let $K$ be a real-valued kernel satisfying \eqref{prop-K}, and let ${\mathcal{T}}$ denote the integral operator associated with $\Sigma$ as in \eqref{sbnvjb}. Then there exists a constant $C\in(0,\infty)$, depending only on $n$, the UR constants of ${\Sigma}$, and $\|\partial^\alpha K\|_{L^\infty(S^{n})}$ for $|\alpha|\leq 2$, such that 
\begin{eqnarray}\label{corvlnGG}
\int_{\mathbb{R}^{n+1}\setminus {\Sigma}}|\nabla{\mathcal{T}}f(x)|^2\,{\rm dist}(x,{\Sigma})\,d x 
\leq C\int_{\Sigma}|f|^2\,d\sigma,\qquad\forall\,f\in L^2({\Sigma},\sigma),
\end{eqnarray}
where $\sigma:={\mathscr{H}}^n\lfloor{{\Sigma}}$ is the measure induced by the $n$-dimensional Hausdorff measure on ${\Sigma}$.
\end{corollary}

\begin{proof}
The set ${\Sigma}$ has (BP)$^2$LG by J. Azzam and R. Schul's characterization 
of uniformly rectifiable sets in~\cite[Corollary~1.7]{AS}, so the result follows 
at once from Theorem~\ref{thm:BPmLGimSFE}.
\end{proof}

\subsection{Square function estimates for integral operators with 
variable kernels}
\label{SSect:PDO}

The square function estimates from Theorem~\ref{thm:BPmLGimSFE} 
and Corollary~\ref{cor:URimSFE} have been formulated for {\it convolution type} 
integral operators and our goal in this subsection is to prove some versions of 
these results which apply to integral operators with variable coefficient kernels.
A first result in this regard reads as follows.

\begin{theorem}\label{thm:BVAR}
Let $k\in\mathbb{N}$ and suppose that ${\Sigma}\subseteq{\mathbb{R}}^{n+1}$ is
compact and has (BP)$^k$LG. Then there exists a positive integer $M=M(n)$ with 
the following significance. Assume that ${\mathcal{U}}$ is a bounded, open 
neighborhood of $\Sigma$ in ${\mathbb{R}}^{n+1}$ and consider a function 
\begin{eqnarray}\label{B-pv5Avb}
{\mathcal{U}}\times\bigl({\mathbb{R}}^{n+1}\setminus\{0\}\bigr)\ni (x,z)
\mapsto b(x,z)\in{\mathbb{R}} 
\end{eqnarray}
which is odd and (positively) homogeneous of degree $-n$ in the 
variable $z\in{\mathbb{R}}^{n+1}\setminus\{0\}$, and which has the property that 
\begin{eqnarray}\label{B-pv5.LFD}
\mbox{$\partial_x^\beta\partial_z^\alpha b(x,z)$ is continuous and bounded on 
${\mathcal{U}}\times S^{n}$ for $|\alpha|\leq M$ and $|\beta|\leq 1$}. 
\end{eqnarray}
Finally, define the variable kernel integral operator 
\begin{eqnarray}\label{B-pv5}
{\mathcal{B}}f(x):=\int_{\Sigma}b(x,x-y)f(y)\,d\sigma(y),\qquad 
x\in{\mathcal{U}}\setminus\Sigma,
\end{eqnarray}
where $\sigma:={\mathscr{H}}^n\lfloor{\Sigma}$ is the measure induced by the $n$-dimensional Hausdorff measure on $\Sigma$. 

Then there exists a constant $C\in(0,\infty)$ depending only on $n$, the (BP)$^k$LG constants of ${\Sigma}$, the diameter of ${\mathcal{U}}$, and 
$\|\partial_x^\beta\partial_z^\alpha b\|_{L^\infty({\mathcal{U}}\times S^{n})}$ 
for $|\alpha|\leq 2$, $|\beta|\leq 1$, such that 
\begin{eqnarray}\label{vlnGG-2ASn}
\int_{{\mathcal{U}}\setminus\Sigma}|\nabla{\mathcal{B}}f(x)|^2\,{\rm dist}(x,\Sigma)\,dx 
\leq C\int_{\Sigma}|f|^2\,d\sigma,\qquad\forall\,f\in L^2(\Sigma,\sigma).
\end{eqnarray}

In particular, \eqref{vlnGG-2ASn} holds whenever $\Sigma$ is uniformly rectifiable
(while retaining the other background assumptions).
\end{theorem}

In preparation for presenting the proof of Theorem~\ref{thm:BVAR}, we state two 
lemmas, of geometric character, from \cite{MMMM-B}.

\begin{lemma}\label{bvv55-DD}
Let $({\mathscr{X}},\rho)$ be a geometrically doubling quasi-metric space 
and let $\Sigma\subseteq{\mathscr{X}}$ be a set with the property that 
$\bigl(\Sigma,\rho\bigl|_\Sigma,{\mathcal{H}}^d_{{\mathscr{X}},\rho}\lfloor\Sigma\bigr)$ 
becomes a $d$-dimensional {\rm ADR} space, for some $d>0$. Assume that $\mu$ 
is a Borel measure on ${\mathscr{X}}$ satisfying 
\begin{eqnarray}\label{bca84r.22.aT}
\sup_{x\in{\mathscr{X}},\,r>0}\frac{\mu\bigl(B_{\rho_{\#}}(x,r)\bigr)}{r^m}<+\infty,
\end{eqnarray}
for some $m\geq 0$. Also, fix a constant $c>0$ 
and select two real numbers $N,\alpha$ such that $\alpha<m-d$ 
and $N<m-\max\,\{\alpha,0\}$. 

Then there exists a constant $C\in(0,\infty)$
depending on the supremum in \eqref{bca84r.22.aT}, the geometric doubling 
constant of $({\mathscr{X}},\rho)$, the {\rm ADR} constant of $\Sigma$, 
as well as $N$, $\alpha$, and $c$, such that
\begin{eqnarray}\label{bca84r.22}
\begin{array}{c}
\displaystyle
\int\limits_{B_{\rho_{\#}}(x,r)\setminus\overline{\Sigma}}
\frac{{\rm dist}_{\rho_{\#}}(y,\Sigma)^{-\alpha}}{\rho_{\#}(x,y)^N}\,d\mu(y)
\leq C\,r^{m-\alpha-N},
\\[30pt]
\forall\,r>0,\quad\forall\,x\in{\mathscr{X}}
\,\,\mbox{ with }\,\,{\rm dist}_{\rho_{\#}}(x,\Sigma)<c\,r.
\end{array}
\end{eqnarray}
\end{lemma}

\begin{lemma}\label{Gkwvr.reV}
Let $({\mathscr{X}},\rho)$ be a quasi-metric space. Suppose 
$E\subseteq{\mathscr{X}}$ is nonempty and $\sigma$ is a measure on $E$ such that
$(E,\rho\bigl|_E,\sigma)$ becomes a $d$-dimensional {\rm ADR} space, for some $d>0$. 
Fix a real number $0\leq N<d$. Then there exists $C\in(0,\infty)$ depending only on 
$N$, $\rho$, and the {\rm ADR} constant of $E$ such that 
\begin{eqnarray}\label{mMji.reV}
\int\limits_{E\cap B_{\rho_{\#}}(x,r)}
\frac{1}{\rho_{\#}(x,y)^N}\,d\sigma(y)\leq C\,r^{d-N},
\qquad\forall\,x\in{\mathscr{X}},\quad\forall\,r>{\rm dist}_{\rho_{\#}}(x,E).
\end{eqnarray}
\end{lemma}

We are now ready to discuss the

\vskip 0.08in
\begin{proof}[Proof of Theorem~\ref{thm:BVAR}]
Set 
\begin{eqnarray}\label{D-HarmK} 
H_0:=1,\quad
H_1:=n+1,\quad\mbox{ and }\quad
H_{\ell}:=\Bigl(\!\!
\begin{array}{c}
n+\ell
\\[-2pt]
\ell
\end{array}
\!\!\Bigr)
-
\Bigl(\!\!
\begin{array}{c}
n+\ell-2
\\[-2pt]
\ell-2
\end{array}
\!\!\Bigr)
\quad\mbox{ if }\,\,\ell\geq 2,
\end{eqnarray} 
and, for each $\ell\in{\mathbb{N}}_0$, let 
$\bigl\{\Psi_{i\ell}\bigr\}_{1\leq i\leq H_\ell}$ be 
an orthonormal basis for the space of spherical harmonics of degree $\ell$ 
on the $n$-dimensional sphere $S^n$. In particular, 
\begin{eqnarray}\label{D-Har-Nr} 
H_{\ell}\leq(\ell+1)\cdot(\ell+2)\cdots(n+\ell-1)\cdot
(n+\ell)\leq C_n\,\ell^n\quad\mbox{ for }\,\,\ell\geq 2
\end{eqnarray} 
and, if $\Delta_{S^{n}}$ denotes the Laplace-Beltrami operator on $S^n$,
then for each $\ell\in{\mathbb{N}}_0$ and $1\leq i\leq H_\ell$, 
\begin{eqnarray}\label{eihen-XS}
\Delta_{S^{n}}\Psi_{i\ell}=-\ell(n+\ell-1)\Psi_{i\ell}\,\,\mbox{ on }\,\,S^n,
\,\,\mbox{ and }\,\,\Psi_{i\ell}\Bigl(\frac{x}{|x|}\Bigr)=\frac{P_{i\ell}(x)}{|x|^{\ell}}
\end{eqnarray} 
for some homogeneous harmonic polynomial $P_{i\ell}$ of degree 
$\ell$ in ${\mathbb{R}}^{n+1}$. Also, 
\begin{eqnarray}\label{eihen-amm-ONB}
\bigl\{\Psi_{i\ell}\bigr\}_{\ell\in{\mathbb{N}}_0,\,1\leq i\leq H_\ell}
\,\,\mbox{ is an orthonormal basis for }\,\,L^2(S^n),
\end{eqnarray} 
hence, 
\begin{eqnarray}\label{eihen-amm}
\|\Psi_{i\ell}\|_{L^2(S^{n})}=1\,\,\mbox{ for each $\ell\in{\mathbb{N}}_0$ and 
$1\leq i\leq H_\ell$}.
\end{eqnarray} 
More details on these matters may be found in, e.g., 
\cite[pp.\,137--152]{STEIN-WEISS} and \cite[pp.\,68--75]{STEIN}.

Assume next that an even integer $d>(n/2)+2$ has been fixed. 
Sobolev's embedding theorem then gives that for each $\ell\in{\mathbb{N}}_0$ and 
$1\leq i\leq H_\ell$ (with $I$ standing for the identity on $S^n$)
\begin{eqnarray}\label{kl-dYU-11}
\|\Psi_{i\ell}\|_{C^2(S^{n})}
\leq C_n\bigl\|(I-\Delta_{S^{n}})^{d/2}\Psi_{i\ell}\bigr\|_{L^2(S^{n})}\leq C_n\ell^{d},
\end{eqnarray}
where the last inequality is a consequence of \eqref{eihen-XS}-\eqref{eihen-amm}.

Fix $\ell\in{\mathbb{N}}_0$ and $1\leq i\leq H_\ell$ arbitrary. If we now define 
\begin{eqnarray}\label{coef-sh}
a_{i\ell}(x):=\int_{S^{n}}b(x,\omega)\Psi_{i\ell}(\omega)\,d\omega,\,\,
\mbox{ for each }\,\,x\in{\mathcal{U}},
\end{eqnarray}
it follows from the last formula in \eqref{eihen-XS} and the assumptions on 
$b(x,z)$ that 
\begin{eqnarray}\label{coef-sh.aC}
\mbox{$a_{i\ell}$ is identically zero whenever $\ell$ is even}.
\end{eqnarray}
Also, for each number $N\in{\mathbb{N}}$ with $2N\leq M$ and 
each multiindex $\beta$ of length $\leq 1$ we have 
\begin{eqnarray}\label{coef-sh.AAA}
\sup_{x\in{\mathcal{U}}}
\bigl|[-\ell(n+\ell-1)]^N(\partial^\beta a_{i\ell})(x)\bigr|
&=& \sup_{x\in{\mathcal{U}}}
\Bigl|\int_{S^{n}}(\partial_x^\beta b)(x,\omega)
\bigl(\Delta_{S^{n}}^N\Psi_{i\ell}\bigr)(\omega)\,d\omega\Bigr|
\nonumber\\[4pt]
&=& \sup_{x\in{\mathcal{U}}}
\Bigl|\int_{S^{n}}\bigl(\partial_x^\beta\Delta_{S^{n}}^N b\bigr)(x,\omega)
\Psi_{i\ell}(\omega)\,d\omega\Bigr|
\nonumber\\[4pt]
&\leq & \sup_{x\in{\mathcal{U}}}
\bigl\|\bigl(\partial_x^\beta\Delta_{S^{n}}^N b\bigr)(x,\cdot)
\bigr\|_{L^2(S^{n})}
\nonumber\\[4pt]
&\leq & C_n\sup_{\stackrel{(x,z)\in{\mathcal{U}}\times S^n}{|\alpha|\leq M}}
\bigl|\bigl(\partial_x^\beta\partial_z^{\alpha}b\bigr)(x,z)\bigr|=:C_b,
\end{eqnarray}
where $C_b$ is a finite constant. 
Hence, for each number $N\in{\mathbb{N}}$ with $2N\leq M$ 
there exists a constant $C_{n,N}$ such that
\begin{eqnarray}\label{coef-HJ}
\sup_{x\in{\mathcal{U}},\,|\beta|\leq 1}\bigl|(\partial^\beta a_{i\ell})(x)\bigr|
\leq C_{n,N} C_b\,\ell^{-2N},\qquad\ell\in{\mathbb{N}}_0,\quad 1\leq i\leq H_\ell. 
\end{eqnarray}
For each fixed $x\in{\mathcal{U}}$, expand the function
$b(x,\cdot)\in L^2(S^n)$ with respect to the orthonormal basis
$\bigl\{\Psi_{i\ell}\bigr\}_{\ell\in{\mathbb{N}}_0,\,1\leq i\leq H_\ell}$ in order 
to obtain that (in the sense of $L^2(S^n)$ in the variable $z/|z|\in S^n$)
\begin{eqnarray}\label{tag{1.12}}
b(x,z) &=& b\Bigl(x,\frac{z}{|z|}\Bigr)|z|^{-n}
=\sum_{\ell\in{\mathbb{N}}}\sum_{i=1}^{H_{\ell}} 
a_{i\ell}(x)\Psi_{i\ell}\Bigl(\frac{z}{|z|}\Bigr)|z|^{-n}
\nonumber\\[4pt]
&=& \sum_{\ell\in 2{\mathbb{N}}+1}\sum_{i=1}^{H_{\ell}} 
a_{i\ell}(x)\Psi_{i\ell}\Bigl(\frac{z}{|z|}\Bigr)|z|^{-n},
\end{eqnarray}
where the last equality is a consequence of \eqref{coef-sh.aC}.
For each $\ell\in 2{\mathbb{N}}+1$ let us now set 
\begin{eqnarray}\label{1.15GH}
k_{i\ell}(z):=\Psi_{i\ell}\Bigl(\frac{z}{|z|}\Bigr)|z|^{-n},
\quad z\in{\mathbb{R}}^{n+1}\setminus\{0\},
\end{eqnarray}
so that, if $d$ is as in \eqref{kl-dYU-11},  
then for each $|\alpha|\leq 2$ we have
\begin{eqnarray}\label{kl-dYU}
\|\partial^\alpha k_{i\ell}\|_{L^\infty(S^{n})}
\leq C_n\|\Psi_{i\ell}\|_{C^2(S^{n})}\leq C_n\ell^{d}.
\end{eqnarray}
Also, given any $f\in L^2(\Sigma,\sigma)$, set 
\begin{eqnarray}\label{1.15ASD}
{\mathcal{B}}_{i\ell}f(x):=\int_{\Sigma}k_{i\ell}(x-y)f(y)\,d\sigma(y),\quad 
x\in{\mathcal{U}}\setminus\Sigma,
\end{eqnarray}
and note that for any compact subset ${\mathcal{O}}$ of ${\mathcal{U}}\setminus\Sigma$
and any multiindex $\alpha$ with $|\alpha|\leq 1$, 
\begin{eqnarray}\label{1.15ASD.u}
\sup_{x\in{\mathcal{O}}}\bigl|\bigl(\partial^\alpha
{\mathcal{B}}_{i\ell}f\bigr)(x)\bigr|\leq C(n,{\mathcal{O}},\Sigma)\ell^{d},
\end{eqnarray}
by \eqref{kl-dYU}. On the other hand, if $N>(d+1)/2$ 
(a condition which we shall assume from now on) then  
\eqref{kl-dYU-11} and \eqref{coef-HJ} imply that the last series
in \eqref{tag{1.12}} converges to $b(x,z)$ uniformly for $x\in{\mathcal{U}}$ 
and $z$ in compact subsets of ${\mathbb{R}}^{n+1}\setminus\{0\}$.
As such, it follows from \eqref{1.15GH} and \eqref{1.15ASD} that
\begin{eqnarray}\label{huBN}
{\mathcal{B}}f(x)=\sum_{\ell\in 2{\mathbb{N}}+1}\sum_{i=1}^{H_{\ell}} 
a_{i\ell}(x){\mathcal{B}}_{i\ell}f(x),\,\,\mbox{ uniformly on compact subsets of }\,\,
{\mathcal{U}}\setminus\Sigma.
\end{eqnarray}
Using this, \eqref{1.15ASD.u} and \eqref{coef-HJ}, the term-by-term differentiation 
theorem for series of functions may be invoked in order to obtain that
\begin{eqnarray}\label{huBN.e}
&& \hskip -0.20in
\bigl(\nabla{\mathcal{B}}f\bigr)(x)
=\sum_{\ell\in 2{\mathbb{N}}+1}\sum_{i=1}^{H_{\ell}} 
a_{i\ell}(x)\bigl(\nabla{\mathcal{B}}_{i\ell}f\bigr)(x)
+\sum_{\ell\in 2{\mathbb{N}}+1}\sum_{i=1}^{H_{\ell}} 
(\nabla a_{i\ell})(x){\mathcal{B}}_{i\ell}f(x),
\nonumber\\[4pt]
&& \hskip 0.80in
\mbox{uniformly for $x$ in compact subsets of }\,\,{\mathcal{U}}\setminus\Sigma.
\end{eqnarray}

Moving on, observe that for each $\ell\in 2{\mathbb{N}}+1$ and 
$1\leq i\leq H_{\ell}$, Theorem~\ref{thm:BPmLGimSFE} gives 
\begin{eqnarray}\label{vln-GFH}
\int_{{\mathcal{U}}\setminus\Sigma}|\nabla{\mathcal{B}}_{i\ell}f(x)|^2\,
{\rm dist}(x,{\Sigma})\,dx\leq C_{i\ell}\int_{\Sigma}|f|^2\,d\sigma,
\qquad\forall\,f\in L^2(\Sigma,\sigma),
\end{eqnarray}
where, with $C\in(0,\infty)$ depending only on the dimension $n$ and 
the (BP)$^k$LG constants of $\Sigma$.  
\begin{eqnarray}\label{vln-Gam}
C_{i\ell}=C\max_{|\alpha|\leq 2}\|\partial^\alpha k_{i\ell}\|_{L^\infty(S^{n})}
\leq C\ell^d,
\end{eqnarray}
thanks to \eqref{kl-dYU-11}. Thus, if 
\begin{eqnarray}\label{kl-dYU.af}
\mbox{$M\in{\mathbb{N}}$ is odd and satisfies $M>d+1$},  
\end{eqnarray}
one may choose $N\in{\mathbb{N}}$ such that $d+1<2N<M$. Such a choice ensures that
for every $f\in L^2(\Sigma,\sigma)$
\begin{eqnarray}\label{vln-GFH.ae}
&& \hskip -0.50in
\sum_{\ell\in 2{\mathbb{N}}+1}\sum_{i=1}^{H_{\ell}} 
\Bigl(\int_{{\mathcal{U}}\setminus\Sigma}|a_{i\ell}(x)|^2
|\nabla{\mathcal{B}}_{i\ell}f(x)|^2\,{\rm dist}(x,\Sigma)\,dx\Bigr)^{1/2}
\nonumber\\[4pt]
&& \hskip 0.60in
\leq C_{n,N} C_b\sum_{\ell\in 2{\mathbb{N}}+1}\sum_{i=1}^{H_{\ell}}\ell^{-2N}
\Bigl(\int_{{\mathcal{U}}\setminus\Sigma}
|\nabla{\mathcal{B}}_{i\ell}f(x)|^2\,{\rm dist}(x,\Sigma)\,dx\Bigr)^{1/2}
\nonumber\\[4pt]
&& \hskip 0.60in
\leq C_{n,N} C_b
\Bigl(\sum_{\ell\in 2{\mathbb{N}}+1}\sum_{i=1}^{H_{\ell}}
\ell^{-2N}C_{i\ell}^{1/2}\Bigr)\Bigl(\int_{\Sigma}|f|^2\,d\sigma\Bigr)^{1/2}
\nonumber\\[4pt]
&& \hskip 0.60in
\leq C_{n,N} C_b
\Bigl(\sum_{\ell\in 2{\mathbb{N}}+1}\sum_{i=1}^{H_{\ell}}
\ell^{d/2-2N}\Bigr)\Bigl(\int_{\Sigma}|f|^2\,d\sigma\Bigr)^{1/2}
\nonumber\\[4pt]
&& \hskip 0.60in
=C\Bigl(\int_{\Sigma}|f|^2\,d\sigma\Bigr)^{1/2},
\end{eqnarray}
by \eqref{coef-HJ}, \eqref{vln-Gam}, and our choice of $N$.

To proceed, let $\ell\in{\mathbb{N}}$ and $1\leq i\leq H_{\ell}$ be arbitrary. 
Also, fix an arbitrary $f\in L^2(\Sigma,\sigma)$. Then 
\begin{eqnarray}\label{vln-GFH.234}
&& \hskip -0.40in
\Bigl(\int_{{\mathcal{U}}\setminus\Sigma}|\nabla a_{i\ell}|^2
|{\mathcal{B}}_{i\ell}f(x)|^2\,{\rm dist}(x,{\Sigma})\,dx\Bigr)^{1/2}
\nonumber\\[4pt]
&& \hskip 0.60in
\leq C_{n,N} C_b\,\ell^{-2N}\Bigl(\int_{{\mathcal{U}}\setminus\Sigma}
\bigl|{\rm dist}(x,\Sigma)^{1/2}{\mathcal{B}}_{i\ell}f(x)\bigr|^2\,dx\Bigr)^{1/2}
\nonumber\\[4pt]
&& \hskip 0.60in
=C_{n,N} C_b\,\ell^{-2N}\Bigl(\int_{{\mathcal{U}}\setminus\Sigma}
\bigl|{\mathcal{T}}_{i\ell}f(x)\bigr|^2\,dx\Bigr)^{1/2}
\end{eqnarray}
where 
\begin{eqnarray}\label{vaj-TBB}
{\mathcal{T}}_{i\ell}:L^2(\Sigma,\sigma)\longrightarrow 
L^2({\mathcal{U}}\setminus\Sigma)
\end{eqnarray}
is the integral operator whose integral kernel is given by 
\begin{eqnarray}\label{vaj-TBB.2i}
K_{i\ell}(x,y):={\rm dist}(x,\Sigma)^{1/2}k_{i\ell}(x-y),\qquad
x\in{\mathcal{U}}\setminus\Sigma,\,\,\,y\in\Sigma.
\end{eqnarray}
Note that 
\begin{eqnarray}\label{vaj-TBB.2}
\sup_{x\in{\mathcal{U}}\setminus\Sigma}
\int_{\Sigma}\bigl|K_{i\ell}(x,y)\bigr|\,d\sigma(y)
&\leq & \|\Psi_{i\ell}\|_{L^\infty(S^n)}\sup_{x\in{\mathcal{U}}\setminus\Sigma}
\int_{\Sigma}\frac{{\rm dist}(x,\Sigma)^{1/2}}{|x-y|^n}\,d\sigma(y)
\nonumber\\[4pt]
&\leq & C\ell^d\sup_{x\in{\mathcal{U}}\setminus\Sigma}
\int_{\Sigma}\frac{1}{|x-y|^{n-1/2}}\,d\sigma(y)
\nonumber\\[4pt]
&\leq & C\ell^d\,{\rm diam}({\mathcal{U}})^{1/2}, 
\end{eqnarray}
by \eqref{kl-dYU-11} and Lemma~\ref{Gkwvr.reV}, and that
\begin{eqnarray}\label{vaj-TBB.3}
\sup_{y\in\Sigma}\int_{{\mathcal{U}}\setminus\Sigma}\bigl|K_{i\ell}(x,y)\bigr|\,dx
&\leq & \|\Psi_{i\ell}\|_{L^\infty(S^n)}\sup_{y\in\Sigma}
\int_{{\mathcal{U}}\setminus\Sigma}\frac{{\rm dist}(x,\Sigma)^{1/2}}{|x-y|^n}\,dx
\nonumber\\[4pt]
&\leq & C\ell^d\,{\rm diam}({\mathcal{U}})^{3/2}, 
\end{eqnarray}
by \eqref{kl-dYU-11} and Lemma~\ref{bvv55-DD}. 
From \eqref{vaj-TBB.2}-\eqref{vaj-TBB.3} and 
Schur's Lemma we then deduce that the operator ${\mathcal{T}}_{i\ell}$
is bounded in the context of \eqref{vaj-TBB}, with norm 
\begin{eqnarray}\label{vaj-TBB.4}
\bigl\|{\mathcal{T}}_{i\ell}\bigr\|_{L^2(\Sigma,\sigma)\rightarrow 
L^2({\mathcal{U}}\setminus\Sigma)}\leq C\ell^d\,{\rm diam}({\mathcal{U}}).
\end{eqnarray}
Combining \eqref{vaj-TBB.4} and \eqref{vln-GFH.234} we therefore arrive at 
the conclusion that, for each $f\in L^2(\Sigma,\sigma)$,
\begin{eqnarray}\label{vln-GFH.234-AG}
\Bigl(\int_{{\mathcal{U}}\setminus\Sigma}|\nabla a_{i\ell}|^2
|{\mathcal{B}}_{i\ell}f(x)|^2\,{\rm dist}(x,{\Sigma})\,dx\Bigr)^{1/2}
\leq C({\mathcal{U}})\ell^d\Bigl(\int_{\Sigma}|f|^2\,d\sigma\Bigr)^{1/2},
\end{eqnarray}
whenever $\ell\in{\mathbb{N}}$ and $1\leq i\leq H_\ell$. As a result,
there exists $C\in(0,\infty)$ such that
\begin{eqnarray}\label{vln-GFH.234.u}
\sum_{\ell\in 2{\mathbb{N}}+1}\sum_{i=1}^{H_{\ell}} 
\Bigl(\int_{{\mathcal{U}}\setminus\Sigma}|\nabla a_{i\ell}|^2
|{\mathcal{B}}_{i\ell}f(x)|^2\,{\rm dist}(x,{\Sigma})\,dx\Bigr)^{1/2}
\leq C\Bigl(\int_{\Sigma}|f|^2\,d\sigma\Bigr)^{1/2},
\end{eqnarray}
for every $f\in L^2(\Sigma,\sigma)$.

Fix now an arbitrary compact subset ${\mathcal{O}}$ 
of ${\mathcal{U}}\setminus\Sigma$. Then \eqref{huBN.e}, \eqref{vln-GFH.ae}
and \eqref{vln-GFH.234.u} allow us to estimate 
\begin{eqnarray}\label{huBN.ef}
\Bigl(\int_{{\mathcal{O}}}
|\nabla{\mathcal{B}}f(x)|^2\,{\rm dist}(x,{\Sigma})\,dx\Bigr)^{1/2}
\leq C\Bigl(\int_{\Sigma}|f|^2\,d\sigma\Bigr)^{1/2},
\end{eqnarray}
where the constant $C$ is independent of ${\mathcal{O}}$ and $f\in L^2(\Sigma,\sigma)$.
Upon letting ${\mathcal{O}}\nearrow{\mathcal{U}}\setminus\Sigma$ in \eqref{huBN.ef},
Lebesgue's Monotone Convergence Theorem then yields \eqref{vlnGG-2ASn}.
Finally, the last claim in the statement of Theorem~\ref{thm:BVAR} is justified in 
a similar manner, based on Corollary~\ref{cor:URimSFE}.
\end{proof}

It is also useful to treat the following variant of \eqref{B-pv5}:
\begin{eqnarray}\label{S5.52}
\widetilde{\mathcal{B}}f(x):=\int\limits_{\Sigma}b(y,x-y)f(y)\,d\sigma(y),
\qquad x\in{\mathcal{U}}\setminus\Sigma.
\end{eqnarray}
The same sort of analysis works, with $x$ replaced by $y$ in 
the spherical harmonic expansion \eqref{tag{1.12}}
(in fact, the argument is simpler since the $a_{i\ell}$'s act this time 
as multipliers in the $y$ variable). Specifically, we have the following.

\begin{theorem}\label{p5.4}
In the setting of Theorem~\ref{thm:BVAR}, with $\widetilde{\mathcal{B}}$ given
by \eqref{S5.52} where, this time, in place of \eqref{B-pv5.LFD} one assumes
\begin{eqnarray}\label{B-pv5.LFD.XXX}
\mbox{$\partial_z^\alpha b(x,z)$ is continuous and bounded on 
${\mathcal{U}}\times S^{n}$ for $|\alpha|\leq M$}, 
\end{eqnarray}
there holds
\begin{eqnarray}\label{vva-UHab}
\int_{{\mathcal{U}}\setminus\Sigma}
|\nabla\widetilde{\mathcal{B}}f(x)|^2\,{\rm dist}(x,\Sigma)\,dx 
\leq C\int_{\Sigma}|f|^2\,d\sigma,\qquad\forall\,f\in L^2(\Sigma,\sigma).
\end{eqnarray}
\end{theorem}

In turn, Theorem~\ref{thm:BVAR} and Theorem~\ref{p5.4} apply to the Schwartz 
kernels of certain pseudodifferential operators. Recall that a pseudodifferential
operator $Q(x,D)$ with symbol $q(x,\xi)$ in H\"ormander's class $S^m_{1,0}$ 
is given by the oscillatory integral
\begin{eqnarray}\label{1.6}
Q(x,D)u &=& (2\pi)^{-(n+1)/2}\int q(x,\xi)\hat{u}(\xi) 
e^{i\langle x,\,\xi\rangle}\,d\xi
\nonumber\\[4pt] 
&=& (2\pi)^{-(n+1)}\int\!\!\int q(x,\xi) e^{i\langle x-y,\,\xi\rangle}u(y)\,dy\,d\xi.
\end{eqnarray}
Here, we are concerned with a smaller class of symbols, $S^m_{\rm cl}$, 
defined by requiring that the (matrix-valued) function $q(x,\xi)$ has an 
asymptotic expansion of the form 
\begin{eqnarray}\label{1.8-B}
q(x,\xi)\sim q_m(x,\xi)+q_{m-1}(x,\xi)+\cdots, 
\end{eqnarray}
with $q_j$ smooth in $x$ and $\xi$ and homogeneous of degree $j$ 
in $\xi$ (for $|\xi|\geq 1$). Call $q_m(x,\xi)$, i.e. the leading term in 
\eqref{1.8-B}, the {\it principal symbol} of $q(x,D)$. In fact, we shall 
find it convenient to work with classes of symbols which only exhibit a 
limited amount of regularity in the spatial variable (while still $C^\infty$
in the Fourier variable). Specifically, for each $r\geq 0$ we define 
\begin{eqnarray}\label{1.2}
C^rS^m_{1,0}:=\bigl\{q(X,\xi):\,\|D^\alpha_\xi q(\cdot,\xi)\|_{C^r}\leq C_\alpha
(1+|\xi|)^{m-|\alpha|},\quad\forall\,\alpha\bigr\}.
\end{eqnarray}
Denote by ${\rm OP}{C^r}S^m_{1,0}$ the class of pseudodifferential operators
associated with such symbols. As before, we write ${\rm OP}{C^r}S^m_{\rm cl}$ 
for the subclass of {\it classical} pseudodifferential operators in 
${\rm OP}{C^r}S^m_{1,0}$ whose symbols can be expanded as in \eqref{1.8-B}, 
where $q_j(x,\xi)\in C^rS^{m-j}_{1,0}$ is homogeneous of degree $j$ in $\xi$ 
for $|\xi|\geq 1$, $j=m,m-1,\dots$. Finally, we set 
$\mbox{\it \O}{\rm P}{C^r}S^m_{\rm cl}$ 
for the space of all formal adjoints of operators in ${\rm OP}{C^r}S^m_{\rm cl}$.

Given a classical pseudodifferential operator 
$Q(x,D)\in{\rm OP}C^rS^{-1}_{\rm cl}$, we denote by $k_Q(x,y)$ and 
${\rm Sym}_Q(x,\xi)$ its Schwartz kernel and its principal symbol, 
respectively. Next, if the sets $\Sigma\subseteq{\mathcal{U}}\subseteq{\mathbb{R}}^{n+1}$ 
are as in Theorem~\ref{thm:BVAR}, we can introduce the integral operator
\begin{eqnarray}\label{1.5-B}
{\mathcal{B}}_Qf(x):=\int_{\Sigma}k_Q(x,y)f(y)\,d\sigma(y),
\qquad x\in{\mathcal{U}}\setminus\Sigma.
\end{eqnarray}
In this context, Theorem~\ref{thm:BVAR} and Theorem~\ref{p5.4} yield the 
following result.

\begin{theorem}\label{T-mmt} 
Let ${\Sigma}\subseteq{\mathbb{R}}^{n+1}$ be compact and uniformly rectifiable,
and assume that ${\mathcal{U}}$ is a bounded, open neighborhood of $\Sigma$ 
in ${\mathbb{R}}^{n+1}$. Let $Q(x,D)\in{\rm OP}C^1S^{-1}_{\rm cl}$ be such that 
${\rm Sym}_Q(x,\xi)$ is odd in $\xi$. Then the operator \eqref{1.5-B} satisfies
\begin{eqnarray}\label{hkah-iyT}
\int_{{\mathcal{U}}\setminus\Sigma}
|\nabla{\mathcal{B}}_Qf(x)|^2\,{\rm dist}(x,\Sigma)\,dx 
\leq C\int_{\Sigma}|f|^2\,d\sigma,\qquad\forall\,f\in L^2(\Sigma,\sigma).
\end{eqnarray}

Moreover, a similar result is valid for a pseudodifferential operator 
$Q(x,D)\in\mbox{\O}{\rm P}C^0S^{-1}_{\rm cl}$.   
\end{theorem}

In fact, since the main claims in Theorem~\ref{T-mmt} are local in nature 
and given the invariance of the class of domains and pseudodifferential 
operators (along with their Schwartz kernels and principal symbols) 
under smooth diffeomorphisms, these results can be naturally extended 
to the setting of domains on manifolds and pseudodifferential operators 
acting between vector bundles. Formulated as such, these in turn extend 
results proved in \cite{MMT} for Lipschitz subdomains of Riemannian manifolds.

\section{$L^p$ Square Function Estimates}
\setcounter{equation}{0}
\label{Sect:5}

We have so far only considered $L^2$ square function estimates. We now consider 
$L^p$ versions for $p\in(0,\infty]$. The natural setting for the consideration 
of these estimates is in term of mixed norm spaces $L^{(p,q)}(\mathscr{X},E)$, 
originally introduced in \cite{MMM} (cf. also \cite{BMMM} for related matters). 
We begin by using the tools developed in Section~\ref{Sect:2} to analyze these 
spaces in the context of an ambient quasi-metric space $\mathscr{X}$ and a 
closed subset $E$. In the case $\mathscr{X}=\mathbb{R}^{n+1}$ and $E=\partial\mathbb{R}^{n+1}_+\eqsim\mathbb{R}^n$, the mixed norm spaces correspond 
to the tent spaces introduced by R.~Coifman, Y.~Meyer and E.M.~Stein in \cite{CoMeSt}.
The preliminary analysis in Subsections \ref{SSect:5.1} and \ref{SSect:5.2} is based 
on the techniques developed in that paper, although we need to overcome a variety of geometric obstructions that arise outside of the Euclidean setting. We build on this 
in Subsection~\ref{SSect:5.3}, where we prove that $L^2$ square function estimates associated with integral operators $\Theta_E$, as defined in Section~\ref{Sect:3}, 
follow from weak $L^p$ square function estimates for any $p\in(0,\infty)$. 
This is achieved by combining the $T(1)$ theorem from Subsection~\ref{SSect:3.1} 
with a weak type John-Nirenberg lemma for Carleson measures, the Euclidean version 
of which appears in~\cite{AHLT}. The theory culminates in Subsection~\ref{SSect:5.4},
where we prove two extrapolation theorems for estimates associated with integral
operators $\Theta_E$, as defined in Section~\ref{Sect:3}. In particular, we prove 
that a weak $L^q$ square function estimate for any $q\in(0,\infty)$ implies that square
functions are bounded from $H^p$ into $L^p$ for all $p\in(\frac{d}{d+\gamma},\infty)$,
where $H^p$ is a Hardy space, $d$ is the the dimension of $E$, and $\gamma$ is a 
finite positive constant depending on the ambient space $\mathscr{X}$ and the 
operator $\Theta_E$.

\subsection{Mixed norm spaces}
\label{SSect:5.1}

We begin by considering the mixed norm spaces $L^{(p,q)}$ from \cite{MMM}
(cf. also \cite{BMMM}) and then, following the theory 
of tent spaces in~\cite{CoMeSt}, record some extensive preliminaries that are used throughout Section~\ref{Sect:5}. In particular, Theorem~\ref{appert} contains an
equivalence for the quasi-norms of the mixed norm spaces that is essential in the 
next subsection.

Let $({\mathscr{X}},\rho)$ be a quasi-metric space, $E$ a nonempty subset of 
${\mathscr{X}}$, and $\mu$ a Borel measure on $({\mathscr{X}},\tau_\rho)$. 
Recall the regularized version $\rho_{\#}$ of the quasi-distance $\rho$ 
discussed in Theorem~\ref{JjEGh}, and recall that we employ the 
notation $\delta_E(y)={\rm dist}_{\rho_{\#}}(y,E)$ for each $y\in {\mathscr{X}}$.
Next, let $\kappa>0$ be arbitrary, fixed, and consider the 
{\tt nontangential approach regions} 
\begin{eqnarray}\label{TLjb}
\Gamma_\kappa(x):=\bigl\{y\in{\mathscr{X}}\setminus E:\,
\rho_{\#}(x,y)<(1+\kappa)\,\delta_E(y)\bigr\},
\qquad\forall\,x\in E.
\end{eqnarray}
Occasionally, we shall refer to $\kappa$ as the {\tt aperture} of 
the nontangential approach region $\Gamma_\kappa(x)$. 
Since both $\rho_{\#}(\cdot,\cdot)$ and $\delta_E(\cdot)$ are continuous
(cf. Theorem~\ref{JjEGh}) it follows that $\Gamma_\kappa(x)$ is an open subset of 
$({\mathscr{X}},\tau_\rho)$, for each $x\in E$. Furthermore, it may be 
readily verified that 
\begin{eqnarray}\label{Tfs23}
{\mathscr{X}}\setminus\overline{E}=\bigcup\limits_{x\in E}\Gamma_\kappa(x),\qquad
\forall\,\kappa>0,
\end{eqnarray}
where $\overline{E}$ denotes the closure of $E$ in the topology $\tau_\rho$.

\begin{lemma}\label{semi-cont}
Let $({\mathscr{X}},\rho)$ be a quasi-metric space, $E$ a proper, nonempty, 
closed subset of $({\mathscr{X}},\tau_\rho)$, and $\mu$ a Borel measure 
on $({\mathscr{X}},\tau_\rho)$. Let $u:{\mathscr{X}}\setminus E\to[0,\infty]$ 
be a $\mu$-measurable function, fix $\kappa>0$ and recall the regions from 
\eqref{TLjb}. Then the function
\begin{eqnarray}\label{Mixed-8}
F:E\longrightarrow[0,\infty],\qquad
F(x):=\int_{\Gamma_\kappa(x)}u(y)\,d\mu(y),\quad\forall\,x\in E,
\end{eqnarray}
is lower semi-continuous (relative to the topology induced by $\tau_\rho$ on $E$).
\end{lemma}

\begin{proof}
Let $x_0\in E$ be arbitrary, fixed, and consider a sequence $\{x_j\}_{j\in{\mathbb{N}}}$
of points in $E$ with the property that 
\begin{eqnarray}\label{Mixed-A}
\lim\limits_{j\to\infty}\rho_{\#}(x_j,x_0)=0.
\end{eqnarray}
We claim that 
\begin{eqnarray}\label{Mixed-9}
\liminf\limits_{j\to\infty}{\mathbf{1}}_{\Gamma_\kappa(x_j)}(x)
\geq{\mathbf{1}}_{\Gamma_\kappa(x_0)}(x),\qquad\forall\,x\in{\mathscr{X}}\setminus E.
\end{eqnarray}
Clearly \eqref{Mixed-9} is true if $x\not\in\Gamma_\kappa(x_0)$. If 
${\mathbf{1}}_{\Gamma_\kappa(x_0)}(x)=1$, then $x\in\Gamma_\kappa(x_0)$, 
thus by definition $\rho_{\#}(x,x_0)<(1+\kappa)\delta_E(x)$. 
Based on the continuity of $\rho_{\#}(x,\cdot)$ and \eqref{Mixed-A}, 
it follows that there exists $j_0\in{\mathbb{N}}$ such that 
$\rho_{\#}(x,x_j)<(1+\kappa)\delta_E(x)$
for $j\geq j_0$. Hence, $x\in\Gamma_\kappa(x_j)$ for $j\geq j_0$ or, equivalently,
${\mathbf{1}}_{\Gamma_\kappa(x_j)}(x)=1$ for $j\geq j_0$. This completes 
the proof of the claim.

Returning to the actual task at hand, Fatou's lemma and \eqref{Mixed-9} then imply
\begin{eqnarray}\label{Mixed-10}
\liminf_{j\to\infty}F(x_j) 
&=& \liminf_{j\to\infty}\int_{{\mathscr{X}}\setminus E}
{\mathbf{1}}_{\Gamma_\kappa(x_j)}u\,d\mu
\geq\int_{{\mathscr{X}}\setminus E}
\liminf_{j\to\infty}\bigl({\mathbf{1}}_{\Gamma_\kappa(x_j)}u\bigr)\,d\mu
\nonumber\\[4pt]
&=&\int_{{\mathscr{X}}\setminus E}
\bigl(\liminf_{j\to\infty}{\mathbf{1}}_{\Gamma_\kappa(x_j)}\bigr)u\,d\mu
\geq\int_{{\mathscr{X}}\setminus E}{\mathbf{1}}_{\Gamma_\kappa(x_0)}u\,d\mu
\nonumber\\[4pt]
&=& F(x_0).
\end{eqnarray}
This shows that $F$ is lower semi-continuous.
\end{proof}

We retain the context of Lemma~\ref{semi-cont}. For each index $q\in(0,\infty)$ and 
constant $\kappa\in(0,\infty)$, define the $L^q$-based {\tt Lusin operator}, or 
{\tt area operator}, ${\mathscr{A}}_{q,\kappa}$ 
for all $\mu$-measurable functions 
$u:{\mathscr{X}}\setminus E\to\overline{\mathbb{R}}:=[-\infty,+\infty]$ by 
\begin{eqnarray}\label{sp-sq}
({\mathscr{A}}_{q,\kappa}u)(x):=\Bigl(\int_{\Gamma_\kappa(x)}|u(y)|^q\,d\mu(y)\Bigr)^{\frac{1}{q}},
\qquad\forall\,x\in E.
\end{eqnarray}
As a consequence of Lemma~\ref{semi-cont}, we have that ${\mathscr{A}}_{q,\kappa}u$ 
is lower semi-continuous, hence 
\begin{eqnarray}\label{Mixed-7A}
\bigl\{x\in E:\,({\mathscr{A}}_{q,\kappa}u)(x)>\lambda\bigr\}\quad
\mbox{ is an open subset of $(E,\tau_\rho)$ for each $\lambda>0$}.
\end{eqnarray}

To proceed, fix a Borel measure $\sigma$ on $(E,\tau_{\rho|_{E}})$. The above considerations then allow us to conclude that
\begin{eqnarray}\label{Mixed-3}
\begin{array}{c}
\mbox{for any $\mu$-measurable function 
$u:{\mathscr{X}}\setminus E\rightarrow\overline{\mathbb{R}}$},
\\[4pt]
\mbox{the mapping }\,{\mathscr{A}}_{q,\kappa}u:E\rightarrow[0,\infty]
\,\mbox{ is well-defined and $\sigma$-measurable.}
\end{array}
\end{eqnarray}
Consequently, given $\kappa>0$ and a pair of integrability indices $p,q$, 
following \cite{MMM} and \cite{BMMM} we may now introduce the {\tt mixed norm 
space of type} $(p,q)$, denoted $L^{(p,q)}({\mathscr{X}},E,\mu,\sigma;\kappa)$, 
or $L^{(p,q)}({\mathscr{X}},E)$ for short, in a meaningful manner as follows. 
If $q\in(0,\infty)$ and $p\in(0,\infty]$ we set 
\begin{eqnarray}\label{Mixed-FF7}
L^{(p,q)}({\mathscr{X}},E,\mu,\sigma;\kappa):=\Bigl\{
u:{\mathscr{X}}\setminus E\to\overline{\mathbb{R}}:\,u\,
\mbox{ $\mu$-measurable and }\,{\mathscr{A}}_{q,\kappa}u\in L^p(E,\sigma)\Bigr\},
\end{eqnarray}
equipped with the quasi-norm 
\begin{eqnarray}\label{Mixed-EEW}
\|u\|_{L^{(p,q)}({\mathscr{X}},E,\mu,\sigma;\kappa)}
:=\|{\mathscr{A}}_{q,\kappa}u\|_{L^p(E,\sigma)}
=\left\{
\begin{array}{l}
\Bigl(\int_{E}\Bigl[\int_{\Gamma_{\kappa}(x)}|u|^q\,d\mu\Bigr]^{p/q}d\sigma(x)
\Bigr)^{1/p}\,\mbox{ if }\,p<\infty,
\\[8pt]
\sigma\mbox{-}{\rm ess}\sup\limits_{x\in E}\,({\mathscr{A}}_{q,\kappa}u)(x)
\quad\mbox{ if }\,\,p=\infty.
\end{array}
\right.
\end{eqnarray}
Also, corresponding to $p\in(0,\infty)$ and $q=\infty$, we set
\begin{eqnarray}\label{Mixed-I}
L^{(p,\infty)}({\mathscr{X}},E,\mu,\sigma;\kappa):=\Bigl\{
u:{\mathscr{X}}\setminus E\to\overline{\mathbb{R}}:\,
\|{\mathcal{N}}_\kappa u\|_{L^{p}(E,\sigma)}<\infty\Bigr\},
\end{eqnarray}
where ${\mathcal{N}}_\kappa$ is the nontangential maximal operator defined by 
\begin{eqnarray}\label{Mixed-N}
({\mathcal{N}}_\kappa u)(x):=\sup_{y\in\Gamma_\kappa(x)}|u(y)|,\qquad
\forall\,x\in E,
\end{eqnarray}
and equip this space with the quasi-norm 
$\|u\|_{L^{(p,\infty)}({\mathscr{X}},E,\mu,\sigma;\kappa)}:=
\|{\mathcal{N}}_{\kappa}u\|_{L^p(E,\sigma)}$. 
Finally, corresponding to $p=q=\infty$, set 
\begin{eqnarray}\label{Mixed-IDF}
L^{(\infty,\infty)}({\mathscr{X}},E,\mu,\sigma;\kappa)
:=L^\infty({\mathscr{X}}\setminus E,\mu).
\end{eqnarray}

We note that the connection of our mixed norm spaces with the Coifman-Meyer-Stein 
tent spaces $T^p_q$ in ${\mathbb{R}}^{n+1}_+$ is as follows
\begin{eqnarray}\label{Phgf}
T^p_q=L^{(p,q)}\Bigl({\mathbb{R}}^{n+1},\partial{\mathbb{R}}^{n+1}_+,
{\mathbf{1}}_{{\mathbb{R}}^{n+1}_+}\frac{dx\,dt}{t^{n+1}},dx\Bigr),
\quad\mbox{for }\,\,p,q\in(0,\infty).
\end{eqnarray}
Thus, results for mixed normed spaces imply results for classical tent spaces.

The next goal is to clarify to what extent the quasi-norm
$\|\cdot\|_{L^{(p,q)}({\mathscr{X}},E,\mu,\sigma;\kappa)}$ depends on the 
parameter $\kappa>0$ associated with the nontangential approach 
regions $\Gamma_k$ defined in \eqref{TLjb} and utilized in \eqref{Mixed-EEW}, \eqref{Mixed-N}. 
This is done in Theorem~\ref{appert} below, but the proof requires a number of
preliminary results and definitions which we now present. 

To set the stage, for each $A\subseteq E$ and $\kappa>0$, define 
the {\tt fan} (or {\tt saw-tooth}) {\tt region} ${\mathcal{F}}_\kappa(A)$ {\tt above} $A$, and the {\tt tent region} ${\mathcal{T}}_\kappa(A)$ {\tt above} $A$, as
\begin{eqnarray}\label{reg-A1}
{\mathcal{F}}_\kappa(A):=\bigcup\limits_{x\in A}\Gamma_\kappa(x)
\quad\mbox{ and }\quad
{\mathcal{T}}_\kappa(A):=\bigl({\mathscr{X}}\setminus E\bigr)\setminus
\Bigl({\mathcal{F}}_\kappa(E\setminus A)\Bigr).
\end{eqnarray}
Also, for each point $y\in{\mathscr{X}}\setminus E$, define the 
``(reverse) conical projection" of $y$ onto $E$ by
\begin{eqnarray}\label{reg-A2}
\pi_y^\kappa:=\bigl\{x\in E:\,y\in\Gamma_\kappa(x)\bigr\}.
\end{eqnarray}

\begin{lemma}\label{T-LL.2}
Let $({\mathscr{X}},\rho)$ be a quasi-metric space, $E$ a proper, nonempty, 
closed subset of $({\mathscr{X}},\tau_\rho)$. For every $A\subseteq E$, denote by 
$\overline{A}$ and $A^\circ$, respectively, the closure and interior of 
$A$ in the topological space $(E,\tau_{\rho|_{E}})$. 
Then for each fixed $\kappa\in(0,\infty)$ the following properties hold.
\begin{enumerate}
\item[(i)] For each $A\subseteq E$ one has 
${\mathcal{F}}_\kappa(A)={\mathcal{F}}_\kappa(\overline{A})$ 
and ${\mathcal{T}}_\kappa(A^\circ)={\mathcal{T}}_\kappa(A)$. 
\item[(ii)] For each $A\subseteq E$ one has 
${\mathcal{T}}_\kappa(A)\subseteq {\mathcal{F}}_\kappa(A)$.
\item[(iii)] For each nonempty proper subset $A$ of $E$ one has
\begin{eqnarray}\label{3.2.64}
&& \hskip -0.40in
{\mathcal{T}}_\kappa(A)=\bigl\{x\in{\mathscr{X}}\setminus E:\,
{\rm dist}_{\rho_{\#}}(x,A)\leq (1+\kappa)^{-1}\,
{\rm dist}_{\rho_{\#}}(x,E\setminus A)\bigr\},
\\[4pt]
&& \hskip 0.60in
{\mathcal{T}}_\kappa(A)=\bigl\{y\in{\mathscr{X}}\setminus E:\,
\pi_y^\kappa\subseteq A\bigr\}.
\label{3.2.BN}
\end{eqnarray}
Moreover, for each nonempty subset $A$ of $E$ one has
\begin{eqnarray}\label{3.2.TTF}
{\mathcal{F}}_\kappa(A)=\bigl\{y\in{\mathscr{X}}\setminus E:\,
{\rm dist}_{\rho_{\#}}(y,A)<(1+\kappa)\,\delta_E(y)\bigr\}.
\end{eqnarray}
\item[(iv)] One has ${\mathcal{F}}_{\kappa}(E)=
{\mathcal{T}}_{\kappa}(E)={\mathscr{X}}\setminus E$. Also, for any family 
$(A_j)_{j\in J}$ of subsets of $E$, 
\begin{eqnarray}\label{Fv-fCC49}
\bigcup_{j\in J}{\mathcal{F}}_\kappa(A_j)
={\mathcal{F}}_\kappa\bigl(\cup_{j\in J}A_j\bigr),\qquad
\bigcap_{j\in J}{\mathcal{T}}_\kappa(A_j)
={\mathcal{T}}_\kappa\bigl(\cap_{j\in J}A_j\bigr),
\end{eqnarray}
and 
\begin{eqnarray}\label{Fv-fCC50}
A_1\subseteq A_2\subseteq E\,\,\Longrightarrow\,\,
{\mathcal{F}}_\kappa(A_1)\subseteq{\mathcal{F}}_\kappa(A_2)\,\,\,\mbox{ and }\,\,\,
{\mathcal{T}}_\kappa(A_1)\subseteq{\mathcal{T}}_\kappa(A_2).
\end{eqnarray}
\item[(v)] Given $A\subseteq E$, it follows that ${\mathcal{F}}_\kappa(A)$ 
is an open subset of $({\mathscr{X}},\tau_\rho)$, while ${\mathcal{T}}_\kappa(A)$ 
is a relatively closed subset of ${\mathscr{X}}\setminus E$ equipped with the 
topology induced by $\tau_\rho$ on this set. 
\item[(vi)] For each $y\in{\mathscr{X}}\setminus E$ it follows that
$\pi_y^\kappa$ is a relatively open set in the topology 
induced by $\tau_\rho$ on $E$. 
\item[(vii)] One has
\begin{eqnarray}\label{Fv-UU45}
B_{\rho_{\#}}\bigl(x,C_\rho^{-1}r\bigr)\setminus E\subseteq
{\mathcal{T}}_\kappa\bigl(E\cap B_{\rho_{\#}}(x,r)\bigr),\qquad
\forall\,r\in(0,\infty),\quad\forall\,x\in E.
\end{eqnarray}
\item[(viii)] Assume that $(E,\rho\bigl|_{E})$ is geometrically doubling.
Then for every $\kappa>0$ there exists a constant $C_o\in(0,\infty)$ with 
the property that if $\mathcal{O}$ is a nonempty, open, proper subset 
of $(E,\tau_{\rho|_{E}})$ and if $\{\Delta_j\}_{j\in J}$, where $x_j\in E$ 
and $\Delta_j:=E\cap B_\rho(x_j,r_j)$ for each $j\in J$, is a Whitney 
decomposition of $\mathcal{O}$ as in Proposition~\ref{H-S-Z}, then 
\begin{eqnarray}\label{3.2.63}
{\mathcal{T}}_{\kappa}(\mathcal{O})
\subseteq\bigcup\limits_{j\in J}B_\rho(x_j,C_or_j).
\end{eqnarray}
In particular, there exists $C\in(0,\infty)$ with the property that 
\begin{eqnarray}\label{3.2.63WS}
\begin{array}{c}
{\mathcal{T}}_{\kappa}\bigl(E\cap B_{\rho}(x,r)\bigr)
\subseteq B_\rho(x,C r)\setminus E\quad\mbox{ whenever}
\\[4pt]
\mbox{$x\in E$ and $r>0$ are such that 
$E\setminus B_{\rho}(x,r)\not=\emptyset$}.
\end{array}
\end{eqnarray}
\item[(ix)] In the case when $E$ is bounded, there exists $C\in(0,\infty)$ 
with the property that 
\begin{eqnarray}\label{3.UGH}
{\mathscr{X}}\setminus B_{\rho_{\#}}\bigl(x_0,C\,{\rm diam}_{\rho}(E)\bigr)\subseteq
\Gamma_\kappa(x),\qquad\forall\,x_0,x\in E.
\end{eqnarray}
Consequently, whenever $E$ is bounded there exists $C\in(0,\infty)$ such that
for each $x_0\in E$ one has
\begin{eqnarray}\label{3.UGH2}
{\mathcal{T}}_{\kappa}(A)\subseteq
B_{\rho_{\#}}\bigl(x_0,C\,{\rm diam}_{\rho}(E)\bigr),
\qquad\forall\,A\,\,\mbox{ proper subset of }\,\,E.
\end{eqnarray}
\end{enumerate}
\end{lemma}

\begin{proof}
With the exception of the first part of {\it (viii)}, these are direct 
consequences of definitions and the fact that both $\rho_{\#}(\cdot,\cdot)$ 
and $\delta_E(\cdot)$ are continuous functions. The remaining portion of the 
proof consists of a verification of \eqref{3.2.63}. To get started, let $x$ 
be an arbitrary point in ${\mathcal{T}}_{\kappa}(\mathcal{O})$. 
This places $x$ in ${\mathscr{X}}\setminus E$ which, given that $E$ 
is closed in $({\mathscr{X}},\tau_\rho)$, means that $x$ does not 
belong to $\overline{{\mathcal{O}}}\subseteq E$. In particular, 
${\rm dist}_{\rho_{\#}}(x,\mathcal{O})>0$. Going further, 
assume that some small $\varepsilon>0$ has been fixed. The above discussion 
then shows that it is possible to pick a point $y\in\mathcal{O}$ with the 
property that 
\begin{eqnarray}\label{3.2.65}
\rho_{\#}(x,y)<(1+\varepsilon)\,{\rm dist}_{\rho_{\#}}(x,\mathcal{O}).
\end{eqnarray}
Then there exists an index $j\in J$ for which $y\in\Delta_j$ and we shall 
show that $\varepsilon$ and $C_o$ can be chosen so as to guarantee that
\begin{eqnarray}\label{3.2.66}
x\in B_{\rho}(x_j,C_or_j).
\end{eqnarray}
Indeed, selecting a real number $\beta\in(0,(\log_2 C_\rho)^{-1}]$ and 
invoking \eqref{3.2.64} we may write
\begin{eqnarray}\label{3.2.67}
\bigl[\rho_{\#}(x,y)\bigr]^\beta 
& < & (1+\varepsilon)^\beta\bigl[{\rm dist}_{\rho_{\#}}(x,\mathcal{O})\bigr]^\beta 
\leq\Bigl(\frac{1+\varepsilon}{1+\kappa}\Bigr)^\beta
\bigl[{\rm dist}_{\rho_{\#}}(x,E\setminus{\mathcal{O}})\bigr]^\beta
\nonumber \\[4pt]
& = & \Bigl(\frac{1+\varepsilon}{1+\kappa}\Bigr)^\beta
{\rm dist}_{(\rho_{\#})^\beta}(x,E\setminus{\mathcal{O}})
\nonumber \\[4pt]
&\leq & \Bigl(\frac{1+\varepsilon}{1+\kappa}\Bigr)^\beta
\Bigl(\bigl[\rho_{\#}(x,y)\bigr]^\beta 
+{\rm dist}_{(\rho_{\#})^\beta}(y,E\setminus{\mathcal{O}})\Bigr)
\nonumber\\[4pt]
& \leq & \Bigl(\frac{1+\varepsilon}{1+\kappa}\Bigr)^\beta
\Bigl(\bigl[\rho_{\#}(x,y)\bigr]^\beta +Cr_j^\beta\Bigr),
\end{eqnarray}
where $C\in(0,\infty)$ depends only on the geometrically doubling 
character of $E$. The last step above uses Theorem~\ref{JjEGh} and the fact that 
$y$ belongs to $\Delta_j=B_\rho(x_j,r_j)\cap E$, which is a 
Whitney ball for ${\mathcal{O}}$. Choosing $\varepsilon=\kappa/2$, this now yields
(on account of the first inequality in \eqref{DEQV1})
\begin{eqnarray}\label{3.2.68}
\rho(x,y)\leq C_\rho^2\,\rho_{\#}(x,y)
<C_\rho^2 C^{1/\beta}\Bigl(\tfrac{1+\kappa/2}
{\bigl[(1+\kappa)^\beta-(1+\kappa/2)^\beta\bigr]^{1/\beta}}\Bigr)r_j
=:C_{\kappa,\beta}\,r_j.
\end{eqnarray}
Hence, since 
$\rho(x_j,x)\leq C_\rho\,\max\{\rho(x_j,y),\rho(y,x)\}
<C_\rho C_{\kappa,\beta}r_j$, the membership in \eqref{3.2.66} holds provided 
we take $C_o:=C_\rho C_{\kappa,\beta}$ to begin with. This finishes the proof of 
\eqref{3.2.63}. 
\end{proof}

\begin{lemma}\label{lbDV}
Let $({\mathscr{X}},\rho)$ be a quasi-metric space, $E$ a proper, nonempty, 
closed subset of $({\mathscr{X}},\tau_\rho)$, $\mu$ a Borel measure on 
$({\mathscr{X}},\tau_\rho)$ and $\sigma$ a Borel measure on $(E,\tau_{\rho|_{E}})$. 
Let $\rho_{\#}$ be associated with $\rho$ as in Theorem~\ref{JjEGh} and recall the 
constant $C_{\rho}\geq 1$ defined in \eqref{C-RHO.111}. Then for each 
real number $\kappa>0$ there holds
\begin{eqnarray}\label{sgbr}
E\cap B_{\rho_{\#}}\bigl(y_\ast,\epsilon\delta_E(y)\bigr)\subseteq\pi^\kappa_y
\subseteq E\cap B_{\rho_{\#}}\bigl(y_\ast,C_{\rho}(1+\kappa)\delta_E(y)\bigr),
\qquad\forall\,y\in{\mathscr{X}}\setminus E,
\end{eqnarray}
where the point $y_\ast$ and the number $\epsilon$ satisfy
\begin{eqnarray}\label{Equ-1}
\begin{array}{l}
\mbox{$y_\ast\in E$ and $\rho_{\#}(y,y_\ast)<(1+\eta)\delta_E(y)$
for some $\eta\in(0,\kappa)$, and }
\\[4pt]
0<\epsilon<\bigl[(1+\kappa)^\beta-(1+\eta)^\beta\bigr]^{1/\beta}
\,\,\mbox{ for some finite $\beta\in(0,(\log_2 C_\rho)^{-1}]$.}
\end{array}
\end{eqnarray}
 
In particular, if $(E,\rho\bigl|_E,\sigma)$ is a space of homogeneous type and
if $\kappa,\kappa'>0$ are two arbitrary real numbers, then  
\begin{eqnarray}\label{equiv11}
c_o^{-1}\,\sigma(\pi^{\kappa}_y)\leq\sigma(\pi^{\kappa'}_y)
\leq c_o\,\sigma(\pi^{\kappa}_y),\qquad\forall\,y\in{\mathscr{X}}\setminus E,
\end{eqnarray}
where $c_o:=C_{\sigma}(C^2_\rho/\epsilon)^{D_\sigma}
(1+\min\{\kappa,\kappa'\})^{D_\sigma}$, with $C_{\sigma}$ and $D_\sigma$ 
the doubling constant and doubling order of $\sigma$.
\end{lemma}

\begin{proof}
Fix an arbitrary point $y\in{\mathscr{X}}\setminus E$ and let 
$y_\ast\in E$ and $\epsilon>0$ be as in \eqref{Equ-1}. 
If $x\in E\cap B_{\rho_{\#}}(y_\ast,\epsilon\delta_E(y))$ then 
$\rho_{\#}(y_\ast,x)<\epsilon\delta_E(y)$ forcing (recall from Theorem~\ref{JjEGh}
that $(\rho_{\#})^\beta$ is a genuine distance)
\begin{eqnarray}\label{H+gz-2}
\rho_{\#}(x,y)^\beta\leq\rho_{\#}(x,y_\ast)^\beta+\rho_{\#}(y_\ast,y)^\beta
<\epsilon^\beta\delta_E(y)^\beta+(1+\eta)^\beta\delta_E(y)^\beta
<(1+\kappa)^\beta\delta_E(y)^\beta.
\end{eqnarray}
Thus $x\in\pi^\kappa_y$, which proves the first inclusion in \eqref{sgbr}. 
Going further, given a point $x\in\pi^\kappa_y$ it follows that
$\rho_{\#}(x,y)<(1+\kappa)\delta_E(y)$, hence 
\begin{eqnarray}\label{H+gz}
\rho_{\#}(x,y_\ast)\leq C_{\rho_{\#}}\max\{\rho_{\#}(x,y),\rho_{\#}(y,y_\ast)\}
<C_{\rho_{\#}}(1+\kappa)\delta_E(y)\leq C_{\rho}(1+\kappa)\delta_E(y), 
\end{eqnarray}
proving the second inclusion in \eqref{sgbr}. 

Suppose now that $(E,\rho\bigl|_E,\sigma)$ is a space of homogeneous type and 
let $\kappa,\kappa'>0$ be given. Assume first that $\kappa\leq\kappa'$. 
Choose $y_\ast$ and $\epsilon$ as in \eqref{Equ-1}. Then \eqref{sgbr} holds 
both as written and with $\kappa$ replaced by $\kappa'$. When combined 
with \eqref{Doub-2}, this yields 
\begin{eqnarray}\label{eqDc}
c_1^{-1}\,\sigma(\pi^{\kappa}_y)\leq\sigma(\pi^{\kappa'}_y)
\leq c_1\,\sigma(\pi^{\kappa}_y),\qquad\forall\,y\in{\mathscr{X}}\setminus E,
\end{eqnarray}
where 
$c_1:=C_{\sigma,\rho_{\#}}\Bigl(\frac{C_\rho(1+\kappa)}{\epsilon}\Bigr)^{D_\sigma}$
with $C_{\sigma,\rho_{\#}}$ and $D_\sigma$ being the constants associated with 
$\sigma$ and $\rho_{\#}$ as in \eqref{Doub-2}. In particular, since 
$C_{\sigma,\rho_{\#}}=C_\sigma(C_{\rho_{\#}}\widetilde{C}_{\rho_{\#}})^{D_\sigma}
\leq C_\sigma(C_{\rho})^{D_\sigma}$, it follows that 
$c_1\leq C_{\sigma}(C^2_\rho/\epsilon)^{D_\sigma}(1+\kappa)^{D_\sigma}$.
If $\kappa'<\kappa$ the same reasoning yields inequalities similar 
to \eqref{eqDc}, this time with $c_1$ replaced by the constant 
$c_2:=C_{\sigma,\rho_{\#}}\Bigl(\frac{C_\rho(1+\kappa')}{\epsilon}\Bigr)^{D_\sigma}
\leq C_{\sigma}(C^2_\rho/\epsilon)^{D_\sigma}(1+\kappa')^{D_\sigma}$.
All these now immediately yield \eqref{equiv11}.
\end{proof}

Moving on, assume now that $(E,\rho,\sigma)$ is a space of homogeneous type and 
let $\rho_{\#}$ be associated with $\rho$ as in Theorem~\ref{JjEGh} in this context. 
Then for each $\sigma$-measurable set $A\subseteq E$ and each $\gamma\in(0,1)$, define 
the set of $\gamma$-{\tt density points}, relative to $A$, as
\begin{eqnarray}\label{Mixed-12}
A^\ast_\gamma:=\Bigl\{x\in E:\,\inf\limits_{r>0}
\Bigl[\frac{\sigma\bigl(B_{\rho_{\#}}(x,r)\cap A\bigr)}
{\sigma\bigl(B_{\rho_{\#}}(x,r)\bigr)}\Bigr]\geq\gamma\Bigr\}.
\end{eqnarray}
In particular, from this definition it follows that 
\begin{eqnarray}\label{Mixed-13}
\inf\limits_{x\in A^\ast_\gamma}\,\Bigl[\,
\inf\limits_{r>0}\frac{\sigma\bigl(B_{\rho_{\#}}(x,r)\cap A\bigr)}
{\sigma\bigl(B_{\rho_{\#}}(x,r)\bigr)}\Bigr]\geq\gamma.
\end{eqnarray}

Some basic properties of the sets of density points in the setting of 
spaces of homogeneous type are collected below. 

\begin{proposition}\label{DJrt}
Let $(E,\rho,\sigma)$ be a space of homogeneous type, $\rho_{\#}$ the 
regularization of $\rho$ as in Theorem~\ref{JjEGh}, $\gamma\in(0,1)$ and 
$A\subseteq E$ a $\sigma$-measurable set. Then the following properties hold:
\begin{enumerate}
\item[(1)] $E\setminus A^\ast_\gamma=\bigl\{x\in E:\,
M_{E}\bigl({\mathbf{1}}_{E\setminus A}\bigr)(x)>1-\gamma\bigr\}$, where $M_E$ 
is the Hardy-Littlewood maximal operator on $E$ (cf. \eqref{HL-MAX}).
\item[(2)] $A^\ast_\gamma$ is closed subset of $(E,\tau_\rho)$.
\item[(3)] $\sigma\bigl(E\setminus A^\ast_\gamma\bigr)
\leq\frac{C}{1-\gamma}\sigma(E\setminus A)$.
\item[(4)] If $A$ is closed (in $\tau_\rho$), then $A^\ast_\gamma\subseteq A$. 
In particular, in this case, 
$\sigma(E\setminus A^\ast_\gamma)\approx\sigma(E\setminus A)$.
\item[(5)] For each $\lambda>0$ there exist $\gamma(\lambda)\in(0,1)$ and $c(\lambda)>0$
such that if $\gamma(\lambda)\leq\gamma<1$ then 
\begin{eqnarray}\label{Mixed-14}
\inf\limits_{x\in E}\,\Bigl[\,
\inf\limits_{r>{\rm dist}_{\rho_{\#}}(x,A^\ast_\gamma)}
\frac{\sigma\bigl(B_{\rho_{\#}}(x,\lambda r)\cap A\bigr)}
{\sigma\bigl(B_{\rho_{\#}}(x,r)\bigr)}\Bigr]\geq c(\lambda).
\end{eqnarray}
\item[(6)] If the measure $\sigma$ is Borel regular, 
then $\sigma(A^\ast_\gamma\setminus A)=0$.
\item[(7)] If $\widetilde{A}$ is $\sigma$-measurable set 
such that $A\subseteq\widetilde{A}\subseteq E$, then 
$A^\ast_\gamma\subseteq(\widetilde{A})^\ast_\gamma$.
\end{enumerate}
\end{proposition}

The remarkable aspect of {\it (3)-(4)} above is that whenever $A$ is a closed 
subset of $(E,\tau_\rho)$ then in a measure-theoretic sense the size of both sets, $A^\ast_\gamma$ and $E\setminus A^\ast_\gamma$, may be controlled in terms of 
sets $A$ and $E\setminus A$, respectively (as opposed to point-set theory). 
The typical application of Proposition~\ref{DJrt} is in estimating the measure 
of a $\sigma$-measurable set $F\subseteq E$ by writing 
\begin{eqnarray}\label{Irh}
\sigma(F)=\sigma(F\cap A^\ast_\gamma)
+\sigma\bigl(F\cap(E\setminus A^\ast_\gamma)\bigr)
\leq \sigma(F\cap A^\ast_\gamma)+\frac{C}{1-\gamma}\sigma(E\setminus A).
\end{eqnarray}

\vskip 0.08in
\begin{proof}[Proof of Proposition~\ref{DJrt}]
Starting with \eqref{Mixed-12} we may write 
\begin{eqnarray}\label{Mixed-15}
E\setminus A^\ast_\gamma 
&=& \Bigl\{x\in E:\,\exists\,r>0\mbox{ such that }
\frac{\sigma\bigl(B_{\rho_{\#}}(x,r)\cap A\bigr)}
{\sigma\bigl(B_{\rho_{\#}}(x,r)\bigr)}<\gamma\Bigr\}
\nonumber\\[4pt]
&=& \Bigl\{x\in E:\,\exists\,r>0\mbox{ such that }
\frac{\sigma\bigl(B_{\rho_{\#}}(x,r)\cap (E\setminus A)\bigr)}
{\sigma\bigl(B_{\rho_{\#}}(x,r)\bigr)}>1-\gamma\Bigr\}
\nonumber\\[4pt]
&=& \Bigl\{x\in E:\,\sup\limits_{r>0}\Bigl(
\meanint_{B_{\rho_{\#}}(x,r)}{\mathbf{1}}_{E\setminus A}\,d\sigma\Bigr)>1-\gamma\Bigr\}
\nonumber\\[4pt]
&=& \Bigl\{x\in E:\,\sup\limits_{0<r\leq{\rm diam}\,_{\rho_{\#}}(E)}\Bigl(
\meanint_{B_{\rho_{\#}}(x,r)}{\mathbf{1}}_{E\setminus A}\,d\sigma\Bigr)>1-\gamma\Bigr\}
\nonumber\\[4pt]
&=& \Bigl\{x\in E:\,
M_{E}\bigl({\mathbf{1}}_{E\setminus A}\bigr)(x)>1-\gamma\Bigr\},
\end{eqnarray}
proving {\it (1)}. We now make the claim that 
\begin{eqnarray}\label{Mixed-SCC}
\mbox{the function
$M_{E}\bigl({\mathbf{1}}_{E\setminus A}\bigr):(E,\tau_\rho)\rightarrow[0,\infty]$ is lower semi-continuous}.
\end{eqnarray}
To prove this claim, we note that since the pointwise supremum of an arbitrary 
family of real-valued, lower semi-continuous functions defined on $E$ 
is itself lower semi-continuous, it suffices to show that 
\begin{eqnarray}\label{MEAS-11.aPP}
\begin{array}{l}
\mbox{for every $\sigma$-measurable set $F\subseteq E$, the function }\,\,
f:(E,\tau_\rho)\to[0,\infty)
\\[4pt]
\mbox{given by }\,f(x):=\sigma\bigl(B_{\rho_{\#}}(x,r)\cap F\bigr)\quad
\forall\,x\in E,\,\mbox{ is lower semi-continuous}.
\end{array}
\end{eqnarray}
To this end, fix $x_o\in E$ arbitrary. The crux of the matter is the fact that
our choice of the quasi-distance ensures that if $\{x_j\}_{j\in{\mathbb{N}}}$ 
is a sequence of points in $E$ with the property that $x_j\to x_o$ as
$j\to\infty$, with convergence understood in the (metrizable) 
topology $\tau_\rho$, then 
\begin{eqnarray}\label{MEAS-12.aPP}
\liminf_{j\to\infty}{\mathbf 1}_{B_{\rho_{\#}}(x_j,r)}(y)
\geq{\mathbf 1}_{B_{\rho_{\#}}(x_o,r)}(y),\qquad\forall\,y\in E,
\end{eqnarray}
as is easily verified by analyzing the cases $y\in B_{\rho_{\#}}(x_o,r)$ 
and $y\in E\setminus B_{\rho_{\#}}(x_o,r)$. In turn, based on this and 
Fatou's lemma we may then estimate 
\begin{eqnarray}\label{MEAS-13.aPP}
f(x_o) &=& \sigma\bigl(B_{\rho_{\#}}(x_o,r)\cap F\bigr)
=\int_{F}{\mathbf 1}_{B_{\rho_{\#}}(x_o,r)}(y)\,d\sigma(y)
\nonumber\\[4pt]
&\leq & \int_{F}\liminf_{j\to\infty}{\mathbf 1}_{B_{\rho_{\#}}(x_j,r)}(y)\,d\sigma(y)
\leq\liminf_{j\to\infty}\int_{F}{\mathbf 1}_{B_{\rho_{\#}}(x_j,r)}(y)\,d\sigma(y)
\nonumber\\[4pt]
&=& \liminf_{j\to\infty}\sigma\bigl(B_{\rho_{\#}}(x_j,r)\cap F\bigr)
=\liminf_{j\to\infty}f(x_j).
\end{eqnarray}
This establishes \eqref{MEAS-11.aPP}, thus finishing the proof of \eqref{Mixed-SCC}.

Moving on, \eqref{Mixed-SCC} implies that the last set in \eqref{Mixed-15} is open
(in $\tau_\rho$), hence {\it (2)} holds true. Also, by combining {\it (1)} with 
the weak-$(1,1)$ boundedness of $M_{E}$ (recall that we are assuming that
$(E,\rho,\sigma)$ is a space of homogeneous type), we obtain
\begin{eqnarray}\label{Mixed-15B}
\sigma\bigl(E\setminus A^\ast_\gamma\bigr)
\leq \frac{C}{1-\gamma}\|{\mathbf{1}}_{E\setminus A}\|_{L^1(E,\sigma)}
=\frac{C}{1-\gamma}\sigma(E\setminus A).
\end{eqnarray}
Hence the inequality in {\it (3)} is proved.

Suppose now that $A$ is a closed subset of $(E,\tau_\rho)$. Then $E\setminus A$ is 
open, so if $x\in E\setminus A$ then there exists $r>0$ such that 
$B_{\rho_{\#}}(x,r)\subseteq E\setminus A$. Consequently, 
$\frac{\sigma\bigl(B_{\rho_{\#}}(x,r)\cap A)\bigr)}
{\sigma\bigl(B_{\rho_{\#}}(x,r)\bigr)}=0<\gamma$, thus $x\not\in A^\ast_\gamma$.
This shows that $A^\ast_\gamma\subseteq A$, hence 
$\sigma(E\setminus A)\leq \sigma(E\setminus A^\ast_\gamma)$. Combining these with 
what we proved in {\it (3)} finishes the proof of {\it (4)}. 

Turning to the proof of {\it (5)}, fix some $\lambda>0$ and $x\in E$, arbitrary, 
and select $r>0$ such that 
\begin{eqnarray}\label{Mixed-16}
{\rm dist}_{\rho_{\#}}(x,A^\ast_\gamma)<r.
\end{eqnarray}
Then there exists $x_0\in A^\ast_\gamma$ such that $\rho_{\#}(x,x_0)<r$, which forces 
\begin{eqnarray}\label{Mixed-17}
B_{\rho_{\#}}(x,\lambda r)\subseteq B_{\rho_{\#}}(x_0,C_{\rho_{\#}}(1+\lambda)r)
\subseteq B_{\rho_{\#}}(x,C_{\rho_{\#}}^2(1+\lambda)r).
\end{eqnarray}
Consequently, since $x_0\in A^\ast_\gamma$ we obtain 
\begin{eqnarray}\label{Mixed-18}
&&\hskip -0.30in
\gamma\sigma\bigl(B_{\rho_{\#}}(x_0,C_{\rho_{\#}}(1+\lambda)r)\bigr)
\leq\sigma\bigl(B_{\rho_{\#}}(x_0,C_{\rho_{\#}}(1+\lambda)r)\cap A\bigr)
\nonumber\\[4pt]
&&\hskip 0.30in
\leq\sigma\bigl(B_{\rho_{\#}}(x_0,C_{\rho_{\#}}(1+\lambda)r)\setminus 
B_{\rho_{\#}}(x,\lambda r)\bigr)+\sigma\bigl(B_{\rho_{\#}}(x,\lambda r)\cap A\bigr)
\nonumber\\[4pt]
&&\hskip 0.30in
=\sigma\bigl(B_{\rho_{\#}}(x_0,C_{\rho_{\#}}(1+\lambda)r)\bigr)
-\sigma\bigl(B_{\rho_{\#}}(x,\lambda r)\bigr)
+\sigma\bigl(B_{\rho_{\#}}(x,\lambda r)\cap A\bigr),
\end{eqnarray}
which further implies that
\begin{eqnarray}\label{Mixed-19}
\sigma\bigl(B_{\rho_{\#}}(x,\lambda r)\bigr)
-(1-\gamma)\sigma\bigl(B_{\rho_{\#}}(x_0,C_{\rho_{\#}}(1+\lambda)r)\bigr)
\leq\sigma\bigl(B_{\rho_{\#}}(x,\lambda r)\cap A\bigr).
\end{eqnarray}
Recalling the second inclusion in \eqref{Mixed-17} and \eqref{Doub-2}, we obtain
\begin{eqnarray}\label{Mixed-20}
\sigma\bigl(B_{\rho_{\#}}(x_0,C_{\rho_{\#}}(1+\lambda)r)\bigr)
&\leq & \sigma\bigl(B_{\rho_{\#}}(x_0,C_{\rho_{\#}}^2(1+\lambda)r)\bigr)
\nonumber\\[4pt]
&\leq & C_{\sigma,\rho_{\#}}\bigl(\tfrac{C_{\rho_{\#}}^2(1+\lambda)}
{\lambda}\bigr)^{D_\sigma}\sigma\bigl(B_{\rho_{\#}}(x,\lambda r)\bigr),
\end{eqnarray}
where $C_{\sigma,\rho_{\#}}$, $D_\sigma$ are associated with $\sigma$, $\rho_{\#}$ 
as in \eqref{Doub-2}. Together, \eqref{Mixed-19} and \eqref{Mixed-20} yield
\begin{eqnarray}\label{Mixed-21}
\sigma\bigl(B_{\rho_{\#}}(x,\lambda r)\bigr)
\Bigl[1-C_{\sigma,\rho_{\#}}(1-\gamma)
\Bigl(\tfrac{C_{\rho_{\#}}^2(1+\lambda)}{\lambda}\Bigr)^{D_{\sigma}}\Bigr]
\leq\sigma\bigl(B_{\rho_{\#}}(x,\lambda r)\cap A\bigr).
\end{eqnarray}
Also, from \eqref{Doub-2} we have that if $\lambda\in(0,1)$ then 
$\sigma\bigl(B_{\rho_{\#}}(x,r)\bigr)\leq C_{\sigma,\rho_{\#}}\lambda^{-D_{\sigma}}
\sigma\bigl(B_{\rho_{\#}}(x,\lambda r)\bigr)$, thus
\begin{eqnarray}\label{Mixed-22}
\sigma\bigl(B_{\rho_{\#}}(x,\lambda r)\bigr)
\geq\min\Bigl\{1,\frac{\lambda^{D_{\sigma}}}{C_{\sigma,\rho_{\#}}}\Bigr\}
\sigma\bigl(B_{\rho_{\#}}(x,r)\bigr),\qquad\forall\,\lambda>0.
\end{eqnarray}
If we now we choose 
\begin{eqnarray}\label{Mixed-22EE}
\gamma(\lambda):=1-\frac{1}{2C_{\sigma,\rho_{\#}}}
\Bigl(\frac{\lambda}{C_{\rho_{\#}}^2(1+\lambda)}\Bigr)^{D_{\sigma}}\in(0,1)
\quad\mbox{ and }\quad
c(\lambda):=\frac{1}{2}\min\Bigl\{1,\frac{\lambda^{D_{\sigma}}}
{C_{\sigma,\rho_{\#}}}\Bigr\}>0,
\end{eqnarray}
then \eqref{Mixed-21} and \eqref{Mixed-22} imply
\begin{eqnarray}\label{Mixed-23}
\sigma\bigl(B_{\rho_{\#}}(x,\lambda r)\cap A\bigr)
\geq\frac{1}{2}\sigma\bigl(B_{\rho_{\#}}(x,\lambda r)\bigr)
\geq c(\lambda)\sigma\bigl(B_{\rho_{\#}}(x,r)\bigr),
\qquad\forall\,\gamma\in\bigl[\gamma(\lambda),1\bigr).
\end{eqnarray}
This proves {\it (5)}. 

If $\sigma$ is Borel-regular, then Lebesgue's Differentiation Theorem holds 
in the current setting. Hence, there exists a set $F\subseteq E$ with 
$\sigma(F)=0$ and such that
\begin{eqnarray}\label{Mixed-24}
\lim\limits_{r\to 0^+}\Bigl(\meanint_{B_{\rho_{\#}}(x,r)}{\mathbf{1}}_A\,d\sigma\Bigr)
={\mathbf{1}}_A(x),\qquad\forall\,x\in E\setminus F.
\end{eqnarray}
In particular, for every $x\in A^\ast_\gamma\setminus F$ we have 
${\mathbf{1}}_A(x)=\lim\limits_{r\to 0^+}\Bigl[
\frac{\sigma\bigl(B_{\rho_{\#}}(x,r)\cap A\bigr)}
{\sigma\bigl(B_{\rho_{\#}}(x,r)\bigr)}\Bigr]\geq\gamma>0$, which implies that
$A^\ast_\gamma\setminus F\subseteq A$, thus $A^\ast_\gamma\setminus A\subseteq F$.
Consequently, since $A^\ast_\gamma\setminus A$ is $\sigma$-measurable, we obtain that
$\sigma(A^\ast_\gamma\setminus A)=0$, proving {\it (6)}. Finally, the statement in 
{\it (7)} is an immediate consequence of \eqref{Mixed-12}. This concludes the proof
of the proposition.
\end{proof}

We continue to state and prove auxiliary lemmas
in preparation for dealing with Theorem~\ref{appert}, advertised earlier. 
To state the lemma below, recall the region 
${\mathcal{F}}_\kappa(A)$ from \eqref{reg-A1}.

\begin{lemma}\label{ap+YH}
Let $({\mathscr{X}},\rho)$ be a quasi-metric space, $\mu$ a Borel measure 
on $({\mathscr{X}},\tau_\rho)$, $E$ a proper, nonempty, closed subset of 
$({\mathscr{X}},\tau_\rho)$ and $\sigma$ a Borel measure on 
$(E,\tau_{\rho|_{E}})$ such that $(E,\rho\bigl|_E,\sigma)$ is a space of 
homogeneous type. If $u:{\mathscr{X}}\setminus E\to[0,\infty]$ is 
$\mu$-measurable, then for every $\kappa>0$ and every $\sigma$-measurable 
set $A\subseteq E$, one has 
\begin{eqnarray}\label{Mix+FR}
\int_A\Bigl(\int_{\Gamma_\kappa(x)}u(y)\,d\mu(y)\Bigr)\,d\sigma(x)
&=& \int_{{\mathscr{X}}\setminus E}u(y)\sigma\bigl(A\cap\pi_y^\kappa\bigr)\,d\mu(y)
\nonumber\\[4pt]
&=& \int_{{\mathcal{F}}_\kappa(A)}u(y)\sigma\bigl(A\cap\pi_y^\kappa\bigr)\,d\mu(y).
\end{eqnarray}
\end{lemma}

\begin{proof}
By Fubini's Theorem (and \eqref{Mixed-3}), we have
\begin{eqnarray}\label{Mix+FR-1}
\int_A\Bigl(\int_{\Gamma_\kappa(x)}u(y)\,d\mu(y)\Bigr)\,d\sigma(x)
&=& \int_{{\mathscr{X}}\setminus E}u(y)\Bigl(
\int_A{\mathbf{1}}_{\pi^\kappa_y}(x)\,d\sigma(x)\Bigr)\,d\mu(y)
\nonumber\\[4pt]
&=& \int_{{\mathscr{X}}\setminus E}u(y)\sigma\bigl(A\cap\pi_y^\kappa\bigr)\,d\mu(y),
\end{eqnarray}
proving the first equality in \eqref{Mix+FR}. The second equality in \eqref{Mix+FR}
follows from \eqref{Mix+FR-1} and the fact that if $y\in{\mathscr{X}}\setminus E$
and $A\cap\pi^\kappa_y\not=\emptyset$ then $y\in{\mathcal{F}}_\kappa(A)$.
\end{proof}

\begin{lemma}\label{ap+YH-2}
Let $({\mathscr{X}},\rho)$ be a quasi-metric space, $\mu$ a Borel measure on $({\mathscr{X}},\tau_\rho)$, $E$ a proper, nonempty, closed subset of 
$({\mathscr{X}},\tau_\rho)$, and $\sigma$ a Borel measure on $(E,\tau_{\rho|_{E}})$ 
such that $(E,\rho\bigl|_E,\sigma)$ is a space of homogeneous type. 
Fix two arbitrary numbers $\kappa,\kappa'>0$. Then there 
exist $\gamma\in(0,1)$ and a finite constant $C>0$ such that for every 
$\sigma$-measurable set $A\subseteq E$ there holds
\begin{eqnarray}\label{Mi+LV}
\int_{A^\ast_\gamma}\Bigl(\int_{\Gamma_\kappa(x)}u(y)\,d\mu(y)\Bigr)\,d\sigma(x)
\leq C\int_A\Bigl(\int_{\Gamma_{\kappa'}(x)}u(y)\,d\mu(y)\Bigr)\,d\sigma(x)
\end{eqnarray}
for every function $u:{\mathscr{X}}\setminus E\to[0,\infty]$ which is 
$\mu$-measurable.
\end{lemma}

\begin{proof}
Recall the notation introduced in \eqref{REG-DDD}. We claim that
\begin{eqnarray}\label{Clsb}
\begin{array}{c}
\mbox{for every }\,\,\kappa,\kappa'>0\quad\mbox{there exist}\,\,\gamma\in(0,1)\,
\mbox{ and }\,\,c>0\,\,\mbox{ such that }
\\[4pt]
\sigma\bigl(A\cap\pi_y^{\kappa'}\bigr)\geq c\,\sigma\bigl(A^\ast_\gamma\cap\pi_y^\kappa\bigr)
\quad\forall\,A\subseteq E\mbox{ $\sigma$-measurable and }
\forall\,y\in{\mathcal{F}}_\kappa(A^\ast_\gamma).
\end{array}
\end{eqnarray}
Assuming this claim for now, let $\kappa,\kappa'>0$ be arbitrary and let $\gamma$ 
and $c>0$ be as in \eqref{Clsb}. Then, if $A$ and $u$ satisfy the hypotheses 
of the proposition, starting with \eqref{Mix+FR} and using the fact that 
${\mathcal{F}}_\kappa(A^\ast_\gamma)\subseteq{\mathscr{X}}\setminus E$
(itself a trivial consequence of \eqref{Tfs23}), we may write
\begin{eqnarray}\label{Mix+FR-5}
\int_A\Bigl(\int_{\Gamma_{\kappa'}(x)}u(y)\,d\mu(y)\Bigr)\,d\sigma(x)
&=& \int_{{\mathscr{X}}\setminus E}u(y)\sigma\bigl(A\cap\pi_y^{\kappa'}\bigr)\,d\mu(y)
\nonumber\\[4pt]
&\geq& \int_{{\mathcal{F}}_\kappa(A^\ast_\gamma)}u(y)
\sigma\bigl(A\cap\pi_y^{\kappa'}\bigr)\,d\mu(y)
\nonumber\\[4pt]
&\geq& c\int_{{\mathcal{F}}_\kappa(A^\ast_\gamma)}u(y)
\sigma\bigl(A^\ast_\gamma\cap\pi_y^\kappa\bigr)\,d\mu(y)
\nonumber\\[4pt]
&=& c\int_{A^\ast_\gamma}\Bigl(\int_{\Gamma_\kappa(x)}u(y)\,d\mu(y)\Bigr)\,d\sigma(x),
\end{eqnarray}
where for the last equality in \eqref{Mix+FR-5} we applied Lemma~\ref{ap+YH}
with $A^\ast_\gamma$ in place of $A$. Hence, to finish the proof of the proposition
we are left with showing \eqref{Clsb}.

Suppose $\kappa,\kappa'>0$ are fixed and pick some $\gamma\in(0,1)$, to be 
made precise later. Also, fix $\eta\in(0,\min\,\{\kappa,\kappa'\})$ and
for each $y\in{\mathcal{F}}_\kappa(A^\ast_\gamma)$ choose $y_\ast\in E$ and 
$\epsilon>0$ as in \eqref{Equ-1} (for $\eta$ as just indicated). Then $y_\ast$
satisfies the conditions in \eqref{Equ-1} corresponding to both $\kappa$ and $\kappa'$.
As such, Lemma~\ref{lbDV} implies that the inclusions in \eqref{sgbr} hold for 
both $\kappa$ and $\kappa'$. 

The fact that $y\in{\mathcal{F}}_\kappa(A^\ast_\gamma)$ entails 
$\pi^\kappa_y\cap A^\ast_\gamma\not=\emptyset$ which, when combined with \eqref{sgbr},
implies $B_{\rho_{\#}}\bigl(y_\ast,C_{\rho}(1+\kappa)\delta_E(y)\bigr)\cap
A^\ast_\gamma\not=\emptyset$ hence, further, 
${\rm dist}_{\rho_{\#}}(y_\ast,A^\ast_\gamma)<C_{\rho}(1+\kappa)\delta_E(y)$. 
Now, {\it (5)} in Proposition~\ref{DJrt} invoked with
$\lambda:=\frac{\epsilon}{C_{\rho}(1+\kappa)}$, $x:=y_\ast$
and $r:=C_{\rho}(1+\kappa)\delta_E(y)$, guarantees the existence of some $\gamma_0=\gamma_0(\lambda)\in(0,1)$ with the property that
\begin{eqnarray}\label{sgbr-2}
\frac{\sigma\bigl(B_{\rho_{\#}}(y_\ast,\epsilon\delta_E(y))\cap A\bigr)}
{\sigma\bigl(B_{\rho_{\#}}(y_\ast,C_{\rho}(1+\kappa)\delta_E(y))\bigr)}
\geq c=c(\kappa)>0\quad\mbox{ if }\gamma\in(\gamma_0,1).
\end{eqnarray}
Hence, if we select $\gamma\in(\gamma_0,1)$ to begin with, 
the estimate in \eqref{sgbr-2} in concert with \eqref{sgbr} implies
\begin{eqnarray}\label{sgbr-3}
\sigma\bigl(B_{\rho_{\#}}(y_\ast,\epsilon\delta_E(y))\cap A\bigr)
\geq c\,\sigma\bigl(B_{\rho_{\#}}(y_\ast,C_{\rho}(1+\kappa)\delta_E(y))\bigr)
\geq c\,\sigma\bigl(\pi^\kappa_y\bigr)
\geq c\,\sigma\bigl(A^\ast_\gamma\cap \pi^\kappa_y\bigr).
\end{eqnarray}
Since \eqref{sgbr} also holds with $\kappa$ replaced by $\kappa'$, we obtain
from this and \eqref{sgbr-3} that
\begin{eqnarray}\label{sgbr-5}
\sigma\bigl(A\cap \pi^{\kappa'}_y\bigr)
\geq \sigma\bigl(B_{\rho_{\#}}(y_\ast,\epsilon\delta_E(y))\cap A\bigr)
\geq c\,\sigma\bigl(A^\ast_\gamma\cap \pi^\kappa_y\bigr).
\end{eqnarray}
This completes the proof of \eqref{Clsb} and, with it, the proof of the lemma.
\end{proof}

\begin{lemma}\label{Biu+F}
Let $({\mathscr{X}},\rho)$ be a quasi-metric space, $\mu$ a Borel measure on $({\mathscr{X}},\tau_\rho)$, $E$ a proper, nonempty, closed subset of 
$({\mathscr{X}},\tau_\rho)$, and $\sigma$ a Borel measure on 
$(E,\tau_{\rho|_{E}})$ such that $(E,\rho\bigl|_E,\sigma)$ is a space of 
homogeneous type. Then for every $\kappa,\kappa'>0$ there exists a constant 
$C\in(0,\infty)$ such that
\begin{eqnarray}\label{Biu+F2}
\int_E\Bigl(\int_{\Gamma_\kappa(x)}u(y)\,d\mu(y)\Bigr)f(x)\,d\sigma(x)
\leq C\int_E\Bigl(\int_{\Gamma_{\kappa'}(x)}u(y)\,d\mu(y)\Bigr)(M_Ef)(x)\,d\sigma(x)
\end{eqnarray}
for every function $u:{\mathscr{X}}\setminus E\to[0,\infty]$ that is $\mu$-measurable,
and every function $f:E\to[0,\infty]$ that is $\sigma$-measurable.
\end{lemma}

\begin{proof}
Based on Fubini's Theorem (and \eqref{Mixed-3}), we may write 
\begin{eqnarray}\label{Biu+F3}
\int_E\Bigl(\int_{\Gamma_\kappa(x)}u(y)\,d\mu(y)\Bigr)f(x)\,d\sigma(x)&=& \int_{{\mathscr{X}}\setminus E}u(y)\Bigl(
\int_E{\mathbf{1}}_{\pi^\kappa_y}(x)f(x)\,d\sigma(x)\Bigr)\,d\mu(y)
\nonumber\\[4pt]
&=& \int_{{\mathscr{X}}\setminus E}u(y)\sigma\bigl(\pi_y^\kappa\bigr)
\Bigl(\meanint_{\pi_y^\kappa}f\,d\sigma\Bigr)\,d\mu(y),
\end{eqnarray}
as well as 
\begin{eqnarray}\label{Biu+F4A}
\int_E\Bigl(\int_{\Gamma_{\kappa'}(x)}u(y)\,d\mu(y)\Bigr)
(M_Ef)(x)\,d\sigma(x)
= \int_{{\mathscr{X}}\setminus E}u(y)\sigma\bigl(\pi_y^\kappa\bigr)
\Bigl(\meanint_{\pi_y^{\kappa'}}M_Ef\,d\sigma\Bigr)\,d\mu(y).
\end{eqnarray}
Hence, in order to conclude \eqref{Biu+F2}, in light of \eqref{Biu+F3}, 
\eqref{Biu+F4A}, and \eqref{equiv11}, it suffices to show that there exists 
a constant $C_1\in(0,\infty)$ such that 
\begin{eqnarray}\label{Biu+F4}
\meanint_{\pi_y^\kappa}f\,d\sigma
\leq C_1\meanint_{\pi_y^{\kappa'}}M_Ef\,d\sigma
\qquad\mbox{for every $y\in{\mathscr{X}}\setminus E$}.
\end{eqnarray}
To this end, fix some $y\in{\mathscr{X}}\setminus E$ and let $y_\ast\in E$, 
$\epsilon>0$ be such that \eqref{Equ-1} holds for some 
$\eta\in(0,\min\{\kappa,\kappa'\})$.
Then \eqref{sgbr} holds when written both for $\kappa$ and $\kappa'$. 
In particular, for each $z\in\pi_y^{\kappa'}$ we have 
$\rho_{\#}(z,y_\ast)<C_{\rho}(1+\kappa')\delta_E(y)$ and, consequently,
\begin{eqnarray}\label{Biu+F5}
\hskip -0.50in
B_{\rho_{\#}}\bigl(y_\ast,C_{\rho}(1+\kappa)\delta_E(y)\bigr)
&\subseteq &
B_{\rho_{\#}}\bigl(z,C^2_{\rho}(1+\max\{\kappa,\kappa'\})\delta_E(y)\bigr)
\nonumber\\[4pt]
&\subseteq &
B_{\rho_{\#}}\bigl(y_\ast,C^3_{\rho}(1+{\max\{\kappa,\kappa'\})\delta_E}(y)\bigr),
\qquad\forall\,z\in\pi_y^{\kappa'}.
\end{eqnarray}
Making now use of \eqref{sgbr}, \eqref{Biu+F5} and \eqref{Doub-2}, we obtain
\begin{eqnarray}\label{Biu+F6}
\meanint_{\pi_y^\kappa}f\,d\sigma
&\leq & \frac{1}{\sigma(B_{\rho_{\#}}(y_\ast,\epsilon\delta_E(y)))}
\int_{B_{\rho_{\#}}\bigl(y_\ast,C_{\rho}(1+\kappa)\delta_E(y)\bigr)}f\,d\sigma
\nonumber\\[4pt]
&\leq & C_{\sigma,\rho_{\#}}\bigl(C^3_{\rho}\epsilon^{-1}(1+\max\{\kappa,\kappa'\})
\bigr)^{D_\sigma}
\meanint_{B_{\rho_{\#}}(z,C^2_{\rho}(1+\max\{\kappa,\kappa'\})\delta_E(y))}
f\,d\sigma
\nonumber\\[4pt]
&\leq & C_{\sigma,\rho_{\#}}\bigl(C^3_{\rho}\epsilon^{-1}
(1+\max\{\kappa,\kappa'\})\bigr)^{D_\sigma}
M_Ef(z),\qquad\forall\,z\in\pi_y^{\kappa'},
\end{eqnarray}
where $C_{\sigma,\rho_{\#}}$, $D_\sigma$ are the constants associated with 
$\sigma$, $\rho_{\#}$ as in \eqref{Doub-2}. Thus, if we now set 
$C_1:=C_{\sigma,\rho_{\#}}
\bigl(C^3_{\rho}\epsilon^{-1}(1+\max\{\kappa,\kappa'\})\bigr)^{D_\sigma}$ then
\begin{eqnarray}\label{Biu+F7}
\meanint_{\pi_y^\kappa}f\,d\sigma
\leq C_1\inf\limits_{z\in\pi_y^{\kappa'}}[M_Ef(z)]
\leq C_1\meanint_{\pi_y^{\kappa'}}M_Ef\,d\sigma,
\end{eqnarray}
proving \eqref{Biu+F4}, and finishing the proof of the lemma.
\end{proof}

We are now prepared to state and prove the following equivalence result
for the quasi-norms of the mixed norm spaces associated with different apertures (of the 
nontangential approach regions). 

\begin{theorem}\label{appert}
Let $({\mathscr{X}},\rho)$ be a quasi-metric space, $\mu$ a Borel measure on $({\mathscr{X}},\tau_\rho)$, $E$ a proper, nonempty, closed subset of 
$({\mathscr{X}},\tau_\rho)$, and $\sigma$ a Borel measure on $(E,\tau_{\rho|_{E}})$ 
such that $(E,\rho\bigl|_E,\sigma)$ is a space of homogeneous type. 
Also, fix two indices $p,q\in(0,\infty]$ with the convention that 
$q=\infty$ if $p=\infty$. Then for each $\kappa,\kappa'>0$ there holds
\begin{eqnarray}\label{Mixed-11CC}
\|u\|_{L^{(p,q)}({\mathscr{X}},E,\mu,\sigma;\kappa)}\approx
\|u\|_{L^{(p,q)}({\mathscr{X}},E,\mu,\sigma;\kappa')},
\end{eqnarray}
uniformly for $\mu$-measurable functions 
$u:{\mathscr{X}}\setminus E\to\overline{\mathbb{R}}$. 

Hence, in particular, for each $p,q\in(0,\infty)$, there holds
\begin{eqnarray}\label{Mixed-11}
\Bigl(\int_E\Bigl[\int_{\Gamma_\kappa(x)}|u(y)|^q\,d\mu(y)\Bigr]^{p/q}
d\sigma(x)\Bigr)^{1/p}\approx\Bigl(\int_E\Bigl[\int_{\Gamma_{\kappa'}(x)}
|u(y)|^q\,d\mu(y)\Bigr]^{p/q}d\sigma(x)\Bigr)^{1/p},
\end{eqnarray}
uniformly for $\mu$-measurable functions 
$u:{\mathscr{X}}\setminus E\to\overline{\mathbb{R}}$.
\end{theorem}

Before presenting the proof of this theorem we shall comment 
on the nature of the limiting case $p=\infty$, $q\in(0,\infty)$ of \eqref{Mixed-11}.
This clarifies the comment at the bottom of page 183 in \cite{STEIN}.

\begin{remark}\label{Fv99-T1}
In the context of Theorem~\ref{appert}, if $q\in(0,\infty)$, in general 
it is not true that 
\begin{eqnarray}\label{Mixed-11-N}
\sup_{x\in E}\Bigl(\int_{\Gamma_\kappa(x)}|u(y)|^q\,d\mu(y)\Bigr)^{1/q}
\approx
\sup_{x\in E}\Bigl(\int_{\Gamma_{\kappa'}(x)}|u(y)|^q\,d\mu(y)\Bigr)^{1/q}.
\end{eqnarray}
To see that this equivalence might fail, consider the case when 
${\mathscr{X}}:={\mathbb{R}}^2$, $E:={\mathbb{R}}\equiv\partial{\mathbb{R}}^2_{+}$, 
and take $\kappa:=\sqrt{2}$, $\kappa'\in(0,\sqrt{2})$. Also, without loss of generality, 
assume that $q=1$ and consider $u:{\mathscr{X}}\setminus E\to\overline{\mathbb{R}}$ 
given by 
\begin{eqnarray}\label{U-Rg89}
u(x,y):=\left\{
\begin{array}{l}
x^{-1}\,\,\mbox{ if $x>0$ and $x<y<x+1$}, 
\\[4pt]
0\,\,\mbox{ otherwise}.
\end{array}
\right.
\end{eqnarray}
Then 
\begin{eqnarray}\label{Mix-Tg.5}
\sup_{z\in{\mathbb{R}}}\Bigl(\int_{\Gamma_\kappa(z)}|u(x,y)|\,dxdy\Bigr)
=\int_{|x|<y}|u(x,y)|\,dxdy=\int_{0<x<y<x+1}x^{-1}\,dxdy=\infty,
\end{eqnarray}
whereas for each $z\in(0,\infty)$, elementary geometry gives that
\begin{eqnarray}\label{Mix-Tg.6}
\int_{\Gamma_{\kappa'}(z)}|u(x,y)|\,dxdy\leq Cz^{-1}\cdot
{\rm Area}\{(x,y)\in\Gamma_{\kappa'}(z):\,0<x<y<x+1\}\leq C,
\end{eqnarray}
for some $C=C(\kappa')\in(0,\infty)$. This shows that 
$\sup_{z\in{\mathbb{R}}}\Bigl(\int_{\Gamma_{\kappa'}(z)}|u(x,y)|\,dxdy\Bigr)<\infty$, 
hence \eqref{Mixed-11-N} fails in this case.
\end{remark}

We now turn to the 

\vskip 0.08in
\begin{proof}[Proof of Theorem~\ref{appert}]
Let the real numbers $\kappa,\kappa'>0$ be arbitrary and fixed. 
Then, recalling \eqref{sp-sq}, it follows that the equivalence in \eqref{Mixed-11CC} 
is proved  once we show that there exists a finite constant $C=C(\kappa,\kappa')>0$ 
such that for every $\mu$-measurable function 
$u:{\mathscr{X}}\setminus E\to\overline{\mathbb{R}}$ we have 
\begin{eqnarray}\label{Mixed-11BB}
\|{\mathscr{A}}_{q,\kappa'}u\|_{L^p(E,\sigma)}
\leq C\|{\mathscr{A}}_{q,\kappa}u\|_{L^p(E,\sigma)},
\end{eqnarray}
with the understanding that, when $q=\infty$, the $q$-th power integral of $u$ 
over nontangential approach regions is replaced by the nontangential maximal 
operator of $u$ (cf. \eqref{Mixed-N}). We proceed by dividing up the proof 
of \eqref{Mixed-11BB} into a number of cases.

\vskip 0.08in
\noindent{\it Case~I: $0<p<q<\infty$.} For $\lambda>0$ arbitrary define 
\begin{eqnarray}\label{Shcd}
A:=\bigl\{x\in E:\,({\mathscr{A}}_{q,\kappa}u)(x)\leq\lambda\bigr\}.
\end{eqnarray}
By \eqref{Mixed-7A} we have that $A$ is closed in $(E,\tau_{\rho|_{E}})$, 
hence $A^\ast_\gamma\subseteq A$ for every $\gamma\in(0,1)$ by virtue of {\it (4)} 
in Proposition~\ref{DJrt}. Let $\gamma=\gamma(\kappa,\kappa')\in(0,1)$ be such 
that \eqref{Mi+LV} holds. Then
\begin{eqnarray}\label{Shcd-2}
&&\hskip -0.40in
\sigma\bigl(\{x\in E:\,({\mathscr{A}}_{q,\kappa'}u)(x)>\lambda\}\bigr)
\leq\sigma\bigl(E\setminus A^\ast_\gamma)
+\sigma\bigl(\{x\in A^\ast_\gamma:\,({\mathscr{A}}_{q,\kappa'}u)(x)>\lambda\}\bigr)
\nonumber\\[4pt]
&&\hskip 0.60in
\leq\frac{C}{1-\gamma}\sigma(E\setminus A)
+\frac{1}{\lambda^q}\int_{A^\ast_\gamma}({\mathscr{A}}_{q,\kappa'}u)(x)^q\,d\sigma(x)
\nonumber\\[4pt]
&&\hskip 0.60in
\leq \frac{C}{1-\gamma}\sigma\bigl(\{x\in E:\,({\mathscr{A}}_{q,\kappa}u)(x)>\lambda\}\bigr)
+\frac{C}{\lambda^q}\int_A({\mathscr{A}}_{q,\kappa}u)(x)^q\,d\sigma(x).
\end{eqnarray}
For the second inequality in \eqref{Shcd-2} we used {\it (3)} in Proposition~\ref{DJrt}
and Tschebyshev's inequality, while for the last inequality we used \eqref{Shcd} and 
\eqref{Mi+LV} (with $\kappa$ and $\kappa'$ interchanged). Thus, if we multiply the
inequality resulting from \eqref{Shcd-2} by $p\lambda^{p-1}$ and then integrate in
$\lambda\in(0,\infty)$, we obtain
\begin{eqnarray}\label{Shcd-3}
\|{\mathscr{A}}_{q,\kappa'}u\|^p_{L^p(E,\sigma)}
\leq \frac{C}{1-\gamma}\|{\mathscr{A}}_{q,\kappa}u\|^p_{L^p(E,\sigma)}
+C\int_0^\infty\lambda^{p-q-1}\Bigl(\int_{\{{\mathscr{A}}_{q,\kappa}u\leq\lambda\}}
({\mathscr{A}}_{q,\kappa}u)^q\,d\sigma\Bigr)\,d\lambda.
\end{eqnarray}
By Fubini's Theorem, we further write 
\begin{eqnarray}\label{Shcd-4}
\int_0^\infty\lambda^{p-q-1}\Bigl(\int_{\{{\mathscr{A}}_{q,\kappa}u\leq\lambda\}}
({\mathscr{A}}_{q,\kappa}u)^q\,d\sigma\Bigr)\,d\lambda
&=& \int_E\Bigl(\int_{({\mathscr{A}}_{q,\kappa}u)(x)}^\infty\lambda^{p-q-1}\,d\lambda\Bigr)
({\mathscr{A}}_{q,\kappa}u)(x)^q\,d\sigma(x)
\nonumber\\[4pt]
&=& (q-p)^{-1}\|{\mathscr{A}}_{q,\kappa}u\|^p_{L^p(E,\sigma)},
\end{eqnarray}
given that we are currently assuming that $p<q$. In concert, \eqref{Shcd-3} 
and \eqref{Shcd-4} now yield \eqref{Mixed-11BB} in the case when $q\in(0,\infty)$ 
and $0<p<q$.

\vskip 0.08in
\noindent{\it Case~II: $p=q\in(0,\infty)$.} Combining \eqref{Mix+FR} 
(corresponding to $A=E$ and applied twice) with \eqref{equiv11} (applied for every
$y\in{\mathscr{X}}\setminus E$), we obtain that
\begin{eqnarray}\label{Shcd-5}
\int_E\Bigl(\int_{\Gamma_{\kappa'}(x)}|u(y)|^p\,d\mu(y)\Bigr)\,d\sigma(x)
&=&\int_{{\mathscr{X}}\setminus E} |u(y)|^p\sigma(\pi^{\kappa'}_y)\,d\mu(y)
\nonumber\\[4pt]
&\approx&\int_{{\mathscr{X}}\setminus E} |u(y)|^p\sigma(\pi^{\kappa}_y)\,d\mu(y)
\nonumber\\[4pt]
&=&\int_E\Bigl(\int_{\Gamma_\kappa(x)}|u(y)|^p\,d\mu(y)\Bigr)\,d\sigma(x),
\end{eqnarray}
and the desired conclusion follows.
\vskip 0.08in

\noindent{\it Case~III: $0<q<p<\infty$.} Let $(p/q)'$ denote the H\"older conjugate 
of $p/q\in(1,\infty)$. By using Riesz's duality theorem for Lebesgue spaces, then 
Lemma~\ref{Biu+F} (with $u$ replaced by $|u|^q$), and then H\"older's 
inequality, we may write 
\begin{eqnarray}\label{Shcd-6}
&&\hskip -0.50in
\left(\int_E\Bigl(\int_{\Gamma_\kappa(x)}|u(y)|^q\,d\mu(y)\Bigr)^{p/q}\,d\sigma(x)
\right)^{q/p}=\Bigl\|\int_{\Gamma_\kappa(x)}|u|^q\,d\mu\Bigr\|_{L^{p/q}_x(E,\sigma)}
\nonumber\\[4pt]
&&\hskip 0.50in
=\sup\limits_{\overset{f\in L^{(p/q)'}(E,\sigma)}
{f\geq 0,\,\,\|f\|_{L^{(p/q)'}(E,\sigma)}\leq 1}}
\left[\int_E\Bigl(\int_{\Gamma_\kappa(x)}|u(y)|^q\,d\mu(y)\Bigr)f(x)\,d\sigma(x)\right]
\nonumber\\[4pt]
&&\hskip 0.50in
\leq C\sup\limits_{\overset{f\in L^{(p/q)'}(E,\sigma)}
{f\geq 0,\,\|f\|_{L^{(p/q)'}(E,\sigma)}\leq 1}}
\left[\int_E\Bigl(\int_{\Gamma_{\kappa'}(x)}|u(y)|^q\,d\mu(y)\Bigr)(M_Ef)(x)
\,d\sigma(x)\right]
\nonumber\\[4pt]
&&\hskip 0.50in
\leq C\sup\limits_{\overset{f\in L^{(p/q)'}(E,\sigma)}
{f\geq 0,\,\|f\|_{L^{(p/q)'}(E,\sigma)}\leq 1}}\!\!\!\Bigl[\,
\left(\int_E\Bigl(\int_{\Gamma_{\kappa'}(x)}|u(y)|^q\,d\mu(y)\Bigr)^{p/q}\,d\sigma(x)
\right)^{q/p}\times
\nonumber\\[4pt]
&&\hskip 2.50in
\times\Bigl(\int_E(M_Ef)(x)^{(p/q)'}\,d\sigma(x)\Bigr)^{1/(p/q)'}\Bigr]
\nonumber\\[4pt]
&&\hskip 0.50in
\leq C\left(\int_E\Bigl(\int_{\Gamma_{\kappa'}(x)}|u(y)|^q\,d\mu(y)\Bigr)^{p/q}\,
d\sigma(x)\right)^{q/p},
\end{eqnarray}
where for the last inequality in \eqref{Shcd-6} we used the boundedness of 
the maximal operator $M_E$ on $L^r(E,\sigma)$ for $r:=(p/q)'\in(1,\infty)$.
This completes the proof of \eqref{Mixed-11BB} when $0<q<p<\infty$.

\vskip 0.08in
\noindent{\it Case~IV: $0<p<\infty$, $q=\infty$.} 
Fix $\lambda>0$ and introduce (recall \eqref{Mixed-N})
\begin{eqnarray}\label{cal-O}
{\mathcal{O}}_\kappa:=\bigl\{x\in E:\,({\mathcal{N}}_\kappa u)(x)>\lambda\bigr\},
\qquad
{\mathcal{O}}_{\kappa'}:=\bigl\{x\in E:\,({\mathcal{N}}_{\kappa'}u)(x)>\lambda\bigr\}.
\end{eqnarray}
Hence, the desired conclusion follows as soon as we show that there exists 
$C\in(0,\infty)$, independent of $u$ and $\lambda$, with the property that 
$\sigma({\mathcal{O}}_{\kappa'})\leq C\sigma({\mathcal{O}}_{\kappa})$. In turn,
by virtue of {\it (3)} in Proposition~\ref{DJrt}, this follows once we prove that
that there exists $\gamma\in(0,1)$ such that 
\begin{eqnarray}\label{incl-Cl}
{\mathcal{O}}_{\kappa'}\subseteq 
E\setminus(E\setminus{\mathcal{O}}_{\kappa})^*_\gamma.
\end{eqnarray}

To justify this inclusion, fix $\eta\in(0,\min\,\{\kappa,\kappa'\})$ and
assume that $x\in{\mathcal{O}}_{\kappa'}$ is an arbitrary point.  
Then there exists $y\in\Gamma_{\kappa'}(x)$ for which $|u(y)|>\lambda$ and 
we select $y_\ast\in E$ and $\epsilon\in(0,1)$ as in \eqref{Equ-1} (for $\eta$ as 
specified above). In particular, $\rho_{\#}(y,y_\ast)<(1+\eta)\delta_E(y)$.
Observe from \eqref{sgbr} and \eqref{cal-O} that in this scenario we have 
\begin{eqnarray}\label{ZDel}
E\cap B_{\rho_{\#}}\bigl(y_\ast,\epsilon\delta_E(y)\bigr)\subseteq\pi^{\kappa}_y
\subseteq{\mathcal{O}}_{\kappa}
\end{eqnarray}
and we also claim that
\begin{eqnarray}\label{ZDel-2}
E\cap B_{\rho_{\#}}\bigl(y_\ast,\epsilon\delta_E(y)\bigr)\subseteq
E\cap B_{\rho_{\#}}\bigl(x,C_{\rho}(1+\kappa')\delta_E(y)\bigr).
\end{eqnarray}
To see this, recall that $\epsilon\in(0,1)$ and note that if $z\in E$ 
satisfies $\rho_{\#}(z,y_\ast)<\delta_E(y)$ then 
\begin{eqnarray}\label{XX-YZ}
\rho_{\#}(x,z) &\leq & C_{\rho}\,\max\,\{\rho_{\#}(x,y),\rho_{\#}(y,z)\}
\nonumber\\[4pt]
& \leq & C_{\rho}\,\max\,\Bigl\{(1+\kappa')\delta_E(y),\,
C_{\rho}\,\max\,\{\rho_{\#}(y,y_\ast),\rho_{\#}(y_\ast,z)\}\Bigr\}
\nonumber\\[4pt]
& \leq & C_{\rho}\,\max\,\Bigl\{(1+\kappa')\delta_E(y),\,
C_{\rho}\,\max\,\{(1+\alpha)\delta_E(y),\delta_E(y)\}\Bigr\}
\nonumber\\[4pt]
& = & C_{\rho}(1+\kappa')\delta_E(y),
\end{eqnarray}
proving \eqref{ZDel-2}. In concert, \eqref{ZDel} and \eqref{ZDel-2} yield
\begin{eqnarray}\label{inCL}
E\cap B_{\rho_{\#}}\bigl(y_\ast,\delta_E(y)\bigr)\subseteq
{\mathcal{O}}_\kappa\cap B_{\rho_{\#}}\bigl(x,C_{\rho}(1+\kappa')\delta_E(y)\bigr).
\end{eqnarray}
Let us also observe that 
\begin{eqnarray}\label{XX-YZ342}
\rho_{\#}(x,y_\ast) &\leq & C_{\rho}\,\max\,\{\rho_{\#}(x,y),\rho_{\#}(y,y_\ast)\}
\nonumber\\[4pt]
& \leq & C_{\rho}\,\max\,\Bigl\{(1+\kappa')\delta_E(y),
(1+\eta)\delta_E(y)\Bigr\}=C_{\rho}(1+\kappa')\delta_E(y).
\end{eqnarray}
Then, for some sufficiently small $c\in(0,1)$ which depends only on
$\kappa,\kappa'$ and background geometrical characteristics, we may write 
\begin{eqnarray}\label{fst-E}
\frac{\sigma\Bigl({\mathcal{O}}_\kappa\cap 
B_{\rho_{\#}}\bigl(x,C_{\rho}(1+\kappa')\delta_E(y)\bigr)\Bigr)}
{\sigma\Bigl(E\cap B_{\rho_{\#}}\bigl(x,C_{\rho}(1+\kappa')\delta_E(y)\bigr)\Bigr)}
\geq 
\frac{\sigma\Bigl(E\cap B_{\rho_{\#}}\bigl(y_\ast,\epsilon\delta_E(y)\bigr)\Bigr)}
{\sigma\Bigl(E\cap B_{\rho_{\#}}\bigl(x,C_{\rho}(1+\kappa')\delta_E(y)\bigr)\Bigr)}
\geq c,
\end{eqnarray}
where the first inequality follows from \eqref{inCL}, while the second
inequality is a consequence of \eqref{XX-YZ342} and the fact that 
$(E,\rho\bigl|_{E},\sigma)$ is a space of homogeneous type (cf. \eqref{Doub-2}).
In particular, if we set $r:=C_{\rho}(1+\kappa')\delta_E(y)$, then 
\begin{eqnarray}\label{O-c}
\frac{\sigma\bigl((E\setminus{\mathcal{O}}_\kappa)\cap
B_{\rho_{\#}}(x,r)\bigr)}{\sigma\bigl(E\cap B_{\rho_{\#}}(x,r)\bigr)}\leq 1-c.
\end{eqnarray}
Thus, if we select $\gamma$ such that $1-c<\gamma<1$, then \eqref{O-c} entails 
$x\notin(E\setminus{\mathcal{O}}_\kappa)^*_\gamma$. This proves the claim
\eqref{incl-Cl}, and finishes the treatment of the current case.

\vskip 0.08in
\noindent{\it Case~V: $p=q=\infty$.} In this case, the desired conclusion 
follows upon observing that if $u:{\mathscr{X}}\setminus E\to\overline{\mathbb{R}}$
is a $\mu$-measurable function then 
\begin{eqnarray}\label{Gf-5tFF}
\Bigl\|E\ni x\mapsto
\|u\|_{L^\infty(\Gamma_\kappa(x),\mu)}\Bigr\|_{L^\infty(E,\sigma)}
=\|u\|_{L^\infty({\mathscr{X}}\setminus E,\mu)}.
\end{eqnarray}
Indeed, the inequality $\bigl\|\|u\|_{L^\infty(\Gamma_\kappa(x),\mu)}\bigr\|_{L_x^\infty(E,\sigma)}
\leq\|u\|_{L^\infty({\mathscr{X}}\setminus E,\mu)}$ is a simple consequence
of the fact that $\Gamma_\kappa(x)\subseteq{\mathscr{X}}\setminus E$ for each 
$x\in E$. In the opposite direction, if $M$ denotes the left-hand side of 
\eqref{Gf-5tFF}, then there exists a $\sigma$-measurable set $F\subseteq E$ satisfying 
$\sigma(F)=0$ and $\|u\|_{L^\infty(\Gamma_\kappa(x),\mu)}\leq M$ for every 
$x\in E\setminus F$. Since $(E,\rho|_{E})$ is geometrically doubling, so is 
$E\setminus F$ when equipped with $\rho|_{E\setminus F}$, hence separable as a
topological space. Consequently, given that $E\setminus F$ is dense in $E$, 
it follows that there exists a countable subset $A:=\{x_j\}_{j\in{\mathbb{N}}}$ of 
$E\setminus F$ which is dense in $E$. Now, for each $j\in{\mathbb{N}}$ there 
exists $N_j\subseteq\Gamma_\kappa(x_j)$, null-set for $\mu$, such that 
$|u(x)|\leq M$ for every $x\in\Gamma_\kappa(x_j)$. 
Thus, $N:=\cup_{j\in{\mathbb{N}}}N_j\subseteq{\mathscr{X}}\setminus E$ 
is a null-set for $\mu$ and $|u(x)|\leq M$ for every point $x$ belonging to 
\begin{eqnarray}\label{Hdsss-YU}
\Bigl(\bigcup\limits_{j\in{\mathbb{N}}}\Gamma_\kappa(x_j)\Bigr)
\setminus N={\mathcal{F}}_\kappa(A)\setminus N
={\mathcal{F}}_\kappa(\overline{A})\setminus N
={\mathcal{F}}_\kappa(E)\setminus N=({\mathscr{X}}\setminus E)\setminus N,
\end{eqnarray}
where the second equality follows from $(i)$ in Lemma~\ref{T-LL.2}, and the 
last equality is a consequence of \eqref{Tfs23} and the fact that $E$ 
is a closed subset of $({\mathscr{X}},\tau_\rho)$.
Hence, $\|u\|_{L^\infty({\mathscr{X}}\setminus E,\mu)}\leq M$, as desired. 
This finishes the justification of \eqref{Gf-5tFF} and finishes the proof 
of the theorem.
\end{proof}

\subsection{Estimates relating the Lusin and Carleson operators}
\label{SSect:5.2}

We now introduce a Carleson operator $\mathfrak{C}$ and show how it can be used instead of the area operator $\mathscr{A}$ to provide an equivalent quasi-norm for the mixed norm spaces. This is essential in Subsection~\ref{SSect:5.4}, and it is achieved by combining Theorem~\ref{appert} with a good $\lambda$ inequality, as in Theorem~3 of~\cite{CoMeSt}.

Let $({\mathscr{X}},\rho)$ be a quasi-metric space, 
$\mu$ a Borel measure on $({\mathscr{X}},\tau_\rho)$, 
$E$ a nonempty, proper, closed subset of $({\mathscr{X}},\tau_\rho)$, and 
$\sigma$ a measure on $E$ such that $(E,\rho\bigl|_E,\sigma)$ is a space 
of homogeneous type. For each index $q\in(0,\infty)$ and constant $\kappa\in(0,\infty)$,
recall the $L^q$-based Lusin (or area) operator ${\mathscr{A}}_{q,\kappa}$ 
from \eqref{sp-sq}, and now define the $L^q$-based {\tt Carleson operator} ${\mathfrak{C}}_{q,\kappa}$ for all $\mu$-measurable functions 
$u:{\mathscr{X}}\setminus E\to\overline{\mathbb{R}}:=[-\infty,+\infty]$ by
\begin{eqnarray}\label{ktEW-7}
({\mathfrak{C}}_{q,\kappa}u)(x):=\sup\limits_{\Delta\subseteq E,\,x\in\Delta}
\Bigl(\tfrac{1}{\sigma(\Delta)}\int_{{\mathcal{T}}_\kappa(\Delta)}|u(y)|^q
\sigma(\pi^{\kappa}_y)\,d\mu(y)\Bigr)^{\frac{1}{q}},\qquad\forall\,x\in E,
\end{eqnarray}
where $\pi^{\kappa}_y$ is from \eqref{reg-A2}, the supremum is taken over {\tt surface balls}, i.e., sets of the form 
\begin{eqnarray}\label{ktEW-7UU}
\Delta:=\Delta(y,r):=E\cap B_{\rho_{\#}}(y,r),\qquad y\in E,\quad r>0
\end{eqnarray}
containing $x$, and ${\mathcal{T}}_\kappa(\Delta)$ is the 
tent region over $\Delta$ from \eqref{reg-A1}.

The following theorem extends the result on ${\mathbb{R}}^{n+1}_+$ from \cite[Theorem~3, p.\,318]{CoMeSt}. To state it, consider a measure space $(E,\sigma)$, and for each $p\in(0,\infty)$ and $r\in(0,\infty]$, let 
$L^{p,r}(E,\sigma)$ denote the Lorentz space equipped with the quasi-norm
\begin{eqnarray}\label{TgEE+56}
&& \|f\|_{L^{p,r}(E,\sigma)}:=\left(\int_0^\infty\lambda^r\sigma\bigl(
\left\{x\in E:\,|f(x)|>\lambda\right\}\bigr)^{r/p}\,\frac{d\lambda}{\lambda}\right)^{1/r},
\quad\mbox{ if }\,\,r<\infty,
\\[4pt] 
&& \|f\|_{L^{p,\infty}(E,\sigma)}:=\sup_{\lambda>0}\Bigl[\lambda\,\sigma\left(
\left\{x\in E:\,|f(x)|>\lambda\right\}\right)^{1/p}\Bigr]\quad\mbox{ if }\,\,r=\infty.
\label{TgEE+56BBB}
\end{eqnarray}
Note that $L^{p,p}(E,\sigma)=L^p(E,\sigma)$ for each $p\in(0,\infty)$.
Also, given a quasi-metric space $({\mathscr{X}},\rho)$, call a Borel measure 
$\mu$ on $({\mathscr{X}},\tau_\rho)$ {\tt locally finite} when 
$\mu\bigl(B_\rho(x,r)^\circ\bigr)<\infty$ for all $x\in{\mathscr{X}}$ and $r>0$, where the interior is taken in the topology 
$\tau_\rho$.

\begin{theorem}\label{AsiC}
Let $({\mathscr{X}},\rho)$ be a quasi-metric space, 
$\mu$ be a locally finite Borel measure on $({\mathscr{X}},\tau_\rho)$, and assume 
that $E$ is a proper, nonempty, closed subset of $({\mathscr{X}},\tau_\rho)$, 
and $\sigma$ a measure on $E$ such that $(E,\rho\bigl|_E,\sigma)$ is a space 
of homogeneous type. Fix $q\in(0,\infty)$ and $\kappa>0$. 
Then the following estimates hold.
\begin{enumerate}
\item[(1)] For each $p\in(0,\infty)$ there exists $C\in(0,\infty)$ such that
$\|{\mathscr{A}}_{q,\kappa}u\|_{L^p(E,\sigma)}\leq C
\|{\mathfrak{C}}_{q,\kappa}u\|_{L^p(E,\sigma)}$ for every 
$\mu$-measurable function $u:{\mathscr{X}}\setminus E\to\overline{\mathbb{R}}$.
\item[(2)] For each $p\in(q,\infty)$ and each $r\in(0,\infty]$ there exists 
a constant $C\in(0,\infty)$ 
such that $\|{\mathfrak{C}}_{q,\kappa}u\|_{L^{p,r}(E,\sigma)}
\leq C\|{\mathscr{A}}_{q,\kappa}u\|_{L^{p,r}(E,\sigma)}$ for every 
$\mu$-measurable function $u:{\mathscr{X}}\setminus E\to\overline{\mathbb{R}}$. 
\item[(3)] Corresponding to the end-point cases $p=q$ and $p=\infty$ in $(2)$, 
there exists $C\in(0,\infty)$ such that for every $\mu$-measurable
function $u:{\mathscr{X}}\setminus E\to\overline{\mathbb{R}}$ 
the following estimates hold:
\begin{eqnarray}\label{Tgbcc-7Y}
\hskip -0.20in
\|{\mathfrak{C}}_{q,\kappa}u\|_{L^{q,\infty}(E,\sigma)}
\leq C\|{\mathscr{A}}_{q,\kappa}u\|_{L^q(E,\sigma)}\,\,\mbox{ and }\,\,
\|{\mathfrak{C}}_{q,\kappa}u\|_{L^{\infty}(E,\sigma)}
\leq C\|{\mathscr{A}}_{q,\kappa}u\|_{L^\infty(E,\sigma)}.
\end{eqnarray}
\end{enumerate}
In particular, 
\begin{eqnarray}\label{sbrn}
\|{\mathscr{A}}_{q,\kappa}u\|_{L^p(E,\sigma)}
\approx\|{\mathfrak{C}}_{q,\kappa}u\|_{L^p(E,\sigma)}
\quad\mbox{ for each }\,p\in(q,\infty), 
\end{eqnarray}
uniformly in $u:{\mathscr{X}}\setminus E\to\overline{\mathbb{R}}$, 
$\mu$-measurable function.
\end{theorem}

\begin{proof}
Fix $q\in(0,\infty)$ and define 
\begin{eqnarray}\label{CPP-77}
c_q:=\left\{
\begin{array}{l}
2^{(1/q)-1}\,\,\mbox{ if }\,\,q<1,
\\[4pt]
1\,\,\mbox{ if }\,\,q\geq 1.
\end{array}
\right.
\end{eqnarray}
Also, fix an arbitrary $\mu$-measurable 
function $u:{\mathscr{X}}\setminus E\to\overline{\mathbb{R}}$. 
We claim that the following good-$\lambda$ inequality is valid:
\begin{eqnarray}\label{kbFF}
\begin{array}{c}
\forall\,\kappa>0,\,\,\,\exists\,\kappa'>\kappa,\,\,\,
\exists\,c\in(0,\infty)\,
\mbox{ such that $\forall\,\gamma\in(0,1]$, $\forall\,\lambda\in(0,\infty)$, there holds}
\\[8pt]
\sigma\bigl(\{x\in E:({\mathscr{A}}_{q,\kappa} u)(x)>2c_q\lambda,
({\mathfrak{C}}_{q,\kappa}u)(x)\leq\!\gamma\lambda\}\bigr)
\leq c\,\gamma^q\sigma\bigl(\{x\in E:\,
({\mathscr{A}}_{q,\kappa'}u)(x)>\lambda\}\bigr),
\end{array}
\end{eqnarray}
where the constant $c\in(0,\infty)$ is independent of $u$.
Suppose for now that the above claim is true. Then, if $\kappa>0$ is fixed, 
let $\kappa',c$ be as in \eqref{kbFF}. Hence, for each fixed $\gamma\in(0,1]$
and every $\lambda>0$ we have 
\begin{eqnarray}\label{kbFF-2}
&&\hskip -0.50in
\sigma\bigl(\{x\in E:\,({\mathscr{A}}_{q,\kappa}u)(x)>2c_q\lambda\}\bigr)
\\[4pt]
&& \hskip 0.30in
\leq\sigma\bigl(\{x\in E:\,({\mathfrak{C}}_{q,\kappa}u)(x)>\gamma\lambda\}\bigr)
+c\,\gamma^q\sigma\bigl(\{x\in E:\,({\mathscr{A}}_{q,\kappa'}u)(x)>\lambda\}\bigr).
\nonumber
\end{eqnarray}
Thus, if we multiply the inequality in \eqref{kbFF-2} by $p\lambda^{p-1}$ and 
then integrate in $\lambda\in(0,\infty)$, we obtain
\begin{eqnarray}\label{kbFF-3}
(2c_q)^{-p}\|{\mathscr{A}}_{q,\kappa}u\|_{L^p(E,\sigma)}^p
\leq\gamma^{-p}\|{\mathfrak{C}}_{q,\kappa}u\|_{L^p(E,\sigma)}^p
+c\,\gamma^q\|{\mathscr{A}}_{q,\kappa'}u\|_{L^p(E,\sigma)}^p,
\qquad\forall\,\gamma\in(0,1].
\end{eqnarray}
Since from Theorem~\ref{appert} we know that there exists a finite constant $C>0$
depending only on $\kappa,\kappa',p,q$ and geometry, 
such that $\|{\mathscr{A}}_{q,\kappa'}u\|_{L^p(E,\sigma)}
\leq C\|{\mathscr{A}}_{q,\kappa}u\|_{L^p(E,\sigma)}$, we arrive at 
\begin{eqnarray}\label{kbFF-3BB}
(2c_q)^{-p}\|{\mathscr{A}}_{q,\kappa}u\|_{L^p(E,\sigma)}^p
\leq\gamma^{-p}\|{\mathfrak{C}}_{q,\kappa}u\|_{L^p(E,\sigma)}^p
+c\,C^p\,\gamma^q\|{\mathscr{A}}_{q,\kappa}u\|_{L^p(E,\sigma)}^p,
\quad\forall\,\gamma\in(0,1].
\end{eqnarray}
In order to hide the last term in the right-hand side into the left-hand side, 
fix a point $x_0\in E$ and, for each $j\in{\mathbb{N}}$, consider 
\begin{eqnarray}\label{kbFF-3CC}
u_j:=\min\,\{|u|,j\}\cdot{\mathbf{1}}_{B_{\rho_{\#}}(x_0,j)\setminus E}\qquad
\mbox{on }\,\,{\mathscr{X}}\setminus E. 
\end{eqnarray}
Observe that 
\begin{eqnarray}\label{kbFF-3DD-1}
{\rm supp}\,\bigl({\mathscr{A}}_{q,\kappa}u_j\bigr)\subseteq 
B_{\rho_{\#}}\bigl(x_0,C_\rho(1+\kappa)j\bigr),\qquad
0\leq{\mathscr{A}}_{q,\kappa}u_j
\leq j\cdot\mu\bigl(B_{\rho_{\#}}(x_0,j)\bigr)^{1/q}<\infty,
\end{eqnarray}
where the last inequality uses the fact that $\mu$ is locally finite.
In turn, this implies that 
\begin{eqnarray}\label{kbFF-3DD2}
\|{\mathscr{A}}_{q,\kappa}u_j\|_{L^p(E,\sigma)}
\leq j\cdot\mu\bigl(B_{\rho_{\#}}(x_0,j)\bigr)^{1/q}
\sigma\bigl(E\cap B_{\rho_{\#}}(x_0,C_\rho(1+\kappa)j)\bigr)^{1/p}<\infty,
\end{eqnarray}
since $\sigma$ is locally finite. Hence, if we choose $\gamma\in(0,1]$ 
so that $(2c_q)^{-p}>2cC^p\,\gamma^q$, then we obtain from \eqref{kbFF-3BB} 
written with $u$ replaced by $u_j$
\begin{eqnarray}\label{kbFF-3DD}
\|{\mathscr{A}}_{q,\kappa}u_j\|_{L^p(E,\sigma)}^p
\leq C\|{\mathfrak{C}}_{q,\kappa}u_j\|_{L^p(E,\sigma)}^p,\qquad\forall\,j\in{\mathbb{N}},
\end{eqnarray}
for some $C\in(0,\infty)$ independent of $j$. Note that 
$0\leq {\mathfrak{C}}_{q,\kappa}u_j\leq{\mathfrak{C}}_{q,\kappa}u$ pointwise in $E$ 
and that $u_j\nearrow |u|$ pointwise $\mu$-a.e. on ${\mathscr{X}}\setminus E$
implies ${\mathscr{A}}_{q,\kappa}u_j\nearrow {\mathscr{A}}_{q,\kappa}u$ everywhere on $E$ 
by Lebesgue's Monotone Convergence Theorem. Based on these observations, 
\eqref{kbFF-3DD} and Fatou's lemma we may then conclude that 
\begin{eqnarray}\label{kbFF-3EE}
\|{\mathscr{A}}_{q,\kappa}u\|_{L^p(E,\sigma)}^p
&=& \int_{E}\liminf_{j\to\infty}\bigl[{\mathscr{A}}_{q,\kappa}u_j\bigr]^p\,d\sigma
\leq\liminf_{j\to\infty}\int_{E}\bigl[{\mathscr{A}}_{q,\kappa}u_j\bigr]^p\,d\sigma
\nonumber\\[4pt]
&=& \liminf_{j\to\infty}\|{\mathscr{A}}_{q,\kappa}u_j\|_{L^p(E,\sigma)}^p
\leq C\|{\mathfrak{C}}_{q,\kappa}u\|_{L^p(E,\sigma)}^p.
\end{eqnarray}
That is, granted \eqref{kbFF}, we have 
\begin{eqnarray}\label{kbFF-3UU}
\|{\mathscr{A}}_{q,\kappa}u\|_{L^p(E,\sigma)}^p
\leq C\|{\mathfrak{C}}_{q,\kappa}u\|_{L^p(E,\sigma)}^p.
\end{eqnarray}
Thus, to finish the proof of part {\it (1)} of the statement of the theorem, we 
are left with proving \eqref{kbFF}. Fix $\kappa'>\kappa>0$ along with $\gamma\in(0,1]$, 
then for an arbitrary $\lambda>0$ define the set 
\begin{eqnarray}\label{xfbl}
{\mathcal{O}}_{\lambda}:=
\bigl\{x\in E:\,({\mathscr{A}}_{q,\kappa'}u)(x)>\lambda\bigr\}.
\end{eqnarray}
By Lemma~\ref{semi-cont}, ${\mathcal{O}}_{\lambda}$ is an open 
subset of $(E,\tau_{\rho|_{E}})$. Also, since ${\mathscr{A}}_{q,\kappa'}u
\geq{\mathscr{A}}_{q,\kappa}u$ pointwise in $E$, we conclude that 
\begin{eqnarray}\label{yrTYY97}
\{x\in E:\,({\mathscr{A}}_{q,\kappa}u)(x)>2c_q\lambda\}
\subseteq{\mathcal{O}}_{\lambda}.
\end{eqnarray}
If ${\mathcal{O}}_{\lambda}=\emptyset$, then by \eqref{yrTYY97} the 
inequality in the second line of \eqref{kbFF} is trivially satisfied. 
Therefore, assume that ${\mathcal{O}}_{\lambda}\not=\emptyset$ in what follows. 
Let us also make the assumption (which will be eliminated {\it a posteriori}) that 
\begin{eqnarray}\label{NON-ZE78}
\begin{array}{c}
\mbox{the $\mu$-measurable function $u:{\mathscr{X}}\setminus E\to\overline{\mathbb{R}}$ 
is such that}
\\[4pt]
\mbox{${\mathcal{O}}_\lambda$ from \eqref{xfbl} is a proper subset of $E$ 
for each $\lambda>0$}.
\end{array}
\end{eqnarray}
In such a scenario, for a fixed, suitably chosen $\lambda_o>1$ we may 
apply Proposition~\ref{H-S-Z} (with $\lambda$ there replaced by $\lambda_o$) 
to obtain a Whitney covering of ${\mathcal{O}}_{\lambda}$ by balls, relative 
to $(E,\rho|_{E})$, which we may assume (given the freedom of choosing 
the parameter $\lambda_o$, and \eqref{RHo-evv}) 
to be of the form $B_j:=E\cap B_{\rho_{\#}}(x_j,r_j)$,
$j\in{\mathbb{N}}$, satisfying properties {\it (1)-(4)} in Proposition~\ref{H-S-Z} 
for some $\Lambda>\lambda_o$. If we now prove that 
\begin{eqnarray}\label{kbFF-5}
\begin{array}{c}
\mbox{$\exists\,\kappa'>\kappa$ and $\exists\,c\in(0,\infty)$
such that $\forall\,\gamma\in(0,1]$, $\forall\,\lambda\in(0,\infty)$, there holds}
\\[6pt]
\sigma\bigl(\{x\in B_j:({\mathscr{A}}_{q,\kappa}u)(x)>2c_q\lambda,\,
({\mathfrak{C}}_{q,\kappa}u)(x)\leq\!\gamma\lambda\}\bigr)
\leq c\,\gamma^q\sigma(B_j)\,\mbox{ for every }\,j\in{\mathbb{N}},
\end{array}
\end{eqnarray}
then combining \eqref{kbFF-5} with \eqref{yrTYY97} and properties 
{\it (1)-(2)} from Proposition~\ref{H-S-Z}, we may estimate
\begin{eqnarray}\label{kbFF-6}
&& \hskip -0.40in
\sigma\bigl(\{x\in E:({\mathscr{A}}_{q,\kappa}u)(x)>2c_q\lambda,
({\mathfrak{C}}_{q,\kappa}u)(x)\leq\!\gamma\lambda\}\bigr)
\nonumber\\[4pt]
&&\hskip 0.20in
=\sigma\bigl(\{x\in {\mathcal{O}}_\lambda:\,
({\mathscr{A}}_{q,\kappa}u)(x)>2c_q\lambda,\,
({\mathfrak{C}}_{q,\kappa}u)(x)\leq\!\gamma\lambda\}\bigr)
\nonumber\\[4pt]
&&\hskip 0.20in
\leq\sum\limits_{j=1}^\infty
\sigma\bigl(\{x\in B_j:({\mathscr{A}}_{q,\kappa}u)(x)>2c_q\lambda,
({\mathfrak{C}}_{q,\kappa}u)(x)\leq\!\gamma\lambda\}\bigr)
\leq c\,\gamma^q\sum\limits_{j=1}^\infty\sigma(B_j)
\nonumber\\[4pt]
&&\hskip 0.20in
\leq C\gamma^q\sigma({\mathcal{O}}_\lambda).
\end{eqnarray}
Hence, \eqref{kbFF} follows. 

Turning now to the proof of \eqref{kbFF-5}, fix $j\in{\mathbb{N}}$, 
and note that without loss of generality we may assume that 
\begin{eqnarray}\label{G.bbn-88}
\bigl\{x\in B_j:({\mathscr{A}}_{q,\kappa}u)(x)>2c_q\lambda,\,\,
({\mathfrak{C}}_{q,\kappa}u)(x)\leq\!\gamma\lambda\bigr\}\not=\emptyset,
\end{eqnarray}
since otherwise there is nothing to prove. Decompose 
$u=u{\mathbf{1}}_{\{\delta_E\geq r_j\}}+u{\mathbf{1}}_{\{\delta_E<r_j\}}=:u_1+u_2$
and let $z_j\in E\setminus{\mathcal{O}}_\lambda$ be such that 
$\rho_{\#}(x_j,z_j)\leq\Lambda r_j$ (the existence of $z_j$ is guaranteed 
by property {\it (3)} in Proposition~\ref{H-S-Z}). We claim that
\begin{eqnarray}\label{kbFF-7}
\begin{array}{c}
\mbox{there exists }\,\,\kappa'>\kappa\,\,\mbox{ independent of $j\in{\mathbb{N}}$ 
with the property that }
\\[4pt]
\mbox{if }\,x\in B_j\,\mbox{ and }\,y\in\Gamma_\kappa(x)\,
\mbox{ is such that }\,\delta_E(y)\geq r_j\,
\mbox{ then }\,y\in\Gamma_{\kappa'}(z_j).
\end{array}
\end{eqnarray}
Indeed, if $x\in B_j$, we have 
$\rho_{\#}(x,z_j)\leq C_{\rho_{\#}}\max\{\rho_{\#}(x,x_j),\rho_{\#}(x_j,z_j)\}
\leq C_{\rho}\Lambda r_j$. Hence, if $y\in\Gamma_\kappa(x)$ is such that 
$\delta_E(y)\geq r_j$ then 
\begin{eqnarray}\label{kbFF-8}
\rho_{\#}(y,z_j) &\leq & C_{\rho_{\#}}\max\{\rho_{\#}(y,x),\rho_{\#}(x,z_j)\}
\leq C_{\rho}\max\{(1+\kappa)\delta_E(y),C_{\rho}\Lambda r_j\}
\nonumber\\[4pt]
&\leq& C_{\rho}\max\{(1+\kappa),C_{\rho}\Lambda\}\delta_E(y).
\end{eqnarray}
Now we choose $\kappa'>C_{\rho}\max\{(1+\kappa),C_{\rho}\Lambda\}-1$, so then 
$\kappa'>\kappa$, and $\kappa'$ depends only on finite positive geometrical constants 
(hence, in particular, it is independent of $j\in{\mathbb{N}}$). 
Based on \eqref{kbFF-8} we obtain that \eqref{kbFF-7} holds true for this 
choice of $\kappa'$. Then, using \eqref{kbFF-7} and recalling that 
$z_j\in E\setminus{\mathcal{O}}_{\lambda}$, we may write 
\begin{eqnarray}\label{kbFF-9}
({\mathscr{A}}_{q,\kappa}u_1)(x)^q &=&
\int\limits_{\stackrel{y\in\Gamma_\kappa(x)}{\delta_E(y)\geq r_j}}|u(y)|^q\,d\mu(y)
\leq\int\limits_{y\in\Gamma_{\kappa'}(z_j)}|u(y)|^q\,d\mu(y)
\nonumber\\[6pt]
&=& ({\mathscr{A}}_{q,\kappa'}u)(z_j)^q\leq\lambda^q,
\qquad\forall\,x\in B_j.
\end{eqnarray}
Next, we make use of \eqref{Mix+FR} to write
\begin{eqnarray}\label{kbFF-10}
\int_{B_j}({\mathscr{A}}_{q,\kappa}u_2)(x)^q\,d\sigma(x)
&=& \int\limits_{y\in{\mathcal{F}}_\kappa(B_j),\,\delta_E(y)<r_j}|u(y)|^q
\sigma\bigl(B_j\cap\pi^\kappa_y\bigr)\,d\mu(y)
\nonumber\\[4pt]
&\leq & \int\limits_{y\in{\mathcal{F}}_\kappa(B_j),\,\delta_E(y)<r_j}|u(y)|^q
\sigma(\pi^\kappa_y)\,d\mu(y).
\end{eqnarray}

In order to proceed further, first make a geometrical observation to 
the effect that (using notation introduced in \eqref{ktEW-7UU})
\begin{eqnarray}\label{kbFF-11}
\begin{array}{c}
\mbox{there exists a finite constant $c_o>0$ such that for every $r>0$ 
and every $x_0\in E$}
\\[4pt]
\mbox{if }\,y\in{\mathcal{F}}_\kappa(\Delta(x_0,r))\,\mbox{ and }\,\delta_E(y)<r
\,\mbox{ then }\,y\in {\mathcal{T}}_\kappa\bigl(E\cap B_{\rho_{\#}}({w},c_or)\bigr)
\,\,\forall\,{w}\in\Delta(x_0,r).
\end{array}
\end{eqnarray}
To see why this is true, consider a point $y\in{\mathcal{F}}_\kappa(\Delta(x_0,r))$ 
with the property that $\delta_E(y)<r$. Then there exists $x\in\Delta(x_0,r)$ 
such that $\rho_{\#}(y,x)<(1+\kappa)\delta_E(y)<(1+\kappa)r$.
Let ${w}\in\Delta(x_0,r)$ be arbitrary and note that
\begin{eqnarray}\label{k-RF5FF}
\rho_{\#}(x,{w})\leq C_{\rho_{\#}}\max\{\rho_{\#}(x,x_0),\rho_{\#}(x_0,{w})\}
<C_\rho r.
\end{eqnarray}
Accordingly, choosing $c_o>C_\rho$ forces $x\in E\cap B_{\rho_{\#}}({w},c_or)$ 
hence, further, 
\begin{eqnarray}\label{k-R446}
{\rm dist}_{\rho_{\#}}\bigl(y, E\cap B_{\rho_{\#}}({w},c_or)\bigr)\leq
\rho_{\#}(y,x)<(1+\kappa)\delta_E(y).
\end{eqnarray}
Let us also observe that 
\begin{eqnarray}\label{kbFF-12}
\rho_{\#}({w},y) &\leq & C_{\rho_{\#}}\max\{\rho_{\#}({w},x),\rho_{\#}(x,y)\}
\nonumber\\[4pt]
&\leq & C_{\rho_{\#}}\max\Bigl\{
C_{\rho_{\#}}\max\{\rho_{\#}({w},x_0),\rho_{\#}(x_0,x)\},\,\rho_{\#}(x,y)\Bigr\}
\nonumber\\[4pt]
&\leq & C_{\rho}\,\max\{C_\rho,1+\kappa\}r.
\end{eqnarray}
Thus, if $z\in E\setminus B_{\rho_{\#}}({w},c_or)$, making use 
of \eqref{kbFF-12} we obtain
\begin{eqnarray}\label{kbFF-14}
c_o r &\leq & \rho_{\#}(z,{w})
\leq C_{\rho_{\#}}\max\{\rho_{\#}(z,y),\rho_{\#}(y,{w})\}
\nonumber\\[4pt]
&\leq & C_{\rho}\max\Bigl\{\rho_{\#}(z,y),\,C_{\rho}\,\max\{C_\rho,1+\kappa\}r\Bigr\}
=C_{\rho}\,\rho_{\#}(z,y),
\end{eqnarray}
where the last equality is necessarily true if we take 
$c_o>C_\rho^2\max\{C_\rho,1+\kappa\}$
(given the nature of the left-most side of \eqref{kbFF-14}).
Consequently, for this choice of $c_o$, \eqref{kbFF-14} gives that 
\begin{eqnarray}\label{kbD-5}
\rho_{\#}(y,z)\geq\frac{c_o}{C_\rho}r>\frac{c_o}{C_\rho}\delta_E(y),\qquad\forall\, 
z\in E\setminus B_{\rho_{\#}}({w},c_or)
\end{eqnarray}
hence, if we also assume $c_o\geq C_\rho(1+\kappa)^2$, then 
\begin{eqnarray}\label{kbD-5BB2}
{\rm dist}_{\rho_{\#}}\bigl(y,E\setminus B_{\rho_{\#}}({w},c_or)\bigr)
\geq(1+\kappa)^2\delta_E(y).
\end{eqnarray}
Together, \eqref{k-R446} and \eqref{kbD-5BB2} allow us to conclude that 
if $c_o>\max\bigl\{ C_\rho(1+\kappa)^2,C_\rho^3,C_\rho^2(1+\kappa)\bigr\}$ then
\begin{eqnarray}\label{kbD-5BB3}
{\rm dist}_{\rho_{\#}}\bigl(y,E\cap B_{\rho_{\#}}({w},c_or)\bigr)
\leq(1+\kappa)^{-1}
{\rm dist}_{\rho_{\#}}\bigl(y,E\setminus B_{\rho_{\#}}({w},c_or)\bigr).
\end{eqnarray}
In light of \eqref{3.2.64}, we deduce from \eqref{kbD-5BB3} that
$y\in {\mathcal{T}}_\kappa\bigl(E\cap B_{\rho_{\#}}({w},c_or)\bigr)$ when $c_o$ is chosen as above. This completes the proof of \eqref{kbFF-11}.

Combining  \eqref{kbFF-10} with \eqref{kbFF-11} (the latter applied with 
$B_j$ in place of $\Delta(x_0,r)$), and keeping in mind that
\begin{eqnarray}\label{kbFF-SA+ii}
\mbox{$\sigma(B_j)\approx\sigma\bigl(E\cap B_{\rho_{\#}}({w},c_or_j)\bigr)$, 
uniformly in $j\in{\mathbb{N}}$ and ${w}\in B_j$},
\end{eqnarray}
which is a consequence of \eqref{Doub-2}, we may then estimate 
\begin{eqnarray}\label{kbFF-15}
\tfrac{1}{\sigma(B_j)}\int_{B_j}
({\mathscr{A}}_{q,\kappa}u_2)(x)^q\,d\sigma(x)
&\leq & \tfrac{C}{\sigma(B_j)}
\int\limits_{y\in{\mathcal{F}}_\kappa(B_j),\,\delta_E(y)<r_j}|u(y)|^q
\sigma(\pi^\kappa_y)\,d\mu(y)
\nonumber\\[4pt]
&\leq & \tfrac{C}{\sigma\bigl(E\cap B_{\rho_{\#}}({w},c_or_j)\bigr)}
\int_{{\mathcal{T}}_\kappa\bigl(E\cap B_{\rho_{\#}}({w},c_or_j)\bigr)}
|u(y)|^q\sigma(\pi^\kappa_y)\,d\mu(y)
\nonumber\\[4pt]
&\leq & C\,\inf_{{w}\in B_j}
\bigl[({\mathfrak{C}}_{q,\kappa}u)({w})\bigr]^q
\leq C\gamma^q\lambda^q,
\end{eqnarray}
where for the last inequality in \eqref{kbFF-15} we have used the assumption
\eqref{G.bbn-88}. In concert with Tschebyshev's inequality, 
\eqref{kbFF-15} gives that
\begin{eqnarray}\label{kbFF-16}
\sigma\bigl(\{x\in B_j:\,({\mathscr{A}}_{q,\kappa}u_2)(x)>\lambda\}\bigr)
\leq C\gamma^q\sigma(B_j),
\end{eqnarray}
for some $C\in(0,\infty)$ independent of $\gamma\in(0,1]$ and $j\in{\mathbb{N}}$.
Also, in view of \eqref{kbFF-9}, we obtain
\begin{eqnarray}\label{ygfgf-654}
\bigl\{x\in B_j:\,({\mathscr{A}}_{q,\kappa}u)(x)>2c_q\lambda\bigr\}
\subseteq\bigl\{x\in B_j:\,({\mathscr{A}}_{q,\kappa}u_2)(x)>\lambda\bigr\},
\end{eqnarray}
since pointwise on $E$ we have ${\mathscr{A}}_{q,\kappa}u\leq c_q\bigl({\mathscr{A}}_{q,\kappa}u_1
+{\mathscr{A}}_{q,\kappa}u_2\bigr)$, where $c_q$ is as in \eqref{CPP-77}. Combined with \eqref{kbFF-16}, this yields the inequality in \eqref{kbFF-5}. The proof of part {\it (1)} of the theorem is then complete, provided we dispense with 
the additional hypothesis in \eqref{NON-ZE78}. To do this, we distinguish two cases. 

\vskip 0.08in
\noindent{\tt Case~I:} {\it Assume that ${\rm diam}_{\rho}(E)=\infty$}. 
An inspection of the proof reveals that estimate \eqref{kbFF-3BB} has only been 
utilized with $u_j$ (from \eqref{kbFF-3CC}) in place of $u$. As such, we only need 
to know that $\bigl\{x\in E:\,({\mathscr{A}}_{q,\kappa'}u_j)(x)>\lambda\bigr\}$ 
is a proper subset of $E$ for each $j\in{\mathbb{N}}$ and each $\lambda>0$.
However, in the case we are currently considering, this follows by observing that,
on the one hand, $\sigma(E)=\infty$ by \eqref{DIA-MEA}, while on the other hand
$\sigma\bigl(\bigl\{x\in E:\,({\mathscr{A}}_{q,\kappa'}u_j)(x)>\lambda\bigr\}\bigr)
<\infty$ by \eqref{kbFF-3DD2} and Tschebyshev's inequality.

\vskip 0.08in
\noindent{\tt Case~II:} {\it Assume that ${\rm diam}_{\rho}(E)<\infty$}. 
Recall from \eqref{DIA-MEA} that this entails $\sigma(E)<\infty$, and set 
$R:={\rm diam}_{\rho_{\#}}(E)\in(0,\infty)$. For some positive, small number
$\varepsilon_o$, to be specified later, decompose
$|u|=u'+u'':=|u|{\mathbf{1}}_{\{\delta_E(\cdot)<\varepsilon_o R\}}
+|u|{\mathbf{1}}_{\{\delta_E(\cdot)\geq\varepsilon_o R\}}$.
Hence, $u',u''$ are $\mu$-measurable and $0\leq u',u''\leq |u|$. 
Note that for each $x\in E$, \eqref{ktEW-7} gives 
\begin{eqnarray}\label{kjvc+yh}
({\mathfrak{C}}_{q,\kappa}u'')(x) &\geq & \Bigl(\tfrac{1}{\sigma(E)}
\int_{{\mathscr{X}}\setminus E}u''(y)^q\sigma(\pi^\kappa_y)\,d\mu(y)\Bigr)^{1/q}
\nonumber\\[4pt]
&\geq & c\Bigl(\int\limits_{\stackrel
{y\in{\mathscr{X}}\setminus E}{\delta_E(y)\geq\varepsilon_o R}}
u''(y)^q\,d\mu(y)\Bigr)^{1/q}\geq c({\mathscr{A}}_{q,\kappa}u'')(x),
\end{eqnarray}
by taking $r>R$ in \eqref{ktEW-7UU} and recalling $(iv)$ in Lemma~\ref{T-LL.2},
and observing that there exists a constant $C\in(0,\infty)$ with the property that 
for each $y\in{\mathscr{X}}\setminus E$ we have (with $y_\ast$ 
and $\epsilon$ as in Lemma~\ref{lbDV})
\begin{eqnarray}\label{kjvc+yhZZ}
\sigma(\pi^\kappa_y)\geq
\sigma\bigl(E\cap B_{\rho_{\#}}(y_\ast,\epsilon\delta_E(y))\bigr)
\geq \sigma\bigl(E\cap B_{\rho_{\#}}(y_\ast,\epsilon\varepsilon_o R)\bigr)
\geq C\sigma(E),
\end{eqnarray}
where the last inequality is a consequence of the doubling condition on $\sigma$.
In turn, \eqref{kjvc+yh} and the monotonicity of the Carleson operator allow us to write 
\begin{eqnarray}\label{kjvc+yh2}
\|{\mathscr{A}}_{q,\kappa}u''\|_{L^p(E,\sigma)}
\leq C\|{\mathfrak{C}}_{q,\kappa}u''\|_{L^p(E,\sigma)}
\leq C\|{\mathfrak{C}}_{q,\kappa}u\|_{L^p(E,\sigma)}.
\end{eqnarray}
To proceed, set $\varepsilon_o:=\tfrac{1}{4C_\rho(1+\kappa')}$ and fix $x_1,x_2\in E$
satisfying $\rho_{\#}(x_1,x_2)>R/2$. We claim that these choices guarantee that 
\begin{eqnarray}\label{yFg-H}
\Gamma_{\kappa'}(x_1)\cap\Gamma_{\kappa'}(x_2)\subseteq
\{x\in{\mathscr{X}}\setminus E:\,\delta_E(x)>\varepsilon_0 R\}.
\end{eqnarray}
Indeed, if $y\in \Gamma_{\kappa'}(x_1)\cap\Gamma_{\kappa'}(x_2)$ then 
$\rho_{\#}(y,x_j)<(1+\kappa')\delta_E(y)$ for $j=1,2$ and we have
$R/2<\rho_{\#}(x_1,x_2)\leq C_\rho\max\{\rho_{\#}(y,x_1),\rho_{\#}(y,x_2)\}
<C_\rho(1+\kappa')\delta_E(y)=\tfrac{1}{4\varepsilon_o}\delta_E(y)$, which shows
that the inclusion in \eqref{yFg-H} is true. If we now further decompose 
\begin{eqnarray}\label{yFg-Hii}
u'=u'_1+u'_2:=u'{\mathbf{1}}_{\Gamma_{\kappa'}(x_1)}
+u'\bigl(1-{\mathbf{1}}_{\Gamma_{\kappa'}(x_1)}\bigr)
\end{eqnarray}
then $0\leq u'_1,u'_2\leq u'$, and both $u'_1,u'_2$ are $\mu$-measurable.
Moreover, due to \eqref{yFg-H} and the fact that $u_1$ has support contained in the 
set $\{\delta_E(\cdot)<\varepsilon_o R\}$, we also obtain that 
$({\mathscr{A}}_{q,\kappa'}u'_1)(x_2)=0$ and $({\mathscr{A}}_{q,\kappa'}u'_2)(x_1)=0$.
The latter imply that the sets constructed according to the same recipe as  
${\mathcal{O}}_\lambda$ in \eqref{xfbl} but with $u$ replaced by either $u'_1$ 
or $u'_2$, are proper subsets of $E$ for every $\lambda>0$. 
Hence, hypothesis \eqref{NON-ZE78} holds for each of the functions $u'_1$, $u'_2$.
As such, the first part of the proof gives that \eqref{kbFF-3UU} holds with $u$ 
replaced by either $u'_1$ or $u'_2$. In concert, these give 
\begin{eqnarray}\label{kjvc+yh2ii}
\|{\mathscr{A}}_{q,\kappa}u'\|_{L^p(E,\sigma)}
&\leq & C\|{\mathscr{A}}_{q,\kappa}u'_1\|_{L^p(E,\sigma)}
+C\|{\mathscr{A}}_{q,\kappa}u'_2\|_{L^p(E,\sigma)}
\\[4pt]
&\leq & C\|{\mathfrak{C}}_{q,\kappa}u'_1\|_{L^p(E,\sigma)}
+C\|{\mathfrak{C}}_{q,\kappa}u'_2\|_{L^p(E,\sigma)}
\leq C\|{\mathfrak{C}}_{q,\kappa}u\|_{L^p(E,\sigma)}.
\nonumber
\end{eqnarray}
Together with \eqref{kjvc+yh2}, this then yields \eqref{kbFF-3UU} for the original 
function $u$. This finishes the treatment of Case~II and completes the proof of 
the estimate in part {\it (1)} of the theorem. 

Moving on to the proof of part {\it (2)}, the key step is establishing 
the pointwise estimate 
\begin{eqnarray}\label{KFf-1}
({\mathfrak{C}}_{q,\kappa}u)(x_0)
\leq C\bigl[M_E({\mathscr{A}}_{q,\kappa}u)^q(x_0)\bigr]^{\frac{1}{q}},
\qquad\forall\,x_0\in E,
\end{eqnarray}
for some $C\in(0,\infty)$ depending only on $\kappa,p,q$ and geometrical characteristics 
of the ambient space. To justify this, fix $r>0$, and let $\Delta$ be a ball of radius 
$r$ in $(E,(\rho|_{E})_{\#})$. Then, upon recalling \eqref{Mix+FR}, 
$(ii)$ in Lemma~\ref{T-LL.2}, and \eqref{3.2.BN}, we may write 
\begin{eqnarray}\label{kbFF-1B}
\int_{\Delta}({\mathscr{A}}_{q,\kappa}u)(x)^q\,d\sigma(x)
&=& \int_{{\mathcal{F}}_\kappa(\Delta)}|u(y)|^q
\sigma\bigl(\Delta\cap\pi^\kappa_y\bigr)\,d\mu(y)
\nonumber\\[4pt]
&\geq & \int_{{\mathcal{T}}_\kappa(\Delta)}|u(y)|^q
\sigma\bigl(\Delta\cap \pi^\kappa_y\bigr)\,d\mu(y)
\nonumber\\[4pt]
&=& \int_{{\mathcal{T}}_\kappa(\Delta)}|u(y)|^q
\sigma\bigl(\pi^\kappa_y\bigr)\,d\mu(y).
\end{eqnarray}
Now \eqref{KFf-1} follows from \eqref{kbFF-1B} by dividing the latter inequality by $\sigma(\Delta)$ and taking the supremum 
over all $\Delta$'s containing an arbitrary given point $x_0\in E$. 

With the pointwise estimate \eqref{KFf-1} in hand, whenever $p\in(q,\infty)$ 
and $r\in(0,\infty]$ we may use the boundedness of $M_E$ on the Lorentz space
$L^{p/q,r/q}(E,\sigma)$, which holds since $p/q>1$, and the general fact that 
for each $\alpha>0$ we have 
\begin{eqnarray}\label{Td-sUU}
\||f|^\alpha\|_{L^{p,r}(E,\sigma)}=C(p,r,\alpha)
\|f\|^{\alpha}_{L^{p\alpha,r\alpha}(E,\sigma)},
\end{eqnarray}
in order to write 
\begin{eqnarray}\label{KFf-1A}
\|{\mathfrak{C}}_{q,\kappa}u\|_{L^{p,r}(E,\sigma)}
&\leq & C\bigl\|[M_E({\mathscr{A}}_{q,\kappa}u)^q]^\frac{1}{q}\bigr\|_{L^{p,r}(E,\sigma)}
=C\|M_E({\mathscr{A}}_{q,\kappa}u)^q\|_{L^{p/q,r/q}(E,\sigma)}^\frac{1}{q}
\nonumber\\[4pt]
&\leq & C\|({\mathscr{A}}_{q,\kappa}u)^q\|_{L^{p/q,r/q}(E,\sigma)}^\frac{1}{q}
=C\|{\mathscr{A}}_{q,\kappa}u\|_{L^{p,r}(E,\sigma)},
\end{eqnarray}
as required.

There remains to observe that the two estimates in {\it (3)} are obtained
by a computation similar to \eqref{KFf-1A} that is based on \eqref{KFf-1}, 
the weak-$(1,1)$ boundedness of $M_E$, and the boundedness of $M_E$ on 
$L^\infty(E,\sigma)$. This finishes the proof of the theorem.
\end{proof}

\begin{remark}\label{yafg-76r}
The case $p=q=r$ of part {\it (2)} of Theorem~\ref{AsiC}, which corresponds 
to the estimate $\|{\mathfrak{C}}_{p,\kappa}u\|_{L^p(E,\sigma)}
\leq C\|{\mathscr{A}}_{p,\kappa}u\|_{L^p(E,\sigma)}$, fails in general. 
A counterexample in Euclidean space when $p=2$ is given in the remarks 
stated below Theorem~3 of~\cite{CoMeSt}.
\end{remark}

\subsection{Weak $L^p$ square function estimates imply $L^2$ square function estimates}
\label{SSect:5.3}

We are now in a position to consider $L^p$ versions of the $L^2$ square function estimates considered in Section~\ref{Sect:3} for integral operators $\Theta_E$. 
The main result is that $L^2$ square function estimates follow automatically from 
weak $L^p$ square function estimates for any $p\in(0,\infty)$. This is stated in Theorem~\ref{VGds-L2XXX} below. The result is achieved by combining the $T(1)$ 
theorem in Theorem~\ref{SChg} with a weak type John-Nirenberg lemma for Carleson 
measures based on Lemma~2.14 in \cite{AHLT} (see also \cite[Lemma IV.1.12]{DaSe93} 
for a similar result). 

\begin{theorem}\label{VGds-L2XXX} 
Let $0<d<m<\infty$. Assume that $({\mathscr{X}},\rho,\mu)$
is an $m$-dimensional {\rm ADR} space, $E$ is a closed subset 
of $({\mathscr{X}},\tau_\rho)$, and $\sigma$ is a Borel regular measure 
on $(E,\tau_{\rho|_{E}})$ with the property that $(E,\rho\bigl|_E,\sigma)$ is a 
$d$-dimensional {\rm ADR} space. Finally, suppose that $\Theta$ is the integral 
operator defined in \eqref{operator} with a kernel $\theta$ as in \eqref{K234}, 
\eqref{hszz}, \eqref{hszz-3}. 

Then whenever $\kappa,p,C_o\in(0,\infty)$ are such that for every 
surface ball $\Delta\subseteq E$ (cf. \eqref{ktEW-7UU}) 
\begin{eqnarray}\label{dtbh-L2iii}
\hskip -0.20in
\sigma\left(\Bigl\{x\in E:\,
\int_{\Gamma_{\kappa}(x)}|(\Theta{\mathbf{1}}_{\Delta})(y)|^2
\,\delta_E(y)^{2\upsilon-m}\,d\mu(y)>\lambda^2\Bigr\}\right)
\leq C_o\lambda^{-p}\sigma(\Delta),\quad\forall\,\lambda>0,
\end{eqnarray}
there exists some $C\in(0,\infty)$ which depends only on $\kappa,p,C_o$ 
and finite positive background constants (including ${\rm diam}_\rho(E)$ in the case when 
$E$ is bounded) with the property that 
\begin{eqnarray}\label{k-tSSiii}
\int\limits_{\mathscr{X}\setminus E}
|(\Theta f)(x)|^2\delta_E(x)^{2\upsilon-(m-d)}\,d\mu(x)
\leq C\int_E|f(x)|^2\,d\sigma(x),\qquad\forall\,f\in L^2(E,\sigma).
\end{eqnarray}
\end{theorem}

The requirement in \eqref{dtbh-L2iii} is actually less restrictive than a weak $L^p$ square function estimate. In particular, it is satisfied whenever the following weak $L^p$ square function estimate holds for every $f\in L^p(E,\sigma)$:
\begin{eqnarray}\label{eqrem}
\hskip -0.30in
\sup_{\lambda>0}\left[\lambda\cdot
\sigma\Bigl(\Bigl\{x\in E:\int_{\Gamma_{\kappa}(x)}|(\Theta f)(y)|^2
\delta_E(y)^{2\upsilon-m}\,d\mu(y)>\lambda^{2}\Bigr\}\Bigr)^{1/p}\right]
\leq C_o\|f\|_{L^{p}(E,\sigma)}.
\end{eqnarray}
Indeed, \eqref{dtbh-L2iii} follows by specializing \eqref{eqrem} to the case when 
$f={\mathbf{1}}_{\Delta}$ for an arbitrary surface ball $\Delta\subseteq E$.

To prove Theorem~\ref{VGds-L2XXX}, we need only set $q=2$ in Proposition~\ref{VGds-L2} below to obtain a Carleson measure estimate, and then apply the $T(1)$ theorem for square functions in Theorem~\ref{SChg}. Therefore, the remainder of this subsection is dedicated to the proof of the following proposition.

\begin{proposition}\label{VGds-L2} 
Retain the same background hypotheses as in the statement of Theorem~\ref{VGds-L2XXX}.
In this context, let ${\mathbb{D}}(E)$ denote a dyadic cube structure on $E$, 
consider a Whitney covering ${\mathbb{W}}_\lambda(\mathscr{X}\setminus E)$ 
of $\mathscr{X}\setminus E$ as in Lemma~\ref{Lem:CQinBQ-N} and, corresponding 
to these, recall the dyadic Carleson tents from \eqref{gZSZ-3}.
Then whenever $\kappa,p,q,C_o\in(0,\infty)$ are such that for every 
surface ball $\Delta\subseteq E$ there holds
\begin{eqnarray}\label{dtbh-L2}
\hskip -0.20in
\sigma\left(\Bigl\{x\in E:\,
\int_{\Gamma_{\kappa}(x)}|(\Theta{\mathbf{1}}_{\Delta})(y)|^q
\,\delta_E(y)^{q\upsilon-m}\,d\mu(y)>\lambda^q\Bigr\}\right)
\leq C_o\lambda^{-p}\sigma(\Delta),\quad\forall\,\lambda>0,
\end{eqnarray}
there exists some $C\in(0,\infty)$ which depends only on $\kappa,p,q,C_o$ 
and finite positive background constants with the property that 
\begin{eqnarray}\label{k-tSS}
\sup_{Q\in{\mathbb{D}}(E)}
\Bigl(\tfrac{1}{\sigma(Q)}\int_{T_E(Q)}|(\Theta 1)(x)|^q
\delta_E(x)^{q\upsilon-(m-d)}\,d\mu(x)\Bigr)\leq C.
\end{eqnarray}
\end{proposition}

In preparation for presenting the proof of Proposition~\ref{VGds-L2} we discuss a 
couple of auxiliary results. The first such result is a variation on the 
theme of Whitney decomposition discussed in Proposition~\ref{H-S-Z}. 

\begin{lemma}\label{PropW-2}
Let $(E,\rho,\sigma)$ be a space of homogeneous type with the property that the
measure $\sigma$ is Borel regular, and let ${\mathbb{D}}(E)$ be a collection 
of dyadic cubes as in Proposition~\ref{Diad-cube}. Also, suppose that 
${\mathcal{O}}$ is an open subset of $(E,\tau_\rho)$ with the property that 
$\bigl({\mathcal{O}},\rho\bigl|_{{\mathcal{O}}},\sigma\lfloor{\mathcal{O}}\bigr)$ 
is a space of homogeneous type. Fix $\lambda\in(1,\infty)$ and suppose $\Omega$ 
is an open, proper, non-empty subset of ${\mathcal{O}}$. Then there exist
$\varepsilon\in(0,1)$, $N\in{\mathbb{N}}$, $\Lambda\in(\lambda,\infty)$ and a subset 
${\mathcal{W}}\subseteq{\mathbb{D}}(E)$ such that the following properties are satisfied:
\begin{enumerate}
\item[(1)] $Q\subseteq\Omega$ for every $Q\in{\mathcal{W}}$ and 
$\sigma\bigl(\Omega\setminus\bigcup_{Q\in{\mathcal{W}}}Q\bigr)=0$;
\item[(2)] $Q\cap Q'=\emptyset$ for every $Q,Q'\in{\mathcal{W}}$ with $Q\not=Q'$;
\item[(3)] for every $x\in\Omega$, the cardinality of the set 
$\bigl\{Q\in{\mathcal{W}}:\,
B_\rho\bigl(x,\varepsilon\,{\rm dist}_\rho(x,{\mathcal{O}}\setminus\Omega)\bigr)
\cap Q\not=\emptyset\bigr\}$ is at most $N$;
\item[(4)] $\lambda Q\subseteq\Omega$ and 
$\Lambda Q\cap[{\mathcal{O}}\setminus\Omega]\not=\emptyset$ 
for every $Q\in{\mathcal{W}}$;
\item[(5)] $\ell(Q)\approx\ell(Q')$ uniformly for  $Q,Q'\in{\mathcal{W}}$ such
that $\lambda Q\cap\lambda Q'\not=\emptyset$;
\item[(6)] $\sum\limits_{Q\in{\mathcal{W}}}{\mathbf{1}}_{\lambda Q}\leq N$.
\end{enumerate}
\end{lemma}

\begin{proof}
Given $\lambda\in(1,\infty)$, apply Proposition~\ref{H-S-Z} to the 
open, proper, non-empty subset $\Omega$ of the space of homogeneous type
$\bigl({\mathcal{O}},\rho\bigl|_{{\mathcal{O}}},\sigma\lfloor{\mathcal{O}}\bigr)$.
This guarantees the existence of parameters $\varepsilon\in(0,1)$, $N\in{\mathbb{N}}$,
$\Lambda\in(\lambda,\infty)$, as well as a covering of $\Omega$ with balls
$\Omega=\bigcup_{j\in{\mathbb{N}}}\bigl({\mathcal{O}}\cap B_\rho(x_j,r_j)\bigr)$ 
such that the analogues of the properties {\it (1)-(4)} in Proposition~\ref{H-S-Z} 
hold in the current setting. Next, for each $j\in{\mathbb{N}}$ consider
\begin{eqnarray}\label{eEE}
I_j:=\bigl\{Q\in{\mathbb{D}}(E):\,\ell(Q)\approx r_j\mbox{ and }
Q\cap B_\rho(x_j,r_j)\not=\emptyset\bigr\},
\end{eqnarray}
and define ${\mathcal{W}}:=\bigcup_{j\in{\mathbb{N}}}I_j$ thinned out, so that 
$Q\cap Q'=\emptyset$ for every $Q,Q'\in{\mathcal{W}}$, $Q\not=Q'$. Granted the 
properties the families ${\mathbb{D}}(E)$ and $\{B_\rho(x_j,r_j)\}_{j\in{\mathbb{N}}}$
satisfy (as listed in Proposition~\ref{Diad-cube} and Proposition~\ref{H-S-Z}) and 
given the nature of the construction of the family ${\mathcal{W}}$, it follows 
that properties {\it (1)-(6)} in the statement of the current lemma hold 
for the family ${\mathcal{W}}$.
\end{proof}

We now state the aforementioned weak type John-Nirenberg lemma for Carleson measures, cf. \cite[Lemma 2.14]{AHLT} for a result similar in spirit in the Euclidean setting.

\begin{lemma}\label{SQ-lema}
Retain the same background hypotheses as in the statement of Theorem~\ref{VGds-L2XXX}. 
In this context, fix two finite numbers $\kappa,\eta>0$, an index $q\in(0,\infty)$ 
and, for each $Q\in{\mathbb{D}}(E)$, define 
\begin{eqnarray}\label{SQ-1}
S_Q(x):=\Bigl(\int\limits_{\stackrel{y\in\Gamma_\kappa(x)}
{\rho_{\#}(x,y)<\eta\ell(Q)}}|(\Theta 1)(y)|^q\delta_E(y)^{q\upsilon-m}
\,d\mu(y)\Bigr)^{\frac{1}{q}},
\qquad\forall\,x\in E.
\end{eqnarray}
Assuming that $\eta$ is sufficiently large (depending only on geometry) and granted 
that there exist two parameters $N\in(0,\infty)$ and $\beta\in(0,1)$ such that 
\begin{eqnarray}\label{SQ-2}
\sigma\Bigl(\bigl\{x\in Q:\,S_Q(x)>N\bigr\}\Bigr)<(1-\beta)\sigma(Q),
\qquad\forall\,Q\in{\mathbb{D}}(E),
\end{eqnarray}
then one may find $C\in(0,\infty)$ depending only on geometry, the estimates 
satisfied by the kernel $\theta$, and $\kappa,\eta$, with the property that 
\begin{eqnarray}\label{k-tSS.22}
\sup_{Q\in{\mathbb{D}}(E)}
\Bigl(\tfrac{1}{\sigma(Q)}\int_{T_E(Q)}|(\Theta 1)(x)|^q
\delta_E(x)^{q\upsilon-(m-d)}\,d\mu(x)\Bigr)\leq\beta^{-1}(C+N^q).
\end{eqnarray}
\end{lemma}

\begin{proof}
For each $i\in{\mathbb{N}}$, let $\Theta_i$ be as in \eqref{LIH-3} and associate
to $\Theta_i$ the function $S^i_Q$, much as $S_Q$ is associated to $\Theta$. 
Note that $S^i_Q$ and $S_Q$ depend on the constant $\kappa$ defining 
$\Gamma_\kappa$. We fix $\widetilde{\kappa}\in(0,\kappa)$ to be specified later 
and we use the notation $S^i_{Q,\widetilde{\kappa}}$ for the function defined 
similarly to $S^i_Q$ but with $\widetilde{\kappa}$ in place of $\kappa$. 
Also, with $N\in(0,\infty)$ and $\beta\in(0,1)$ satisfying \eqref{SQ-2}, define 
\begin{eqnarray}\label{SQ-55}
\Omega_Q^{N,i}:=\bigl\{x\in Q:\,S^i_{Q}(x)>N\bigr\},
\qquad\forall\,Q\in{\mathbb{D}}(E),\quad\forall\,i\in{\mathbb{N}}.
\end{eqnarray}
Since, thanks to Lemma~\ref{semi-cont}, for each $Q\in{\mathbb{D}}(E)$ and 
$i\in{\mathbb{N}}$ the function $S^i_Q$ is lower semi-continuous,  
from \eqref{SQ-55}, \eqref{SQ-2} and the fact that $S^i_Q\leq S_Q$ pointwise in $Q$ 
we deduce that
\begin{eqnarray}\label{SQ-56MM}
\forall\,Q\in{\mathbb{D}}(E),\quad\forall\,i\in{\mathbb{N}},\quad
\mbox{ $\Omega_Q^{N,i}$ is an open, proper subset of $Q$}.
\end{eqnarray}
To proceed, consider 
\begin{eqnarray}\label{SQ-56}
A^i:=\sup_{Q\in{\mathbb{D}}(E)}
\Bigl(\tfrac{1}{\sigma(Q)}\int_Q(S^i_{Q,\widetilde{\kappa}}(x))^q\,d\sigma(x)\Bigr),
\qquad\forall\,i\in{\mathbb{N}}.
\end{eqnarray}
Then, based on \eqref{Mix+FR} (applied to the function 
$u:=|(\Theta_i 1)|^q\delta_E^{q\upsilon-m}
{\mathbf{1}}_{\{{\rm dist}_{\rho_{\#}}(\cdot,Q)\leq C\ell(Q)\}}$) we may write, 
with $x_Q$ denoting the center of $Q\in{\mathbb{D}}(E)$, 
\begin{eqnarray}\label{SQ-57A}
\hskip -0.30in
\int_Q(S^i_{Q,\widetilde{\kappa}}(x))^q\,d\sigma(x)
& \leq &
\int_Q\Bigl(
\int\limits_{\stackrel{y\in\Gamma_{\widetilde{\kappa}}(x)}
{{\rm dist}_{\rho_{\#}}(y,Q)\leq C\ell(Q)}}
|(\Theta_i 1)(y)|^q\delta_E(y)^{q\upsilon-m}\,d\mu(y)\Bigr)\,d\sigma(x)
\nonumber\\[4pt]
& \leq & \int\limits_{\stackrel{y\in{\mathcal{F}}_{\widetilde{\kappa}}(Q)}
{{\rm dist}_{\rho_{\#}}(y,Q)\leq C\ell(Q)}}
|(\Theta_i 1)(y)|^q\delta_E(y)^{q\upsilon-m}
\sigma\bigl(Q\cap\pi_y^{\widetilde{\kappa}}\bigr)\,d\mu(y)
\nonumber\\[4pt]
& \leq & C\int\limits_{B_{\rho_{\#}}(x_Q,C\ell(Q))}
|(\Theta_i 1)(y)|^q\delta_E(y)^{q\upsilon-(m-d)}\,d\mu(y),
\end{eqnarray}
where the last step in \eqref{SQ-57A} uses the inequality 
$\sigma(Q\cap\pi_y^{\widetilde{\kappa}})\leq C\delta_E(y)^d$ which, in turn, 
is a consequence of Lemma~\ref{lbDV}, the fact that $(E,\rho\bigl|_E,\sigma)$ 
is a $d$-dimensional {\rm ADR} space, and the observation that 
$\delta_E(y)\leq{\rm dist}_{\rho_{\#}}(y,Q)\leq C\,\ell(Q)\leq C\,{\rm diam}_{\rho}(E)$
on the domain of integration (of the third integral in \eqref{SQ-57A}). 
Moreover, reasoning as in \eqref{LIH-5} we obtain 
\begin{eqnarray}\label{LIH-5B}
\int_{B_{\rho_{\#}}(x_Q,C\ell(Q))}
|(\Theta_i 1)(y)|^q\delta_E(y)^{q\upsilon-(m-d)}\,d\mu(y)
\leq Ci^{2q\upsilon}\ell(Q)^d\leq Ci^{2q\upsilon}\sigma(Q)
\end{eqnarray}
for every $Q\in{\mathbb{D}}(E)$, so by combining \eqref{SQ-57A} 
and \eqref{LIH-5B} we arrive at the conclusion that 
\begin{eqnarray}\label{SQ-57}
\tfrac{1}{\sigma(Q)}\int_Q(S^i_{Q,\widetilde{\kappa}}(x))^q\,d\sigma(x)\leq C(i)<\infty,
\end{eqnarray}
for each cube $Q\in{\mathbb{D}}(E)$ and each $i\in{\mathbb{N}}$. 
Thus, in particular, $A^i<\infty$ for every $i\in{\mathbb{N}}$. 

At this stage in the proof, the incisive step is the claim that, in fact, 
\begin{eqnarray}\label{SQ-58}
\begin{array}{c}
\exists\,A\in(0,\infty)\quad\mbox{independent of $i$ such that }
\\[6pt]
\tfrac{1}{\sigma(Q)}\int_Q(S^i_{Q,\widetilde{\kappa}}(x))^q\,d\sigma(x)\leq A,
\quad\forall\,Q\in{\mathbb{D}}(E).
\end{array}
\end{eqnarray}
In the process of proving this claim we shall show that one can take 
$A:=\beta^{-1}(C+N^q)$ where $C\in(0,\infty)$ is a constant which depends 
only on geometry, the estimates satisfied by $\theta$, and $\kappa$.
To get started, fix $i\in{\mathbb{N}}$ and first observe that if 
$Q\in{\mathbb{D}}(E)$ is such that $\Omega_Q^{N,i}=\emptyset$, then 
$S^i_{Q,\widetilde{\kappa}}\leq S^i_Q\leq N$ on $Q$, hence for such $Q$'s 
\eqref{SQ-58} will hold if we impose the condition that $A\geq N^q$. 
Next, let $Q\in{\mathbb{D}}(E)$ be such that $\Omega_Q^{N,i}\not=\emptyset$. 
Then, thanks to \eqref{SQ-56MM}, it follows that $\Omega_Q^{N,i}$ is 
an open, nonempty, proper subset of $Q$. Recall from \eqref{ihgc} that 
$\bigl(Q,\rho|_{Q},\sigma\lfloor{Q}\bigr)$ is a space of homogeneous 
type and the doubling constant of the measure $\sigma\lfloor{Q}$ is 
independent of $Q$. Then there exists a Whitney decomposition of 
$\Omega_Q^{N,i}$ relative to $Q$ via dyadic cubes $\{Q_k\}_{k\in I^{N,i}_Q}$ 
as described in Lemma~\ref{PropW-2} (used with ${\mathcal{O}}:=Q$ and 
$\Omega:=\Omega_Q^{N,i}$). Introducing
$F^{N,i}_Q:=Q\setminus\Omega_Q^{N,i}$ we may then write 
\begin{eqnarray}\label{SQ-59}
\hskip -0.20in
\int_Q(S^i_{Q,\widetilde{\kappa}}(x))^q\,d\sigma(x)
=\int_{F_Q^{N,i}}(S^i_{Q,\widetilde{\kappa}}(x))^q\,d\sigma(x)
+\sum\limits_{k\in I_Q^{N,i}}\int_{Q_k}(S^i_{Q,\widetilde{\kappa}}(x))^q\,d\sigma(x)
=:I+II.
\end{eqnarray}
Since $\widetilde{\kappa}<\kappa$ forces  
$S^i_{Q,\widetilde{\kappa}}\leq S^i_Q\leq N$ on 
$F^{N,i}_Q$, we further have 
\begin{eqnarray}\label{SQ-f77}
I\leq\int_{F_Q^{N,i}}(S^i_Q(x))^q\,d\sigma(x)\leq N^q\sigma(Q).
\end{eqnarray}
To estimate $II$, we write
\begin{eqnarray}\label{SQ-78}
&& \hskip -0.20in
II =\sum\limits_{k\in I_Q^{N,i}}\int_{Q_k}(S^i_{Q_k,\widetilde{\kappa}}(x))^q\,
d\sigma(x)
\nonumber\\[4pt]
&& \hskip 0.20in
+\sum\limits_{k\in I_Q^{N,i}}\int_{Q_k}\Bigl(\hskip -0.05in
\int\limits_{\stackrel{y\in\Gamma_{\widetilde{\kappa}}(x)}
{\eta\ell(Q_k)\leq\rho_{\#}(y,x)<\eta\ell(Q)}}
\hskip -0.30in
|(\Theta_i 1)(y)|^q\delta_E(y)^{q\upsilon-m}d\mu(y)\Bigr)\,d\sigma(x)
\nonumber\\[4pt]
&& \hskip 0.08in
=:III+IV.
\end{eqnarray}
By recalling \eqref{SQ-56}, the fact that the family $\{Q_k\}_{k\in I^{N,i}_Q}$ 
consists of pairwise disjoint cubes from ${\mathbb{D}}(E)$ contained in 
$\Omega_Q^{N,i}$, as well as assumption \eqref{SQ-2}, we have 
\begin{eqnarray}\label{SQ-79}
III\leq \sum\limits_{k\in I_Q^{N,i}}A^i\sigma(Q_k)
\leq A^i\sigma\bigl(\Omega_Q^{N,i}\bigr)\leq A^i(1-\beta)\sigma(Q).
\end{eqnarray}
Moving on, from \eqref{hszz} and \eqref{mMji} (given that $\upsilon-a>0$) we see that 
$|(\Theta_i 1)(y)|\leq\frac{C}{\delta_E(y)^\upsilon}$ for every
$y\in{\mathscr{X}}\setminus E$. Thus, if $C_0>0$ is some large finite fixed constant
which will be specified later (just below \eqref{SQ-83ii-B}, to be precise),  
and if $k\in I^{N,i}_Q$, then for each $x\in Q_k$ there holds 
\begin{eqnarray}\label{SQ-80}
&& \hskip -0.60in
\int\limits_{\stackrel{y\in\Gamma_{\widetilde{\kappa}}(x)}
{\eta\ell(Q_k)\leq\rho_{\#}(x,y)\leq C_0\ell(Q_k)}}
\hskip -0.40in
|(\Theta_i 1)(y)|^q\delta_E(y)^{q\upsilon-m}\,d\mu(y)
\leq C\int\limits_{\stackrel{y\in\Gamma_{\widetilde{\kappa}}(x)}
{\eta\ell(Q_k)\leq\rho_{\#}(x,y)\leq C_0\ell(Q_k)}}\hskip -0.40in
\frac{d\mu(y)}{\delta_E(y)^{m}}
\\[4pt]
&&\hskip 0.10in
\leq C \ell(Q_k)^{-m}\mu\Bigl(\bigl\{y\in\Gamma_{\widetilde{\kappa}}(x):\,
\eta\ell(Q_k)\leq\rho_{\#}(x,y)\leq C_0\ell(Q_k)\bigr\}\Bigr)
\leq C<\infty,
\nonumber
\end{eqnarray}
for some $C>0$ independent of $x$, $k$, $Q$ and $i$. In turn, \eqref{SQ-80} entails 
\begin{eqnarray}\label{SQ-81}
&& \hskip -0.50in
\sum\limits_{k\in I_Q^{N,i}}\int_{Q_k}\Bigl(
\int\limits_{\stackrel{y\in\Gamma_{\widetilde{\kappa}}(x)}
{\eta\ell(Q_k)\leq\rho_{\#}(x,y)\leq C_0\ell(Q_k)}}
\hskip -0.40in
|(\Theta_i 1)(y)|^q\delta_E(y)^{q\upsilon-m}\,d\mu(y)\Bigr)\,d\sigma(x)
\nonumber\\[4pt]
&&\hskip 1.50in
\leq C\sum\limits_{k\in I_Q^{N,i}}\sigma(Q_k)
\leq C\sigma\bigl(\Omega_Q^{N,i}\bigr)\leq C\sigma(Q),
\end{eqnarray}
which once again suits our purposes.
Next, since $\{Q_k\}_{k\in I^{N,i}_Q}$ is a Whitney decomposition of $\Omega_Q^{N,i}$ 
relative to $Q$, for each $k\in I^{N,i}_Q$ there exists $x_k\in F^{N,i}_Q$ such that
\begin{eqnarray}\label{SQ-MMh}
{\rm dist}_{\rho_{\#}}(x_k,Q_k)\leq c\,\ell(Q_k), 
\end{eqnarray}
for some finite $c>0$ independent of $k$, $Q$ and $i$. We now claim that there 
exits $\widetilde{\kappa}\in(0,\kappa)$ depending on the constants 
associated with the Whitney decomposition of $\Omega_Q^{N,i}$
(hence, ultimately, on finite positive geometric constants associated with $(E,\rho\bigl|_{E},\sigma)$),
as well as on $\kappa$ and the constant $C_0$, but independent of $k$, $Q$ and $i$, 
such that
\begin{eqnarray}\label{SQ-83ii}
x\in Q_k,\,\,\,y\in\Gamma_{\widetilde{\kappa}}(x)\,\,\mbox{ and }\,\,
C_0\ell(Q_k)<\rho_{\#}(x,y)\,\Longrightarrow\,
y\in\Gamma_\kappa(x_k).
\end{eqnarray}
To justify this claim, suppose that $\widetilde{\kappa}\in(0,\kappa)$ and 
fix $x\in Q_k$ along with $y\in\Gamma_{\widetilde{\kappa}}(x)$ such that 
$C_0\ell(Q_k)<{\rm dist}_{\rho_{\#}}(y,Q)$. Then 
\begin{eqnarray}\label{SQ-83ii-A}
C_0\ell(Q_k)<\rho_{\#}(y,x)<(1+\widetilde{\kappa})\delta_E(y)
<(1+\kappa)\delta_E(y).
\end{eqnarray}
Also, if we choose a finite number $\vartheta\in\bigl(0,(\log_2C_\rho)^{-1}\bigr]$, 
Theorem~\ref{JjEGh} gives that $(\rho_{\#})^\vartheta$ is a genuine distance. As such, 
we may estimate based on \eqref{SQ-MMh}, \eqref{SQ-83ii-A} and hypotheses 
\begin{eqnarray}\label{SQ-83ii-B}
\rho_{\#}(y,x_k)^\vartheta &\leq & \rho_{\#}(y,x)^\vartheta+\rho_{\#}(x,x_k)^\vartheta
<(1+\widetilde{\kappa})^\vartheta\delta_E(y)^\vartheta+c^\vartheta\ell(Q_k)^\vartheta
\nonumber\\[4pt]
&\leq & (1+\widetilde{\kappa})^\vartheta\delta_E(y)^\vartheta
+c^\vartheta\frac{(1+\kappa)^\vartheta}{C_0^\vartheta}\delta_E(y)^\vartheta
\nonumber\\[4pt]
&\leq & (1+\kappa)^\vartheta\delta_E(y)^\vartheta,
\end{eqnarray}
provided $C_0>c\Bigl[1-\Bigl(\frac{1}{1+\kappa}\Bigr)^\vartheta\Bigr]^{-1/\vartheta}$ 
and $0<\widetilde{\kappa}<(1+\kappa)
\Bigl[1-\Bigl(\frac{c}{C_0}\Bigr)^\vartheta\Bigr]^{1/\vartheta}-1$. Assuming that 
this is the case, \eqref{SQ-83ii} now follows from \eqref{SQ-83ii-B}. 

Going further, with \eqref{SQ-83ii} in hand and upon recalling that 
$S^i_{Q}(x_k)\leq N$, we may estimate 
\begin{eqnarray}\label{SQ-83}
&&\hskip -0.60in
\sum\limits_{k\in I_Q^{N,i}}\int_{Q_k}\Bigl(
\int\limits_{\stackrel{y\in\Gamma_{\widetilde{\kappa}}(x)}
{C_0\ell(Q_k)<\rho_{\#}(x,y)<\eta\ell(Q)}}
\hskip -0.40in
|(\Theta_i 1)(y)|^q\delta_E(y)^{q\upsilon-m}\,d\mu(y)\Bigr)\,d\sigma(x)
\nonumber\\[4pt]
&&\hskip 0.20in
\leq \sum\limits_{k\in I_Q^{N,i}}\int_{Q_k}\Bigl(\hskip -0.08in
\int\limits_{\stackrel{y\in\Gamma_{\kappa}(x_k)}
{\rho_{\#}(x,y)<\eta\ell(Q)}}
\hskip -0.10in
|(\Theta_i 1)(y)|^q\delta_E(y)^{q\upsilon-m}\,d\mu(y)\Bigr)\,d\sigma(x)
\nonumber\\[4pt]
&&\hskip 0.20in
=\sum\limits_{k\in I_Q^{N,i}}\int_{x\in Q_k}
(S^i_{Q}(x_k))^q \,d\sigma(x)
\leq N^q\sum\limits_{k\in I_Q^{N,i}}\sigma(Q_k)\leq N^q\sigma(Q),
\end{eqnarray}
which is of the right order.
In concert, \eqref{SQ-81}-\eqref{SQ-83} prove that there exists 
$C\in(0,\infty)$ depending only on geometry, the estimates satisfied by the 
kernel $\theta$, and $\kappa$, with the property that $IV\leq (C+N^q)\sigma(Q)$. 
In combination with \eqref{SQ-59}-\eqref{SQ-79}, this then allows us to 
conclude that
\begin{eqnarray}\label{SQ-84}
\int_Q(S^i_{Q,\widetilde{\kappa}}(x))^q d\sigma
\leq A^i(1-\beta)\sigma(Q)+(C+N^q)\sigma(Q),\qquad\forall\,Q\in{\mathbb{D}}(E).
\end{eqnarray}
In particular, if we divide \eqref{SQ-84} by $\sigma(Q)$, then take the supremum 
over $Q\in{\mathbb{D}}(E)$ we arrive at the conclusion that 
$A^i\leq A^i(1-\beta)+C$ for each $i\in{\mathbb{N}}$. 
Upon recalling that $A^i\in(0,\infty)$ for each $i\in{\mathbb{N}}$ and that 
$\beta\in(0,1)$, it follows from this that 
$\sup_{i\in{\mathbb{N}}}A^i\leq\beta^{-1}(C+N^q)<\infty$.
Hence, \eqref{SQ-58} is true. 

Consider now the function $S_{Q,\widetilde{\kappa}}$ defined analogously to 
$S_Q$ but with $\widetilde{\kappa}$ in place of $\kappa$. Given that $\lim\limits_{i\to\infty}S^i_{Q,\widetilde{\kappa}}=S_{Q,\widetilde{\kappa}}$ 
pointwise in $E$, from \eqref{SQ-58} and Lebesgue's Monotone Convergence Theorem
we may conclude that
\begin{eqnarray}\label{SQ-85}
\exists\,C\in(0,\infty)\quad\mbox{such that }\,\,
\tfrac{1}{\sigma(Q)}\int_Q(S_{Q,\widetilde{\kappa}}(x))^q\,d\sigma(x)\leq C,
\qquad\forall\,Q\in{\mathbb{D}}(E).
\end{eqnarray}
Next, observe that
\begin{eqnarray}\label{SQ-8vv}
x,y\in B_{\rho_{\#}}\bigl(x_Q,\eta C_\rho^{-1}\ell(Q)\bigr)
\,\Longrightarrow\,\rho_{\#}(x,y)\leq\eta\ell(Q).
\end{eqnarray}
Then, based on \eqref{SQ-1}, \eqref{SQ-8vv}, \eqref{Mix+FR}, 
$(ii)$ in Lemma~\ref{T-LL.2}, \eqref{3.2.BN}, and the fact that 
$\bigl(E,\rho\bigl|_{E},\sigma\bigr)$ is a $d$-dimensional {\rm ADR} space, 
we may estimate (using notation introduced in \eqref{ktEW-7UU}):
\begin{eqnarray}\label{kbF-YH.3}
&& \hskip -0.50in
\int\limits_{\Delta(x_Q,\eta C_\rho^{-1}\ell(Q))}
(S_{Q,\widetilde{\kappa}}(x))^q\,d\sigma(x)
\nonumber\\[4pt]
&& \hskip 0.25in
=\int\limits_{\Delta(x_Q,\eta C_\rho^{-1}\ell(Q))}
\Bigl(\int\limits_{\stackrel{y\in\Gamma_{\widetilde{\kappa}}(x)}
{\rho_{\#}(x,y)<\eta\ell(Q)}}
|(\Theta 1)(y)|^q\delta_E(y)^{q\upsilon-m}\,d\mu(y)\Bigr)d\sigma(x)
\nonumber\\[4pt]
&& \hskip 0.25in
\geq\int\limits_{\Delta(x_Q,\eta C_\rho^{-1}\ell(Q))}
\Bigl(\int\limits_{\stackrel{y\in\Gamma_{\widetilde{\kappa}}(x)}
{\rho_{\#}(y,x_Q)<\eta C_\rho^{-1}\ell(Q)}}
|(\Theta 1)(y)|^q\delta_E(y)^{q\upsilon-m}\,d\mu(y)\Bigr)d\sigma(x)
\nonumber\\[4pt]
&& \hskip 0.25in
=\int\limits_{\stackrel{y\in{\mathcal{F}}_{\widetilde{\kappa}}
(\Delta(x_Q,\eta C_\rho^{-1}\ell(Q)))}
{\rho_{\#}(y,x_Q)<\eta C_\rho^{-1}\ell(Q)}}
|(\Theta 1)(y)|^q\delta_E(y)^{q\upsilon-m}
\sigma\bigl(\Delta(x_Q,\eta C_\rho^{-1}\ell(Q))
\cap\pi^{\widetilde{\kappa}}_y\bigr)\,d\mu(y)
\nonumber\\[4pt]
&& \hskip 0.25in
\geq\int\limits_{\stackrel{y\in{\mathcal{T}}_{\widetilde{\kappa}}
(\Delta(x_Q,\eta C_\rho^{-1}\ell(Q)))}
{\rho_{\#}(y,x_Q)<\eta C_\rho^{-1}\ell(Q)}}
|(\Theta 1)(y)|^q\delta_E(y)^{q\upsilon-m}
\sigma\bigl(\Delta(x_Q,\eta C_\rho^{-1}\ell(Q))
\cap\pi^{\widetilde{\kappa}}_y\bigr)\,d\mu(y)
\nonumber\\[4pt]
&& \hskip 0.25in
=\int\limits_{\stackrel{y\in{\mathcal{T}}_{\widetilde{\kappa}}
(\Delta(x_Q,\eta C_\rho^{-1}\ell(Q)))}
{\rho_{\#}(y,x_Q)<\eta C_\rho^{-1}\ell(Q)}}
|(\Theta 1)(y)|^q\delta_E(y)^{q\upsilon-m}
\sigma\bigl(\pi^{\widetilde{\kappa}}_y\bigr)\,d\mu(y)
\nonumber\\[4pt]
&& \hskip 0.25in
\approx\int\limits_{\stackrel{y\in{\mathcal{T}}_{\widetilde{\kappa}}
(\Delta(x_Q,\eta C_\rho^{-1}\ell(Q)))}
{\rho_{\#}(y,x_Q)<\eta C_\rho^{-1}\ell(Q)}}
|(\Theta 1)(y)|^q\delta_E(y)^{q\upsilon-(m-d)}\,d\mu(y),
\end{eqnarray}
uniformly for $Q\in{\mathbb{D}}(E)$. Let us also observe that there exists an 
integer $M_o\in{\mathbb{N}}$ (depending only on geometry) with the property that for 
every $Q\in{\mathbb{D}}(E)$ the ball $\Delta(x_Q,\eta C_\rho^{-1}\ell(Q))$ 
may be covered by at most $M_o$ dyadic cubes of the same generation as $Q$, 
and that for every such cube $\widetilde{Q}$ there holds $S_{\widetilde{Q}}=S_Q$.
Having noticed this, we then deduce from \eqref{SQ-85}, \eqref{kbF-YH.3}, 
and \eqref{Fv-UU45} that there exists $C\in(0,\infty)$ satisfying  
\begin{eqnarray}\label{k-tSS.ps}
\frac{1}{\sigma(Q)}\int
\limits_{B_{\rho_{\#}}(x_Q,\,\eta C_\rho^{-2}\ell(Q))\setminus E}
|(\Theta 1)(x)|^q\delta_E(x)^{q\upsilon-(m-d)}\,d\mu(x)
\leq C,\qquad\forall\,Q\in{\mathbb{D}}(E).
\end{eqnarray}
With this in hand, \eqref{k-tSS.22} now follows with the help of \eqref{dFvK},
if $\eta$ is sufficiently large to begin with (depending only on geometry). 
\end{proof}

Our last auxiliary result is an estimate of geometrical nature, on 
a nontangential approach region. For a proof (and for more general 
results of this type) see \cite{MMMM-B}.

\begin{lemma}\label{bzrg1}
Let $({\mathscr{X}},\rho,\mu)$ be an $m$-dimensional {\rm ADR} space 
for some $m>0$. Assume that $E$ is a closed subset of $({\mathscr{X}},\tau_{\rho})$
with the property that there exists a Borel measure $\sigma$ on $(E,\tau_{\rho|_{E}})$
such that $\bigl(E,\rho\bigl|_{E},\sigma\bigr)$ is a $d$-dimensional {\rm ADR} 
space for some $d\geq 0$. Then for each $\kappa>0$, $\beta<m$, $M>m-\beta$, there 
exists a finite constant $C>0$ depending on $\kappa$, $M$, $\beta$, and the 
{\rm ADR} constants of ${\mathscr{X}}$ and $E$, such that
\begin{eqnarray}\label{ae4t}
\hskip -0.20in
\int_{\Gamma_{\kappa}(z)}
\frac{\delta_E(x)^{-\beta}}{\rho_{\#}(x,y)^M}\,d\mu(x)\leq C \rho(y,z)^{m-\beta-M},\quad\mbox{for all }\,\,z,y\in E\mbox{ with }z\not=y.
\end{eqnarray}
\end{lemma}

Now we are ready to proceed with the 

\vskip 0.08in
\begin{proof}[Proof of Proposition~\ref{VGds-L2}]
Based on Lemma~\ref{SQ-lema}, it suffices to prove that if the hypotheses of
Proposition~\ref{VGds-L2} are satisfied, then there exist $N<\infty$ and $\beta\in(0,1)$ 
such that \eqref{SQ-2} holds. To this end, let $N>0$ be a large finite constant, 
to be specified later, and fix an arbitrary $Q\in{\mathbb{D}}(E)$. 
Also, recall $c_q$ from \eqref{CPP-77} and fix an arbitrary number $\eta>0$. 
Then, with $S_Q$ as in \eqref{SQ-1} and some finite constant $c>0$ to be specified 
later, we may write  
\begin{eqnarray}\label{SQ-3}
&&
\hskip -0.30in
\sigma\Bigl(\bigl\{x\in Q:\,S_Q(x)>N\bigr\}\Bigr)
\nonumber\\[4pt]
&& \hskip 0.10in
\leq \sigma\Bigl(\bigl\{x\in Q:\,\Bigl(
\int\limits_{y\in\Gamma_\kappa(x),\,\rho_{\#}(x,y)<\eta\ell(Q)}
|(\Theta{\mathbf{1}}_{c\,Q})(y)|^q
\delta_E(y)^{q\upsilon-m}\,d\mu(y)\Bigr)^{\frac{1}{q}}>N/2\bigr\}\Bigr)
\nonumber\\[4pt]
&&\hskip 0.20in
+\sigma\Bigl(\bigl\{x\in Q:\,\Bigl(\int\limits_{y\in\Gamma_\kappa(x),
\,\rho_{\#}(x,y)<\eta\ell(Q)}
|(\Theta{\mathbf{1}}_{E\setminus c\,Q})(y)|^q
\delta_E(y)^{q\upsilon-m}\,d\mu(y)\Bigr)^{\frac{1}{q}}>N/2\bigr\}\Bigr)
\nonumber\\[4pt]
&& \hskip 0.10in
=:I+II,
\end{eqnarray}
where we have used the notation $c\,Q:=E\cap B_{\rho_{\#}}\bigl(x_Q,c\,\ell(Q)\bigr)$.
Note that under the assumption \eqref{dtbh-L2} (and the fact that $\sigma$ is 
doubling) we may estimate
\begin{eqnarray}\label{SQ-3B}
\hskip -0.20in
I\leq\sigma\Bigl(\bigl\{x\in Q:\,\Bigl(\int_{\Gamma_\kappa(x)}
|(\Theta{\mathbf{1}}_{c\,Q})(y)|^q\delta_E(y)^{q\upsilon-m}\,d\mu(y)\Bigr)^{\frac{1}{q}}
>N/2\bigr\}\Bigr)\leq\tfrac{C}{N^p}\sigma(Q),
\end{eqnarray}
which suits our purposes. 

Going further, select some finite constant $c_o\geq\sup_{Q'\in{\mathbb{D}}(E)}
\bigl({\rm diam}_{\rho_{\#}}(Q')/\ell(Q')\bigr)$. Given $x\in Q$ fixed, 
note that for each $y\in B_{\rho_{\#}}(x,\eta\ell(Q))$ we have
\begin{eqnarray}\label{SQ-ZZ}
\rho_{\#}(y,x_Q) &\leq & C_{\rho_{\#}}\max\,\{\rho_{\#}(y,x),\rho_{\#}(x,x_Q)\}
\\[4pt]
&\leq & C_\rho\max\,\{\eta,c_o\}\,\ell(Q)\leq c^{-1}C_\rho\max\,\{\eta,c_o\}\,
\rho_{\#}(z,x_Q),\qquad\forall\,z\in E\setminus{\mathbf{1}}_{c\,Q}.
\nonumber
\end{eqnarray}
Consequently, if $z\in E\setminus{\mathbf{1}}_{c\,Q}$, then 
\begin{eqnarray}\label{SQ-Zs}
\rho_{\#}(z,x_Q)&\leq& C_{\rho_{\#}}\max\{\rho_{\#}(z,y),\rho_{\#}(y,x_Q)\}
\\[4pt]
&\leq & C_\rho\rho_{\#}(z,y)+c^{-1}C^2_\rho\max\,\{\eta,c_o\}\,\rho_{\#}(z,x_Q),
\quad\forall\,y\in B_{\rho_{\#}}(x,\eta\ell(Q)).
\nonumber
\end{eqnarray}
Hence, choosing the finite constant $c>0$ sufficiently large so that 
$c^{-1}C^2_\rho\max\,\{\eta,c_o\}<\frac{1}{2}$ forces 
$\rho_{\#}(z,x_Q)\leq 2C_\rho\rho_{\#}(z,y)$ for all 
$y\in B_{\rho_{\#}}(x,\eta\ell(Q))$.
Making use of this, \eqref{hszz}, and \eqref{WBA} we may then write 
\begin{eqnarray}\label{SQ-4}
|(\Theta{\mathbf{1}}_{E\setminus c\,Q})(y)|
&\leq & C\int\limits_{E\setminus c\,Q}\frac{\delta_E(y)^{-a}}{\rho_{\#}(z,y)^{d+\upsilon-a}}\,d\sigma(z)
\leq C\delta_E(y)^{-a}\!\!\!\!
\int\limits_{z\in E,\,\rho_{\#}(z,x_Q)>c\,\ell(Q)}
\frac{d\sigma(z)}{\rho_{\#}(z,x_Q)^{d+\upsilon-a}}
\nonumber\\[4pt]
&\leq & C\frac{\delta_E(y)^{-a}}{\ell(Q)^{\upsilon-a}},
\qquad\forall\,y\in B_{\rho_{\#}}(x,\eta\ell(Q)).
\end{eqnarray}
Pick now $1<c_1<c_2<c$ such that there exists ${w}\in c_2Q\setminus c_1Q$
(which may be assured by further increasing $c$ if needed, given that
$(E,\rho|_{E},\sigma)$ is a $d$-dimensional {\rm ADR} space). 
Then clearly $\rho_{\#}(x,{w})\approx\ell(Q)$ and we claim that also 
\begin{eqnarray}\label{SQ-5}
\rho_{\#}(y,{w})\approx\ell(Q),\quad\mbox{ uniformly for }\,\,
y\in\Gamma_\kappa(x)\cap B_{\rho_{\#}}(x,\eta\ell(Q)).
\end{eqnarray}
Indeed, on the one hand, if the point $y\in{\mathscr{X}}$ is 
such that $\rho_{\#}(y,x)<\eta\ell(Q)$ then we obtain 
$\rho_{\#}(y,{w})\leq C_{\rho}\max\{\rho_{\#}(y,x),\rho_{\#}(x,{w})\}
\leq C\ell(Q)$. On the other hand, if we additionally know that $y\in\Gamma_\kappa(x)$, 
then $\rho_{\#}(y,x)<(1+\kappa)\delta_E(y)\leq(1+\kappa)\rho_{\#}(y,{w})$, hence
\begin{eqnarray}\label{BHh}
C\ell(Q) &\leq & \rho_{\#}(x,{w})
\leq C_{\rho_{\#}}\max\{\rho_{\#}(x,y),\rho_{\#}(y,{w})\}
\nonumber\\[4pt]
& \leq & C_\rho(1+\kappa)\rho_{\#}(y,{w})\leq C\ell(Q),
\end{eqnarray}
proving \eqref{SQ-5}. 

Select now a real number $M>q(\upsilon-a)$. 
Combining \eqref{SQ-4} and \eqref{SQ-5} we then obtain
\begin{eqnarray}\label{SQ-6}
&& \hskip -0.80in
\int\limits_{\stackrel{y\in\Gamma_\kappa(x)}{\rho_{\#}(x,y)<\eta\ell(Q)}}
\!\!\!\!\!
|(\Theta{\mathbf{1}}_{E\setminus c\,Q})(y)|^q\delta_E(y)^{q\upsilon-m}\,d\mu(y)
\\
&& \hskip 0.80in
\leq C\int\limits_{\stackrel{y\in\Gamma_\kappa(x)}{\rho_{\#}(x,y)<\eta\ell(Q)}}
\!\!\!\!\!
\frac{1}{\ell(Q)^{q(\upsilon-a)}}\cdot\delta_E(y)^{q(\upsilon-a)-m}\,d\mu(y)
\nonumber\\[4pt]
&& \hskip 0.80in
\leq C\ell(Q)^{M-q(\upsilon-a)}\int\limits_{\Gamma_\kappa(x)}
\frac{\delta_E(y)^{-[m-q(\upsilon-a)]}}{\rho_{\#}(y,{w})^{M}}\,d\mu(y)
\nonumber\\[4pt]
&& \hskip 0.80in
\leq C\ell(Q)^{M-q(\upsilon-a)}\rho_{\#}(x,{w})^{-M+q(\upsilon-a)}\leq C,
\qquad\forall\,x\in Q,
\nonumber
\end{eqnarray}
where for the penultimate inequality in \eqref{SQ-6} we have relied on 
Lemma~\ref{bzrg1} (used here with $\beta:=m-q(\upsilon-a)$).

With this in hand, we are now ready to estimate the term $II$ (appearing in 
\eqref{SQ-3}). Concretely, applying first Tschebyshev's inequality and then invoking  
\eqref{SQ-6} we obtain 
\begin{eqnarray}\label{SQ-7}
II \leq \tfrac{C}{N}\int_Q\Bigl(
\int\limits_{\stackrel{y\in\Gamma_\kappa(x)}{\rho_{\#}(x,y)<\eta\ell(Q)}}
\!\!\!\!\!
|(\Theta{\mathbf{1}}_{E\setminus c\,Q})(y)|^q\frac{d\mu(y)}{\delta_E(y)^{m-q\upsilon}}
\Bigr)^\frac{1}{q}\,d\sigma(x)\leq\tfrac{C}{N}\sigma(Q).
\end{eqnarray}
Combining \eqref{SQ-3}, \eqref{SQ-3B} and \eqref{SQ-7} we see that, for each 
$\beta\in(0,1)$, if we choose $N>0$ sufficiently large, then 
\begin{eqnarray}\label{SQ-8}
\sigma\Bigl(\bigl\{x\in Q:\,S_Q(x)>N\bigr\}\Bigr)
\leq\tfrac{C}{N^{\min\{1,p\}}}\sigma(Q)<(1-\beta)\sigma(Q),
\qquad\forall\,Q\in{\mathbb{D}}(E).
\end{eqnarray}
Hence, \eqref{SQ-2} holds and the proof of the proposition is complete.
\end{proof}

\subsection{Extrapolating square function estimates}
\label{SSect:5.4}

We now combine our results to prove two extrapolation theorems for square 
function estimates associated with integral operators $\Theta_E$, as defined in Section~\ref{Sect:3}. First, we use Theorem~\ref{AsiC} to prove the extrapolation 
result in Theorem~\ref{VGds-2}, and then we combine this with Theorem~\ref{VGds-L2XXX} 
to obtain another extrapolation result in Theorem~\ref{VGds-2.33}.

In the first part of this subsection we digress to clarify terminology and background results concerning the scale of Hardy spaces $H^p$ for $p\in(0,\infty)$ in the context 
of a $d$-dimensional Ahlfors-David Regular space. In particular, we consider an atomic characterization for these spaces based on the work of R.R.~Coifman and G.~Weiss in \cite{CoWe77}, as well as a maximal function characterization based on the work of R.A.~Mac\'{i}as and C.~Segovia in~\cite{MaSe79II}. The theory of Hardy spaces in the context considered here has also been developed by D.~Mitrea, I.~Mitrea, M.~Mitrea 
and S.~Monniaux in~\cite{MMMM-G}.

Consider a $d$-dimensional {\rm ADR} space $(E,\rho,\sigma)$ and let
$\beta\in(0,\infty)$. Given a real-valued function $f$ on $E$, define its 
{\tt H\"older\! semi-norm} (of order $\beta$, relative to the quasi-distance $\rho$) 
by setting 
\begin{eqnarray}\label{Hol.T2}
\|f\|_{\dot{\mathscr{C}}^\beta(E,\rho)}:=
\sup_{x,y\in E,\,x\not=y}\frac{|f(x)-f(y)|}{\rho(x,y)^\beta}, 
\end{eqnarray}
and define the homogeneous H\"older space $\dot{\mathscr{C}}^\beta(E)$ as 
\begin{eqnarray}\label{Hol.T2Li}
\hskip -0.30in
\dot{\mathscr{C}}^\beta(E,\rho):=
\bigl\{f:E\to{\mathbb{R}}:\,\|f\|_{\dot{\mathscr{C}}^\beta(E,\rho)}<\infty\bigr\}.
\end{eqnarray}
Going further, set $\dot{\mathscr{C}}^\beta_c(E,\rho)$ for the subspace of 
$\dot{\mathscr{C}}^\beta(E,\rho)$ consisting of functions which vanish identically 
outside a bounded set. Then define the class of {\tt test\! functions} on $E$ as 
\begin{eqnarray}\label{TTT-E.1}
{\mathscr{D}}(E,\rho):=\bigcap\limits_{0<\beta<[\log_2C_\rho]^{-1}}
\dot{\mathscr{C}}^\beta_c(E,\rho), 
\end{eqnarray}
equipped with a certain topology, $\tau_{\mathscr{D}}$, which we shall 
describe next. Specifically, fix a nested family $\{K_n\}_{n\in{\mathbb{N}}}$ 
of $\rho$-bounded subsets of $E$ with the property that any $\rho$-ball is 
contained in one of the $K_n$'s. Then, for each $n\in{\mathbb{N}}$, denote 
by ${\mathscr{D}}_n(E,\rho)$ the collection of functions 
from ${\mathscr{D}}(E,\rho)$ which vanish in $E\setminus K_n$. 
With $\|\cdot\|_{\infty}$ standing for the supremum norm on $E$, 
this becomes a Frech\'et space when equipped with the topology 
$\tau_n$ induced by the family of norms 
\begin{eqnarray}\label{TFF-E.a2}
\bigl\{\|\cdot\|_{\infty}+
\|\cdot\|_{\dot{\mathscr{C}}^\beta(E,\rho)}:\,\beta\mbox{ rational number such that }
0<\beta<[\log_2C_\rho]^{-1}\bigr\}.
\end{eqnarray}
That is, ${\mathscr{D}}_n(E,\rho)$ is a Hausdorff topological space, 
whose topology is induced by a countable family of semi-norms, and which 
is complete (as a uniform space with the uniformity canonically induced 
by the aforementioned family of semi-norms or,
equivalently, as a metric space when endowed with a metric yielding the 
same topology as $\tau_{n}$). 
Since for any $n\in{\mathbb{N}}$ the topology induced by $\tau_{n+1}$ 
on ${\mathscr{D}}_n(X,\rho)$ coincides with $\tau_n$, we may turn 
${\mathscr{D}}(X,\rho)$ into a topological space, 
$({\mathscr{D}}(X,\rho),\tau_{\mathscr{D}})$, by regarding it as 
the strict inductive limit of the family of topological spaces 
$\bigl\{({\mathscr{D}}_n(X,\rho),\tau_n)\bigr\}_{n\in{\mathbb{N}}}$. 
Having accomplished this, we then define the {\tt space\! of\! distributions}
${\mathscr{D}}'(E,\rho)$ on $E$ as the (topological) dual of 
${\mathscr{D}}(E,\rho)$, and denote by $\langle\cdot,\cdot\rangle$ 
the natural duality pairing between distributions in ${\mathscr{D}}'(E,\rho)$ 
and test functions in ${\mathscr{D}}(E,\rho)$.

To proceed, for each number $\gamma\in\bigl(0,\bigl[\log_2 C_\rho\bigr]^{-1}\bigr)$
and each point $x\in E$ define the class ${\mathcal{B}}^{\,\gamma}_\rho(x)$ 
of $(\rho,\gamma)$-{\tt normalized bump-functions supported near} $x$ by 
\begin{eqnarray}\label{ASW43}
&& \hskip -1.00in
{\mathcal{B}}^{\,\gamma}_\rho(x):=\Bigl\{\psi\in{\mathscr{D}}(E,\rho):\,
\exists\,r>0\mbox{ such that $\psi=0$ on $E\setminus B_\rho(x,r)$}  
\mbox{ and}
\nonumber\\[4pt] 
&& \hskip 2.20in
\|\psi\|_{\infty}+r^\gamma\|\psi\|_{\dot{\mathscr{C}}^\gamma(E,\rho)}
\leq r^{-d}\Bigr\}.
\end{eqnarray}
In this setting, define the {\tt grand maximal function} of a distribution 
$f\in{\mathscr{D}}'(E,\rho)$ by setting (with the duality paring understood as before) 
\begin{eqnarray}\label{NMC22-1}
f^*_{\rho,\gamma}(x):=
\sup_{\psi\in{\mathcal{B}}^{\,\gamma}_\rho(x)}\bigl|\langle f,\psi\rangle\bigr|,
\qquad\forall\,x\in E.
\end{eqnarray}

Given an exponent $p$ satisfying
\begin{eqnarray}\label{range-p}
\frac{d}{d+[\log_2C_\rho]^{-1}}<p<\infty,
\end{eqnarray}
define the {\tt Hardy\! space} $H^p(E,\rho,\sigma)$ by setting 
\begin{eqnarray}\label{NMC22-2BB}
&&\hskip -0.80in
H^p(E,\rho,\sigma):=
\Bigl\{f\in{\mathscr{D}}'(E,\rho):\,\forall\,\gamma\in{\mathbb{R}}\mbox{ so that }
d\bigl({\textstyle{\frac{1}{p}}}-1\bigr)<\gamma<[\log_2 C_\rho]^{-1}
\\[4pt]
&& \hskip 2.20in
\mbox{ it follows that }f^*_{\rho_{\#},\gamma}\in L^p(E,\sigma)\Bigr\}.
\nonumber
\end{eqnarray}
A closely related version of the above Hardy space is 
$\widetilde{H}^p(E,\rho,\sigma)$, with $p$ as before, defined as 
\begin{eqnarray}\label{NMC22-2}
&& \hskip -0.30in
\widetilde{H}^p(E,\rho,\sigma):=
\Bigl\{f\in{\mathscr{D}}'(E,\rho):\,\exists\,\gamma\in{\mathbb{R}}\mbox{ so that }
d\bigl({\textstyle{\frac{1}{p}}}-1\bigr)<\gamma<[\log_2 C_\rho]^{-1}
\\[4pt]
&& \hskip 1.90in
\mbox{ and with the property that }f^*_{\rho_{\#},\gamma}\in L^p(E,\sigma)\Bigr\}.
\nonumber
\end{eqnarray}
Moving on, given an index 
\begin{eqnarray}\label{range-p-1}
\frac{d}{d+[\log_2C_\rho]^{-1}}<p\leq 1,
\end{eqnarray}
call a function $a\in L^\infty(E,\sigma)$ a $p$-{\tt atom} provided 
there exist $x_0\in E$ and a real number $r>0$ with the property that 
\begin{eqnarray}\label{jk-AM}
{\rm supp}\,a\subseteq E\cap B_{\rho}(x_0,r),\quad 
\|a\|_{L^\infty(E,\sigma)}\leq r^{-d/p},\quad\int_Ea\,d\sigma=0.
\end{eqnarray}
In the case when $E$ is bounded we also agree to consider the constant 
function $\sigma(E)^{-1/p}$ as a $p$-atom. 
Then, for each $p$ as in \eqref{range-p-1}, define the {\tt atomic Hardy space} 
$H^p_{at}(E,\rho,\sigma)$ as 
\begin{eqnarray}\label{hp-at33}
&&\hskip -0.40in
H^p_{at}(E,\rho,\sigma):=
\Bigl\{f\in\bigl(\dot{\mathscr{C}}^{d(1/p-1)}(E,\rho)\bigr)^\ast:\,
\exists\,\{\lambda_j\}_{j\in{\mathbb{N}}}\in\ell^p({\mathbb{N}})
\mbox{ and $p$-atoms $\{a_j\}_{j\in{\mathbb{N}}}$}
\nonumber\\[4pt]
&&\hskip 1.20in
\mbox{ such that }
f=\sum_{j\in{\mathbb{N}}}\lambda_ja_j\,\,\mbox{ in }\,\,
\bigl(\dot{\mathscr{C}}^{d(1/p-1)}(E,\rho)\bigr)^\ast\Bigr\},
\end{eqnarray}
and equip this space with the quasi-norm $\|\cdot\|_{H^p_{at}(E,\rho,\sigma)}$ defined 
for each $f\in H^p_{at}(E,\rho,\sigma)$ by 
\begin{eqnarray}\label{Mac-12AAA}
\|f\|_{H^p_{at}(E,\rho,\sigma)}:=\inf\,\Bigl\{
\Bigl(\sum_{j\in{\mathbb{N}}}|\lambda_j|^p\Bigr)^{1/p}:\,
f=\sum_{j\in{\mathbb{N}}}\lambda_ja_j\,\,\mbox{ as in \eqref{hp-at33}}\Bigr\}. 
\end{eqnarray}
The following atomic decomposition theorem, extending work in \cite{MaSe79II}, 
has been established in \cite{MMMM-G}.

\begin{theorem}\label{MacSeg-2}
Assume that $(E,\rho,\sigma)$ is a $d$-dimensional {\rm ADR}. Then 
\begin{eqnarray}\label{p-good.LL}
H^p(E,\rho,\sigma)=\widetilde{H}^p(E,\rho,\sigma)=L^p(E,\sigma)
\quad\mbox{for each }\,p\in(1,\infty).
\end{eqnarray}

Suppose now that $p$ is as in \eqref{range-p-1}
and, for every functional $f\in H^p_{at}(E,\rho,\sigma)$, denote by $\widetilde{f}$
the distribution in ${\mathscr{D}}'(E,\rho)$ defined as the restriction of $f$ to 
${\mathscr{D}}(E,\rho)$. Then the assignment $f\mapsto\widetilde{f}$ induces a 
well-defined, injective linear mapping from $H^p_{at}(E,\rho,\sigma)$ onto the space 
$\widetilde{H}^p(E,\rho,\sigma)$. Moreover, for each 
\begin{eqnarray}\label{Ugv-888}
\gamma\in{\mathbb{R}}\,\,\mbox{ with }\,\, 
d\bigl({\textstyle{\frac{1}{p}}}-1\bigr)<\gamma<[\log_2 C_\rho]^{-1}
\end{eqnarray}
there exist two finite constants $c_1,c_2>0$ such that 
\begin{eqnarray}\label{MacSeg-21}
c_1\|f\|_{H^p_{at}(E,\rho,\sigma)}\leq
\|(\widetilde{f}\,)^*_{\rho_{\#},\gamma}\|_{L^p(E,\sigma)}
\leq c_2\|f\|_{H^p_{at}(E,\rho,\sigma)}
\quad\mbox{for all }\,f\in H^p_{at}(E,\rho,\sigma).
\end{eqnarray}
Consequently, the spaces $H^p(E,\rho,\sigma)$, $\widetilde{H}^p(E,\rho,\sigma)$ 
are naturally identified with $H^p_{at}(E,\rho,\sigma)$. In particular, they do 
not depend on the particular choice of the index $\gamma$ as in \eqref{Ugv-888}.

As a corollary, whenever \eqref{Ugv-888} holds one can find a finite 
constant $c=c(p,\rho,\gamma)>0$ such that for every distribution
$f\in{\mathscr{D}}'(E,\rho)$ with the property that its grand maximal 
function $f^*_{\rho_{\#},\gamma}$ belongs to $L^p(E,\sigma)$ there exist 
a sequence of $p$-atoms $\{a_j\}_{j\in{\mathbb{N}}}$ on $X$ and a numerical 
sequence $\{\lambda_j\}_{j\in{\mathbb{N}}}\in\ell^p({\mathbb{N}})$ for which 
\begin{eqnarray}\label{MacSeg-11}
f=\sum_{j\in{\mathbb{N}}}\lambda_ja_j\quad\mbox{ in }\,\,{\mathscr{D}}'(E,\rho)
\end{eqnarray}
and
\begin{eqnarray}\label{MacSeg-12}
\Bigl(\sum_{j\in{\mathbb{N}}}|\lambda_j|^p\Bigr)^{1/p}
\leq c\|f^*_{\rho_{\#},\gamma}\|_{L^p(E,\sigma)}.
\end{eqnarray}
Finally,  whenever \eqref{Ugv-888} holds one can find a finite 
constant $c'=c'(p,\rho,\gamma)>0$ such that, given a distribution
$f\in{\mathscr{D}}'(E,\rho)$, a sequence of $p$-atoms 
$\{a_j\}_{j\in{\mathbb{N}}}$, and a numerical sequence
$\{\lambda_j\}_{j\in{\mathbb{N}}}\in\ell^p({\mathbb{N}})$ 
such that \eqref{MacSeg-11} holds, then 
\begin{eqnarray}\label{MacSeg-12B}
\|f^*_{\rho_{\#},\gamma}\|_{L^p(E,\sigma)}\leq
c'\Bigl(\sum_{j\in{\mathbb{N}}}|\lambda_j|^p\Bigr)^{1/p}.
\end{eqnarray}
\end{theorem}

Consider now the setting of Section \ref{SSect:3.1} and suppose that $\theta$ is a function
as in \eqref{K234} which satisfies \eqref{hszz} and such that there exists 
$\alpha\in(0,\infty)$ with the property that for all 
$x\in\mathscr{X}\setminus E$ and $y\in E$ there holds
\begin{eqnarray}\label{hszz-3noalpha}
\begin{array}{l}
\displaystyle|\theta(x,y)-\theta(x,\widetilde{y})|\leq C_\theta \frac{\rho(y,\widetilde{y})^\alpha}{\rho(x,y)^{d+\upsilon+\alpha}}\,\Bigl(
\frac{{\rm dist}_\rho(x,E)}{\rho(x,y)}\Bigr)^{-a}, 
\\[12pt]
\qquad\forall\,\widetilde{y}\in E\,\,\mbox{ with }\,\,
\rho(y,\widetilde{y})\leq\tfrac{1}{2}\rho(x,y).
\end{array}
\end{eqnarray}

We are now ready to present the first main result in this subsection.

\begin{theorem}\label{VGds-2}
Let $d,m$ be two real numbers such that $0<d<m$. Assume that $({\mathscr{X}},\rho,\mu)$
is an $m$-dimensional {\rm ADR} space, $E$ is a closed subset 
of $({\mathscr{X}},\tau_\rho)$, and $\sigma$ is a Borel measure 
on $(E,\tau_{\rho|_{E}})$ with the property that $(E,\rho\bigl|_E,\sigma)$ is a 
$d$-dimensional {\rm ADR} space. 

Furthermore, suppose that $\Theta$ is the integral operator defined in \eqref{operator} 
with a kernel $\theta$ as in \eqref{K234}, \eqref{hszz}, \eqref{hszz-3noalpha}. 
Finally, fix $\kappa>0$ and, with $\alpha_\rho$ as in \eqref{Cro} 
and $\alpha$ as in \eqref{hszz-3noalpha}, set 
\begin{eqnarray}\label{WQ-tDD}
\gamma:=\min\,\bigl\{\alpha_\rho,\alpha\bigr\}.
\end{eqnarray}

Given $q\in(1,\infty)$ and $p\in\bigl(\frac{d}{d+\gamma},\infty\bigr)$ 
consider the estimate 
\begin{eqnarray}\label{kt-Dc}
\hskip -0.20in
\left\|\Bigl(\int_{\Gamma_{\kappa}(x)}|(\Theta f)(y)|^q\,
\frac{d\mu(y)}{\delta_E(y)^{m-q\upsilon}}\Bigr)^{\frac{1}{q}}
\right\|_{L^p_x(E,\sigma)}\!\!\!
\leq C\|f\|_{H^p(E,\rho|_{E},\sigma)},\quad\forall\,f\in H^p(E,\rho|_{E},\sigma),
\end{eqnarray}
where $C>0$ is a finite constant. 

\begin{enumerate}
\item[(I)] Assume that $q\in(1,\infty)$ has the property that, for some 
finite constant $C>0$, either 
\begin{eqnarray}\label{kt-Dc-BIS}
\hskip -0.20in
\left\|\Bigl(\int_{\Gamma_{\kappa}(x)}|(\Theta f)(y)|^q\,
\frac{d\mu(y)}{\delta_E(y)^{m-q\upsilon}}\Bigr)^{\frac{1}{q}}
\right\|_{L^q_x(E,\sigma)}\!\!\!
\leq C\|f\|_{L^q(E,\sigma)},\quad\forall\,f\in L^q(E,\sigma),
\end{eqnarray}
or there exists $p_o\in(q,\infty)$ such that for every 
$f\in L^{p_o}(E,\sigma)$ there holds
\begin{eqnarray}\label{dtbjHT}
\hskip -0.30in
\sup_{\lambda>0}\left[\lambda\cdot
\sigma\Bigl(\Bigl\{x\in E:\int_{\Gamma_{\kappa}(x)}|(\Theta f)(y)|^q\,
\frac{d\mu(y)}{\delta_E(y)^{m-q\upsilon}}>\lambda^{q}\Bigr\}\Bigr)^{1/p_o}\right]
\leq C\|f\|_{L^{p_o}(E,\sigma)}.
\end{eqnarray}
Then \eqref{kt-Dc} holds for every $p\in\bigl(\frac{d}{d+\gamma},\infty\bigr)$.
\item[(II)] Assume that $q\in(1,\infty)$ is such that there exist $p_o\in(1,\infty)$ 
and a finite constant $C>0$ such that \eqref{dtbjHT} holds for every 
$f\in L^{p_o}(E,\sigma)$. Then \eqref{kt-Dc} holds for every $p\in(1,p_o)$ and, 
in addition, for every $f\in L^1(E,\sigma)$ one has
\begin{eqnarray}\label{d-YD23}
\hskip -0.30in
\sup_{\lambda>0}\left[\lambda\cdot
\sigma\Bigl(\Bigl\{x\in E:\int_{\Gamma_{\kappa}(x)}|(\Theta f)(y)|^q\,
\frac{d\mu(y)}{\delta_E(y)^{m-q\upsilon}}>\lambda^{q}\Bigr\}\Bigr)\right]
\leq C\|f\|_{L^1(E,\sigma)}.
\end{eqnarray}
\end{enumerate}
\end{theorem}

It is worth mentioning that the conclusion \eqref{kt-Dc} in Theorem~\ref{VGds-2}
may be re-phrased as saying that the operator 
\begin{eqnarray}\label{ki-DUD}
\delta_E^{\upsilon-m/q}\Theta:H^p(E,\rho|_{E},\sigma)
\longrightarrow L^{(p,q)}({\mathscr{X}},E)
\end{eqnarray}
is well-defined, linear and bounded.

To set the stage for presenting the proof of Theorem~\ref{VGds-2}, 
we state a lemma containing an estimate for a Marcinkiewicz-type
integral (cf. \cite{MMMM-B} for a proof).

\begin{lemma}\label{P-Marc}
Assume that $(E,\rho,\sigma)$ is a $d$-dimensional {\rm ADR} space for some $d>0$.
Then for each $\alpha>0$ there exists $C\in(0,\infty)$ such that whenever 
$F$ is a nonempty closed subset of $(E,\tau_\rho)$ one has  
\begin{eqnarray}\label{VG+ds}
\int_{F}\int_{E}\frac{{\rm dist}_{\rho_{\#}}\,(y,F)^\alpha}
{\rho_{\#}(x,y)^{d+\alpha}}\,d\sigma(y)\,d\sigma(x)\leq C\sigma(E\setminus F).
\end{eqnarray}
\end{lemma}

We are now prepared to present the 

\vskip 0.08in
\begin{proof}[Proof of Theorem~\ref{VGds-2}]
We divide the proof into a number of cases.

\vskip 0.10in
{\tt Case~1}: {\it Let $q\in[1,\infty)$, $p_o\in(1,\infty)$ be 
such that \eqref{dtbjHT} holds for each $f\in L^{p_o}(E,\sigma)$}.
The main step in this scenario is proving that the operator
${\mathcal{A}}_{q,\kappa}\circ(\delta_E^{\upsilon-m/q}\Theta)$ 
is of weak type $(1,1)$, that is, that there exists $C>0$ such that 
for every $\lambda>0$ there holds
\begin{eqnarray}\label{DCvh}
\sigma\Bigl(\bigl\{x\in E:\,
{\mathcal{A}}_{q,\kappa}\bigl(\delta_E^{\upsilon-m/q}(\Theta f)\bigr)(x)
>\lambda\bigr\}\Bigr)\leq C\frac{\|f\|_{L^1(E,\sigma)}}{\lambda},
\qquad\forall\,f\in L^1(E,\sigma).
\end{eqnarray}
Assuming \eqref{DCvh} for the moment, we proceed as follows. 
The operator ${\mathcal{A}}_{q,\kappa}\circ(\delta_E^{\upsilon-m/q}\Theta)$ 
is subadditive, of weak type $(1,1)$ by \eqref{DCvh}, and of weak 
type $(p_o,p_o)$ by \eqref{dtbjHT}. Hence, by the Marcinkiewicz interpolation theorem,
${\mathcal{A}}_{q,\kappa}\circ(\delta_E^{\upsilon-m/q}\Theta)$ is of strong 
type $(p,p)$ for every $p\in(1,p_o)$, yielding \eqref{kt-Dc} 
(after unraveling notation), for the specified range of $q,p_o,p$. 
As such, this takes care of the claim made in the first part of $(II)$ in the 
statement of the theorem. Moreover, \eqref{d-YD23} corresponds to 
\eqref{DCvh}, whose proof we now consider.  

To get started, assume that $f\in L^1(E,\sigma)$ has been fixed. 
When $0<\lambda\leq\|f\|_{L^1(E,\sigma)}/\sigma(E)$ (which may only happen 
in the case when $E$ is bounded), we have 
\begin{eqnarray}\label{DCvh-BB}
\sigma\Bigl(\bigl\{x\in E:\,
{\mathcal{A}}_{q,\kappa}\bigl(\delta_E^{\upsilon-m/q}(\Theta f)\bigr)(x)
>\lambda\bigr\}\Bigr)\leq\sigma(E)\leq\frac{\|f\|_{L^1(E,\sigma)}}{\lambda},
\end{eqnarray}
so \eqref{DCvh} holds in this case if we choose $C\geq 1$.

Consider now the case when $\lambda>\|f\|_{L^1(E,\sigma)}/\sigma(E)$. 
There is no loss of generality in assuming that $f$ has bounded support, and 
we shall perform a Calder\'on-Zygmund decomposition of $f$ at level $\lambda$. 
More precisely, there exist two finite constants $C>0$, $N\in{\mathbb{N}}$ 
(depending only on geometry), along with an at most countable family of balls 
$(Q_j)_{j\in J}$, say $Q_j:=B_{\rho}(x_j,r_j)$ for each $j\in J$, and two functions 
$g,b:E\to{\mathbb{R}}$ satisfying the following properties (cf., e.g., \cite{CoWe71}):
\begin{eqnarray}\label{PRO-1}
&& \hskip -0.40in
f=g+b\mbox{ on $E$},
\\[4pt]
&& \hskip -0.40in
g\in L^1(E,\sigma)\cap L^\infty(E,\sigma),\quad
\|g\|_{L^1(E,\sigma)}\leq C\|f\|_{L^1(E,\sigma)},\quad
|g(x)|\leq C\lambda,\,\,\,\forall\,x\in E,
\label{PRO-2}
\\[4pt]
&& \hskip -0.40in
b=\sum_{j\in J}b_j\,\,\mbox{ with }\,\,
{\rm supp}\,b_j\subseteq Q_j,\,\,\,\int_{E}b_j\,d\sigma=0,\,\,\mbox{ and }\,\,
\meanint_{Q_j}|b_j|\,d\sigma\leq C\lambda,\,\,\,\,\forall\,j\in J,
\label{PRO-3}
\\[4pt]
&& \hskip -0.40in
\begin{array}{l}
\mbox{if }{\mathcal{O}}:=\bigcup_{j\in{\mathbb{N}}}Q_j\subseteq E
\,\,\mbox{ and }\,\,F:=E\setminus{\mathcal{O}},\,\,\mbox{ then }\,\,
\sum_{j\in J}{\mathbf{1}}_{Q_j}\leq N,
\\[8pt]
\sigma({\mathcal{O}})\leq\frac{C}{\lambda}\|f\|_{L^1(E,\sigma)}\,\,\mbox{ and }\,\,
{\rm dist}_{\rho}(Q_j,F)\approx r_j\,\,\mbox{ uniformly in }j\in J.
\end{array}
\label{PRO-4}
\end{eqnarray}
Note that the above properties also entail 
$\sum_{j\in J}\|b_j\|_{L^1(E,\sigma)}\leq C\|f\|_{L^1(E,\sigma)}$, so the series
in \eqref{PRO-3} converges absolutely in $L^1(E,\sigma)$.

By the quasi-subadditivity of
${\mathcal{A}}_{q,\kappa}\circ(\delta_E^{\upsilon-m/q}\Theta)$ 
and \eqref{PRO-1} we have 
\begin{eqnarray}\label{DCvh-gUUU}
{\mathcal{A}}_{q,\kappa}\bigl(\delta_E^{\upsilon-m/q}(\Theta f)\bigr)\leq
{\mathcal{A}}_{q,\kappa}\bigl(\delta_E^{\upsilon-m/q}(\Theta g)\bigr)
+{\mathcal{A}}_{q,\kappa}\bigl(\delta_E^{\upsilon-m/q}(\Theta b)\bigr)
\end{eqnarray}
so, as far as \eqref{DCvh} is concerned, it suffices to prove that
\begin{eqnarray}\label{DCvh-g}
\sigma\Bigl(\bigl\{x\in E:\,
{\mathcal{A}}_{q,\kappa}\bigl(\delta_E^{\upsilon-m/q}(\Theta g)\bigr)(x)>
\lambda/2\bigr\}\Bigr)\leq C\frac{\|f\|_{L^1(E,\sigma)}}{\lambda},
\end{eqnarray}
and 
\begin{eqnarray}\label{DCvh-b}
\sigma\Bigl(\bigl\{x\in E:\,
{\mathcal{A}}_{q,\kappa}\bigl(\delta_E^{\upsilon-m/q}(\Theta b)\bigr)(x)>
\lambda/2\bigr\}\Bigr)\leq C\frac{\|f\|_{L^1(E,\sigma)}}{\lambda}.
\end{eqnarray}
Making use of \eqref{dtbjHT} (with $f$ replaced by $g$), 
\eqref{PRO-2} and keeping in mind that $p_o>1$ we obtain 
\begin{eqnarray}\label{bmKK}
&& \hskip -0.80in
\sigma\Bigl(\bigl\{x\in E:\,
{\mathcal{A}}_{q,\kappa}\bigl(\delta_E^{\upsilon-m/q}(\Theta g)\bigr)(x)
>\lambda/2\bigr\}\Bigr)
\leq C\Bigl(\frac{\|g\|_{L^{p_o}(E,\sigma)}}{\lambda}\Bigr)^{p_o}
\nonumber\\[4pt]
&& \hskip 0.80in
\leq C\frac{\|g\|_{L^{\infty}(E,\sigma)}^{p_o-1}\|g\|_{L^1(E,\sigma)}}{\lambda^{p_o}}
\leq C\frac{\|f\|_{L^1(E,\sigma)}}{\lambda},
\end{eqnarray}
thus \eqref{DCvh-g} is proved. We are therefore left with proving \eqref{DCvh-b}. 
To justify this, first note that by \eqref{PRO-4} we have 
\begin{eqnarray}\label{bmKK-2}
\sigma\Bigl(\bigl\{x\in{\mathcal{O}}:\,
{\mathcal{A}}_{q,\kappa}\bigl(\delta_E^{\upsilon-m/q}
(\Theta b)\bigr)(x)>\lambda/2\bigr\}\Bigr)\leq\sigma({\mathcal{O}})
\leq C\frac{\|f\|_{L^1(E,\sigma)}}{\lambda}.
\end{eqnarray}
Second, it is immediate that 
\begin{eqnarray}\label{bmKK-3}
\sigma\Bigl(\bigl\{x\in F:\,
{\mathcal{A}}_{q,\kappa}\bigl(\delta_E^{\upsilon-m/q}
(\Theta b)\bigr)(x)>\lambda/2\bigr\}\Bigr)
\leq \frac{1}{\lambda}\int_F{\mathcal{A}}_{q,\kappa}\bigl(\delta_E^{\upsilon-m/q}
(\Theta b)\bigr)\,d\sigma.
\end{eqnarray}
Therefore, since $E={\mathcal{O}}\cup F$, in view of \eqref{bmKK-2} and 
\eqref{bmKK-3}, estimate \eqref{DCvh-b} will follow as soon as we prove that
\begin{eqnarray}\label{bmKK-3B}
\int_F{\mathcal{A}}_{q,\kappa}\bigl(\delta_E^{\upsilon-m/q}
(\Theta b)\bigr)\,d\sigma\leq C\|f\|_{L^1(E,\sigma)}
\end{eqnarray}
for some $C>0$ independent of $f$. With this goal in mind, we fix $j\in J$ 
and $x\in F$ arbitrary and look for a pointwise estimate for 
\begin{eqnarray}\label{bmKK-3A}
{\mathcal{A}}_{q,\kappa}\bigl(\delta_E^{\upsilon-m/q}(\Theta b_j)\bigr)(x)
=\Bigl(\int_{\Gamma_{\kappa}(x)}|(\Theta b_j)(y)|^q\delta_E(y)^{q\upsilon-m}\,d\mu(y)
\Bigr)^{\frac{1}{q}}.
\end{eqnarray}
With $x_j$ and $r_j$ denoting, respectively, the center and radius of $Q_j$, 
based on the third condition in \eqref{PRO-3}, for each $y\in\Gamma_{\kappa}(x)$, 
we may write
\begin{eqnarray}\label{bmKK-4}
\bigl|(\Theta b_j)(y)\bigr| 
& = & \Bigl|\int_{E}\theta(y,z)b_j(z)\,d\sigma(z)\Bigr|
=\Bigl|\int_{E}\bigl[\theta(y,z)-\theta(y,x_j)\bigr]
b_j(z)\,d\sigma(z)\Bigr|
\nonumber\\[4pt]
& = & \Bigl|\int_{Q_j}\bigl[\theta(y,z)-\theta(y,x_j)\bigr]
b_j(z)\,d\sigma(z)\Bigr|\leq I_1+I_2,
\end{eqnarray}
where, for some small $\epsilon>0$ to be determined momentarily, we have set 
\begin{eqnarray}\label{bmKK-4B}
I_1:=\!\!\!\!
\int\limits_{\stackrel{z\in Q_j}{\rho_{\#}(z,x_j)<\epsilon\rho_{\#}(y,x_j)}}
\!\!\!\!\bigl|\theta(y,z)-\theta(y,x_j)\bigr||b_j(z)|\,d\sigma(z),
\\[4pt]
I_2:=\!\!\!\!
\int\limits_{\stackrel{z\in Q_j}{\rho_{\#}(z,x_j)\geq\epsilon\rho_{\#}(y,x_j)}}
\!\!\!\!\bigl|\theta(y,z)-\theta(y,x_j)\bigr||b_j(z)|\,d\sigma(z).
\label{bmKK-4B.2}
\end{eqnarray}
Note that, by \eqref{DEQV1}, 
if $\rho_{\#}(z,x_j)<\epsilon\rho_{\#}(y,x_j)$ then
\begin{eqnarray}\label{b-YH53}
\rho(z,x_j)\leq C_\rho^2\rho_{\#}(z,x_j)<\epsilon C_\rho^2\rho_{\#}(y,x_j)
\leq\epsilon\widetilde{C}_{\rho}C_\rho^2\rho_{\#}(y,x_j)
<\tfrac{1}{2}\rho(y,x_j)
\end{eqnarray}
if $0<\epsilon<2^{-1}\widetilde{C}_{\rho}^{-1}C_\rho^{-2}$.
Hence, for this choice of $\epsilon$, we have $\rho(z,x_j)<\frac{1}{2}\rho(y,x_j)$ 
on the domain of integration in $I_1$. Based on this, \eqref{hszz-3} and 
\eqref{PRO-3}, we may then estimate this term as follows
\begin{eqnarray}\label{bmKK-4B.3}
I_1 & \leq & C\int_{Q_j}\frac{\rho_{\#}(z,x_j)^\alpha\delta_E(y)^{-a}}
{\rho_{\#}(y,x_j)^{d+\upsilon+\alpha-a}}|b_j(z)|\,d\sigma(z)
\nonumber\\[4pt]
& \leq & C\frac{r_j^\alpha\delta_E(y)^{-a}}{\rho_{\#}(y,x_j)^{d+\upsilon+\alpha-a}}
\int_{Q_j}|b_j(z)|\,d\sigma(z)
\leq C\lambda\,\frac{r_j^\alpha\delta_E(y)^{-a}\sigma(Q_j)}
{\rho_{\#}(y,x_j)^{d+\upsilon+\alpha-a}}.
\end{eqnarray}
Estimating $I_2$ requires a few geometrical preliminaries. 
Recall that $x\in F$, $y\in\Gamma_\kappa(x)$ and fix an arbitrary point $z\in Q_j$ 
such that $\rho(z,x_j)\geq\epsilon\rho(y,x_j)$, where $\epsilon>0$ is as above. 
Since, on the one hand, 
\begin{eqnarray}\label{bmKK-4b}
r_j & \approx & {\rm dist}_\rho(Q_j,F)\leq\widetilde{C}_{\rho}
\rho(x,x_j)\leq C\rho(x,y)+C\rho(y,x_j)
\nonumber\\[4pt]
& \leq & C(1+\kappa)\delta_E(y)+C\rho(y,x_j)\leq C\rho(y,x_j),
\end{eqnarray}
while, on the other hand, the fact that $z\in Q_j$ forces 
$\rho(x_j,z)<r_j$ which in turn allow us to estimate 
$\rho(y,x_j)\leq\epsilon^{-1}\rho(z,x_j)\leq\epsilon^{-1}
\widetilde{C}_{\rho}\rho(x_j,z)<\epsilon^{-1}\widetilde{C}_{\rho}r_j$. 
Hence, ultimately, 
\begin{eqnarray}\label{bmKK-4b.L}
r_j\approx\rho(y,x_j),\quad
\mbox{uniformly in $j\in J$ and $y\in\Gamma_{\kappa}(x)$ with $x\in F$}.
\end{eqnarray}
In addition, the same type of estimate as in \eqref{bmKK-4b} written with $x_j$ 
replaced by $z$ yields $r_j\leq C\rho(y,z)$, which further implies 
\begin{eqnarray}\label{bmKK-4b.F}
\rho(y,x_j)\leq C\rho(y,z)+C\rho(z,x_j)\leq C\rho(y,z)+Cr_j\leq C\rho(y,z),
\end{eqnarray}
for some constant $C\in(0,\infty)$ independent of $j,x,y,z$. Hence, from 
\eqref{bmKK-4b.L} and \eqref{bmKK-4b.F}, we obtain 
\begin{eqnarray}\label{bmKK-4b.H}
\frac{1}{\rho(y,z)^{d+\upsilon}}\leq\frac{C}{\rho(y,x_j)^{d+\upsilon}}
\leq\frac{Cr_j^\alpha}{\rho(y,x_j)^{d+\upsilon+\alpha}}.
\end{eqnarray}
Consequently, on the domain of integration in $I_2$ we have thanks to 
\eqref{hszz} and \eqref{bmKK-4b.H}
\begin{eqnarray}\label{4b.H2A}
\bigl|\theta(y,z)-\theta(y,x_j)\bigr|\leq\frac{C\delta_E(y)^{-a}}{\rho_{\#}(y,z)^{d+\upsilon-a}}
+\frac{C\delta_E(y)^{-a}}{\rho_{\#}(y,x_j)^{d+\upsilon-a}}
\leq\frac{Cr_j^\alpha\delta_E(y)^{-a}}{\rho_{\#}(y,x_j)^{d+\upsilon+\alpha-a}}.
\end{eqnarray}
Together with \eqref{PRO-3}, this allows us to estimate (recall that $I_2$ has been 
defined in \eqref{bmKK-4B.2})
\begin{eqnarray}\label{b-4b.H2}
I_2 &\leq & \frac{Cr_j^\alpha\delta_E(y)^{-a}}{\rho_{\#}(y,x_j)^{d+\upsilon+\alpha-a}}
\!\!\!\!
\int\limits_{\stackrel{z\in Q_j\,\,\mbox{\tiny{such that}}}
{\rho_{\#}(z,x_j)\geq\epsilon\rho_{\#}(y,x_j)}}
\!\!\!\!|b_j(z)|\,d\sigma(z)
\leq C\lambda\,\frac{r_j^\alpha\delta_E(y)^{-a}\sigma(Q_j)}
{\rho_{\#}(y,x_j)^{d+\upsilon+\alpha-a}}.
\end{eqnarray}
Cumulatively, \eqref{bmKK-4}, \eqref{bmKK-4B.3} and \eqref{b-4b.H2} prove that 
there exists $C\in(0,\infty)$ with the property that, for every $j\in J$,
\begin{eqnarray}\label{bmKK-4i}
x\in F\,\Longrightarrow\,
\bigl|(\Theta b_j)(y)\bigr| 
\leq C\lambda\,\frac{r_j^\alpha\delta_E(y)^{-a}\sigma(Q_j)}
{\rho_{\#}(y,x_j)^{d+\upsilon+\alpha-a}},
\qquad\forall\,y\in\Gamma_\kappa(x).
\end{eqnarray}
Utilizing \eqref{bmKK-4i} in \eqref{bmKK-3A}, it follows that for every $j\in J$ 
and $x\in F$
\begin{eqnarray}\label{bmKK-5}
{\mathcal{A}}_{q,\kappa}\bigl(\delta_E^{\upsilon-m/q}(\Theta b_j)\bigr)(x)
\leq C\lambda\,r_j^\alpha\sigma(Q_j)\Bigl(\int_{\Gamma_{\kappa}(x)}
\frac{\delta_E(y)^{q(\upsilon-a)-m}}{\rho_{\#}(y,x_j)^{q(d+\upsilon+\alpha-a)}}
\,d\mu(y)\Bigr)^{\frac{1}{q}}.
\end{eqnarray}
At this point we make use of Lemma~\ref{bzrg1} (recall that $\nu-a>0$) to further 
bound the last integral in \eqref{bmKK-5} and obtain that for every $j\in J$ and $x\in F$
\begin{eqnarray}\label{bmKK-7}
{\mathcal{A}}_{q,\kappa}\bigl(\delta_E^{\upsilon-m/q}(\Theta b_j)\bigr)(x)
\leq C\lambda\,r_j^\alpha\sigma(Q_j)\rho(x,x_j)^{-d-\alpha}
\leq C\lambda\int_{Q_j}\frac{{\rm dist}_{\rho_{\#}}(z,F)^\alpha}
{\rho_{\#}(x,z)^{d+\alpha}}\,d\sigma(z).
\end{eqnarray}
Two geometrical inequalities that have been used in the last step in \eqref{bmKK-7}
are as follows. First, ${\rm dist}_{\rho_{\#}}(z,F)\approx r_j$, uniformly for 
$z\in Q_j$ and, second, for every $z\in Q_j$ we have
\begin{eqnarray}\label{QLVp}
\rho(x,z) &\leq & C\rho(x,x_j)+C\rho(x_j,z)\leq C\rho(x,x_j)+Cr_j
\nonumber\\[4pt]
&\leq & C\rho(x,x_j)+C\,{\rm dist}_{\rho}(Q_j,F)\leq C\rho(x,x_j).
\end{eqnarray}

Summing up inequalities of the form \eqref{bmKK-7} over $j\in{\mathbb{N}}$ 
and using the sublinearity of the operator
${\mathcal{A}}_{q,\kappa}\bigl(\delta_E^{\upsilon-m/q}\Theta(\cdot)\bigr)$ 
(recall that $q\geq 1$), as well as the finite overlap property in \eqref{PRO-4}, 
we obtain
\begin{eqnarray}\label{bmKK-8}
{\mathcal{A}}_{q,\kappa}\bigl(\delta_E^{\upsilon-m/q}(\Theta b)\bigr)(x)
\leq C\lambda\int_{{\mathcal{O}}}\frac{{\rm dist}_{\rho_{\#}}(z,F)^\alpha}
{\rho_{\#}(x,z)^{d+\alpha}}\,d\sigma(z),\qquad\forall\,x\in F.
\end{eqnarray}
Consequently, from \eqref{bmKK-8}, Lemma~\ref{P-Marc} and \eqref{PRO-4}, we deduce 
that
\begin{eqnarray}\label{bmKK-9}
&& \hskip -0.20in
\int_F{\mathcal{A}}_{q,\kappa}\bigl(\delta_E^{\upsilon-m/q}(\Theta b)\bigr)(x)\,dx
\leq C\lambda\int_F \int_{{\mathcal{O}}}
\frac{{\rm dist}_{\rho_{\#}}(z,F)^\alpha}{\rho_{\#}(x,z)^{d+\alpha}}\,d\sigma(z)
\,d\sigma(x)
\\[4pt]
&& \hskip 0.30in
\leq C\lambda\int_F\int_{E}
\frac{{\rm dist}_{\rho_{\#}}(z,F)^\alpha}{\rho_{\#}(x,z)^{d+\alpha}}\,d\sigma(z)
\,d\sigma(x)
\leq C\lambda\,\sigma(E\setminus F)=C\lambda\,\sigma({\mathcal{O}})
\leq C\|f\|_{L^1(E,\sigma)}.
\nonumber
\end{eqnarray}
This proves \eqref{bmKK-3B}, thus completing the proof of \eqref{DCvh}. 
In summary, the analysis so far proves part $(II)$ in the statement of the theorem.

\vskip 0.10in
{\tt Case~2}: {\it Assume that \eqref{dtbjHT} holds for some $1<q<p_o<\infty$.} 
As a preliminary step, we make the observation that, in this scenario,
granted \eqref{dtbjHT} and the conclusion in the Case~1,  
\begin{eqnarray}\label{Area-p2XXX}
{\mathscr{A}}_{q,\kappa}\circ\bigl(\delta_E^{\upsilon-m/q}\Theta\bigr):
L^r(E,\sigma)\to L^r(E,\sigma)\quad\mbox{is bounded whenever $r\in(1,p_o)$}.
\end{eqnarray}
Because of the equivalence \eqref{sbrn} in Theorem~\ref{AsiC}, estimate \eqref{kt-Dc} for the range 
$p\in(q,\infty)$ will follow once we show that 
\begin{eqnarray}\label{PaSeD}
{\mathfrak{C}}_{q,\kappa}\circ(\delta_E^{\upsilon-m/q}\Theta):L^p(E,\sigma)\to
L^p(E,\sigma)\quad\mbox{is bounded for $q<p<\infty$}. 
\end{eqnarray}
Fix $p\in(q,\infty)$. The proof of the boundedness of the operator in 
\eqref{PaSeD} relies on the following pointwise estimate 
\begin{eqnarray}\label{PaSeD-2}
&& \hskip -0.40in
{\mathfrak{C}}_{q,\kappa}\bigl(\delta_E^{\upsilon-m/q}(\Theta f)\bigr)(x_0)
\\[4pt]
&& \hskip 0.30in
\leq C\bigl[
\bigl(M_E(|f|^q)(x_0)\bigr)^{\frac{1}{q}}+(M_E(M_E(f)))(x_0)\bigr],\qquad
\forall\,x_0\in E,
\nonumber
\end{eqnarray}
for each $f\in L^p(E,\sigma)$. Indeed, fix such a function $f$. After raising the 
inequality in \eqref{PaSeD-2} to the $p$-th power and then integrating over $E$, 
we obtain
\begin{eqnarray}\label{PaSeD-3}
&& \hskip -0.40in
\int_{E}\bigl[{\mathfrak{C}}_{q,\kappa}\bigl(\delta_E^{\upsilon-m/q}
(\Theta f)\bigr)(x)\bigr]^p\,d\sigma(x)
\\[4pt]
&& \hskip 0.30in
\leq C\int_{E}\bigl[M_E(|f|^q)(x)\bigr]^{\frac{p}{q}}\,d\sigma(x)
+C\int_{E}\bigl[(M_E^2f)(x)\bigr]^p\,d\sigma(x)
\leq C\int_{E}|f|^p\,d\sigma, 
\nonumber
\end{eqnarray}
where the last inequality in \eqref{PaSeD-3} uses the boundedness on $L^p(E,\sigma)$ 
and $L^{p/q}(E,\sigma)$ of the Hardy-Littlewood maximal operator $M_E$
(here we make use of the fact that in the current case $p>\max\{q,1\}$). 
This shows that \eqref{PaSeD} holds assuming \eqref{PaSeD-2}. 

Returning to the proof of \eqref{PaSeD-2}, fix $f\in L^p(E,\sigma)$
along with $r>0$ and $x_0\in E$. For some finite constant $c>0$ to be specified later, set
$\Delta:=E\cap B_{\rho_{\#}}(x_0,r)$ and $c\Delta:=E\cap B_{\rho_{\#}}(x_0,cr)$, 
then write $f=f_1+f_2$, where $f_1:=f{\mathbf{1}}_{c\Delta}$ 
and $f_2:=f{\mathbf{1}}_{E\setminus c\,\Delta}$. 
First we estimate the contribution from $f_1$ by writing
\begin{eqnarray}\label{PaSeD-4}
&&\hskip -0.50in 
\tfrac{1}{\sigma(\Delta)}\int_{{\mathcal{T}}_\kappa(\Delta)}
|(\Theta f_1)(x)|^q\delta_E(x)^{q\upsilon-(m-d)}\,d\mu(x)
\nonumber\\[4pt]
&&\hskip 0.50in
 \leq \tfrac{1}{\sigma(\Delta)}\int_{{\mathscr{X}}\setminus E}
|(\Theta f_1)(x)|^q\delta_E(x)^{q\upsilon-(m-d)}\,d\mu(x)
\nonumber\\[4pt]
&&\hskip 0.50in
\leq \tfrac{C}{\sigma(\Delta)}
\int_E\Bigl(\int_{\Gamma_{\kappa}(x)}|(\Theta f_1)(y)|^q\,
\delta_E(y)^{q\upsilon-m}\,d\mu(y)\Bigr)\,d\sigma(x)
\nonumber\\[4pt]
&&\hskip 0.50in
=C\bigl\|{\mathscr{A}}_{q,\kappa}\bigl(\delta_E^{\upsilon-m/q}(\Theta f_1)\bigr)
\bigr\|_{L^q(E,\sigma)}^q\leq\tfrac{C}{\sigma(\Delta)}\int_{E}|f_1|^q\,d\sigma
\nonumber\\[4pt]
&&\hskip 0.50in
=\tfrac{C}{\sigma(\Delta)}\int_{c\,\Delta}|f|^q\,d\sigma\leq C\,M_E(|f|^q)(x_0).
\end{eqnarray}
For the second inequality in \eqref{PaSeD-4} we have used \eqref{Mix+FR}, 
Lemma~\ref{lbDV}, and the fact that $(E,\rho|_{E},\sigma)$ is a 
$d$-dimensional {\rm ADR} space, while the third inequality 
follows from \eqref{Area-p2XXX} used with $r:=q\in(1,p_o)$.

To treat the term corresponding to $f_2$, observe that if $c>C_\rho$, then 
for every $y\in E\setminus c\,\Delta$ we have 
$cr<\rho_{\#}(y,x_0)\leq C_\rho\max\{\rho_{\#}(y,{w}),\rho_{\#}({w},x_0)\}
\leq C_\rho\rho_{\#}(y,{w})$ for every ${w}\in\Delta$.
Hence, $E\setminus c\,\Delta\subseteq\{y\in E:\,\rho_{\#}(y,{w})>r\}$ and
$\rho_{\#}(y,x_0)\approx\rho_{\#}(y,{w})$, uniformly for 
$y\in E\setminus c\,\Delta$ and ${w}\in\Delta$. Furthermore, for every 
$z\in {\mathcal{T}}_\kappa(\Delta)$, we have
\begin{eqnarray}\label{brz}
\rho_{\#}(y,x_0) &\leq & C_\rho\max\{\rho_{\#}(y,z),\rho_{\#}(z,x_0)\}
\leq C_\rho\max\{\rho_{\#}(y,z),(1+\kappa)\delta_E(z)\}
\nonumber\\[4pt]
&\leq & C\rho_{\#}(y,z).
\end{eqnarray}
Based on these considerations as well as \eqref{hszz} and \eqref{WBA}, 
if $z\in {\mathcal{T}}_\kappa(\Delta)$ we may write
\begin{eqnarray}\label{PaSeD-5}
|(\Theta f_2)(z)|
& \leq & C\int\limits_{E\setminus c\,\Delta}
\frac{\delta_E(z)^{-a}}{\rho_{\#}(z,y)^{d+\upsilon-a}}|f(y)|\,d\sigma(y)
\nonumber\\[4pt]
& \leq & \frac{C\delta_E(z)^{-a}}{r^{\upsilon-a}}\int\limits_{y\in E,\,\rho_{\#}(y,{w})>r}
\frac{r^{\upsilon-a}}{\rho(y,{w})^{d+\upsilon-a}}|f(y)|\,d\sigma(y)
\nonumber\\[4pt]
& \leq & \frac{C\delta_E(z)^{-a}}{r^{\upsilon-a}}(M_Ef)({w}), 
\quad\mbox{uniformly for }{w}\in\Delta.
\end{eqnarray}
Thus, \eqref{PaSeD-5} implies
\begin{eqnarray}\label{PaSeD-6}
|(\Theta f_2)(z)|
\leq \frac{C\delta_E(z)^{-a}}{r^{\upsilon-a}}\inf_{{w}\in\Delta}(M_Ef)({w}),\qquad
\forall\,z\in {\mathcal{T}}_\kappa(\Delta).
\end{eqnarray}
In concert with \eqref{3.2.63WS} and Lemma~\ref{geom-lem} (which uses $\upsilon-a>0$),
estimate \eqref{PaSeD-6} further yields
\begin{eqnarray}\label{PaSeD-7}
&&\hskip -0.50in
\Bigl[\tfrac{1}{\sigma(\Delta)}\int_{{\mathcal{T}}_\kappa(\Delta)}
|(\Theta f_2)(z)|^q\delta_E(z)^{q\upsilon-(m-d)}\,d\mu(z)\Bigr]^{\frac{1}{q}}
\nonumber\\[4pt]
&& \hskip 0.50in
\leq \frac{C}{r^{\upsilon-a}}\inf_{{w}\in\Delta}(M_Ef)({w})
\Bigl[\tfrac{1}{\sigma(\Delta)}\int_{B_{\rho_{\#}}(x_0,Cr)\setminus E}
\delta_E(z)^{q(\upsilon-a)-(m-d)}\,d\mu(z)\Bigr]^{\frac{1}{q}}
\nonumber\\[4pt]
&& \hskip 0.50in
\leq C\inf_{{w}\in\Delta}(M_Ef)({w})
\leq C\meanint_{\Delta}M_Ef\,d\sigma\leq CM_E(M_Ef)(x_0).
\end{eqnarray}
Now \eqref{PaSeD-2} follows from \eqref{PaSeD-4} and \eqref{PaSeD-5} 
in view of \eqref{ktEW-7} and the fact that
${\mathfrak{C}}_{q,\kappa}\circ\bigl(\delta_E^{\upsilon-m/q}\Theta\bigr)$ 
is, in the current case, sub-linear. 

In summary, the analysis in this case proves that, under the assumption \eqref{dtbjHT}, 
estimate \eqref{kt-Dc} holds whenever $1<q<p_o<\infty$ and $q<p<\infty$.

\vskip 0.10in
{\tt Case~3}: {\it Assume that $q\in(1,\infty)$ is such that \eqref{kt-Dc-BIS} holds.}
We claim that 
\begin{eqnarray}\label{PaSeD-BIS}
{\mathfrak{C}}_{q,\kappa}\circ(\delta_E^{\upsilon-m/q}\Theta):L^p(E,\sigma)\to
L^p(E,\sigma)\quad\mbox{is bounded for $q<p<\infty$}. 
\end{eqnarray}
The proof of \eqref{PaSeD-BIS} largely parallels that of \eqref{PaSeD}.
More specifically, the only significant difference occurs in the third inequality 
in \eqref{PaSeD-4} which, this time, follows directly from \eqref{kt-Dc-BIS}.
Once this has been established, the equivalence \eqref{sbrn} in Theorem~\ref{AsiC}, and the current 
assumption yield \eqref{kt-Dc} for the range $p\in[q,\infty)$.

\vskip 0.10in
{\tt Case~4}: {\it Assume $\frac{d}{d+\gamma}<p\leq1$ and $q\in[p,\infty)$, 
and suppose that}
\begin{eqnarray}\label{Ar-EEE}
{\mathscr{A}}_{q,\kappa}\circ\bigl(\delta_E^{\upsilon-m/q}\Theta\bigr):
L^q(E,\sigma)\to L^q(E,\sigma)\quad\mbox{is bounded}.
\end{eqnarray}
In this case, we shall prove that there exists $C\in(0,\infty)$ such that
\begin{eqnarray}\label{hnv}
\bigl\|{\mathcal{A}}_{q,\kappa}\bigl(\delta_E^{\upsilon-m/q}\Theta(a)\bigr)
\bigr\|_{L^p(E,\sigma)}^p\leq C,\qquad\mbox{for every $p$-atom $a$}.
\end{eqnarray}
With this goal in mind, fix a $p$-atom $a$ and let $x_0\in E$ and $r>0$ be 
such that the conditions in \eqref{jk-AM} hold. In particular, 
\begin{eqnarray}\label{hnv-Uj}
{\rm supp}\,a\subseteq B_{\rho_{\#}}(x_0,\widetilde{C}_\rho r).
\end{eqnarray}
Then, for some finite constant $c>1$ to be specified later, and with
$\Delta:=E\cap B_{\rho_{\#}}(x_0,cr)$, we have
\begin{eqnarray}\label{hnv-2}
&& \hskip -0.50in
\bigl\|{\mathcal{A}}_{q,\kappa}\bigl(\delta_E^{\upsilon-m/q}\Theta(a)\bigr)
\bigr\|_{L^p(E,\sigma)}^p=\int_{\Delta}\Bigl(\int_{\Gamma_\kappa(x)}
|(\Theta a)(y)|^q\delta_E(y)^{q\upsilon-m}\,d\mu(y)\Bigr)^{\frac{p}{q}}d\sigma(x)
\\[4pt]
&& \hskip 1.00in 
+\int_{E\setminus \Delta}\Bigl(\int_{\Gamma_\kappa(x)}
|(\Theta a)(y)|^q\delta_E(y)^{q\upsilon-m}\,d\mu(y)\Bigr)^{\frac{p}{q}}d\sigma(x)
=:I_1+I_2.
\nonumber
\end{eqnarray}
Using H\"older's inequality (with exponent $q/p\geq 1$), the fact 
that $(E,\rho|_{E},\sigma)$ is a $d$-dimensional {\rm ADR} space, 
\eqref{Ar-EEE}, and \eqref{jk-AM}, we may write
\begin{eqnarray}\label{jnbj}
I_1 &\leq & C\left[\int_{\Delta}\Bigl(\int_{\Gamma_\kappa(x)}
|(\Theta a)(y)|^q\delta_E(y)^{q\upsilon-m}\,d\mu(y)\Bigr)\,d\sigma(x)
\right]^{\frac{p}{q}}\,r^{d\bigl(1-\frac{p}{q}\bigr)}
\nonumber\\[4pt]
& \leq & C\bigl\|{\mathcal{A}}_{q,\kappa}\bigl(\delta_E^{\upsilon-m/q}\Theta(a)\bigr)
\bigr\|_{L^q(E,\sigma)}^p\,r^{d\bigl(1-\frac{p}{q}\bigr)}
\leq C\|a\|^p_{L^q(E,\sigma)}\,r^{d\bigl(1-\frac{p}{q}\bigr)}\leq C,
\end{eqnarray}
for some finite $C>0$ independent of $a$. We are left with estimating $I_2$. 
First, we look for a pointwise estimate for $\Theta a$. 
Fix $x\in E\setminus\Delta$ and $y\in\Gamma_\kappa(x)$. Then
for every $z\in E\cap B_{\rho_{\#}}(x_0,\widetilde{C}_\rho r)$ we have
\begin{eqnarray}\label{Jnb-H}
\rho_{\#}(x_0,z) &\leq & \widetilde{C}_\rho r\leq\tfrac{1}{c}\,\rho_{\#}(x,x_0)
\leq \tfrac{1}{c}\,\widetilde{C}_\rho 
C_{\rho_{\#}}\max\{\rho_{\#}(x,y),\rho_{\#}(y,x_0)\}
\nonumber\\[4pt]
&\leq & \tfrac{1}{c}\,\widetilde{C}_\rho
C_{\rho}\max\{(1+\kappa)\delta_E(y),\rho_{\#}(y,x_0)\}
\nonumber\\[4pt]
&\leq & \tfrac{1}{c}\,\widetilde{C}_\rho C_{\rho}(1+\kappa)\rho_{\#}(y,x_0).
\end{eqnarray}
Now, based on this and \eqref{DEQV1}, by choosing $c$ sufficiently large 
we conclude that
\begin{eqnarray}\label{Jnb-HsA}
\rho(z,x_0)\leq\tfrac{1}{2}\rho(y,x_0)\quad\mbox{ for every }\,\,
z\in E\cap B_{\rho_{\#}}(x_0,\widetilde{C}_\rho r). 
\end{eqnarray}
At this point, we may use the last condition in \eqref{jk-AM}, \eqref{hnv-Uj},
\eqref{Jnb-HsA}, \eqref{hszz-3noalpha}, the second condition in \eqref{jk-AM} 
and the fact that $(E,\rho|_E,\sigma)$ is a $d$-dimensional {\rm ADR} 
space in order to obtain
\begin{eqnarray}\label{jnbj-2}
|(\Theta a)(y)| & = & \left|\int_E[\theta(y,z)-\theta(y,x_0)]a(z)\,d\sigma(z)\right|
\nonumber\\[4pt]
& = & \Bigl|\int\limits_{E\cap B_{\rho_{\#}}(x_0,\widetilde{C}_\rho r)}
[\theta(y,z)-\theta(y,x_0)]a(z)\,d\sigma(z)\Bigr|
\nonumber\\[4pt]
& \leq & C \int\limits_{E\cap B_{\rho_{\#}}(x_0,\widetilde{C}_\rho r)}
\frac{\rho_{\#}(z,x_0)^\alpha\delta_E(y)^{-a}}
{\rho_{\#}(y,x_0)^{d+\upsilon+\alpha-a}}\,|a(z)|\,d\sigma(z)
\nonumber\\[4pt]
& \leq & C 2^{\gamma-\alpha}\delta_E(y)^{-a}
\int\limits_{E\cap B_{\rho_{\#}}(x_0,\widetilde{C}_\rho r)}
\frac{\rho_{\#}(z,x_0)^\gamma}{\rho_{\#}(y,x_0)^{d+\upsilon-a+\gamma}}\,|a(z)|\,d\sigma(z)
\nonumber\\[4pt]
& \leq &C \frac{\delta_E(y)^{-a}r^{\gamma+d\bigl(1-\frac{1}{p}\bigr)}}
{\rho_{\#}(y,x_0)^{d+\upsilon-a+\gamma}},
\qquad\forall\,y\in \Gamma_\kappa(x).
\end{eqnarray}
In turn, \eqref{jnbj-2} yields 
\begin{eqnarray}\label{jnbj-3}
\int_{\Gamma_\kappa(x)}|(\Theta a)(y)|^q\frac{d\mu(y)}{\delta_E(y)^{m-q\upsilon}}
 & \leq & C r^{q\gamma+qd\bigl(1-\frac{1}{p}\bigr)}
\int_{\Gamma_\kappa(x)}\frac{\delta_E(y)^{q(\upsilon-a)-m}}
{\rho_{\#}(y,x_0)^{q(d+\upsilon-a+\gamma)}}\,d\mu(y)
\nonumber\\[4pt]
& \leq & C \frac{r^{q\gamma+qd\bigl(1-\frac{1}{p}\bigr)}}
{\rho_{\#}(x,x_0)^{qd+q\gamma}},\qquad\forall\,x\in E\setminus\Delta,
\end{eqnarray}
where for the last inequality in \eqref{jnbj-3} we applied Lemma~\ref{bzrg1}.
Estimate \eqref{jnbj-3} used in $I_2$ further implies
\begin{eqnarray}\label{jnbj-4}
I_2 & \leq & C\,r^{p\gamma+pd\bigl(1-\frac{1}{p}\bigr)}
\int_{E\setminus\Delta}\frac{d\sigma(x)}{\rho_{\#}(x,x_0)^{pd+p\gamma}}
\nonumber\\[4pt]
& \leq & C \frac{r^{p\gamma+pd\bigl(1-\frac{1}{p}\bigr)}}
{r^{pd+p\gamma-d}}=C,
\end{eqnarray}
where the last inequality in \eqref{jnbj-4} is a consequence of 
\eqref{WBA} (used with $f\equiv 1$) and the fact that $p(d+\gamma)>d$. 
Now \eqref{hnv} follows from \eqref{hnv-2}, \eqref{jnbj} and \eqref{jnbj-4}.

\vskip 0.10in
{\tt Case~5}: {\it Assume $\frac{d}{d+\gamma}<p\leq1\leq q<\infty$, 
and suppose that \eqref{Ar-EEE} holds.} Then we claim that
\begin{eqnarray}\label{Ar-EEE.2+}
\begin{array}{c}
\delta_E^{\upsilon-m/q}\Theta:
H^p(E,\rho|_E,\sigma)\to L^{(p,q)}({\mathscr{X}},E,\mu,\sigma;\kappa)\quad\mbox{is bounded}
\\[4pt]
\mbox{whenever $\tfrac{d}{d+\gamma}<p\leq 1\leq q<+\infty$}.
\end{array}
\end{eqnarray}
To proceed with the proof of this claim, fix $p$ and $q$ as in \eqref{Ar-EEE.2+} and 
define the sets 
\begin{eqnarray}\label{AKp-DD}
\dot{\mathscr{C}}_{b,0}^\gamma(E,\rho|_E):=\Bigl\{f\in \dot{\mathscr{C}}^\gamma(E,\rho|_E)
:\,f\mbox{ has bounded support and }
{\textstyle\int_Ef\,d\sigma=0}\Bigr\}
\end{eqnarray}
and 
\begin{eqnarray}\label{AKp-RR}
{\mathcal{F}}(E):=&\left\{
\begin{array}{l}
\dot{\mathscr{C}}_{b,0}^\gamma(E,\rho|_E)\,\mbox{ if $E$ is unbounded},
\\[8pt]
\dot{\mathscr{C}}_{b,0}^\gamma(E,\rho|_E)\cup\{{\mathbf{1}}_E\}\,\mbox{ if $E$ is bounded}.
\end{array}
\right.
\end{eqnarray}
Then letting 
\begin{eqnarray}\label{AKp-2Z}
{\mathcal{D}}_0(E):=\mbox{the finite linear span of functions in ${\mathcal{F}}(E)$},
\end{eqnarray}
we shall show that
\begin{eqnarray}\label{AKp-2}
{\mathcal{D}}_0(E)\,\,\mbox{ is dense in $H^p(E,\rho|_E,\sigma)$.}
\end{eqnarray}
Indeed, since finite linear spans of $p$-atoms are dense in $H^p(E,\rho|_E,\sigma)$, 
the density result formulated in 
\eqref{AKp-2} will follow once we show that individual $p$-atoms may be approximated  
in $H^p(E,\rho|_E,\sigma)$ with functions from ${\mathcal{D}}_0(E)$. To prove the 
latter, recall the approximation to the identity of order $\gamma$ as given in
Proposition~\ref{Besov-ST} and observe that from the properties of the integral 
kernels from Definition~\ref{Besov-S} we have 
that ${\mathcal{S}}_l\,a\in{\mathcal{D}}_0(E)$ for every $p$-atom $a$ and  each 
$l\in{\mathbb{N}}$. This and \cite[Lemma~3.2, (iii), p.\,108]{HYZ}, which gives that
\begin{eqnarray}\label{BG-IF}
\begin{array}{c}
\{{\mathcal{S}}_l\}_{l\in{\mathbb{N}}}\,\,\mbox{is uniformly bounded from $H^p(E,\rho|_E,\sigma)$ to $H^p(E,\rho|_E,\sigma)$}\,\,\mbox{ and}
\\[4pt]
{\mathcal{S}}_l f\to f\,\,\mbox{ in $H^p(E,\rho|_E,\sigma)$ as $l\to +\infty$},
\quad\forall\,f\in H^p(E,\rho|_E,\sigma),
\end{array}
\end{eqnarray}
now yield the desired conclusion, finishing the proof of \eqref{AKp-2}.

The next task is to prove that there exists $C\in(0,\infty)$ such that 
\begin{eqnarray}\label{BB-ZZ}
\|\delta_E^{\upsilon-m/q}\Theta f\|_{L^{(p,q)}({\mathscr{X}},E,\mu,\sigma;\kappa)}
\leq C\|f\|_{H^p(E,\rho|_E,\sigma)},\quad\forall\,f\in{\mathcal{D}}_0(E).
\end{eqnarray}
Assume for the moment \eqref{BB-ZZ}. Then, it follows that the linear operator
$\delta_E^{\upsilon-m/q}\Theta$ is bounded from ${\mathcal{D}}_0(E)$ into
$L^{(p,q)}({\mathscr{X}},E,\mu,\sigma;\kappa)$. Based on this, \eqref{AKp-2} and
the fact hat the mixed-norm spaces $L^{(p,q)}({\mathscr{X}},E,\mu,\sigma;\kappa)$ 
are quasi-Banach (see \cite{MMMZ}, \cite{BMMM}), it follows that 
$\delta_E^{\upsilon-m/q}\Theta$ extends in a standard way to a linear operator from
$H^p(E,\rho|_E,\sigma)$ into $L^{(p,q)}({\mathscr{X}},E,\mu,\sigma;\kappa)$. Since the
latter spaces are only quasi-normed, to show that this extension is also bounded 
we use the following property of quasi-normed spaces 
(for a proof see \cite[Theorem~1.5, (6)]{MMMM-G})
\begin{eqnarray}\label{Q-Nor}
\begin{array}{c}
\mbox{if $(X,\|\cdot\|)$ is a quasi-normed vector space, then 
$\exists\,C\in[1,\infty)$ such that}
\\[6pt] 
\mbox{if $x_j\to x_\ast$ in $X$ as $j\to\infty$, in the topology induced on 
$X$ by $\|\cdot\|$, then}
\\[6pt]
C^{-1}\|x_\ast\|\leq\liminf\limits_{j\to\infty}\|x_j\|
\leq \limsup\limits_{j\to\infty}\|x_j\|\leq C\|x_\ast\|.
\end{array}
\end{eqnarray}
In summary, the boundedness claimed in \eqref{Ar-EEE.2+} follows, once \eqref{BB-ZZ} is 
proved. 

With the goal of establishing \eqref{BB-ZZ}, fix a function
$f\in\dot{\mathscr{C}}_{b,0}^\gamma(E,\rho|_E)$. 
By \cite[Proposition~3.1, p.\,112]{HYZ}, we have that
\begin{eqnarray}\label{HH-FD}
\begin{array}{c}
\exists\,(\lambda_j)_{j\in{\mathbb{N}}}\in \ell^p,\quad
\exists\,(a_j)_{j\in{\mathbb{N}}}\,\,\mbox{$p$-atoms, such that }\,\,
\Bigl(\sum\limits_{j=1}^\infty|\lambda_j|^p\Bigr)^{1/p}\leq C\|f\|_{H^p(E,\rho|_E,\sigma)}
\\[4pt]
\mbox{and }\,\,f=\sum\limits_{j=1}^\infty\lambda_ja_j\,\,
\mbox{ both in $H^p(E,\rho|_E,\sigma)$ and in $L^q(E,\sigma)$},
\end{array}
\end{eqnarray}
for some $C\in(0,\infty)$ independent of $f$. Also, from our assumption \eqref{Ar-EEE} 
we deduce that
\begin{eqnarray}\label{Ar-EEE-S}
\delta_E^{\upsilon-m/q}\Theta:L^q(E,\sigma)\to 
L^{(q,q)}({\mathscr{X}},E,\mu,\sigma;\kappa)\quad\mbox{is linear and bounded}.
\end{eqnarray}
Combining \eqref{Ar-EEE-S} with \eqref{HH-FD} it follows that, with $f$ as above,
\begin{eqnarray}\label{HH-FD-A}
\delta_E^{\upsilon-m/q}\Theta f & = & 
\delta_E^{\upsilon-m/q}\Theta\Bigl(\lim_{N\to\infty}\sum\limits_{j=1}^N\lambda_ja_j\Bigr)
=\lim_{N\to\infty}\delta_E^{\upsilon-m/q}\Theta\Bigl(\sum\limits_{j=1}^N\lambda_ja_j\Bigr)
\nonumber\\[4pt]
& = & \lim_{N\to\infty}\sum\limits_{j=1}^N\lambda_j\delta_E^{\upsilon-m/q}\Theta a_j
\quad\mbox{in }\,\,L^{(q,q)}({\mathscr{X}},E,\mu,\sigma;\kappa).
\end{eqnarray}
%
Granted this, we may apply \cite[Theorem~1.5]{MMMZ} to conclude that 
\begin{eqnarray}\label{PP}
\begin{array}{c}
\exists\,(N_k)_{k\in{\mathbb{N}}},\,\,N_k\nearrow +\infty\,\,\mbox{ as }\,\,k\to +\infty,
\,\,\mbox{ such that}
\\[4pt] 
\mbox{$\sum\limits_{j=1}^{N_k}\lambda_j\delta_E^{\upsilon-m/q}\Theta a_j\to 
\delta_E^{\upsilon-m/q}\Theta f$ pointwise
$\mu$-a.e. on ${\mathscr{X}}\setminus E$ as $k\to +\infty$}.
\end{array}
\end{eqnarray}
Since we are currently assuming that $0<p\leq 1\leq q<\infty$, an inspection 
of definition \eqref{Mixed-EEW} of the quasi-norm for the space 
$L^{(p,q)}({\mathscr{X}},E,\mu,\sigma;\kappa)$ reveals that
$\|\cdot\|_{L^{(p,q)}({\mathscr{X}},E,\mu,\sigma;\kappa)}^p$ is subadditive. 
As such, for each $k\in{\mathbb{N}}$, we may estimate
\begin{eqnarray}\label{DF-ar}
\Bigl\|\sum\limits_{j=1}^{N_k}\lambda_j\delta_E^{\upsilon-m/q}\Theta a_j
\Bigr\|^p_{L^{(p,q)}({\mathscr{X}},E,\mu,\sigma;\kappa)}
& \leq & \sum\limits_{j=1}^{N_k}\Bigl\|\lambda_j\delta_E^{\upsilon-m/q}\Theta a_j
\Bigr\|^p_{L^{(p,q)}({\mathscr{X}},E,\mu,\sigma;\kappa)}
\nonumber\\[4pt]
& = & \sum\limits_{j=1}^{N_k}|\lambda_j|^p\Bigl\|\delta_E^{\upsilon-m/q}\Theta a_j
\Bigr\|^p_{L^{(p,q)}({\mathscr{X}},E,\mu,\sigma;\kappa)}
\nonumber\\[4pt]
& \leq & C\sum\limits_{j=1}^{N_k}|\lambda_j|^p,
\end{eqnarray}
where for the last inequality in \eqref{DF-ar} we used estimate \eqref{hnv}
(note that the assumptions in Case~4 are currently satisfied). Next,
introduce 
\begin{eqnarray}\label{DF-arZZ}
F_k:=\sum\limits_{j=1}^{N_k}\lambda_j\delta_E^{\upsilon-m/q}\Theta a_j,\qquad
\forall\,k\in{\mathbb{N}}.
\end{eqnarray}
To proceed, observe that Fatou's lemma holds in the 
space $L^{(p,q)}({\mathscr{X}},E,\mu,\sigma;\kappa)$ 
(this is seen directly from \eqref{Mixed-EEW} by applying twice the classical 
Fatou's lemma in Lebesgue spaces). When used for the 
sequence $\{F_k\}_{k\in{\mathbb{N}}}$, this yields
\begin{eqnarray}\label{DF-ar-2}
\|\delta_E^{\upsilon-m/q}\Theta f\|_{L^{(p,q)}({\mathscr{X}},E,\mu,\sigma;\kappa)}
& = & \bigl\|\liminf_{k\to\infty} |F_k|\bigr\|_{L^{(p,q)}({\mathscr{X}},E,\mu,\sigma;\kappa)}
\nonumber\\[4pt]
&\leq &\liminf_{k\to\infty}\|F_k\|_{L^{(p,q)}({\mathscr{X}},E,\mu,\sigma;\kappa)}
\leq C \|f\|_{H^p(E,\rho|_E,\sigma)}.\quad
\end{eqnarray}
The equality in \eqref{DF-ar-2} is a consequence of \eqref{PP} and the fact that
\begin{eqnarray}\label{Hkuc}
\|u\|_{L^{(p,q)}({\mathscr{X}},E,\mu,\sigma;\kappa)}
=\|\,|u|\,\|_{L^{(p,q)}({\mathscr{X}},E,\mu,\sigma;\kappa)},\quad
\forall\,u\in L^{(p,q)}({\mathscr{X}},E,\mu,\sigma;\kappa), 
\end{eqnarray}
the first inequality is due to Fatou's Lemma and \eqref{Hkuc}, 
while the last inequality follows from \eqref{DF-ar}, \eqref{DF-arZZ} and \eqref{HH-FD}.

At this stage, we have established \eqref{DF-ar-2} for any function
$f\in\dot{\mathscr{C}}_{b,0}^\gamma(E,\rho|_E)$, so in order to finish the proof of 
\eqref{BB-ZZ} there remains to consider the case when $E$ is bounded and
$f={\mathbf{1}}_E$. In this setting, since $E$ is $d$-dimensional ADR, 
we have $\sigma(E)<\infty$, and we may write
\begin{eqnarray}\label{DF-ar-3}
\bigl\|\delta_E^{\upsilon-m/q}\Theta {\mathbf{1}}_E
\bigr\|_{L^{(p,q)}({\mathscr{X}},E,\mu,\sigma;\kappa)}
& = & \bigl\|{\mathscr{A}}_{q,\kappa}\circ\bigl(\delta_E^{\upsilon-m/q}
\Theta{\mathbf{1}}_E\bigr)\bigr\|_{L^p(E,\sigma)}
\nonumber\\[4pt]
& \leq & \sigma(E)^{\frac{1}{p}-\frac{1}{q}}
\bigl\|{\mathscr{A}}_{q,\kappa}\circ\bigl(\delta_E^{\upsilon-m/q}
\Theta{\mathbf{1}}_E\bigr)\bigr\|_{L^q(E,\sigma)}
\nonumber\\[4pt]
& \leq & C\sigma(E)^{\frac{1}{p}}=C<+\infty.
\end{eqnarray}
The first inequality in \eqref{DF-ar-3} uses H\"older's inequality for the integrability
index $q/p\geq 1$, while the second inequality uses \eqref{Ar-EEE}. Now \eqref{BB-ZZ}
follows by combining \eqref{DF-ar-2} and \eqref{DF-ar-3}, and with it the proof of
\eqref{Ar-EEE.2+} is finished. In particular, \eqref{BB-ZZ} may be rewritten as 
\begin{eqnarray}\label{Ar-EEE.2}
\begin{array}{c}
{\mathscr{A}}_{q,\kappa}\circ\bigl(\delta_E^{\upsilon-m/q}\Theta\bigr):
H^p(E,\rho|_E,\sigma)\to L^p(E,\sigma)\quad\mbox{is bounded}
\\[4pt]
\mbox{whenever $\tfrac{d}{d+\gamma}<p\leq 1\leq q<+\infty$, and \eqref{Ar-EEE} holds}.
\end{array}
\end{eqnarray}

The end-game in the proof of part $(I)$ in the statement of the 
theorem is now as follows. Assume first that $q\in(1,\infty)$, $p_o\in(q,\infty)$ 
are such that \eqref{dtbjHT} holds for every $f\in L^{p_o}(E,\sigma)$. 
Based on these assumptions and Case~1 we conclude that 
\begin{eqnarray}\label{Ar-EEE.A}
{\mathscr{A}}_{q,\kappa}\circ\bigl(\delta_E^{\upsilon-m/q}\Theta\bigr):
L^p(E,\sigma)\to L^p(E,\sigma)\quad\mbox{is bounded whenever $p\in (1,p_o)$}.
\end{eqnarray}
In particular, \eqref{Ar-EEE.A} with $p:=q\in(1,p_o)$ and \eqref{Ar-EEE.2} yield 
\begin{eqnarray}\label{Ar-EEE.Ai}
{\mathscr{A}}_{q,\kappa}\circ\bigl(\delta_E^{\upsilon-m/q}\Theta\bigr):
H^p(E,\rho|_{E},\sigma)\to L^p(E,\sigma)\quad\mbox{is bounded if 
$\tfrac{d}{d+\gamma}<p\leq 1$}.
\end{eqnarray}
To proceed, fix an exponent 
\begin{eqnarray}\label{Ar-EEE.3}
p_o'\in(q,p_o).
\end{eqnarray}
Then \eqref{Ar-EEE.A} corresponding to $p:=p_o'$ together with 
Case~2 used here with $p_o$ replaced by $p_o'$ imply that
\begin{eqnarray}\label{Ar-EEE.B}
{\mathscr{A}}_{q,\kappa}\circ\bigl(\delta_E^{\upsilon-m/q}\Theta\bigr):
L^p(E,\sigma)\to L^p(E,\sigma)\quad\mbox{is bounded for each $p\in (q,\infty)$}.
\end{eqnarray}
Now the claim in part $(I)$ in the statement of the theorem corresponding 
to the current working hypotheses follows from \eqref{Ar-EEE.A}, 
\eqref{Ar-EEE.Ai}, and \eqref{Ar-EEE.B}.

There remains to consider the situation when $q\in(1,\infty)$ is such that  
\eqref{kt-Dc-BIS} holds. From Case~3 we know that \eqref{kt-Dc} is valid in 
the range $p\in[q,\infty)$. Then in combination with  Case~1
(used with $p_o:=q\in(1,\infty)$), this gives that 
\begin{eqnarray}\label{Ar-EEE.A-BIS}
{\mathscr{A}}_{q,\kappa}\circ\bigl(\delta_E^{\upsilon-m/q}\Theta\bigr):
L^p(E,\sigma)\to L^p(E,\sigma)\quad\mbox{is bounded for each $p\in (1,q]$}.
\end{eqnarray}
Now reasoning as before we obtain that \eqref{Ar-EEE.Ai} holds. In summary, the 
above analysis shows that \eqref{kt-Dc} is valid in the range
$p\in\bigl(\tfrac{d}{d+\gamma},\infty\bigr)$, under the assumption that 
$q\in(1,\infty)$ is such that \eqref{kt-Dc-BIS} holds.
This concludes the proof of part $(I)$, and finishes the proof of the theorem. 
\end{proof}

The second main result in this subsection is a combination of 
Theorem~\ref{VGds-L2XXX} and Theorem~\ref{VGds-2}.

\begin{theorem}\label{VGds-2.33} 
Suppose that $d,m$ are real numbers such that $0<d<m$. 
Assume that $({\mathscr{X}},\rho,\mu)$ is an $m$-dimensional {\rm ADR} space, 
$E$ is a closed subset of $({\mathscr{X}},\tau_\rho)$, and $\sigma$ is a 
Borel regular measure on $(E,\tau_{\rho|_{E}})$ with the property 
that $(E,\rho\bigl|_E,\sigma)$ is a $d$-dimensional {\rm ADR} space. 
In addition, suppose that $\Theta$ is the integral operator defined in \eqref{operator} 
with a kernel $\theta$ as in \eqref{K234}, \eqref{hszz}, \eqref{hszz-3noalpha}. 
Finally, fix $\kappa>0$ and recall the exponent $\gamma$ from \eqref{WQ-tDD}.

If there exist $p_o\in(0,\infty)$ and a finite constant $C_o>0$ such that for every 
$f\in L^{p_o}(E,\sigma)$ 
\begin{eqnarray}\label{GvBh}
\hskip -0.30in
\sup_{\lambda>0}\left[\lambda\cdot
\sigma\Bigl(\Bigl\{x\in E:\int_{\Gamma_{\kappa}(x)}|(\Theta f)(y)|^2\,
\frac{d\mu(y)}{\delta_E(y)^{m-2\upsilon}}>\lambda^{2}\Bigr\}\Bigr)^{1/p_o}\right]
\leq C_o\|f\|_{L^{p_o}(E,\sigma)},
\end{eqnarray}
then for each $p\in\bigl(\frac{d}{d+\gamma},\infty\bigr)$ there holds
\begin{eqnarray}\label{Cfrd}
\hskip -0.20in
\left\|\Bigl(\int_{\Gamma_{\kappa}(x)}|(\Theta f)(y)|^2\,
\frac{d\mu(y)}{\delta_E(y)^{m-2\upsilon}}\Bigr)^{\frac{1}{2}}
\right\|_{L^p_x(E,\sigma)}\!\!\!\leq C\|f\|_{H^p(E,\rho|_{E},\sigma)},\quad\forall\,f\in H^p(E,\rho|_{E},\sigma),
\end{eqnarray}
where $C>0$ is a finite constant which is allowed to depend only on 
$p,C_o,\kappa,C_\theta$, and geometry. 
\end{theorem}

\begin{proof}
The assumption that the operator ${\mathcal{A}}_{2,\kappa}\circ(\delta_E^{\upsilon-m/2}\Theta):
L^{p_o}(E,\sigma)\rightarrow L^{p_o,\infty}(E,\sigma)$ is bounded
implies that \eqref{dtbh-L2iii} holds. Consequently, Theorem~\ref{VGds-L2XXX} applies 
and yields that ${\mathcal{A}}_{2,\kappa}\circ(\delta_E^{\upsilon-m/2}\Theta):
L^2(E,\sigma)\rightarrow L^2(E,\sigma)$ is bounded as well. With this in hand, 
part $(I)$ in Theorem~\ref{VGds-2} (pertaining to condition \eqref{kt-Dc-BIS} 
with $q=2$) applies and gives that \eqref{Cfrd} holds for every 
$p\in\bigl(\frac{d}{d+\gamma},\infty\bigr)$.
\end{proof}

\section{Conclusion}\label{Sect:6}
\setcounter{equation}{0}

Theorem~\ref{M-TTHH} asserts the equivalence of a number of the properties encountered
in the body of the manuscript. A formal proof is presented below.

\vskip 0.08in
\begin{proof}[Proof of Theorem~\ref{M-TTHH}]
The fact that {\it (1)} $\Rightarrow$ {\it (2)} is a consequence of Theorem~\ref{SChg}.
It is easy to see that if {\it (2)} holds, then {\it (7)} holds by taking
$b_Q:={\mathbf{1}}_Q$ for each $Q\in{\mathbb{D}}(E)$, hence 
{\it (2)} $\Rightarrow$ {\it (7)}.
The implication {\it (7)} $\Rightarrow$ {\it (1)} is proved in Theorem~\ref{Thm:localTb}.
The implication {\it (9)} $\Rightarrow$ {\it (1)} is proved in Theorem~\ref{Thm:BPSFtoSF}.
Moreover, {\it (1)} $\Leftrightarrow$ {\it (9)} $\Leftrightarrow$ {\it (10)}
by Theorem~\ref{Thm:BPSFtoSF.XXX}.
The implication {\it (11)} $\Rightarrow$ {\it (12)} is proved in Theorem~\ref{VGds-2.33}.
Clearly {\it (12)} $\Rightarrow$ {\it (11)}, while {\it (11)} $\Rightarrow$ {\it (1)}
is contained in Theorem~\ref{VGds-L2XXX}. To show that {\it (1)} $\Rightarrow$ {\it (11)},
suppose {\it (1)} holds and take $f\in L^2(E,\sigma)$ and $\lambda>0$ arbitrary. Then
starting with Tschebyshev's inequality we may write
\begin{eqnarray}\label{dSHV}
&& \hskip -0.40in
\lambda^2\cdot\sigma\Bigl(\Bigl\{x\in E:\int_{\Gamma_{\kappa}(x)}
|(\Theta f)(y)|^2\,\frac{d\mu(y)}{\delta_E(y)^{m-2\upsilon}}>\lambda^{2}\Bigr\}\Bigr)
\nonumber\\[4pt]
&& \hskip 0.30in
\leq \int_{\Bigl\{x\in E:\int_{\Gamma_{\kappa}(x)}
|(\Theta f)(y)|^2\,\frac{d\mu(y)}{\delta_E(y)^{m-2\upsilon}}>\lambda^{2}\Bigr\}}\Bigl(
\int_{\Gamma_{\kappa}(x)}|(\Theta f)(y)|^2\,\frac{d\mu(y)}{\delta_E(y)^{m-2\upsilon}}
\Bigr)d\sigma(x)
\nonumber\\[4pt]
&& \hskip 0.30in
\leq \int_E\Bigl(\int_{\Gamma_{\kappa}(x)}
|(\Theta f)(y)|^2\,\frac{d\mu(y)}{\delta_E(y)^{m-2\upsilon}}\Bigr)d\sigma(x)
\nonumber\\[4pt]
&& \hskip 0.30in
\leq \int_{{\mathscr{X}}\setminus E}\frac{|(\Theta f)(y)|^2}{\delta_E(y)^{m-2\upsilon}}
\,\sigma(\pi_y^\kappa)\,d\mu(y)
\nonumber\\[4pt]
&& \hskip 0.30in
\leq \int_{{\mathscr{X}}\setminus E}\frac{|(\Theta f)(y)|^2}{\delta_E(y)^{m-2\upsilon}}
\,\sigma\Bigl(E\cap B_{\rho_{\#}}\bigl(y_\ast,C_{\rho}(1+\kappa)\delta_E(y)\bigr)\Bigr)
\,d\mu(y)
\nonumber\\[4pt]
&& \hskip 0.30in
\leq C\int_{{\mathscr{X}}\setminus E}|(\Theta f)(y)|^2\delta_E(y)^{2\upsilon-(m-d)}
\,d\mu(y)
\nonumber\\[4pt]
&& \hskip 0.30in
\leq C\|f\|_{L^2(E,\sigma)}^2.
\end{eqnarray}
The third inequality in \eqref{dSHV} is due to \eqref{Mix+FR} (recall \eqref{reg-A2}),
the fourth uses \eqref{sgbr} in Lemma~\ref{lbDV}, the fifth uses the fact that 
$\bigl(E,\rho\bigl|_E,\sigma\bigr)$ is a $d$-dimensional {\rm ADR} space,
and the last inequality is a consequence of \eqref{G-UFXXX.2}. Thus,  
{\it (1)} $\Rightarrow$ {\it (11)} as desired. 
Since \eqref{ki-DUDXXX} is a rewriting of \eqref{CfrdXXX}, it is immediate that 
{\it (12)} $\Leftrightarrow$ {\it (13)}. In summary, so far we have shown that 
{\it (1)}, {\it (2)}, {\it (7)}, {\it (9)}, {\it (10)}, {\it (11)}, {\it (12)}, 
and {\it (13)} are equivalent.

The implication {\it (6)} $\Rightarrow$ {\it (4)} is trivial and, based on 
\eqref{dFvK}, we have that {\it (4)} $\Rightarrow$ {\it (2)}. We focus next on 
{\it (1)} $\Rightarrow$ {\it (6)}. Suppose {\it (1)} holds and fix
$f\in L^\infty(E,\sigma)$, $x\in E$, and $r\in(0,\infty)$ arbitrary. 
Then, using the notation $B_{cr}:=B_{\rho_{\#}}(x,cr)$ for $c>0$, we may write
\begin{eqnarray}\label{esL}
\int_{B_r\setminus E}|\Theta f|^2\delta_E^{2\upsilon-(m-d)}\,d\mu
& \leq & \int_{B_r\setminus E}
|\Theta (f{\mathbf{1}}_{E\cap B_{2rC_\rho}})|^2\delta_E^{2\upsilon-(m-d)}\,d\mu
\nonumber\\[4pt]
&& 
+\int_{B_r\setminus E}|\Theta (f{\mathbf{1}}_{E\setminus B_{2rC_\rho}})|^2
\delta_E^{2\upsilon-(m-d)}\,d\mu=:I+II.
\end{eqnarray}
To estimate $I$ we apply \eqref{G-UFXXX.2} and the property of $E$ being $d$-dimensional
ADR to obtain
\begin{eqnarray}\label{esL-2}
I\leq C\int_{E\cap B_{2rC_\rho}}|f|^2\,d\sigma
\leq C\|f\|^2_{L^\infty(E,\sigma)}\sigma(E\cap B_{r}).
\end{eqnarray}
As regards $II$, we first note that if $z\in B_r\setminus E$ and 
$y\in E\setminus B_{2rC_\rho}$ are arbitrary points then $\rho_{\#}(x,y)
\leq C_\rho(\rho_{\#}(x,z)+\rho_{\#}(z,y))< C_\rho r+C_\rho\rho_{\#}(z,y)
\leq \tfrac{1}{2}\rho_{\#}(x,y)+C_\rho\rho_{\#}(z,y)$ which implies 
$\rho_{\#}(z,y)\geq \rho_{\#}(x,y)/(2C_\rho)$. This, \eqref{hszz-AXXX}, 
and \eqref{WBA} then yield
\begin{eqnarray}\label{esL-3}
|\Theta (f{\mathbf{1}}_{E\setminus B_{2rC_\rho}})(z)|
&\leq & C\|f\|_{L^\infty(E,\sigma)}
\int_{E\setminus B_{2rC_\rho}}\frac{1}{\rho_{\#}(x,y)^{d+\upsilon}}\,d\sigma(y)
\nonumber\\[4pt]
&\leq & C\|f\|_{L^\infty(E,\sigma)}\,r^{-\upsilon},
\qquad\forall\,z\in B_r\setminus E.
\end{eqnarray}
Using this last estimate in $II$ and applying \eqref{lbzF} (with $R:=r$ 
and $\gamma:=m-d-2\upsilon$) we obtain
\begin{eqnarray}\label{esL-4}
II &\leq & C\|f\|^2_{L^\infty(E,\sigma)}r^{-2\upsilon}
\int_{B_r\setminus E}\delta_E^{2\upsilon-(m-d)}\,d\mu
\nonumber\\[4pt]
& \leq & C\|f\|^2_{L^\infty(E,\sigma)}r^{-2\upsilon}r^{d+2\upsilon}
\leq C\|f\|^2_{L^\infty(E,\sigma)}\sigma(E\cap B_{r}).
\end{eqnarray}
At this point, \eqref{UEHgXXX.2S} follows from \eqref{esL}, \eqref{esL-2}, and 
\eqref{esL-4}, completing the proof of {\it (1)} $\Rightarrow$ {\it (6)}. 
Based on \eqref{dFvK} we have that {\it (6)} $\Rightarrow$ {\it (3)} while 
{\it (3)} $\Rightarrow$ {\it (2)} is trivial.

Next, we shall show that {\it (8)} $\Rightarrow$ {\it (7)}. To this end, 
suppose {\it (8)} holds and let $\varepsilon_o:=\min\{\varepsilon,a_0\}$,
where $\varepsilon$ is as in Lemma~\ref{b:SV} and $a_0$ as in \eqref{ha-GVV}. 
Fix an arbitrary $Q\in{\mathbb{D}}(E)$ and define 
$\Delta_Q:=B_{\rho_{\#}}\Bigl(x_Q,\tfrac{\varepsilon_o\ell(Q)}{2C\rho}\Bigr)\cap E$.
Then \eqref{ha-GVV}, \eqref{zjrh} and the fact that $E$ is $d$-dimensional ADR
imply 
\begin{eqnarray}\label{VCV}
\Delta_Q\subseteq Q,\quad
B_{\rho_{\#}}\bigl(x_Q,\varepsilon_o\ell(Q)\bigr)\setminus E\subseteq T_E(Q),
\quad\sigma(\Delta_Q)\approx\sigma(Q)=C\ell(Q)^d.
\end{eqnarray} 
Hence, if we now define $b_Q:=b_{\Delta_Q}$, where $b_{\Delta_Q}$ is the function 
associated to $\Delta_Q$ as in {\it (8)}, then $b_{\Delta_Q}$ satisfies
\eqref{CON-BB.789} which, when combined with the support condition of $b_{\Delta_Q}$
and the last condition in \eqref{VCV}, implies that $b_Q$ satisfies the first two 
conditions in \eqref{CON-BB} (with $\widetilde{Q}=Q$). In order to show that $b_Q$ 
also verifies the last condition in \eqref{CON-BB}, we write
\begin{eqnarray}\label{DB-N}
&& \hskip -0.70in
\int_{T_E(Q)}|(\Theta\,b_Q)(x)|^2\delta_E(x)^{2\upsilon-(m-d)}\,d\mu(x)
\nonumber\\[4pt]
&& =\int_{T_E(Q)\setminus B_{\rho_{\#}}\bigl(x_Q,\varepsilon_o\ell(Q)\bigr)}
|(\Theta\,b_Q)(x)|^2\delta_E(x)^{2\upsilon-(m-d)}\,d\mu(x)
\nonumber\\[4pt]
&& \quad+\int_{B_{\rho_{\#}}\bigl(x_Q,\varepsilon_o\ell(Q)\bigr)}|(\Theta\,b_Q)(x)|^2
\delta_E(x)^{2\upsilon-(m-d)}\,d\mu(x)=:I_1+I_2.
\end{eqnarray}
To further estimate $I_2$, observe that if 
$x\in T_E(Q)\setminus B_{\rho_{\#}}\bigl(x_Q,\varepsilon_o\ell(Q)\bigr)$ and $y\in\Delta_Q$,
then $\rho_{\#}(x,y)\geq\tfrac{\varepsilon_o}{2C_\rho}\ell(Q)$. This, \eqref{hszz-AXXX},
the first estimate in \eqref{CON-BB}, and the last condition in \eqref{VCV}, imply $|(\Theta\,b_Q)(x)|\leq C\ell(Q)^{-\upsilon}$ for every 
$x\in T_E(Q)\setminus B_{\rho_{\#}}\bigl(x_Q,\varepsilon_o\ell(Q)\bigr)$. Hence,
if we also recall \eqref{dFvK} we have
\begin{eqnarray}\label{DB-N-2}
I_1 &\leq & 
C \ell(Q)^{-2\upsilon}
\int_{T_E(Q)\setminus B_{\rho_{\#}}\bigl(x_Q,\varepsilon_o\ell(Q)\bigr)}
\delta_E(x)^{2\upsilon-(m-d)}\,d\mu(x)
\nonumber\\[4pt]
&\leq & 
C \ell(Q)^{-2\upsilon}
\int_{B_{\rho_{\#}}\bigl(x_Q,C\ell(Q)\bigr)\setminus E}
\delta_E(x)^{2\upsilon-(m-d)}\,d\mu(x)
\nonumber\\[4pt]
&\leq & 
C \ell(Q)^{-2\upsilon}\ell(Q)^{d+2\upsilon}\leq C\sigma(Q),
\end{eqnarray}
where the third inequality in \eqref{DB-N-2} is a consequence of \eqref{lbzF} 
(applied with $R=r=\ell(Q)$ and $\gamma:=m-d-2\upsilon$). As for $I_2$, by recalling
\eqref{CON-BB.789} and the last condition in \eqref{VCV}, it is immediate that 
$I_2\leq C_0\sigma(\Delta_Q)\leq C\sigma(Q)$. This, \eqref{DB-N} and \eqref{DB-N-2}
show that $b_Q$ also satisfies the last condition in \eqref{CON-BB} since the constants
in our estimates are finite positive geometric, and independent of the choice of $Q$. This completes the
proof of  {\it (8)} $\Rightarrow$ {\it (7)}. 

It is not difficult to see that {\it (1)} $\Rightarrow$ {\it (8)}. Indeed, 
by taking $b_\Delta:={\mathbf{1}}_\Delta$ for each surface ball $\Delta$, the first two
estimates in \eqref{CON-BB.789} are immediate while the third one is a consequence of
\eqref{G-UFXXX.2} written for $f:=b_\Delta$.

Trivially, {\it (2)} $\Rightarrow$ {\it (5)}. If we assume that {\it (5)} holds and
for each $Q\in{\mathbb{D}}(E)$ we set $b_Q:=b{\mathbf{1}}_Q$, then it is easy to verify
based on \eqref{UEHgXXX.2PA} and the fact that $b$ is para-accretive that \eqref{CON-BB}
is satisfied by the family $\{b_Q\}_{Q\in{\mathbb{D}}(E)}$. 
Hence, {\it (5)} $\Rightarrow ${\it (7)}. The proof of Theorem~\ref{M-TTHH} 
is therefore complete.
\end{proof}

In the last part of this section we present the 

\vskip 0.08in
\begin{proof}[Proof of Theorem~\ref{sfe-cor}]
The idea is to apply Theorem~\ref{VGds-2} in the setting 
${\mathscr{X}}:=E\times[0,\infty)$ and $E\equiv E\times\{0\}$ (i.e., we identify
$(y,0)\equiv y$ for every $y\in E$). Moreover, we let
\begin{eqnarray}\label{GCvca-LL1}
&& \rho((x,t),(y,s)):=\max\{|x-y|,|t-s|\}\,\,
\mbox{ for every }\,\,(x,t),(y,s)\in E\times[0,\infty),
\\[4pt]
&& \mu:=\sigma\otimes{\mathcal{L}}^1, 
\label{GCvca-LL2}
\end{eqnarray}
where ${\mathcal{L}}^1$ is the one-dimensional Lebesgue measure on $[0,\infty)$, 
and consider the integral kernel
\begin{eqnarray}\label{K-spc}
\begin{array}{l}
\theta:({\mathscr{X}}\setminus E)\times E\to{\mathbb{R}}
\\[4pt]
\theta\bigl((x,t),y\bigr):=2^{-k}\psi_k(x-y)\,\,\mbox{ if }\,\,x,y\in E, t>0
\,\,\mbox{ and }\,\,k\in{\mathbb{Z}},\,\,2^k\leq t<2^{k+1}.
\end{array}
\end{eqnarray}
Also, we let $\Theta$ be the integral operator defined in \eqref{operator}
corresponding to this choice of $\theta$. Then it is not difficult to verify 
that $({\mathscr{X}},\rho,\mu)$ is a $(d+1)$-ADR space, that $\alpha_\rho=1$, 
that $\theta$ satisfies \eqref{K234}-\eqref{hszz-3} for 
$a:=0$, $\alpha:=1$, $\upsilon:=1$, and that $\delta_E(x,t)=t$ for every 
$x\in E$ and $t\in[0,\infty)$. In particular, $\gamma$ defined in \eqref{WQ-tDD}
now equals $1$. Fix some $\kappa>0$ and observe that 
\begin{eqnarray}\label{agv}
\Gamma_\kappa(x)=\bigl\{(y,t)\in E\times(0,\infty):\,
|x-y|<(1+\kappa)t\bigr\},\quad\forall\,x\in E.
\end{eqnarray}
In this context, for $f\in L^2(E,\sigma)$, we consider the square
of the term in the left hand-side of \eqref{kt-Dc-BIS} 
corresponding to $p=q=2$ and use Fubini's Theorem, the property 
that $E$ is $d$-ADR, and \eqref{K-spc} to write
\begin{eqnarray}\label{sBBv}
&& \hskip -0.20in
\left\|\Bigl(\int_{\Gamma_{\kappa}(x)}|(\Theta f)(y,t)|^2\,
\frac{d\mu(y,t)}{\delta_E(y,t)^{d-1}}\Bigr)^{\frac{1}{2}}
\right\|_{L^2_x(E,\sigma)}^2
=\int_E\int_{\Gamma_{\kappa}(x)}|(\Theta f)(y,t)|^2\,t^{1-d}d\mu(y,t)
\,d\sigma(x)
\nonumber\\[4pt]
&&\quad =\int_{E\times(0,\infty)}|(\Theta f)(y,t)|^2\,t^{1-d}
\sigma\bigl(\{x\in E:\,y\in\Gamma_\kappa(x)\}\bigr)\,d\mu(y,t)
\nonumber\\[4pt]
&&\quad =\int_{E\times(0,\infty)}|(\Theta f)(y,t)|^2\,t^{1-d}
\sigma\bigl(E\cap B(y,(1+\kappa)t)\bigr)\,d\mu(y,t)
\nonumber\\[4pt]
&&\quad \approx C\int_0^\infty\int_E|(\Theta f)(y,t)|^2\,t\,d\sigma(y)\,dt
\nonumber\\[4pt]
&&\quad =C\sum\limits_{k=-\infty}^{+\infty}\int_{2^k}^{2^{k+1}}
\int_E\Bigl|\int_E2^{-k}\psi_k(y-z)f(z)\,d\sigma(z)\Bigr|^2 
d\sigma(y)\,t\,dt
\nonumber\\[4pt]
&&\quad =C\sum\limits_{k=-\infty}^{+\infty}
\int_E\Bigl|\int_E\psi_k(y-z)f(z)\,d\sigma(z)\Bigr|^2 d\sigma(y).
\end{eqnarray}
However, under the current assumptions on $E$, it was proved in 
\cite[Theorem, p.\,10]{DaSe91} that there exists $C\in(0,\infty)$ 
with the property that 
\begin{eqnarray}\label{PP-Zs}
\sum\limits_{k=-\infty}^{+\infty}\int_E
\Bigl|\int_E\psi_k(x-y)f(y)\,d\sigma(y)\Bigr|^2\,d\sigma(x)
\leq C\int_E|f|^2\,d\sigma,\quad\forall\,f\in L^2(E,\sigma).
\end{eqnarray}
 Hence, we may apply
Theorem~\ref{VGds-2} to conclude that there exists $C\in(0,\infty)$ such that
estimate \eqref{kt-Dc} is valid for every $q\in(1,\infty)$ and every
$p\in\bigl(\frac{d}{d+1},\infty\bigr)$. In turn, reasoning as in \eqref{sBBv}, 
estimate \eqref{kt-Dc} may be rewritten in the form 
\begin{eqnarray}\label{Df-Ch}
\left\|\Bigl(\sum\limits_{k=-\infty}^{+\infty}
\meanint_{y\in\Delta(x,(1+\kappa)2^k)}\Bigl|\int_E\psi_k(z-y)f(z)\,d\sigma(z)
\Bigr|^q\,d\sigma(y)\Bigr)^{1/q}
\right\|_{L^p_x(E,\sigma)}\leq C'\|f\|_{H^p(E,\sigma)}
\end{eqnarray}
for every $f\in H^p(E,\sigma)$ and some $C'\in(0,\infty)$ independent
of $f$. The desired conclusion now follows by observing that if $q\in(1,\infty)$ 
and $p\in\bigl(\frac{d}{d+1},\infty\bigr)$ are fixed, then there exists 
some $C\in(0,\infty)$ such that \eqref{bh} holds for every $f\in H^p(E,\sigma)$ 
if and only if there exist $\kappa,C'\in(0,\infty)$ such that estimate 
\eqref{Df-Ch} holds for every $f\in H^p(E,\sigma)$. Indeed, one direction 
is obvious, while the opposite one may be handled by observing that if 
$\psi\in C^\infty_0({\mathbb{R}}^{n+1})$ is odd then 
$\widetilde{\psi}(x):=\psi(x/2^N)$, for some fixed sufficiently large 
$N\in{\mathbb{N}}$, is also odd, smooth and compactly supported, and satisfies
$\widetilde{\psi}_k=2^{dN}\psi_{k+N}$ for every $k\in{\mathbb{Z}}$. 
Writing \eqref{Df-Ch} for $\widetilde{\psi}$ in place of $\psi$ and 
shifting the index of summation in the left-hand side, the desired conclusion follows.
\end{proof}


\begin{thebibliography}{999}
\small

\bibitem{AMM} R.\,Alvarado, I.\,Mitrea, and M.\,Mitrea, {\it Whitney-Type extensions 
in geometrically doubling quasi-metric spaces}, Communications in 
Pure and Applied Analysis, 12 (2013), no.\,1, 59--88.

\bibitem{Au} P.\,Auscher, {\it Lectures on the Kato square root problem}, 
Surveys in analysis and operator theory (Canberra, 2001), Proc. Centre Math. 
Appl. Austral. Nat. Univ. 40, Austral. Nat. Univ., Canberra, 2002, 1--18.

\bibitem{AHLMcT} P.\,Auscher, S.\,Hofmann, M.\,Lacey, A.\,McIntosh, and 
P.\,Tchamitchian, {\it The solution of the Kato Square Root Problem for second order elliptic operators on ${\mathbb{R}}^n$}, Annals of Math., 156 (2002), 633--654.

\bibitem{AHLT} P.\,Auscher, S.\,Hofmann, J.L.\,Lewis and P.\,Tchamitchian, 
{\it Extrapolation of Carleson measures and the analyticity of Kato's square
root operators}, Acta Math., 187 (2001), no.\,2, 161--190.

\bibitem{AHMTT} P.\,Auscher, S.\,Hofmann, C.\,Muscalu, T.\,Tao, and 
C.\,Thiele, {\it Carleson measures, trees, extrapolation, and $T(b)$ 
theorems}, Publ. Mat., 46 (2002), no.\,2, 257--325.

\bibitem{AR} P.\,Auscher and E.\,Routin, {\it Local $TB$ theorems and Hardy 
inequalities}, to appear in Journal of Geometric Analysis, (2012).

\bibitem{AY} P. Auscher and Q. X. Yang, On local $T(b)$ Theorems, 
{\it Publ. Math.} {\bf 53} (2009), 179-196.

\bibitem{AS} J.\,Azzam and R.\,Schul, {\it Hard Sard: Quantitative implicit 
function and extension theorems for Lipschitz maps}, arXiv:1105.4198v3, (2012).

\bibitem{BMMM} D.\,Brigham, D.\,Mitrea, I.\,Mitrea, and M.\,Mitrea, 
{\it Triebel-Lizorkin sequence spaces are genuine mixed-norm spaces}, 
Math. Nachr., 1--15 (2012) / DOI 10.1002/mana.201100184.

\bibitem{BGS} D.L.\,Burkholder, R.F.\,Gundy and M.\,Silverstein, 
{\it Distribution function inequalities for the area integral}, 
Studia Math., 44 (1972), 527--544.

\bibitem{Cal50} A.P.\,Calder\'on, {\it On a theorem of Marcinkiewicz and Zygmund}, 
Trans. Amer. Math. Soc., 68 (1950), 55--61.

\bibitem{Cal65} A.P.\,Calder\'on, {\it Commutators of singular integral operators},
Proc. Nat. Acad. Sci. U.S.A., 53 (1965), 1092-1099.
 
\bibitem{Ch} M.\,Christ, {\it Lectures on Singular Integral Operators}, 
CBMS Series in Math., Vol.\,77, Amer. Math. Soc., 1990.

\bibitem{Christ} M.\,Christ, {\it A $T(b)$ theorem with remarks on analytic 
capacity and the {C}auchy integral}, Colloq. Math., 60/61 (1990), no.\,2, 601--628.

\bibitem{CJ} M.\,Christ and J.-L.\,Journ\'e, {\it Polynomial growth estimates for
multilinear singular integral operators}, Acta Math., 4 (1988), 219--225.

\bibitem{CoMeSt} R.\,Coifman, Y.\,Meyer, and E.M.\,Stein, {\it Some new function spaces
and their applications to Harmonic Analysis}, Journal of Functional Analysis, 
62 (1985), 304--335.

\bibitem{CM} R.\,Coifman and Y.\,Meyer, {\it Non-linear harmonic analysis and P.D.E.}, 
pp.\,3--45 in ``Beijing Lectures in Harmonic Analysis," E.M. Stein, editor, 
Annals of Math. Studies, Vol.\,112, Princeton Univ. Press, 1986.

\bibitem{CoWe71} R.R.\,Coifman and G.\,Weiss, {\it Analyse Harmonique
Non-Commutative sur Certains Espaces Homog\`enes}, Lecture Notes in Mathematics,
Vol.\,242, Springer-Verlag, 1971.

\bibitem{CoWe77} R.R.\,Coifman and G.\,Weiss, {\it Extensions of Hardy spaces 
and their use in analysis}, Bull. Amer. Math. Soc., 83 (1977), no.\,4, 569--645.

\bibitem{Dah80} B.E.J.\,Dahlberg, {\it Weighted norm inequalities for the Lusin 
area integral and the nontangential maximal function for functions harmonic 
in a Lipschitz domain}, Studia Math., 67 (1980), 297--314.

\bibitem{DJK} B.\,Dahlberg, D.\,Jerison and C.\,Kenig, {\it Area integral estimates 
for elliptic differential operators with nonsmooth coefficients}, Ark. Mat., 22 
(1984), no. 1, 97--108.

\bibitem{David1988} G.\,David, {\it Morceaux de graphes lipschitziens et int\'egrales
singuli\`eres sur une surface}, Rev. Mat. Iberoamericana, 4 (1988), no.\,1, 73--114.

\bibitem{DJ84} G.\,David and J.L.\,Journ\'e, {\it A boundedness criterion for 
generalized Calder\'on-Zygmund operators}, Ann. of Math. (2), 120 (1984), no.\,2, 
371--397.

\bibitem{DJS} G.\,David, J.L.\,Journ\'e and S.\,Semmes, {\it Op\'erateurs de
Calder\'on-Zygmund, fonction para-accretive et interpolation}, Rev. Mat. 
Iberoamericana, 1 (1985), 1--56.

\bibitem{DaSe91} G.\,David and S.\,Semmes, {\it Singular Integrals and Rectifiable 
Sets in ${\mathbb{R}}^n$: Beyond Lipschitz Graphs}, Ast\'erisque, No.\,193, 1991.

\bibitem{DaSe93} G.\,David and S.\,Semmes, {\it Analysis of and on Uniformly 
Rectifiable Sets}, Mathematical Surveys and Monographs, AMS Series, 1993.

\bibitem{DeHa09} D.\,Deng and Y.\,Han, {\it Harmonic Analysis on Spaces of
Homogeneous Type}, Lecture Notes in Mathematics, Vol.\,1966, Springer-Verlag, 2009. 

\bibitem{Fa88} E.\,Fabes, {\it Layer potential methods for boundary value problems 
on Lipschitz domains}, pp.\,55-80 in ``Potential Theory, Surveys and Problems,"
J. Kral et al. eds., Lecture Notes in Math., Vol.\,1344, Springer-Verlag, New York, 1988.

\bibitem{Feff74} C.\,Fefferman, {\it Recent progress in classical Fourier analysis}, 
Proceedings of the International Congress of Mathematicians, Vancouver, 1974, 
pp.\,95--118.

\bibitem{FeSt72} C.\,Fefferman and E.M.\,Stein, {\it $H^p$ spaces of several variables},
Acta Math., 129 (1972), no.\,1, 137--193.

\bibitem{F} T.\,Figiel, {\it Singular integral operators: a martingale approach},
pp.\,95--110 in ``Geometry of Banach spaces" (Strobl, 1989), P.F.X.\,M\"uller and
W.\,Schachermayer eds., London Math. Soc. Lecture Note, Ser.\,158, Cambridge Univ. 
Press, Cambridge, 1990.  

\bibitem{G} A.\,Grau de la Herran, Ph.D. Thesis, University of Missouri, 2012.

\bibitem{GM} A.\,Grau de la Herran and M.\,Mourgoglou, {\it A $Tb$ theorem for 
square functions in domains with Ahlfors-David regular boundaries}, preprint, (2012). 

\bibitem{HMY} Y.\,Han, D.\,M\"uller, and D.\,Yang, {\it Littlewood-Paley
characterizations for Hardy spaces on spaces of homogeneous type}, Mathematische Nachrichten, 279 (2006), no.\,13-14, 1505--1537.

\bibitem{HaSa94} Y.\,Han and E.\,Sawyer, {\it Littlewood-Paley Theory on Spaces
of Homogeneous Type and the Classical Function Spaces}, Memoirs of the Amer. Math. 
Soc., Vol.\,530, 1994.

\bibitem{Heil} C.\,Heil, {\it A Basis Theory Primer}, Applied and Numerical Harmonic
Analysis, Expanded Edition, Birkh\"auser, 2011.

\bibitem{Ho} S.\,Hofmann, {\it Parabolic singular integrals of Calder\'on-type, 
rough operators and caloric layer potentials}, Duke Math. J., 90 (1997), 209--260.

\bibitem{Ho2} S.\,Hofmann, {\it A proof of the local $Tb$ Theorem for standard Calder\'on-Zygmund operators}, unpublished manuscript,
http://www.math.missouri.edu/$\sim$hofmann/

\bibitem{Ho3} S.\,Hofmann, {\it Local $Tb$ Theorems and applications in PDE}, 
Proceedings of the ICM Madrid, Vol.\,II, pp.\,1375--1392, European Math. Soc., 2006.

\bibitem{Ho4} S.\,Hofmann, {\it A local $Tb$ theorem for square functions}, 
pp.\,175--185 in ``Perspectives in Partial Differential Equations, Harmonic Analysis 
and Applications: A Volume in Honor of Vladimir G. Maz'ya's 70th Birthday,"
D. Mitrea and M. Mitrea eds., Proc. Sympos. Pure Math., Vol.\,79, Amer. Math. Soc., 
Providence, RI, 2008.

\bibitem{HLMc} S.\,Hofmann, M.\,Lacey and A.\,McIntosh, {\it The solution of the 
Kato problem for divergence form elliptic operators with Gaussian heat kernel bounds},
Annals of Math., 156 (2002), 623--631.

\bibitem{HL} S.\,Hofmann and J.L.\,Lewis, {\it Square functions of Calder\'on 
type and applications}, Rev. Mat. Iberoamericana, 17 (2001), 1--20.

\bibitem{HM} S.\,Hofmann and J.M.\,Martell, {\it Uniform rectifiability 
and harmonic measure I}, preprint.

\bibitem{HMU} S.\,Hofmann, J.M.\,Martell and I.\,Uriarte-Tuero, {\it Uniform
rectifiability and harmonic measure II: Poisson kernels in $L^p$ imply 
uniform rectifiability}, preprint.

\bibitem{HMc} S.\,Hofmann and A.\,McIntosh, {\it The solution of the Kato problem 
in two dimensions}, pp.\,143--160 in ``Proceedings of the Conference on Harmonic 
Analysis and PDE" (El Escorial, 2000), Publ. Mat., Vol. extra, 2002.

\bibitem{HMc2} S.\,Hofmann and A.\,McIntosh, {\it Boundedness and applications of singular integrals and square functions: a survey}, Bull. Math. Sci., DOI 10.1007/s13373-011-0014-3.

\bibitem{HMT} S.\,Hofmann, M.\,Mitrea and M.\,Taylor {\it Geometric and 
transformational properties of Lipschitz domains, Semmes-Kenig-Toro domains, 
and other classes of finite perimeter domains}, Journal of Geometric Analysis, 
17 (2007), no.\,4, 593--647.

\bibitem{HYZ} G.\,Hu, D.\,Yang and Y.\,Zhou, {\it Boundedness of singular integrals in
Hardy spaces on spaces of homogeneous type}, Taiwanese Journal of Mathematics, 13
(2009), no.\,1, 91--135.

\bibitem{Hy} T.\,Hyt\"onen, {\it An operator-valued $Tb$ theorem},
J. Funct. Anal., 234 (2006), 420--463.

\bibitem{HyM1} T.\,Hyt\"onen and H.\,Martikainen, {\it Non-homogeneous $Tb$
theorem and random dyadic cubes on metric measure spaces}, preprint.

\bibitem{HyM2} T.\,Hyt\"onen and H.\,Martikainen, {\it On general local $Tb$ theorems},
preprint.

\bibitem{HyWe} T.\,Hyt\"onen and L.\,Weis, {\it A $T1$ Theorem for integral 
transformations with operator-valued kernel}, J. Reine Angew. Math., 599 (2006), 
155--200. 

\bibitem{JeKe82} D.S.\,Jerison and C.E.\,Kenig, {\it Hardy spaces, $A_\infty$, 
and singular integrals on chord-arc domains}, Math. Scand., 50 (1982), 221--247.

\bibitem{J} P.W.\,Jones, {\it Square functions, Cauchy integrals, analytic capacity, 
and harmonic measure}, in ``Harmonic Analysis and Partial Differential Equations", 
(El Escorial, 1987), pp.\,24--68, Lecture Notes in Math., 1384, Springer, Berlin, 1989.

\bibitem{Ke} C.E.\,Kenig, {\it Weighted $H^p$ spaces on Lipschitz domains}, 
American Journal of Mathematics, 102 (1980), no.\,1, 129--163.   

\bibitem{MaSe79} R.A.\,Mac\'{i}as and C.\,Segovia, {\it Lipschitz functions 
on spaces of homogeneous type}, Adv. in Math., 33 (1979), 257--270. 

\bibitem{MaSe79II} R.A.\,Mac\'{\i}as and C.\,Segovia, {\it A decomposition into
atoms of distributions on spaces of homogeneous type}, Adv. in Math., 33 (1979), 
no.\,3, 271--309.

\bibitem{McM85} A.\,McIntosh and Y.\,Meyer, {\it Alg\`ebres d'op\'erateurs d\'efinis 
par des int\'egrales singuli\`eres}, C. R. Acad. Sci. Paris, S\'erie 1, 301 (1985),
395--397.

\bibitem{MMM} D.\,Mitrea, I.\,Mitrea, and M.\,Mitrea, {\it Weighted mixed-normed
spaces on spaces of homogeneous type}, preprint, (2012).

\bibitem{MMMM-G} D.\,Mitrea, I.\,Mitrea, M.\,Mitrea and S.\,Monniaux,
{\it Groupoid Metrization Theory with Applications to Analysis on Quasi-Metric Spaces 
and Functional Analysis}, Birkh\"auser, 2012.

\bibitem{MMMM-B} D.\,Mitrea, I.\,Mitrea and M.\,Mitrea,
{\it A Treatise on the Theory of Elliptic Boundary Value Problems, 
Singular Integral Operators, and Smoothness Spaces in Rough Domains}, 
book manuscript, 2011.

\bibitem{MMMZ} D.\,Mitrea, I.\,Mitrea, M.\,Mitrea and E.\,Ziad\'e,  
{\it Abstract capacitary estimates and the completeness 
and separability of certain classes of non-locally convex topological vector spaces},
Dorina Mitrea, Irina Mitrea, Marius Mitrea, Elia Ziad\'e, Journal of Functional Analysis,
Vol. 262 (2012), 4766--4830.

\bibitem{MMT} D.\,Mitrea, M.\,Mitrea and M.\,Taylor, {\it Layer potentials, 
the Hodge Laplacian and global boundary problems in nonsmooth Riemannian manifolds}, 
Memoirs of AMS, Vol.\,150, No.\,713, 2001.

\bibitem{M-LNM} M.\,Mitrea, {\it Clifford Wavelets, Singular Integrals, 
and Hardy Spaces}, Springer-Verlag Lecture Notes in Mathematics, No.\,1575, 
Berlin/\,Heidelberg/\,New York, 1994.

\bibitem{M-DAH} M.\,Mitrea, {\it On Dahlberg's Lusin area integral theorem},  
Proc. Amer. Math. Soc., 123 (1995), no.\,5, 1449--1455.

\bibitem{NTV} F.\,Nazarov, S.\,Treil and A.\,Volberg, {\it Accretive system 
$Tb$-theorems on nonhomogeneous spaces}, Duke Math. J., 113 (2002), no.\,2, 259--312.

\bibitem{Se69} C.\,Segovia, {\it On the area function of Lusin}, 
Studia Math., 33 (1969), 312--343.

\bibitem{Sem90} S.\,Semmes, {\it Square function estimates and the $T(b)$ Theorem}, 
Proc. Amer. Math. Soc., 110 (1990), no.\,3, 721--726.

\bibitem{St70} E.M.\,Stein, {\it Singular Integrals and Differentiability
Properties of Functions}, Princeton Mathematical Series, No.\,30,
Princeton University Press, Princeton, NJ, 1970.

\bibitem{St82} E.M.\,Stein, {\it The development of square functions in the work of 
A. Zygmund}, Bull. Amer. Math. Soc. (N.S.), 7 (1982), no.\,2, 359--376.

\bibitem{STEIN} E.M.\,Stein, {\it Harmonic Analysis: Real-Variable Methods, 
Orthogonality, and Oscillatory Integrals}, Princeton Mathematical Series, 
Vol.\,43, Monographs in Harmonic Analysis, III, Princeton University Press, 
Princeton, NJ, 1993. 

\bibitem{STEIN-WEISS} E.M.\,Stein and G. Weiss, {\it Introduction to Fourier
Analysis on Euclidean Spaces}, Princeton University Press, 
Princeton, NJ, 1971. 

\bibitem{Tor86} A.\,Torchinsky, {\it Real-Variable Methods in Harmonic Analysis}, 
Dover Publications, Inc., Mineola, New York, 2004.


\end{thebibliography}
\end{document}